\documentclass[12pt]{article}
\usepackage{amsmath,amssymb,latexsym, amsfonts, amscd, amsthm}
\usepackage{graphicx,epsfig}

 \DeclareMathOperator{\gal}{Gal}

\theoremstyle{plain}
\newtheorem{theorem}{Theorem}[section]

\newtheorem{conjecture}[theorem]{Conjecture}

\newtheorem{proposition}[theorem]{Proposition}
\newtheorem{corollary}[theorem]{Corollary}

\newtheorem{lemma}[theorem]{Lemma}

\newtheorem{example}[theorem]{Example}

\theoremstyle{definition}
\newtheorem{definition}[theorem]{Definition}
\newtheorem{remark}[theorem]{Remark}

\newcommand{\ds}{\displaystyle}

\newcommand{\lra}{\longrightarrow}
\newcommand{\ra}{\rightarrow}

\newcommand{\tensor}{{\otimes}}
\newcommand{\Hom}{\mbox{Hom}}
\newcommand{\uHom}{\mbox{\underline{Hom}}}

\newcommand{\Spec}{\mbox{Spec}}
\newcommand{\Spf}{\mbox{Spf}}

\newcommand{\Ker}{\mbox{Ker}}
\newcommand{\Fil}{\mbox{\rm Fil}}

\newcommand{\Coker}{\mbox{Coker}}
\newcommand{\cont}{{\rm cont}}

\newcommand{\tT}{{\widetilde T}}

\newcommand{\GL}{{\rm \mathbf{GL}}}

\newcommand{\Z}{{\mathbb Z}}
\newcommand{\Q}{{\mathbb Q}}
\newcommand{\C}{{\mathbb C}}
\newcommand{\R}{{\mathbb R}}
\newcommand{\F}{{\mathbb F}}
\newcommand{\N}{{\mathbb N}}
\newcommand{\V}{{\mathbb V}}

\newcommand{\A}{{\underline{A}}}

\newcommand{\Kbar}{{\overline{K}}}

\newcommand{\Xbar}{{\overline{X}}}
\newcommand{\Bbar}{{\overline{B}}}
\newcommand{\Ubar}{{\overline{V}}}
\newcommand{\kbar}{{\overline{k}}}
\newcommand{\OKbar}{{\cO_\Kbar}}
\newcommand{\Mun}{M_0}
\newcommand{\OMun}{\cO_{M_0}}
\newcommand{\bareta}{{\overline{\eta}}}
\newcommand{\Rbar}{{\overline{R}}}
\newcommand{\hR}{\widehat{{\overline{R}}}}

\newcommand{\hatcO}{\widehat{{\overline{\cO}}}}

\newcommand{\bDcrisgeo}{\bD_{\rm cris}^{\rm geo}}

\newcommand{\bDcrisM}{\bD_{\rm cris,M}}

\newcommand{\bVcrisarQ}{\bV_{\rm cris}^{\rm ar}}
\newcommand{\bDcrisar}{\bD_{\rm cris}^{\rm ar}}

\newcommand{\X}{{\underline{X}}}

\newcommand{\cL}{{\mathbb L}}

\newcommand{\cH}{{\cal H}}
\newcommand{\cI}{{\cal I}}
\newcommand{\cJ}{{\cal J}}
\newcommand{\cU}{{\cal U}}
\newcommand{\cP}{{\cal P}}

\newcommand{\cF}{{\cal F}}
\newcommand{\cG}{{\cal G}}
\newcommand{\cE}{{\cal E}}
\newcommand{\cV}{{\cal V}}

\newcommand{\bD}{{\mathbb D}}
\newcommand{\bM}{{\mathbb M}}
\newcommand{\bV}{{\mathbb V}}
\newcommand{\bH}{{\mathbb H}}
\newcommand{\bA}{{\mathbb A}}
\newcommand{\bB}{{\mathbb B}}

\newcommand{\bK}{{\mathbb K}}

\newcommand{\cO}{{\cal O}}
\newcommand{\cA}{{\cal A}}

\newcommand{\Sh}{{\mbox{Sh}}}
\newcommand{\Mod}{{\mbox{Mod}}}

\newcommand{\Isoc}{{\mbox{Isoc}}}
\newcommand{\rig}{{\mbox{rig}}}
\newcommand{\fet}{{\mbox{fet}}}

\newcommand{\Rep}{{\mbox{Rep}}}

\newcommand{\II}{{\mathbb I}}

\newcommand{\bL}{{\mathbb L}}
\newcommand{\WW}{{\mathbb W}}

\newcommand{\cR}{\underline{\cal{R}}}
\newcommand{\crR}{{\cal{R}}}

\newcommand{\cW}{{\cal W}}

\newcommand{\fX}{{\mathfrak{X}}}

\newcommand{\fU}{{\mathfrak{U}}}
\newcommand{\fUM}{{\mathfrak{U}^\bullet_M}}
\newcommand{\fXKbar}{{\mathfrak{X}_\Kbar}}

\newcommand{\fp}{{\mathfrak{p}}}

\newcommand{\cM}{{\cal M}}
\newcommand{\cN}{{\cal N}}

\begin{document}

\title{Comparison Isomorphisms for Smooth Formal Schemes}
\author{Fabrizio Andreatta \\ Adrian Iovita\\} \maketitle

\tableofcontents \pagebreak
\section{Introduction}
\label{sec:intro}

Let  $p>0$ denote a prime integer and $K$ a complete discrete
valuation field of characteristic~$0$ and perfect residue field
$k$ of characteristic~$p$. This article proposes a new point of
view on the problem of comparison isomorphisms for algebraic
varieties over~$K$ which allows the extension of the results to
smooth $p$-adic formal schemes over the ring of integers of~$K$.

\noindent More precisely, we suppose throughout the introduction that $K$ is absolutely unramified i.~e., that the prime $p$ is a uniformizer of $K$, denote by
$\Kbar$ a fixed algebraic closure of it with ring of integers $\cO_{\Kbar}$ and by $A_{\rm cris}$ and $B_{\rm cris}$ the period rings defined by J.-M.~Fontaine in
\cite{Fontaineperiodes} (see section \S \ref{sec:Notation} for a review of the construction). We set $G_K:=\gal(\Kbar/K)$.

The ``comparison isomorphism problem" was first alluded to by Grothendieck as the existence of 
a {\it mysterious functor} associating, for a given algebraic variety
over $K$, the \'etale cohomology groups of the variety over $\Kbar$ to the 
respective de Rham cohomology groups. This was precisely formulated  by J.-M. Fontaine in
\cite{fontaine} as {\it the crystalline comparison conjecture}. Let $X$ be a smooth 
proper scheme over $\cO_K$. If $R$ is an $\cO_K$-algebra we denote by $X_R$ the
base change $X_R=X\times_{\Spec(\cO_K)}\Spec(R)$. In particular we'll denote by $\Xbar$ the 
special fiber $X_k$ of $X$.

\begin{conjecture}[\cite{fontaine}]
\label{conj:cris} In the notations above for every $i\ge 0$  there
is a canonical and functorial isomorphism commuting with all the
additional structures (namely filtrations, $G_K$--actions and
Frobenii)
$$
{\rm H}^i(X_{\Kbar}^{\rm et}, \Q_p)\otimes_{\Q_p}B_{\rm cris}\cong
{\rm H}^i_{\rm cris}(\Xbar/\cO_K)\otimes_{\cO_K}B_{\rm cris}.
$$
\end{conjecture}

The first case of the conjecture was proved by Fontaine himself in
\cite{fontaine} for  abelian schemes. There followed other proofs
of the conjecture for abelian varieties and curves over $K$ with
good or semi-stable reduction in preprints of Fontaine-Messing (unpublished),
\cite{colmez}, \cite{cristante_candilera}.

\noindent The first general result was proved by Bloch and Kato in
\cite{bloch_kato} for proper and smooth schemes with ordinary
reduction. The next breakthrough was due to Fontaine and Messing
in \cite{fontainemessing}. They noticed that the ring $A_{\rm
cris}$ has a geometric interpretation as global sections of a
certain sheaf on the crystalline site of $\Xbar$ and thus the syntomic cohomology
on $X$ calculates the
right hand side of the isomorphism in the conjecture
\ref{conj:cris}. The fact that the syntomic cohomology can be
related to the left hand side (i.~e., to \'etale cohomology of
$X_{\Kbar}$) is  more complicated and the conjecture was
proved in \cite{fontainemessing} under the assumption that
 $1+2\mbox{dim}(X_K)<p.$

\noindent G.~Faltings fully proved the conjecture in
\cite{faltingscrystalline}, in fact he proved more: one can drop
the assumption that $K$ is absolutely unramified and he allowed
certain non-trivial coefficients, more precisely
$\Q_p$--adic local systems $\bL$ on $X_K$ for which there exist
``associated $F$-isocrystals" $\cE$ on $X_k$ (see
\cite{faltingscrystalline}). Faltings's strategy was to define a new
cohomology theory associated to $X$ and to prove that it calculated
both the left hand side (via the theory of almost \'etale
extensions) and the right hand side of conjecture \ref{conj:cris}.
Let us be more precise: we  suppose that $X$ is geometrically
connected and denote by $\eta=\Spec(\bK)$ a geometric generic
point of $X_\Kbar$. Let $X_\bullet\lra X$ be an \'etale
hyper-covering of $X$ by ``small affine schemes over $\cO_K$". If
$X_i=\Spec(R_i)$ ($i$ is a multi-index) let $\Rbar_i$ be the
maximal normal extension of $R_i$ in $\bK$ such that
$\Rbar_i\bigl[p^{-1}\bigr]$ is the union of finite and \'etale
extensions of~$R_i\bigl[p^{-1}\bigr]$. By $B_{\rm
cris}^\nabla(\Rbar_i)$ we denote the relative Fontaine ring $B_{\rm
cris}$ constructed using the pair $(R_i, \Rbar_i)$
(in \cite{faltingscrystalline} this ring is
denoted $B_{\rm cris}(R_i)$). If $\Delta_i=
\pi_1(X_{i,\Kbar},\eta)$ consider the double complex
$K^{\bullet\bullet}(\bL):=C^\bullet(\Delta_\bullet,
\bL_\eta\otimes B_{\rm cris}^\nabla(\Rbar_\bullet))$, where
$C^\bullet(\Delta, - )$ is the standard chain complex
computing continuous group cohomology of $\Delta$. The new
cohomology with coefficients in $\bL$ defined by Faltings is the cohomology of the
total complex  $K^{\bullet\bullet}(\bL)$ (or rather the limit of
such over all hyper-coverings.)

\noindent Faltings' cohomology theory seemes easier to handle than
the syntomic cohomology but there are two inconveniences related
to it. One is conceptual: the association $U=\Spec(R)\lra B_{\rm
cris}^\nabla(\Rbar)$ is not a sheaf and no geometric
interpretation of the above defined cohomology theory is given in
terms of sheaf cohomology. The second inconvenience is that in
order to prove the isomorphism of the new cohomology theory of
$\bL$ with the crystalline cohomology of the associated
$F$-isocrystal tensored with $B_{\rm cris}$ one had to prove that
this new cohomology theory satisfied Poincar\'e duality compatible
with the known Poincar\'e dualities on the two sides of
\ref{conj:cris}. This complicates the proof of the crystalline
comparison conjecture and limits the applications of these ideas
to proper schemes, or complements of a normal crossing divisor in
a proper scheme.

\noindent To finish our history of ``comparison isomorphisms", K.~Kato in \cite{kato} 
adapted the proof in \cite{fontainemessing} to schemes over $K$ with
semi-stable reduction by the systematic use of log structures (using  results in 
\cite{hyodo_kato}) and proved the ``semi-stable conjecture" for trivial
coefficients and under the assumption that $1+2\mbox{dim}(X_K)$ is less than $p$. 
Finally T.~Tsuji was able to circumvent the technical difficulties related to
syntomic and log syntomic cohomology and proved both the crystalline and semi-stable 
conjectures for trivial coefficients in \cite{tsuji}. Next, Faltings extended
his results to open varieties over $K$ with semi-stable reduction in 
\cite{faltingsAsterisque} and W.~Niziol re-proved in \cite{niziol} the crystalline and
semi-stable conjectures for trivial coefficients using a new idea namely a comparison 
isomorphism in $K$-theory. Recently Go Yamashita gave a new proof of the comparison
isomorphism for open varieties over $K$ with semi-stable reduction and trivial 
coefficients using syntomic methods (see \cite{yamashita}).

\bigskip
\noindent The new point of view introduced in the present article is the systematic use 
of a topos which we call  ``Faltings's topos" associated to $X$ and of
certain new (``ind-continuous'') sheaves of rings $\bB_{\rm cris}^\nabla$ and 
$\bB_{\rm cris}$ in it. Faltings's topos is the category of sheaves on a certain site
which we denote $\fX$ (for a more precise definition see the sections \ref{sec:description} 
and \ref{sec:Groth}). Despite the suggestion of the notations, the
sections of $\bB_{\rm cris}^\nabla$ are not the rings $B_{\rm cris}^\nabla(\Rbar)$ 
used by Faltings (for the precise relationship between the two see the next
section.) If $\bL$ is a $\Q_p$-local system on $X_K$ then $\bL$ may be viewed as a 
sheaf on $\fX$  and we have:

a) ${\rm H}^i(\fX, \bL\otimes_{\Q_p} \bB_{\rm cris}^\nabla)$ is
Faltings' cohomology  associated  to $\bL$ as above. Thus the
present theory gives a geometric interpretation of Faltings's
construction.

b) In this setting one may attach to $\bL$ in a geometric way a
sheaf on $X$ with a connection and thus define crystalline local
systems on $X_K^{\rm et}$ and their associated $F$-isocrystals.

c) We also provide a  new idea of the proof of the comparison
isomorphism. The main reason the comparison isomorphism in
the algebraic setting over $K$ fails to follow the classical
pattern over the complex numbers is because the de Rham complex of
sheaves of $X$
$$
\cO_X\stackrel{d_X}{\lra}\Omega^1_{X/\cO_K}\stackrel{d_X}{\lra}
\Omega^2_{X/\cO_K}\stackrel{d_X}{\lra}\cdots
$$
is not exact (i.e. there is no algebraic Poincar\'e lemma).
Let us just mention that even if we replace $X$ (or $X_K$) by the
rigid analytic variety $X^{\rm rig}$ associated to $X_K$ its de
Rham complex is still not exact. But if we now pass to the finer
topology $\fX$ and remark that the sheaf $\bB_{\rm cris}$ is a
sheaf of $\cO_X$-modules with a connection $\nabla$ such that
$\bB_{\rm cris}^\nabla$ is its sheaf of horizontal sections we
have an exact sequence of sheaves on $\fX$:
$$
0\lra \bB_{\rm cris}^\nabla\lra \bB_{\rm
cris}\stackrel{\nabla}{\lra} \bB_{\rm
cris}\otimes_{\cO_X}\Omega^1_{X/\cO_K}\stackrel{\nabla}{\lra}\bB_{\rm
cris}\otimes_{\cO_X}\Omega^2_{X/\cO_K}\stackrel{\nabla}{\lra}\cdots
$$
In other words working on the site $\fX$ and after tensoring with
the sheaf of periods $\bB_{\rm cris}$ the de Rham complex of $X$
becomes exact and a resolution of the sheaf $\bB_{\rm
cris}^\nabla$. Now an acyclicity property of this resolution (for
a more precise formulation see the next section)  permits the
calculation of the cohomology on $\fX$ of the sheaf  $\bL\otimes
\bB_{\rm cris}^\nabla$ as the hyper-cohomology of the de Rham
complex on $X$ of the associated $F$-isocrystal without any use of
Poincar\'e duality. Therefore these results extend to $p$-adic
formal schemes.  In the next sub-section we present a more precise
description of the article.

\subsection{Description of the paper}
\label{sec:description}

Let $X$ denote either a smooth scheme over $\cO_K$ or a smooth
$p$-adic formal scheme over $\cO_K$ with special fiber denoted
$\Xbar$. If $X$ is a scheme we denote its generic fiber by $X_K$
and if $X$ is a formal scheme we denote its rigid analytic generic
fiber $X^{\rm rig}$, also by $X_K$.

Let us define (in the algebraic setting) the category $E_{X_\Kbar}$ whose objects are pairs $(\cU,\cW)$ where $\cU\lra X$ is an \'etale morphism, $\cW\lra
\cU_\Kbar$ is a finite \'etale morphism and the morphisms between two pairs are pairs of morphisms satisfying natural compatibilities. G.~Faltings defined in
\cite{faltingsAsterisque} a certain topology on $E_{X_\Kbar}$ but as recently noticed by A.~Abbes a fundamental error occurred in that construction namely to put is
as directly as possible, the topos defined in \cite{faltingsAsterisque} is not the category of sheaves on the topology he defines there (see section \ref{sec:Groth}
for a counterexample.) A salvage was suggested in a letter of P.~Deligne to L.~Illusie  in 1995 raising a series of questions related to section 3.4 in
\cite{illusie}. He pointed out that the correct definition of Faltings's topos should 
be as a certain ``topos flech\'e". The general theory of these topoi was
further developed by O.~Gabber, L.~Illusie, F.~Orgogozo and was implemented in the case 
of interest to us by  A. Abbes. We chose not to follow this  direction and
fix the problem in an equivalent way here (see also \cite{erratum}) by defining a 
different pre-topology, which we call ${\rm PT}_{X_\Kbar}$ on the category
$E_{X_\Kbar}$. See section \ref{sec:Groth} for the precise statements and the definition 
in the formal context.

There is a fundamental operation on sheaves and continuous sheaves
of abelian groups on ${\rm PT}_{X_\Kbar}$ called {\bf
localization} and defined as follows. Suppose that $\cF$ is a
continuous sheaf on ${\rm PT}_{X_\Kbar}$ (let us recall that such
an object is a projective system $\{\cF_n\}_n$ of $p$-power
torsion sheaves on $\fX$) and let $\cU=\Spec(R_\cU)$ be an affine
connected object of $X^{\rm et}$ (by which we mean the \'etale
site on $X$). We fix a geometric generic point
$\eta=\Spec(\bK)$ and denote by
$$
\cF(\Rbar_\cU):=\lim_{\leftarrow,n}\cF_n(\Rbar_\cU)\mbox{ with }
\cF_n(\Rbar_\cU):= \lim_{\rightarrow,S} \cF_n(\cU,\Spec(S)),
$$
where in the inductive limit
 $S$ runs over all $R_\cU\otimes_{\cO_K}K$-sub algebras of $\bK$ which are finite and
\'etale. Let us remark that $\cF(\Rbar_\cU)$ is a continuous
representation of the fundamental group $\pi_1^{\rm
alg}(\cU_K,\eta)$.

We have natural functors $v\colon X^{\rm et}\lra \fX,
u\colon\fX\lra X_\Kbar^{\rm et}$ defined (in the algebraic
setting) as follows: $v(\cU)=(\cU,\cU_\Kbar)$ and respectively
$u(\cU,\cW)=\cW$. The functors $v^\ast$ and $u_\ast$ allow us to
view sheaves on $X^{\rm et}$ and respectively on $X_\Kbar^{\rm
et}$ as sheaves on $\fX$. The functor $v_\ast$ is left exact and its right derived
functors are central to our theory. We have the following result
proved in \cite{andreatta_iovita} describing these
functors in terms of localizations.

\begin{theorem}[\cite{andreatta_iovita}, theorem 6.12]
\label{thm:coh_loc} Let $\cF$ be a continuous sheaf on $\fX$
satisfying certain conditions (Assumptions 6.10 in
\cite{andreatta_iovita}). Then for all $i\ge 0$, $R^iv_\ast\cF$
are the sheaves associated to the pre-sheaves on $X^{\rm et}$
$$
\cU={\rm Spec}(R_\cU)\lra {\rm H}^i_{\rm cont}(\pi_1^{\rm
alg}(\cU_\Kbar, \ast), \cF(\Rbar_\cU)).
$$
\end{theorem}

This theorem reduces via the Leray spectral sequence the
calculation of the sheaf cohomology groups ${\rm H}^i(\fX, \cF)$
to local calculations of fundamental group cohomology with values
in localizations and sheaf cohomology on $X^{\rm et}$.

We now pass to the description  of the sheaves $\bB_{\rm
cris}^\nabla$ and  $\bB_{\rm cris}$ announced at the beginning of
this introduction. We  call these ``Fontaine sheaves" and prove
that they enjoy the following properties.\smallskip

\noindent a) Both are ind-continuous sheaves of $B_{\rm
cris}$-algebras on $\fX$ such that for a ``small" affine
$\cU=\Spec(R_\cU)$ the localizations $\bB_{\rm
cris}^\nabla(\Rbar_\cU)$ and $\bB_{\rm cris}(\Rbar_\cU)$ are
respectively isomorphic to the rings $B_{\rm
cris}^\nabla(\Rbar_\cU)$ and $B_{\rm cris}(\Rbar_\cU)$ defined in
\cite{faltingscrystalline} and \cite{brinon}.\smallskip

\noindent b) $\bB_{\rm cris}^\nabla$ is endowed with a filtration
${\rm Fil}^\bullet (\bB_{\rm cris}^\nabla)$ by sub-sheaves and a
Frobenius endomorphism.\smallskip

\noindent c) $\bB_{\rm cris}$ is a sheaf of
$\cO_X\otimes_{\cO_K}B_{\rm cris}$-algebras, is endowed with a
filtration ${\rm Fil}^\bullet(\bB_{\rm cris})$ by sub-sheaves, a
Frobenius endomorphism and a quasi-nilpotent and integrable
connection $\nabla$ such that

i) $\nabla$ satisfies the Griffith transversality property.

ii) $\bB_{\rm cris}^\nabla$ is exactly the sub-sheaf of $\bB_{\rm
cris}$ of horizontal sections for $\nabla$.\smallskip

\noindent d) For $i\ge 1$ we have  $R^iv_\ast\bB_{\rm
cris}=0$.
\smallskip

We  remark that  property d) above is a
deep result stating that the sheaf $\bB_{\rm cris}$ is acyclic for
the functor $v_\ast$. It is a consequence of theorem
\ref{thm:coh_loc} and the results in
\cite{andreatta_brinonacyclicity}. As we now have  Fontaine
sheaves on $\fX$ we can start developing a Fontaine theory with
sheaves. To start let $\bL$ denote a locally constant
$\Q_p$-sheaf on $X_K^{\rm et}$ which, let us recall, we view via
base change and $u_\ast$ as a sheaf on ${\rm PT}_{X_\Kbar}$. We
define $\bD_{\rm cris}^{\rm geo}(\bL):= v_\ast(\bL\otimes\bB_{\rm
cris}).$ It is a sheaf of $\cO_{X_K}\otimes_{K}B_{\rm
cris}$-modules on $X_K^{\rm et}$ endowed with a filtration,
Frobenius  endomorphism, quasi-nilpotent and integrable connection
and a continuous $G_K$-action. Here  $\cO_{X_K}$ denotes the
sheaf $\cO_X [p^{-1}]$ on $X$. We set $\bD_{\rm cris}^{\rm
ar}(\bL) :=\bigl(\bD_{\rm cris}^{\rm geo}(\bL)\bigr)^{G_K}$.

\begin{definition}
\label{def:crys} We say that $\bL$ is a {\bf crystalline sheaf} on
$X_K^{\rm et}$ if

$\bullet$ $\bD_{\rm cris}^{\rm ar}(\bL)$ is a coherent sheaf of
$\cO_{X_K}$-modules on $X_K^{\rm et}$.

$\bullet$ The natural morphism $\bD_{\rm cris}^{\rm
ar}(\bL)\otimes_{\cO_X} \bB_{\rm cris}\lra \bL\otimes_{\Q_p}
\bB_{\rm cris}$ is an isomorphism.
\end{definition}

The definition \ref{def:crys} is the sheaf theoretic analogue of the usual definition 
of crystalline representations in $p$-adic Hodge theory. We prove that it
coincides with the notions of ``locally crystalline representations" in the relative 
setting due to \cite{brinon} and with that of ``associated sheaves'' due to
Faltings. If $X$ is a formal scheme as at the beginning of this section and 
$\bL$ is a crystalline sheaf on $X_K^{\rm et}$ then $\bD_{\rm cris}^{\rm ar}(\bL)$ is a
filtered convergent $F$-isocrystal on $X_K^{\rm et}$ in the sense of \cite{berthelot} 
and $\bD_{\rm cris}^{\rm geo}(\bL)\cong \bD_{\rm cris}^{\rm ar}(\bL)\tensor_K
B_{\rm cris}$. We remark that in the recent preprint \cite{tsuji1} T.~Tsuji developed 
systematically a theory of crystalline \'etale local systems on schemes $X_K$
in the case where $X$ is the complement of a divisor with normal crossings in a proper 
formal scheme over $\cO_K$ with semi-stable special fiber and such that the
horizontal divisor has normal crossings also with the special fiber. The paper uses 
different methods and does not contain comparison isomorphisms. If $X$ is a
smooth formal scheme (and the horizontal divisor is trivial) our notion of a crystalline 
\'etale local system on $X_K$ coincides with the one in \cite{tsuji1}.

We can now
list the main result of this paper.

\begin{theorem}
\label{thm:formal_comp} Suppose that $X$ is a smooth $p$-adic formal scheme over $\cO_K$ and let $\bL$ be a crystalline sheaf on $X_K^{\rm et}$. For every $i\ge 0$
we have a natural isomorphism of $\delta$--functors with values in $B_{\rm cris}$-modules respecting the filtrations, Frobenii and the $G_K$-actions
$$
{\rm H}^i(\fX, \bL\otimes\bB_{\rm cris}^\nabla)\cong {\rm
H}^i_{\rm cris}(\Xbar, \bD_{\rm cris}^{\rm geo}(\bL)).
$$
\end{theorem}

\bigskip
\noindent The theorem \ref{thm:formal_comp} has two main
applications, one which is the comparison isomorphism for smooth
proper schemes over $\cO_K$ (theorem \ref{thm:alg_comp} below) and
the other which is an application to modular forms,
more precisely an overconvergent Eichler-Shimura isomorphism (see
\cite{andreatta_iovita_stevens}).

\begin{theorem}
\label{thm:alg_comp} Suppose that $X$ is a smooth proper scheme
over $\cO_K$ and $\bL$ is a crystalline sheaf on $X_K^{\rm et}$.
For every $i\ge 0$ we have a canonical isomorphism of
$\delta$--functors with values in $B_{\rm cris}$-modules, which
respects the filtrations, the Frobenii and the $G_K$-actions
$$
{\rm H}^i(X_\Kbar^{\rm et}, \bL)\otimes_{\Q_p} B_{\rm cris}\cong
{\rm H}^i_{\rm cris}(\Xbar, \bD_{\rm cris}^{\rm ar}(\bL))\otimes_K
B_{\rm cris}.
$$
\end{theorem}

The theorem \ref{thm:alg_comp} is a consequence of theorem
\ref{thm:formal_comp} and of the following two results:\smallskip

$\bullet$ If $X$ is a smooth, proper formal scheme over $\cO_K$
then we have an isomorphism of filtered, Frobenius modules: ${\rm
H}^i_{\rm cris}\bigl(\Xbar, \bD_{\rm cris}^{\rm
geo}(\cL)\bigr)\cong {\rm H}^i_{\rm cris}(\Xbar, \bD_{\rm
cris}^{\rm ar}(\cL))\otimes_{K}B_{\rm cris}$.\smallskip

\noindent and\smallskip

$\bullet$ The natural morphism of sheaves on $\fX$,
 $\bL\lra \bL\otimes \bB_{\rm cris}^\nabla$
induces for every $i\ge 0$ canonical isomorphisms as
$B_{\rm cris}$-modules respecting all the structure ${\rm
H}^i(X_{\Kbar}^{\rm et}, \cL)\otimes_{\Q_p}B_{\rm cris}\cong {\rm
H}^i\bigl(\fX, \bL\otimes_{\Z_p}\bB_{\rm
cris}^\nabla\bigr)$.\smallskip

This last isomorphism which is also proved in
\cite{faltingsAsterisque} as being one of the central and deep
results of that paper is in our theory  an
elementary consequence of theorem \ref{thm:formal_comp} and of the
criterion for ``admisibility'' of filtered, Frobenius modules
in \cite{colmez_fontaine}.

Let us remark that in lemma 3.14 we prove that $\bL$ is a
crystalline sheaf on $X_K^{\rm et}$ if and only if
$\bL$ and $\bD_{\rm cris}^{\rm ar}(\bL)$ are associated in
Faltings's sense. This and the fact that
$H^i\bigl(\fX, \bL\otimes\bB_{\rm cris}^\nabla\bigr)$ is naturally
isomorphic to the $i$-th cohomology group defined by
Faltings shows that the comparison isomorphism of theorem
\ref{thm:alg_comp} is the same as the one defined by Faltings, and
hence it is the same as all the other period maps defined in the literature.

\bigskip
\noindent Finally, in a future work we are planning to show
how to produce examples of crystalline sheaves and how to explicitly
calculate their $\bD_{\rm cris}^{\rm ar}$. More precisely let $X$
and $Y$ denote smooth $p$-adic formal schemes over $\cO_K$  and
suppose that $f\colon X\lra Y$ is a smooth proper morphism which
is algebrizable Zariski locally on $Y$. We believe that we
would be able to prove:

\begin{theorem}
\label{thm:relative_coh} Let us suppose that $\bM$ is a
crystalline sheaf on $X_K^{\rm et}$ and for every $i\ge 0$ let us
denote by $\bL_i:={\rm R}^if_{\rm et,\ast} \bM$. Then $\bL_i$ is a
crystalline sheaf on $Y_K^{\rm et}$ and we have an isomorphism
$\bD_{\rm cris,Y}^{\rm ar}(\bL_i)\cong \R^i f_{\rm cris,\ast}
(\bD_{\rm cris,X}^{\rm ar}(\bM)\otimes \Omega^\bullet_{X/Y})$ of
$\delta$--functors with values in the category of filtered
convergent $F$-isocrystals on $Y_k$.
\end{theorem}

\bigskip

The reader will remark throughout the paper the presence of an
auxiliary field $M$ which is an extension of $K$ contained in
$\Kbar$ and which indexes all the objects appearing: $\fX_M$,
$\bA_{\rm cris,M}^\nabla$, $\bA_{\rm cris,M}$ etc. If $M$ is a
finite extension of~$K$, this allows us to prove the theorems above
also for the base change of~$X$ to the ring of integers of~$M$.
Equivalently the above results are valid without assuming that
$K$ is absolutely unramified but under the hypothesis that~$X$
(and the morphism $f\colon X\to Y$ in \ref{thm:relative_coh}) is
defined over~$\WW(k)$. Since the notations become more complicated
and possibly obscure some of the simple ideas present in the
proofs we have chosen to sketch these ideas in the introduction
in the simplified assumption that $K$ is unramified.

We would also like to point out that the methods presented here
seem suitable for pursuing further inquiries into this problem.
Namely we have already worked out the comparison theorems for
schemes and formal schemes over the ring of integers of a finite
extension $K$ of $\Q_p$ (hence removing the ``unramified-ness''
assumption present in this paper) with semi-stable special fibers
and hope to be able to report on these results soon. Moreover we
think that for smooth schemes over $\cO_K$ one may replace
the locally constant $\Q_p$-sheaf $\bL$ on $X_K^{\rm et}$ by a
constructible $\Q_p$-sheaf and obtain interesting comparison
isomorphisms. We also believe that we should be able to derive
{\it integral} comparison isomorphisms which would work better
than the existing ones.

\bigskip

\noindent
{\bf Acknowledgements} We thank Ahmed Abbes for
pointing out the error in \cite{faltingsAsterisque} (see the
section \ref{sec:Groth}) and for
many helpful discussions and email exchanges on this
subject and 
Luc Illusie for kindly providing us with copies of his correspondence with
Deligne on Faltings' topology.
We are grateful to the two referees for the careful reading
of the paper and for pointing out some mistakes and suggesting ways to remedy them.
Some of their remarks have been incorporated in the article (e.g. see remark
\ref{remark:normal}.)
Finally, part of the work on this article was done while both authors were guests of
L'Institut Henri Poincar\'e, Paris during the Galois semester 2010.
We thank this institution for its hospitality.

\subsection{Notations}\label{sec:Notation} Let $p>0$ be a prime integer,
$\cO_K$ a complete discrete valuation ring with fraction field $K$
and perfect residue field $k$. Fix an algebraic closure $K \subset \Kbar$,
 let $\kbar$ denote its residue field and  $\OKbar$ the normalization
of~$\cO_K$ in~$\Kbar$. Write $G_K$ for the Galois group of~$\Kbar$
over~$K$. Fix a field extension $K\subset M \subset \Kbar$. We
write $\Mun \subseteq M$ for the maximal absolutely unramified
subfield of~$M$ and $\OMun$ for its ring of integers.\smallskip

The following notations will be used throughout the paper (some of
the objects denoted here will be defined in this very section  and
the rest in the next sections): \bigskip

\noindent $\bullet$ Rings:\smallskip

\noindent $W_n:=\WW_n(\cO_{\Kbar}/p\cO_\Kbar)$, $A_{\rm
inf}^+:=A_{\rm inf}^+(\OKbar)$, $A_{\rm inf}:=A_{\rm
inf}(\OKbar)$, $A_{\rm cris,n}:=A_{\rm cris,n}(\cO_\Kbar)$,
$A'_{\rm cris,n}:=A'_{\rm cris,n}(\cO_\Kbar)$, $A_{\rm
cris}:=A_{\rm cris}(\cO_\Kbar)= A'_{\rm cris}(\cO_\Kbar)$, $B_{\rm
cris}:=B_{\rm cris}(\OKbar)$.\bigskip

\noindent $\bullet$ Sheaves on $\fX_M$:
\smallskip

\noindent $\WW_{n,M}:=\WW_n(\cO_{\fX_M}/p\cO_{\fX_M})$,
$\WW_n:=\WW_{n,\Kbar}$, $\bA_{\rm inf,M}^+:=\{\WW_{n,M}\}_n$,
$\bA_{\rm inf}:=\{\WW_n\}_n$.

\bigskip
\noindent We recall a few facts regarding the properties and (one
of) the constructions of $A_{\rm cris}$ needed in the sequel. For
details we refer to \cite[\S1\&\S2]{Fontaineperiodes}. Choose a
compatible sequence of roots $(p^{1/p^{n-1}})_{n\ge 1}$ in
$\OKbar$ (compatible means that $(p^{1/p^{n}})^p=p^{1/p^{n-1}}$,
for all $n\ge 1$). For every $n\in\N$ we have a ring homomorphism
$\theta_n\colon W_n:=\WW_n(\OKbar/p\OKbar) \lra \OKbar/p^n\OKbar$
given by $(s_0,\ldots,s_{n-1}) \mapsto
\sum_{i=0}^{n-1}p^i\tilde{s}_i^{p^{n-1-i}}$ where $\tilde{s}_i\in
\OKbar/p^n\OKbar$ is a lift of~$s_i$ for every $i$.
Write~$\varphi$ for Frobenius on $W_n$. Denote by $\ds
\widetilde{p}_n:=\bigl[p^{1/p^{n-1}}\bigr] \in W_n$ the
Teichm\"uller lift of $p^{1/p^{n-1}}\in \OKbar/ p\OKbar$. Let
$\xi_n:=\widetilde{p}_n-p\in W_n$, then $\xi_n$ generates
$\Ker(\theta_n)$. Denote by $A_{\rm cris,n}$ the
$\WW(k)$--DP--envelope of $W_n$ with respect to the ideal
$\Ker(\theta_n)$ (where $\WW(k)$--DP--envelope means that the
divided powers are compatible with the standard divided powers
on~$p\WW_n(k)$). Note that $A_{\rm cris,n}$ is naturally endowed
with an action of $G_K$. Denote by $\Ker(\theta_n)^{\rm DP}$ the
PD--ideal on $A_{\rm cris,n}$. Note
that~$\varphi(\xi_n)=\varphi(\widetilde{p}_n-p)=(\widetilde{p}_n^p-p^p)+
(p^p-p)$. Since~$\widetilde{p}_n^p-p^p\in \Ker(\theta_n)$ and~$p$
admits divided powers in $A_{\rm cris,n}$, also~$\varphi(\xi_n)$
does. Thus Frobenius on $W_n$ extends to an operator called
Frobenius and denoted by $\varphi$, on $A_{\rm cris,n}$.

Let $\ds \cR(\OKbar):= \lim_{\leftarrow } \OKbar/p\OKbar$ where
the inverse limit is taken with respect to Frobenius. Put $\ds
A^+_{\rm inf}:=\WW\bigl(\cR(\OKbar)\bigr)\cong \lim_{\infty
\leftarrow n} W_n$ where the latter inverse limit is taken with
respect to the map $u_{n+1}\colon W_{n+1}\lra W_n$ defined by the
natural projection composed with Frobenius. Remark that the maps
$\theta_n$ are compatible i.~e., $\theta_n=\theta_{n+1}\circ
u_{n+1}$, that the sequence $\xi:=\{\xi_n\}_n$ is compatible i.e.,
$u_{n+1}(\xi_{n+1})=\xi_n$ for all $n\ge 0$, and that
$\Ker(\theta)$ is generated by~$\xi$. Denote by $\ds A_{\rm
cris}:=\lim_{\infty \leftarrow n} A_{\rm cris,n}$. It is the
$p$--adic completion of the $\WW(k)$--DP--envelope of $A^+_{\rm
inf}$ with respect to the ideal $\Ker(\theta)$. We then have
$$A_{\rm
cris}=A_{\rm inf}^+\bigl\{\langle \xi \rangle\bigr\}= A_{\rm
inf}^+ \bigl\{\delta_0,\delta_1,\ldots\bigr\}/(p\delta_0-\xi^p,p
\delta_{m+1}-\delta_m^p\bigr)_{m\geq 0}$$where $\delta_i=
\gamma^{i+1}(\xi)$ and $\gamma$ is the application on the kernel
of~$\theta$ on~$A_{\rm cris}$ given by~$z\mapsto (p-1)!z^{[p]}$;
cf.~\cite[Prop.~6.1.2]{brinon}. Note
that~$\WW_n\bigl(\cR(\OKbar)\bigr)\cong A_{\rm inf}^+/p^n A_{\rm
inf}^+$ since $A_{\rm inf}^+=\WW\bigl(\cR(\OKbar)\bigr)$ and
$\cR(\OKbar)=\ds \lim_\leftarrow \OKbar/p\OKbar$ is a perfect ring
by construction.

\begin{lemma} The kernel of the ring homomorphism $q_n\colon
\WW_n\bigl(\cR(\OKbar)\bigr)\lra \WW_n(\OKbar/p\OKbar)$ induced by $\overline{q}_n$ is the ideal generated by $\{[\widetilde{p}]^{p^n}, V([\widetilde{p}]^{p^n}),
V^2([\widetilde{p}]^{p^n}),\ldots,V^{n-1}([\widetilde{p}]^{p^n})\}$. In particular $$A_{\rm cris}/p^n A_{\rm cris} \cong W_n\bigl[\delta_0,
\delta_1,\ldots\bigr]/(p\delta_0-\xi_{n+1}^p,p \delta_{m+1}-\delta_m^p\bigr)_{m\geq 0}$$via the map which sends $\delta_i\mapsto \delta_i$ and induces
on~$\WW_n\bigl(\cR(\OKbar)\bigr)$ the morphism $q_n\colon \WW_n\bigl(\cR(\OKbar)\bigr)\to W_n $ associated to~$\overline{q}_n$.
\end{lemma}
\begin{proof} We prove the first claim.
We have~$\widetilde{p}^{p^n}=(\xi+p)^{p^n}\equiv \xi^{p^n}\mbox{
mod}~p^nA_{\rm inf}^+$ and $\xi^{p^n}=p^{p^{n-1}}
\delta_0^{p^{n-1}}=0$ mod~$p^n A_{\rm inf}^+$. The kernel of the
projection $\overline{q}_n\colon \cR(\OKbar)=\ds \lim_\leftarrow
\OKbar/p\OKbar\to \OKbar/p\OKbar$ on the $n+1$--th factor of the
limit is generated by~$\widetilde{p}^{p^n}$. This proves the lemma
for $n=1$. The general case follows  by induction on $n$ using the
exact sequence
$$0\lra \WW_{n-1}\bigl(\cR(\OKbar)\bigr)\stackrel{V}{\lra}\WW_n\bigl(\cR(\OKbar)\bigr)\lra
\WW_1\bigl(\cR(\OKbar)\bigr)\lra 0.
$$

Now let us recall that $[\widetilde{p}]^{p^n}=0$ in $A_{\rm cris}/p^nA_{\rm cris}$ and similarly, for every $0\le i\le n-1$ we have $V^i([\widetilde{p}]^{p^n})=
p^i[\widetilde{p}]^{p^{n-i}}=0$ in $A_{\rm cris}/p^nA_{\rm cris}$. Then the second claim follows.
\end{proof}

In particular $A_{\rm cris}/p^n A_{\rm cris}$ is the $\WW(k)$--DP envelope of $W_n$ with respect to $\xi_{n+1}W_n=\Ker\bigl(\theta_n\circ \varphi\bigr)$.  We then
get a surjective map of DP algebras
$$q_n\colon A_{\rm cris}/p^n
A_{\rm cris} \lra  A_{\rm cris, n} $$sending
$\xi_{n+1}^{[i]}\mapsto \xi_n^{[i]}$ and inducing Frobenius
on~$W_n$. We also have a map
$$u_n\colon A_{\rm cris, n+1}\lra  A_{\rm cris}/p^n
A_{\rm cris}$$sending $\xi_{n+1}^{[i]}\mapsto \xi_{n+1}^{[i]}$ and
inducing the natural projection $W_{n+1}\to W_n$.\smallskip

We  introduce the following ideal~$\II\subset A_{\rm cris}$.
Let~$\{\zeta_n\}_{n\in\N}$ be a compatible system of primitive
$p^n$--th roots of unity: $\zeta_2\neq 1$
and~$\zeta_{n+1}^p=\zeta_n$. It defines an
element~$\varepsilon=(1,\zeta_2,\zeta_3,\ldots)\in \crR(\OKbar)$.
Let~$[\varepsilon]\in A_{\rm inf}^+$ be its Teichm\"uller lift.
Let~$\II$ be the ideal generated
by~$\{\varphi^{-n}([\varepsilon])-1\}_{n\in\N}$ and the
Teichm\"uller lifts~$[x]$ of elements $x=(x_0,x_1,\ldots)\in
\crR(\OKbar)$ such that~$x_0$ lies in the maximal ideal
of~$\OKbar/p\OKbar$. It is proven in~\cite[Lem.~6.3.1]{brinon}
that~$\II^2=\II$ mod~$p^n A_{\rm cris}$.
Since~$\theta\bigl([\varepsilon]-1\bigr)=0$, the
element~$[\varepsilon]-1$ admits divided powers. In~$A_{\rm cris}$
we have the following important element
$$t:=\log\bigl([\varepsilon]\bigr)=\sum_{n=1}^\infty (n-1)!
\bigl([\varepsilon]-1\bigr)^{[n]}.$$We have $\varphi(t)=p t$ and for~$\sigma\in G_K$, $\sigma(t)=\chi(\sigma) t$ where $\chi\colon g_K \to \Z_p ^\ast$ is the
cyclotomic character defined by $\sigma(\zeta_n)=\zeta_n^{\chi(\sigma)}$ for every~$n\in\N$. Put~$B_{\rm cris}:=A_{\rm cris}[1/t]$.  Since~$t$ lies
in~$\Ker(\theta)$, it admits divided powers in~$A_{\rm cris}$ so that~$t^p=p !t^{[p]}$ and~$p$ is invertible in~$B_{\rm cris}$. Then $B_{\rm cris}$ is a
$\WW(\kbar)$--algebra, endowed with an action of~$G_K$, a Frobenius operator~$\varphi$ and a separated and exhaustive filtration $\Fil^r B_{\rm cris}:=\lim_{n\in\N}
\Fil^{r+n}A_{\rm cris}\cdot t^{-n}$ for every~$r\in\Z$.

\section{Fontaine sheaves}
\label{sec:fontain}

Let $X$ denote a smooth scheme over $\cO_K$ or a smooth $p$--adic formal
scheme topologically of finite type, over $\cO_K$. In this
section we introduce several sites describing their underlying
categories and giving pre-topology structures i.~e., for each object, we
describe the covering families. The topologies underlying the sites
will be the topologies generated by the given pre-topologies. See
\cite[\S II.1]{SGAIV} for details.

\subsection{Faltings' topos; the algebraic setting}
\label{sec:Groth}

Let us first treat the case when $X$ is a scheme of finite type over $\cO_K$. We denote by $X^{\rm et}$ the small \'etale site on $X$ and by $X_M^{\rm fet}$ the
finite \'etale site of $X_M$. Then $\Sh(X^{\rm et})$ and $\Sh(X_M^{\rm fet})$ will denote the categories  of sheaves of abelian groups on these sites, respectively.

\bigskip

\begin{definition}
\label{def:opensets} Let $E_{X_M}$ be the category defined as follows
\smallskip

i) the objects consist of pairs $\bigl(g\colon U\lra X, f\colon
W\lra U_M\bigr)$ such that $g$ is an \'etale morphism of finite
type and $f$ is a finite \'etale morphism. We will usually denote
by $(U,W)$ this object to shorten notations; \smallskip

ii) a morphism $(U',W')\lra (U,W)$ in $E_{X_M}$ consists of a pair
$(\alpha,\beta)$, where $\alpha\colon U'\lra U$ is a morphism over
$X$ and $\beta\colon W'\lra W$ is a morphism commuting with
$\alpha\otimes_{\cO_K}{\rm Id}_M$.\smallskip
\end{definition}

Let us remark that the pair $(X,X_M)$ is a final object of
$E_{X_M}$. Moreover, finite projective limits are representable in
$E_{X_M}$ and, in particular, fibre products exist: the fibre
product of the objects $(U',W')$ and $(U'',W'')$ over $(U,W)$ is
$\bigl(U'\times_U U'',W'\times_W W''\bigr)$. See \cite{erratum}.

\smallskip Faltings defined in \cite[p. 214]{faltingsAsterisque} a
pre-topology on $E_{X_M}$ by defining a family of morphisms
$\{(U_i,W_i)\lra (U,W)\}_{i\in I}$ to be a covering family if
$\{U_i\lra U\}_{i\in I}$ is a covering in $X^{\rm et}$ and
$\{W_i\lra W\}_{i\in I}$ is a covering family in $X_M^{\rm fet}$.
He then defined the presheaf $\cO_{\fX}$ on $E_{X_M}$ by
$$\cO_\fX(U,W):=\mbox{ the normalization of }\Gamma(U,\cO_U)\mbox{
in } \Gamma(W, \cO_W)$$ and stated that this was a sheaf. However,
this is not true in general due to point b) of the following
example. Moreover point c) below shows that even if one sheafified
the presheaf $\cO_\fX$ on Faltings' site the theory of
``localizations'' of sheaves, as developed later in this paper,
would not work. It should be noticed, though, that even if the
definition of the topology is not correct, the topos of sheaves
described by Faltings coincides with the one defined in this paper.

\begin{example}
\label{ex:faltings} Assume that $M=\Kbar$. Let $p>2$ be a prime
and let us denote by $\displaystyle A:=\Z_p\bigl[X,
\frac{1}{X^2+p}\bigr]$ and $\displaystyle
B:=\Z_p\bigl[X,Y,\frac{1}{X^2+p}\bigr]/(Y^2-X^2-p)$. For $i=1,2$
we define $\displaystyle B_i:=B\bigl[\frac{1}{Y+(-1)^iX}\bigr]$
and let $f_i$ denote the composition of the natural $\Z_p$-algebra
morphisms $A\lra B\lra B_i$. We denote $U:=\Spec(A)$,
$V:=\Spec(B)$, $U_i:=\Spec(B_i)$ and $W=W_i:=\Spec(B_\Kbar)$. Fix
$i\in \{1,2\}$, then we have

a) The pairs $(U,W)$ and $(U_i,W_i)$ are objects of $E_{U_\Kbar}$
and if we denote by $F_i:(U,W)\lra (U_i,W_i)$ the morphism induced
by  the pair $(f_i, Id)$, then this morphism is a coverings in
Faltings' sense.

b) For the covering above the presheaf $\cO_\fX$ does not satisfy the sheaf property.

c) Let us denote by $\cF$ the sheaf associated to $\cO_\fX$ on the topology defined by
Faltings and by $\cG$ the sheaf
$\cF/p\cF$. Then the natural map: $\cO_\fX(U,W)/p\cO_\fX(U,W)\lra \cG(U,W)$
is the zero map.
\end{example}

\begin{proof} a) Let us observe that
$$
\Bbar:=B/pB\cong \F_p[X, 1/X]/\bigl((X+Y)(X-Y)\bigr)\cong \F_p[X, 1/X]\times \F_p[X, 1/X],
$$
and we let $\Ubar:=\Spec(\Bbar)\cong \Ubar_1\coprod \Ubar_2$ where
we have denoted by $\Ubar_i\cong \Spec(\F_p[X,1/X])$ for $i=1,2$ the components of $\Ubar$.

Then let us remark that $U_i\cong V-\Ubar_i$. As $V\lra U$ is \'etale and surjective it follows that
the morphisms induced by $f_i$, $U_i\lra U$ are \'etale and surjective. Moreover
the natural $\Z_p$-algebra morphisms
$B\lra B_i$ for $i=1,2$ induce isomorphisms as $\Kbar$-algebras $B_\Kbar\cong B_{1,\Kbar}\cong B_{2,\Kbar}$.
Now a) follows.

b) We fix $i\in \{1,2\}$ as in the statement and we have the following commutative diagram
$$
\begin{array}{cccccccccc}
0&\lra&\cO_\fX(U,W)&\stackrel{h_i}{\lra}&\cO_\fX(U_i,W_i)&\stackrel{g_i}{\lra}&\cO_\fX(U_i\times_UU_i, W_i\times_WW_i)\\
&&\downarrow&&\downarrow&&\downarrow\\
0&\lra&B_\Kbar&\lra&B_{i,\Kbar}&\stackrel{\gamma}{\lra}&B_{i,\Kbar}\otimes_{B_\Kbar}B_{i,\Kbar}
\end{array}
$$
The vertical arrows in the diagram are inclusions therefore they are injective. Moreover $\gamma$ is defined by: $\gamma(b):=b\otimes 1 - 1\otimes b=0$ for all
$b\in B_{i,\Kbar}$ in view of the remarks above, therefore $\gamma=0$ which implies that $g_i=0$. If $\cO_\fX$ were  a sheaf then the top sequence would be exact,
i.e. $h_i$ would be an isomorphism. Thus all elements of $B_i\subset \cO_\fX(U_i,W_i)$ would be integral over $A$. In particular as $B_i$ is a finitely generated
$A$-algebra, $B_i$  would be finite over $A$. Since $U=\Spec(A)$ is connected the degree of $B_i$ as an $A$-module would be constant. But $A/pA\lra B_i/pB_i$ is an
isomorphism while $B_{i,\Kbar}$ is a free $A_\Kbar$-module of rank $2$. Therefore $\cO_\fX$ is not a sheaf.

c) As $\cF$ is a sheaf, for each $i=1,2$  we have a commutative diagram with the bottom row exact

$$
\begin{array}{cccccccccc}
0&\lra&\cO_\fX(U,W)&\stackrel{h_i}{\lra}&\cO_\fX(U_i,W_i)&\stackrel{g_i}{\lra}&\cO_\fX(U_i\times_UU_i, W_i\times_WW_i)\\
&&\downarrow&&\downarrow&&\downarrow\\
0&\lra&\cF(U,W)&\stackrel{u_i}{\lra}&\cF(U_i,W_i)&\stackrel{v_i}{\lra}&\cF(U_i\times_UU_i, W_i\times_WW_i)
\end{array}
$$

The arguments at b) above show that the map $g_i=0$ therefore the
image of the natural map $\cO_\fX(U_i,W_i)\lra \cF(U_i,W_i)$ is
contained in the image of $u_i$. More precisely the natural map:
$\varphi: \cO_\fX(U,W)\lra \cF(U,W)$ has the property that
$\varphi=w_i\circ h_i$ for $i=1,2$, where
$w_i:\cO_\fX(U_i,W_i)\lra \cF(U,W)$ is the map defined by the
above diagram.

We remark that $\cO_\fX(U,W)=B$ as $B$ is integral over $A$ and being smooth it is normal, similarly
$\cO_\fX(U_i,W_i)=B_i$, for $i=1,2$. Moreover we have $h_i=f_i$ for $i=1,2.$
As $\cF$ is the sheaf associated to the presheaf $\cO_\fX$ the natural map
$\cO_\fX(U,W)/p\cO_\fX(U,W)\lra \cG(U,W)$ is  the composition
$\cO_\fX(U,W)/p\cO_\fX(U,W)\lra \cF(U,W)/p\cF(U,W)\lra \cG(U,W)$. But the map
$$\overline{\varphi}:
\cO_\fX(U,W)/p\cO_\fX(U,W)\cong \Bbar\cong \Bbar_1\times\Bbar_2 \lra \cF(U,W)/p\cF(U,W)$$
induced by $\varphi$ has the property that it factors through
$\overline{f}_i$, for $i=1,2$ and $\overline{f}_i:\Bbar\lra \Bbar_i$
is the natural projection on the $i-th$ factor. We deduce that $\overline{\varphi}=0$.
\end{proof}

\bigskip
\noindent {\it Faltings' site} ${\rm PT}_{X_M}$. Let $X$ be a
scheme of finite type over $\cO_K$ and let $M$ be an algebraic
extension of $K$. We denote by  $E_{X_M}$ the category defined in
definition \ref{def:opensets}.

\begin{definition}
\label{def:coverings} Let$\{(U_i,W_i)\lra
(U,W)\}_{i\in I}$ be a family of morphisms in $E_{X_M}$.
We say that it is of type $\alpha$ respectively $\beta$ if:\smallskip

\noindent $\alpha$) $\{U_i\lra U\}_{i\in I}$ is a covering in
$X^{\rm et}$ and $W_i\cong W\times_UU_i$ for every $i\in I$. Here
the morphism $W\lra U$ used in the fibre product is the
composition $W\lra U_M\lra U$.\smallskip

or

\noindent $\beta$) $U_i\cong U$ for all $i\in I$ and $\{W_i\lra
W\}_{i\in I}$ is a covering in $X_M^{\rm fet}$.
\end{definition}
We endow $E_{X_M}$ with the topology ${\rm T}_{X_M}$ {\bf
generated} by the families of type $\alpha$ and $\beta$ described in definition
\ref{def:coverings} and denote by $\fX_M$ the associated site. We
call ${\rm T}_{X_M}$ Faltings' topology and $\fX_M$ Faltings'
site associated to $(X,M)$. Note that ${\rm T}_{X_M}$ can
be described differently as follows.

\begin{definition}
\label{def:strict} A family $\{(U_{ij},W_{ij})\lra (U,W)\}_{i\in
I,j\in J}$ of morphisms in $E_{X_M}$ is called a {\bf strict
covering family} if\smallskip

\noindent a) For each $i\in I$ there exists an \'etale morphism
$U_i\lra X$ such that we have isomorphisms $U_i\cong U_{ij}$ over
$X$ for every $j\in J$. \smallskip

\noindent b) $\{U_i\lra U\}_{i\in I}$ is a covering in $X^{\rm
et}$.\smallskip

\noindent c) For every $i\in I$ the family $\{W_{ij}\lra
W\times_UU_i\}_{j\in J}$ is a covering in $X_M^{\rm
fet}$.\smallskip

To simplify notations we will henceforth denote a strict covering
family $\{(U_{ij},W_{ij})\lra (U,W)\}_{i\in I,j\in J}$ by
$\{(U_i,W_{ij})\lra (U,W)\}_{i\in I,j\in J}$.
\end{definition}

\begin{remark}
\label{rm:strict} The families of type $\alpha$ and $\beta$ in definition \ref{def:coverings} are examples of strict coverings. Conversely a strict covering family
$\{(U_i,W_{ij})\lra (U,W)\}_{i\in I,j\in J}$ can be obtained as a composite of the covering $\{(U_i,W\times_U,U_i) \lra (U,W)\}_{i\in I,j\in J}$, which is of type
$\alpha$) and for every $i\in I$ the covering $\{(U_i, W_{ij})\lra (U_i,W\times_UU_i)\}_{j\in J}$, which is of type $\beta$). In particular the topology generated
by the strict coverings coincides with ${\rm T}_{X_M}$.
\end{remark}

\begin{remark}
\label{rm:solvedfaltings} The morphisms $(U_i,W_i)\lra (U,W)$, for
$i=1,2$ in example \ref{ex:faltings} are {\bf not} coverings in
the sense of \ref{def:coverings}. In fact, it follows from
\ref{prop:sheaf} and \ref{ex:faltings} Faltings' topology
associated to $(X,\Kbar)$ is coarser than the one originally
introduced by Faltings.
\end{remark}

\begin{remark}
\label{rm:error} The category $E_{X_M}$ with the strict covering
families do not form a pre-topology. Indeed, since finite
projective limits exist in $E_{X_M}$ the strict covering families satisfy PT0,
PT1 and PT3 of \cite[Def II 1.3]{SGAIV} but contrary to what was
written in \cite{andreatta_iovita} and as was pointed out to us by
A.~Abbes, they do not satisfy PT2. However, one may define
tautologically the generated pre-topology ${\rm PT}_{X_M}$ by
considering as covering families the composite of finitely many
strict coverings (or of finitely many  families of type
$\alpha$) and $\beta$) of definition \ref{def:coverings}). The
associated topology is ${\rm T}_{X_M}$.
\end{remark}

\begin{remark}
\label{rm:salvage} It follows from \cite[Cor. II 2.3]{SGAIV} or by
a direct check using the definitions that a pre-sheaf on $E_{X_M}$
is a sheaf if and only if it satisfies the usual exactness
property for the strict covering families.
\end{remark}

\noindent The next lemma and \cite[Remark II 3.3]{SGAIV} show that
it is enough to use strict covering families in order to sheafify
a presheaf on $E_{X_M}$, as done in \cite{andreatta_iovita}.

\begin{lemma}
\label{lemma:cofinal} Let $(U,W)$ be an object of $E_{X_M}$. Then
the strict covering families are cofinal in the collection of all
covering families of $(U,W)$ in ${\rm PT}_{X_M}$.
\end{lemma}

\begin{proof} See \cite{erratum}.
\end{proof}

\begin{definition}
\label{def:OX} We define  the pre-sheaf of $\cO_M$-algebras on
$E_{X_M}$, denoted $\cO_{\fX_M}$, by
$$
\cO_{\fX_M}(U,W):=\mbox{ the normalization of
}\Gamma\bigl(U,\cO_U\bigr)\mbox{ in } \Gamma\bigl(W, \cO_W\bigr).
$$We also define the sub pre-sheaf of $\cO_{\Mun}$-algebras $\cO_{\fX_M}^{\rm un}$ of
$\cO_{\fX_M}$ whose sections over $(U,W)\in E_{X_M}$ consist of elements $x\in \cO_{\fX_M}(U,W)$ for which there exist a finite unramified extension $K\subset L$, a
finite \'etale morphism $U' \to U\otimes_{\cO_K} \cO_L$ and a  morphism $W\to U'_K\otimes_L M$ over $U_M$ such that $x$, viewed in $\Gamma\bigl(W, \cO_W\bigr)$,
lies in the image of $\Gamma\bigl(U',\cO_{U'}\bigr)$.
\end{definition}

We have

\begin{proposition}
\label{prop:sheaf} The pre-sheaves $\cO_{\fX_M}$ and\/
$\cO_{\fX_M}^{\rm un}$ are sheaves.
\end{proposition}

\begin{proof} We first prove that $\cO_{\fX_M}$ is a sheaf.
Let $\{(U_\alpha, W_{\alpha,i})\lra (U,W)\}_{\alpha,i}$ be a
strict covering family. We set
$U_{\alpha\beta}:=U_\alpha\times_UU_\beta$ and $W_{\alpha\beta
ij}:=W_{\alpha,i}\times_WW_{\beta,j}$. We have the following
commutative diagram
$$
\begin{array}{cccccccccc}
0&\lra&\cO_{\fX_M}(U,W)&\stackrel{f}{\lra}&
\prod_{i,\alpha}\cO_{\fX_M}(U_\alpha,W_{\alpha,i})&\stackrel{g}{\lra}&
\prod_{(\alpha,i),(\beta,j)}\cO_{\fX_M}(U_{\alpha\beta}, W_{\alpha\beta ij})\\
&&\downarrow&&\downarrow&&\downarrow\\
0&\lra&\Gamma(W,
\cO_W)&\lra&\prod_{\alpha,i}\Gamma(W_{\alpha,i},\cO_{W_{\alpha,i}})&\lra&
\prod_{(\alpha,i),(\beta,j)} \Gamma(W_{\alpha\beta ij},
\cO_{\alpha\beta ij})
\end{array}
$$
Since the $\{U_\alpha\lra U\}_\alpha$ is a covering in $X^{\rm et}$ and for every $\alpha$, $\{W_{\alpha,i}\lra W\times_UU_\alpha\}_i$ is a covering in
$\left(W\times_U U_{\alpha,M}\right)^{\rm fet}$ it follows that $\{W_{\alpha,i}\lra W\}_{\alpha,i}$ is a covering in $X_M^{\rm et}$. In particular the bottom row of
the above diagram is exact. Moreover the vertical maps are all inclusions therefore $f$ is injective, i.e. $\cO_{\fX_M}$ is a separable pre-sheaf. Let $x\in
\Ker(g)$. Then $\displaystyle x\in \Gamma(W,\cO_W)\cap\prod_{\alpha,i} \cO_{\fX_M}(U_\alpha, W_{\alpha,i})$. We are left to prove that $x$ is integral over
$\Gamma(U, \cO_U)$. Without loss of generality we may assume that $W$ is connected and that $U_\alpha=\Spec(A_\alpha)$ is affine for every $\alpha$. Note that there
exists a finite extension $K\subset L$ in $M$ and a finite and \'etale morphism $W'\to U_L$ so that its base change via $L \to M$ is $W \to U_M$ and $x\in
\Gamma(W',\cO_{W'})$. Let us denote by $x_\alpha$ the image of $x$ in $\Gamma\bigl(W'\times_UU_\alpha, \cO_{W'\times_UU_\alpha}\bigr)$. Because the family
$\{W_{\alpha,i}\lra W\times_UU_\alpha\}_i$ is a covering family in $\left(W\times_U U_{\alpha,M}\right)^{\rm fet}$ and the image $x_{\alpha,i}$ of $x_\alpha$ in
$\Gamma(W_{\alpha,i}, \cO_{W_{\alpha,i}})$ is in fact in $\cO_{\fX_M}(U_\alpha, W_{\alpha,i})$, hence integral over $A_\alpha$, it follows that $x_\alpha$ is
integral over $A_\alpha$. Let $P_\alpha(X)\in A_\alpha[X]$ be the (monic) characteristic polynomial of $x_\alpha$ over $A_\alpha$ with respect to the finite and
\'etale extension $W'\times_UU_\alpha \lra U_{\alpha,K}$ (see remark \ref{remark:normal} below.) Then
$P_\alpha(X)\vert_{U_{\alpha\beta}}=P_\beta(X)\vert_{U_{\alpha\beta}}$ for all $\alpha$ and $\beta$ and, therefore, there is a monic polynomial $P(X)\in
\Gamma(U,\cO_U)$ such that $P(X)\vert_{U_\alpha}=P_\alpha(X)$. As $P(x)\vert_{U_\alpha}=P_\alpha(x_\alpha)=0$ for every $\alpha$ it follows that $P(x)=0$, i.e. that
$x$ is integral over $\Gamma(U,\cO_U)$.

Since $\cO_{\fX_M}^{\rm un}\subset \cO_{\fX_M}$ by construction, it follows that $\cO_{\fX_M}^{\rm un}$ is a separated pre-sheaf. Using the previous notations, it
suffices to show that given $W$ connected and $U_\alpha$'s affine and  given $x\in \prod_{i,\alpha}\cO_{\fX_M}^{\rm un}(U_\alpha,W_{\alpha,i})$, whose image in
$\prod_{i,\alpha}\cO_{\fX_M}(U_\alpha,W_{\alpha,i})$ lies in $\Ker(g)$, then $x\in \cO_{\fX_M}^{\rm un}(U,W)$. As before for every $\alpha$ let $x_\alpha$ be the
image of $x$ in $\Gamma(W_\alpha,\cO_{W_\alpha})$. By replacing the $U_\alpha$'s by a finite subcover, we may assume that we have only finitely many $\alpha$'s. By
definition there is a finite unramified extension $K \subset L'$ such that each $x_\alpha$ is defined over a finite and \'etale cover of $U_\alpha\otimes_{\cO_K}
\cO_{L'}$. Possibly after enlarging $L$, we may assume that $L'\subset L$. Let $W_\alpha'$ (resp.~$Z_\alpha'$) be the spectrum of the normalization of the
sub-algebra $A_\alpha\otimes_{\cO_K} L [x_\alpha]$ (resp.~$A_\alpha\otimes_{\cO_K} \cO_{L'}[x_\alpha]$) in $\Gamma\bigl(W\times_U U_\alpha,\cO_{W\times_U
U_\alpha}\bigr)$. By construction $Z_\alpha'$ is finite and \'etale over $U \otimes_{\cO_K} \cO_L$ and we have morphisms $W_\alpha'\to Z_{\alpha,K}'\otimes_{L'} L$
over $U_{\alpha,L}$. Moreover, $x_\alpha\in\Gamma\bigl(W_\alpha', \cO_{W_\alpha'}\bigr)$ is in the image of $\Gamma\bigl(Z_\alpha', \cO_{Z_\alpha'}\bigr)$. Note
that $W_\alpha' \times_{U_\alpha} U_{\alpha\beta}\cong W_\beta' \times_{U_\beta} U_{\alpha\beta}$ so that the various $W_\alpha'$ glue to a finite and \'etale
morphism $W' \to U_L$ and there is a morphism $W\to W'$ as schemes over $U_L$ such that $x\in \Gamma\bigl(W',\cO_{W'}\bigr)$. Moreover, also the various $Z_\alpha'$
glue to a scheme $Z'$ finite and \'etale over $U \otimes_{\cO_K} \cO_{L'}$ and we have a morphism $W' \to  Z' \otimes_{\cO_{L'}} L$. Then $x$ is in the image of
$\Gamma\bigl(Z',\cO_{Z'}\bigr)$ and we conclude that $x\in \cO_{\fX_M}^{\rm un}(U,W)$ as claimed.
\end{proof}

\bigskip
\noindent
The following argument was offered by the referee of the paper.

\begin{remark}
\label{remark:normal}
Let $A$ be a noetherian normal domain and let $B$ be the integral closure of
$A$ in a finite \'etale extension of $A[1/p]$. Let $x\in B[1/p]$
be an element and let $Q(X)\in A[1/p][X]$ be the characteristic polynomial of
$x$ (it exists as $B[1/p]$ is a finitely generated projective $A[1/p]$-module.)
Then $x\in B$ if and only if $Q(X)\in A[X]$.
\end{remark}

\begin{proof}
The sufficiency is clear and to prove necessity, as $A$ is a noetherian normal
domain, it is enough to prove that $Q(X)\in A_{\mathfrak p}$, for all ${\mathfrak p}$
prime ideal of height $1$ of $A$ which contains $p$. Hence we may assume
that $A$ is a DVR and in this case $B$ is a free $A$-module of finite rank
and $Q(X)$ is the characteristic polynomial of the matrix associated to
the endomorphism $B\lra B; b\rightarrow xb$ with respect to a
basis of $B$ over $A$. Therefore $Q(X)\in A[X]$.
\end{proof}

\subsection{Faltings' topos; the formal setting}
\label{sec:formal_Groth}

Let now  $X$ denote a formal scheme. Denote by
$X^{\rm et}$ the \'etale site on~$X$ and by~$\Sh(X^{\rm et})$ the
category of sheaves of abelian groups on~$X^{\rm et}$. Of
particular importance will be the so called {\it small affine
opens}  of $X^{\rm et}$. These are objects~$\cU$ such that
$\cU=\Spf(R_\cU)$ is affine and connected and there are parameters
$T_1,T_2,\ldots,T_d\in R_\cU^\times$ such that the map
$R_0:=\cO_K\{T_1^{\pm 1},\ldots,T_d^{\pm 1}\}\subset R_\cU$ is
formally \'etale.
\bigskip

\noindent {\it The site $X_{M, {\rm et}}$.} For every finite
extension $K\subset L$ in $M$ let $X_{L,{\rm et}}$ be the site of
\'etale and quasi--compact morphisms $\cW \to X_L$  of $L$--rigid
analytic spaces. Here $X_K $ denotes the $K$--rigid analytic space
associated to~$X$ and $X_L$ is its base change to $L$. We refer to
\cite[\S3.1\&3.2]{deJong_vdPut} for generalities about \'etale
morphisms of rigid analytic spaces. Given extensions $L \subset
L'$ of $K$ contained in $M$ the base change from $L$ to $L'$
provides a morphism of sites $X_{L,{\rm et}} \to X_{L',{\rm et}}$.
We then get a fibred site $X_{\ast,{\rm et}}$ over the category of
finite extensions of $K$ contained in $M$ in the sense of \cite[\S
VI.7.2.1]{SGAIV}. We let $X_{M,{\rm et}}$ be the site defined by
the projective limit of the fibred site $X_{\ast,{\rm et}}$; see
\cite[Def. VI.8.2.5]{SGAIV}.

We can give the following explicit description. The objects in
$X_{M,{\rm et} }$ consist of pairs $(\cW,L)$ where~$L$ is a finite
extension of\/~$K$ contained in~$M$ and~$\cW \to X\tensor_K L$ is
an \'etale and quasi--compact map  of $L$--rigid analytic spaces.
Given~$(\cW,L)$ and~$(\cW',L')$ define $\Hom_{X_{M,{\rm
et}}}\bigl((\cW',L'), (\cW,L)\bigr)$ as the direct limit $\ds
\lim_\rightarrow \Hom_{L''}\bigl(\cW'\otimes_{L'} L'',
\cW\otimes_L L''\bigr)$ over all finite extensions~$L''\subset M$,
containing both~$L$ and~$L'$, of the morphism $\cW'\otimes_{L'}
L''\to \cW\otimes_L L''$ as rigid analytic spaces over~$X\otimes_L
L''$. The coverings of a pair~$(\cW,L)$ in~$X_{M,{\rm et}}$ are
finite families of pairs~$\{(\cW_\alpha,L_\alpha)\}_\alpha$
over~$(\cW,L)$ such that~$L\subset L_\alpha$ for every~$\alpha$
and there exists a finite extension of $K$, contained in $M$ and
containing $L_\alpha$ for every $\alpha$ such that the induced map
$\amalg_\alpha \cW_\alpha\otimes_{L_\alpha} L' \to \cW\otimes_{L}
L'$ is surjective.

\bigskip

\noindent {\it The site $\cU_{M,{\rm fet}}$.} Let $\cU\to X$ be an
\'etale map topologically of finite type of $p$-adic formal
schemes. Define~$\cU^{\ast,{\rm fet}}$ to be the following site
fibred over the category of finite extensions $K\subset L$
contained in $M$. For very such $L$ write~$\cU^{L,{\rm fet}}$ to
be the category of finite \'etale covers $\cW \to \cU_L$ as
$L$--rigid analytic spaces. Given extensions $L\to L'$ we consider
the base change map $\cU^{L,\fet}\to \cU^{L',{\rm fet}}$. Define
$\cU_{\rm M,fet}$ as the projective limit site. For notational
purposes we write $(\cW,L)$, or simply $\cW$, for an object of
$\cU_{M,{\rm fet}}$. In the first notation we implicitly assume
that $\cW\in \cU^{L,{\rm fet}}$. We refer to \cite[\S
4.1]{andreatta_iovita} for an explicit description. Note that the
fiber product of two pairs over a given one exists in~$\cU_{M,{\rm
fet}}$ and, if those are defined in $\cU_{L,{\rm fet}} $ for some
$L$, it coincides with the image of the fibre product
in~$\cU_{L,{\rm fet}}$.

\noindent Let~$\cU_2\to\cU_1$ be a map of formal schemes over~$X$.
Assume that they are \'etale over~$X$. We then have a functor $
\cU_{1,\ast,{\rm fet}} \to \cU_{2,\ast,{\rm fet}}$ of sites fibred
over the category of finite extensions $K\subset L$ contained in
$M$. We let $$\rho_{\cU_1,\cU_2}\colon \cU_{1,M,{\rm fet}} \to
\cU_{2,M,{\rm fet}}$$ be the induced morphisms of projective
limits. It is given on objects by $$(\cW,L) \mapsto
\left(\cW\times_{\cU_{1,K}} \cU_{2,K},L\right).$$

\bigskip
\noindent {\it The category $E_{X_M}$ and Faltings' topology ${\rm
T}_{X_M}$.} Define $E_{X_M}$ to be the category of
pairs~$(\cU,\cW)$ where $\cU\to X$ is an \'etale map of formal
schemes and $\cW$ is an object of~$\cU_{\rm M,fet}$. A morphism of
pairs~$(\cU',\cW')\to (\cU,\cW)$ is defined to be a morphism
$\cU'\to \cU$ as formal schemes over~$X$ and a map $\cW' \to
\cW\times_{\cU_K} \cU^{\rm '}_K$ in~$\cU'_{\rm M,fet}$.

We define strict covering families exactly as in definition
\ref{def:strict} and Faltings' topology ${\rm T}_{X_M}$ to be the
topology generated by the strict covering families. We call the
associated site the locally Galois site  attached to the data $(X,
M)$ and denote it by $\fX_M$.

We define the pre-sheaves $\cO_{\fX_M}$ and $\cO_{\fX_M}^{\rm un}$
on $E_{X_M}$ as in definition \ref{def:OX}. The analogue of
proposition \ref{prop:sheaf} holds in our formal context i.e.,
$\cO_{\fX_M}$ and $\cO_{\fX_M}^{\rm un}$ are sheaves.

\subsection{Continuous functors.
Localization functors}\label{sec:localization} We define:

\begin{enumerate}

\item[I.a] if $X$ is a scheme over $\cO_K$, we have $u_{X,M}\colon
\fX_M \lra X_{M, {\rm et}}  $ with $u_{X,M}(U,W):=W$;

\item[I.b] if $X$ is a $p$--adic formal scheme  over~$\cO_K$, let
$u_{X,M}\colon \fX_M \lra X_{M, {\rm et}} $ be
$u_{X,M}\bigl(\cU,(\cW,L)\bigr):=(\cW,L)$;

\item[II.a] if $X$ is a scheme of finite type over $\cO_K$, let
$v_{X,M}\colon X_{{\rm et}}\lra \fX_M$ be given by
$v_{X,M}(U):=\bigl(U,U_M\bigr)$;

\item[II.b] if $X$ is a formal scheme locally topologically of
finite type over~$\cO_K$, we have $v_{X,M}\colon X_{{\rm et}}\lra
\fX_M$ given by $v_{X,M}(\cU):=\bigl(\cU,\cU_K\bigr)$;

\noindent Let~$K\subset M_1 \subset M_2\subset \Kbar$ be field
extensions. Define

\item[III] $\beta_{M_1,M_2}\colon \fX_{M_1}\to  \fX_{M_2}$ by
$\beta_{M_1,M_2}(U,W)=\bigl(U,W\otimes_{M_1} M_1\bigr)$
(resp.~$\beta_{M_1,M_2}(\cU,\cW)$ equal to $\bigl(\cU,\cW\bigr)$
viewed in  $\fX_{M_2}$) in the algebraic (resp.~formal) setting.

\end{enumerate}

\noindent It is clear that the above functors send covering families to covering families and commute with fiber products.  In particular they define continuous
functors of sites by~\cite[Prop.~III.1.6]{SGAIV}. They also send final objects to final objects so that they induce morphisms of the associated topoi of sheaves.
\smallskip

\noindent Following \cite{erratum} we define a geometric point of
$\fX$ to be a pair $(x,y)$ where  $x$ is a geometric point of $X$
and  $y$ is a geometric point of $X_K$ specializing to $x$ i.e., a
geometric point of the henselization of $X$ at $x$. In loc.~cit.,
we define the stalk $\cF_{(x,y)}$ of a sheaf $\cF$ on $\fX$ to be
the direct limit $\lim \cF(U,W)$ over all pairs
$\bigl((U,x'),(W,y')\bigr)$ where $x'$ is a point of $U$ mapping
to $x$ and $y'$ is a point of $W$ specializing to $x'$ and mapping
to $y$. We proved in loc.~cit.~that there are enough geometric
point in $\fX$ i.e., that a sequence of sheaves is exact if an
only if the induced sequence on stalks is exact for all geometric
points.

\begin{lemma}\label{lemma:vastOX} Both in the algebraic and in the
formal setting we have an isomorphism of sheaves
$v_{X,M}^\ast\bigl(\cO_X\bigr) \cong \cO_{\fX_M}^{\rm un}$ on
$\fX_M$.
\end{lemma}
\begin{proof} Let $Q$ be the pre-sheaf on $\fX$ defined by
$Q(U,W):=\Gamma(U,\cO_U)$ if $W\ne \phi$ and $Q(U,\phi)=0$. It is a separated pre-sheaf. Note that if $W\ne \phi$, $(U,U_K)$ is the initial object in the category
of all pairs $(U',W')$ admitting a morphism $(U,W)\to (U',W')$ in $\fX$. Thus, $v_{X,M}^\ast\bigl(\cO_X\bigr)$ is the sheaf on $\fX$ associated to the pre-sheaf
$Q$.  Note also that we have a natural map $Q \to \cO_{\fX_M}^{\rm un}$. Let $a\in \cO_{\fX_M}^{\rm un}(U,W)$ and view it in $\Gamma(W,\cO_W)$.  By definition there
exists a finite extension $K\subset L$ in $M$ and a finite and \'etale morphism $U' \lra U\otimes_{\cO_K} \cO_L$ so that we have a  map $W \to U'_K\otimes_L M$ over
$U_M$ and $a$ is in the image of $\Gamma(U',\cO_{U'})$ in $\Gamma(W, \cO_W)$. Note that $U'$ is a direct factor of $U'\otimes_{\cO_K} \cO_L$ so that $a$ is in the
image of $Q(U',W)$ in $\Gamma(W, \cO_W)$ as wanted. This proves that the natural morphism $Q\lra \cO_{\fX_M}^{\rm un}$ is surjective. To prove injectivity let
$(U,W)$ be such that $U$ is connected and $W\ne \phi$.  Since the composition
$$
Q(U,W)\lra \cO_{\fX_M}^{\rm un}(U,W)\subset \Gamma(W, \cO_W)
$$
is injective we deduce that the first map is injective. It follows that
the induced morphism from the sheaf associated to
$Q$ to $\cO_{\fX_M}^{\rm un}$ is injective.
\end{proof}

\bigskip
{\it The localization functors.} For this section we suppose that
$X$ is either a smooth scheme or a smooth formal scheme  over
$\cO_K$.  Let $\cU$ be a connected affine open in the  \'etale
site of $X$ with underlying algebra $R_\cU$.  Write
$R_{\cU}\otimes_{\cO_K} M:=\prod_{i=1}^n R_{\cU,i}$ with
$\Spec\bigl(R_{\cU,i}\bigr)$ connected. Fix a geometric generic
point $\bareta_i=\Spec(\C_{\cU,i})$ of
$\Spec\bigl(R_{\cU,i}\bigr)$ and denote by $\Rbar_{\cU,i}$ the
union of all finite normal $R_{\cU}$ sub-algebras of $\C_{\cU,i}$,
which are finite and \'etale over $R_{\cU,i}$ after inverting $p$.
We let $\cG_{\cU_M,i}$ be the Galois group of $R_{\cU,i}\subset
\Rbar_{\cU,i}\otimes_{\cO_K} K$. Eventually, let
$\Rbar_\cU:=\prod_{i=1}^n \Rbar_{\cU,i}$ and let
$$\cG_{\cU_M}:=\prod_{i=1}^n \cG_{\cU_M,i}.$$Let
$\Rep\bigl(\cG_{\cU_M}\bigr)$
(resp.~$\Rep\bigl(\cG_{\cU_M}\bigr)^\N$) be the category of
discrete abelian groups (resp.~the category of inverse systems of
finite abelian groups indexed by~$\N$) with continuous action of
$\cG_{\cU_M}$. We have natural functors, which we'll call
localization functors
$$\Sh\bigl(\fX_M\bigr)\longrightarrow \Rep\bigl(\cG_{\cU_M}
\bigr)\qquad \hbox{{\rm and}}\qquad
\Sh\bigl(\fX_M\bigr)^\N\longrightarrow
\Rep\bigl(\cG_{\cU_M}\bigr)^\N$$defined as follows (we only define
the functor in the case $X$ is a scheme over $\cO_K$ as above and
leave it to the reader to fill in the details for the other
cases): if $\cG\in \Sh\bigl(\fX_M\bigr)$ is a sheaf of abelian
groups, its localization is $\ds
\cG\bigl(\Rbar_\cU\bigr):=\oplus_{i=1}^n \ds
\cG\bigl(\Rbar_{\cU,i}\bigr)$ where
$\cG\bigl(\Rbar_{\cU,i}\bigr):=\ds
\lim_{\to}\cG\bigl(\cU,\Spec(S)\bigr)$, for $S$ running over all
$R_{\cU,i}$ sub-algebras of $\Rbar_{\cU,i}\tensor_{\cO_K} K$ which
are finite and \'etale. It is a set with the discrete topology
and it is endowed with a continuous action of $\cG_{\cU_M}$. The
objects~$(\cU,\cW)$ of~$\fX_M$, with~$\cW=\Spec(S)$ and $R_{\cU,i}
\to  S \subset \Rbar_{\cU,i}\tensor_{\cO_K} K$, correspond to
finite index sub-groups $G_\cW\subset \cG_{\cU_M,i}$. For any such,
we can recover $\cG(\cU,\cW)$ from the localization of~$\cG$ by
the formula $\cG(\cU,\cW)=\cG\bigl(\Rbar_{\cU,i}\bigr)^{G_\cW}$.
This allows to recover $\cG(\cU,\cW)$ for every $\cW\to \cU_M$
finite and \'etale (see~\cite[Lemma 4.5.3]{andreatta_iovita}).

\begin{lemma}\label{lemma:betaastG}
Let $\cG$ be a sheaf on~$\fX_{M_1}$. Let $M_1\subset M_2$ be a Galois field extension. Then
\begin{enumerate}

\item[i.] the sheaf $\beta_{M_1,M_2}^\ast (\cG)$ coincides with
the pre-sheaf $\beta_{M_1,M_2}^{-1} (\cG)$;

\item[ii.] for every object~$(\cU,\cW)$ of $\fX_{M_1}$ the group
$\beta_{M_1,M_2,\ast}\left(\beta_{M_1,M_2}^\ast
(\cG)\right)(\cU,\cW)$ is endowed with an action of ${\rm
Gal}(M_2/M_1)$ and $\cG(\cU,\cW)= \left(\beta_{M_1,M_2,\ast}\left(
\beta_{M_1,M_2}^\ast (\cG)\right)(\cU,\cW)\right)^{{\rm
Gal}(M_2/M_1)}$.

\item[iii.] take $\cU=\Spf(R_\cU)$ to be a connected affine open in the \'etale site of $X$. Then $$\beta_{M_1,M_2,\ast}\left(\beta_{M_1,M_2}^\ast
(\cG)\right)(\Rbar_\cU)\cong \cG\bigl(\Rbar_\cU\bigr)^{[M_2:M_1]}.$$

\end{enumerate}
\end{lemma}

\begin{proof} We prove the statements in the formal case, leaving  the
algebraic case to the reader.

(i) Given an object~$(\cU,\cW)$ of $\fX_{M_2}$, the group $\beta_{M_1,M_2}^{-1} (\cG)(\cU,\cW) $ is $\ds \lim_{\cW'} \cG(\cU,\cW')$ where the direct limit is taken
over all objects $(\cU,\cW')\in \fX_{M_1}$ and all morphisms from $(\cU,\cW)$ to~$(\cU,\cW')$ in~$\fX_{M_2}$. Note that $\cW$ is finite and \'etale over $\cU_L$ for
some finite extension $K\subset L$ contained in $M_2$. In particular it is a finite and \'etale over $\cU_{L'}$ for $L':=M_1\cap L$. Thus the direct limit admits as
final object the group $\cG(\cU,\cW)$ with $(\cU,\cW)$ viewed as an object of $\fX_{M_1}$. The first claim follows; see also~\cite[Pf.
Prop.~4.4.2(4)]{andreatta_iovita}.

(ii) Take an object~$(\cU,\cW)$  of $\fX_{M_1}$. Assume that $\cW$ is finite and \'etale over $\cU_L$ with $L\subset M_1$. Then $\beta_{M_1,M_2,\ast}\left(
\beta_{M_1,M_2}^\ast (\cG)\right)(\cU,\cW)$ coincides with the direct limit $\cG(\cU,\cW_{L'})$ over all finite extensions $L \subset L'\subset M_2$ where
$\cW_{L'}$ is considered as an \'etale covering of $\cU_{L''}$ for $L'':=L'\cap M_1$.  The Galois group ${\rm Gal}(M_2/M_1)$ acts on this set and  the invariants
under ${\rm Gal}(L'M_1/M_1)$ are exactly $\cG(\cU,\cW)$. The claim follows.

(iii) It follows from (ii) and the fact
that~$\Rbar_\cU[p^{-1}]\tensor_{M_1}
M_2=\Rbar_\cU[p^{-1}]^{[M_2:M_1]}$.
\end{proof}

\subsection{The sheaf $\bA_{\rm inf,M}^+$.} \label{sec:sheafAinf}

Let us recall the following definitions from \S5 of
\cite{andreatta_iovita}.  Denote by $\widehat{\cO}_{\fX_M}$ the inverse
system of sheaves of $\cO_M$-algebras
$\left\{\cO_{\fX_M}/p^n\cO_{\fX_M}\right\}_n\in \Sh(\fX_M)^\N$.

For every~$s\in\N$ define
$\WW_{s,M}:=\WW_s\bigl(\cO_{\fX_M}/p\cO_{\fX_M}\bigr)$; it is the
sheaf~$\bigl(\cO_{\fX_M}/p\cO_{\fX_M}\big)^s$ with ring operations
defined by Witt polynomials and the transition maps in the inverse
system defined by Frobenius. Let $\bA_{\rm inf,M}^+$
in\/~$\Sh(\fX_M)^\N$ be the inverse system of sheaves of
$\WW(k)$-algebras $\left\{\WW_{n,M}\right\}_n$ where the
transition maps are defined as the composite of the natural
projection $\WW_{n+1,M}\to \WW_{n,M}$ and Frobenius on
$\WW_{n,M}$. Note that~$\bA_{\rm inf,M}^+$ is endowed with a
Frobenius operator, denoted by~$\varphi$, and is a sheaf of
$\OMun$--algebras.

If $M_1\subset M_2$ is a field extension, it follows from \ref{lemma:betaastG} that we have a natural isomorphism $\beta_{M_1,M_2}^{\ast}\bigl(\cO_{\fX_{M_1}}\bigr)
\cong\cO_{\fX_{M_2}}$. In particular we have a natural map
$$
\beta_{M_1,M_2}^{\ast}\left(\WW_{s,M_1}\right)\lra \WW_{s,M_2}$$which is an isomorphism since $\beta_{M_1,M_2}^\ast$ is exact. In particular we have  natural
isomorphisms of inverse systems of sheaves $\beta_{M_1,M_2}^\ast\left(\widehat{\cO}_{\fX_{M_1}}\right)\cong\widehat{\cO}_{\fX_{M_2}}$ and
$\beta_{M_1,M_2}^\ast\left(\bA_{\rm inf,M_1}^+\right)\cong\bA_{\rm inf,M_2}^+$.

\begin{proposition}
\label{prop:localization} Let $\cU$ be a small affine object of the site $X^{\rm et}$; see \S\ref{sec:formal_Groth} for the definition. Then the natural maps

\begin{enumerate}
\item[a)] $\ds \hR_\cU \longrightarrow \lim_{\infty \leftarrow
n}\left(\cO_{\fX_M}/p^n\cO_{\fX_M}\right)
(\Rbar_\cU)=:\widehat{\cO}_{\fX_M}(\Rbar_\cU)$,

\item[b)] $\ds A_{\rm inf}^+(\Rbar_\cU)\longrightarrow
\lim_{\infty \leftarrow n} \WW_{n,M}(\Rbar_\cU)=:\bA_{\rm
inf,M}^+(\Rbar_\cU)$.

\end{enumerate}
\noindent are isomorphisms.
\end{proposition}

We'll only prove the result in the formal case and leave it
to the reader to repeat the arguments in the algebraic case. First
of all we prove

\begin{lemma}\label{lemma:barOmodpnisseparated}
For every~$n\in\N$ the pre-sheaf $\cO_{\fX_M}/p^n \cO_{\fX_M}$ is
separated i.~e., if $(\cU',\cW')\to (\cU,\cW)$ is a covering, the
natural map $$\cO_{\fX_M}(\cU,\cW)/p^n \cO_{\fX_M}(\cU,\cW) \lra
\cO_{\fX_M}(\cU',\cW')/p^n \cO_{\fX_M}(\cU',\cW')$$ is
injective.\end{lemma}
\begin{proof}
The lemma is a direct consequence of \cite{artin} Miscellany (1.8), (iv) or in this particular case one may reason as follows. We write $\cO_{\fX_M}(\cU,\cW)=\cup_i
S_i$ (resp.~$\cO_{\fX_M}(\cU',\cW')=\cup_j S_j'$) as the union of normal and finite $R_\cU$--algebras (resp.~$R_{\cU'}$--algebras), \'etale after inverting~$p$ such
that for every~$i$ there exists~$j_i$ so that $S_i$ is contained in~$S_{j_i}'$ and the map $\Spec(S_{j_i}') \to \Spec(S_i)$ is surjective on prime ideals
containing~$p$. Let~$x\in S_i\cap p^n S_{j_i}'$. Let~$\cP\subset S_i$ be a prime ideal over~$p$ and let~$\cP'\subset S_{j_i}'$ be a height one prime ideal over it.
Then $x\in S_{i,\cP}\cap p^n S_{j_i,\cP'}'$. Hence $x\in p^n S_{i,\cP}$. Thus $x$ lies in the intersection of all height one prime ideals of~$S_i$ so that~$x\in
S_i$. We conclude that the map $S_i/p^n S_i \to S_{j_i}'/p^nS_{j_i}'$ is injective. The claim follows.
\end{proof}

\noindent The lemma implies that we have an injective map
$$\Rbar_\cU/p^n\Rbar_\cU= \cO_{\fX_M}(\Rbar_\cU)
/p^n\cO_{\fX_M}(\Rbar_\cU) \lra
\bigl(\cO_{\fX_M}/p^n\cO_{\fX_M}\bigr)(\Rbar_\cU).$$The
proposition follows then from the following

\begin{lemma}\label{lemma:localizationofbarOmodpn} 1) The cokernel
of $\Rbar_\cU/p^n\Rbar_\cU \lra
\bigl(\cO_{\fX_M}/p^n\cO_{\fX_M}\bigr)(\Rbar_\cU)$ is annihilated
by the maximal ideal of~$\OKbar$.\smallskip

2) The image of the map
$\bigl(\cO_{\fX_M}/p^{n+1}\cO_{\fX_M}\bigr)(\Rbar_\cU)\to
\bigl(\cO_{\fX_M}/p^n\cO_{\fX_M}\bigr)(\Rbar_\cU)$ factors via
$$\Rbar_\cU/p^n\Rbar_\cU \subset
\bigl(\cO_{\fX_M}/p^n\cO_{\fX_M}\bigr)(\Rbar_\cU).$$\smallskip

3) The image of Frobenius on
$\bigl(\cO_{\fX_M}/p\cO_{\fX_M}\bigr)(\Rbar_\cU)$ factors via
$\Rbar_\cU/p\Rbar_\cU \subset
\bigl(\cO_{\fX_M}/p\cO_{\fX_M}\bigr)(\Rbar_\cU)$.

\end{lemma}
\begin{proof}
It follows from~\ref{lemma:barOmodpnisseparated} that the value of the sheaf $\cO_{\fX_M}/p^n \cO_{\fX_M}$ on $(\cU,\cW)$ is given by the direct limit, over all
coverings $(\cU',\cW')$ of~$(\cU,\cW)$ with~$\cU'$ affine, of the elements $b$ in $ \cO_{\fX_M}(\cU',\cW')/p^n\cO_{\fX_M}(\cU',\cW')$ such that the image of~$b$
in~$\cO_{\fX_M}(\cU'',\cW'')/p^n \cO_{\fX_M}(\cU'',\cW'')$ is~$0$, where~$(\cU'',\cW'') $ is the fiber product of~$(\cU',\cW')$ with itself over~$(\cU,\cW)$. Hence
$$\bigl(\cO_{\fX_M}/p^n\cO_{\fX_M}\bigr)\bigl(\Rbar_\cU\bigr) =\lim_{S,T}\Ker_{S,T,n}$$where the notation is as follows. The direct limit is taken over all normal
$R_{\cU,\infty}$ sub-algebras~$S$ of~$\Rbar_\cU$, finite and \'etale after inverting~$p$ over~$R_{\cU,\infty}[1/p]$, all affine covers~$\cU'\to \cU$ and all normal
extensions~$R_{\cU',\infty} \tensor_{R_\cU} S \to T$, finite, \'etale and Galois after inverting~$p$. Eventually, we put $\cU'':=\Spf(R_{\cU''})$ to be the fiber
product of~$\cU'$ with itself over~$\cU$ i.~e., $R_{\cU''}:=\widehat{R_{\cU'}\otimes_{R_\cU}R_{\cU'}}$. We let $\Ker_{S,T,n}:=\Ker\left(T/p^n T \rightrightarrows
\widetilde{T\otimes_S T}/p^n \widetilde{T\otimes_S T} \right)$, where~$\widetilde{T\otimes_S T}$ is the normalization of $T\otimes_S T$. Write $\widetilde{T}_S$ for
the normalization of
 $T\otimes_{(R_{\cU',\infty}
\otimes_{R_{\cU,\infty}} S)} T$.   We have a natural morphism
$\widetilde{T\otimes_S T} \lra \widetilde{T}_S$ of
$R_{\cU',\infty}$--algebras. Let
$$\Ker'_{S,T,n}:= \Ker\left(T/p^n T \rightrightarrows
\widetilde{T}_S/p^n \widetilde{T}_S \right).$$Then $$\bigl(\cO_{\fX_M}/p^n\cO_{\fX_M}\bigr)\bigl(\Rbar_\cU\bigr) \subset \lim_{S,T}\Ker'_{S,T,n}.$$

\smallskip

{\it Study of $\Ker'_{S,T,n}$.} For every~$S$ and~$T$ as above, write~$G_{S,T}$ for the Galois group of $T\tensor_{\cO_K} K$ over
$S\otimes_{R_{\cU,\infty}}R_{\cU',\infty}\otimes_{\cO_K} K$. Then $\widetilde{T}_S$  is simply the product $\prod_{g\in G_{S,T}} T$.   Hence we have
$$\Ker'_{S,T,n}=\Ker\left(T/p^nT
\rightrightarrows \prod_{g\in G_{S,T}} \frac{T}{p^nT}\right)=(T/p^nT)^{G_{S,T}},$$where the
two maps in the display are $a\mapsto (a,\cdots,a)$ and $a\mapsto
\bigl(g(a)\bigr)_{g\in G_{S,T}}$.
\smallskip

{\it Study of $\Coker(S/p^nS\lra \Ker_{S,T,n})$.} For the rest of
this proof we make the following notations: if $B$ is a normal
$R_{\cU,\infty}$-algebra we denote by
$B':=B\otimes_{R_{\cU,\infty}}R_{\cU',\infty}=B\otimes_{R_\cU}R_{\cU'}$,
also
$B'':=B'\otimes_{R_{\cU',\infty}}R_{\cU'',\infty}=B'\otimes_{R_{\cU'}}R_{\cU''}$
(the second equalities above follow from \cite[Lemma
6.19]{andreatta_iovita}). Note that $B'$ and $B''$ are normal.
Indeed, $B=\cup B_i$
is the union of finite and normal $R_{\cU'}$--algebras $B_i$.
Since~$R_{\cU'}$ is an excellent ring by \cite{Valabrega}, it
follows from \cite[\S7.8.3(ii)]{EGAIV} that each $B_i$ is
excellent. Thanks to \cite[\S7.8.3(v)]{EGAIV} we conclude that
$B_i\tensor_{R_{\cU}} R_{\cU'}$ is normal since it is the
$p$-adic completion of an \'etale $B_i$-algebra of finite type.
Thus, $B':=B\tensor_{R_{\cU}} R_{\cU'}$ is normal as well.
Similarly one shows that $B''$ is normal.

We
then get a commutative diagram
$$
\begin{array}{cccccccccc}
0 & \to & S/p^nS & \lra & S'/p^nS' & \rightrightarrows &
S''/p^nS''\cr & & \big\downarrow & & \big\downarrow{\alpha} & &
\big\downarrow{\beta}\cr 0 & \to & \Ker_{S,T,n} & \lra & T/p^nT &
\rightrightarrows &\widetilde{T\otimes_ST}/p^n\widetilde{T\otimes_ST}\cr
&&\big\downarrow &&||&&\big\downarrow\cr
0&\to & \Ker'_{S,T,n}& \lra & T/p^nT &\rightrightarrows &\widetilde{T}_S/p^n\widetilde{T}_S.\cr
\end{array}
$$The top row is exact by \'etale descent and the middle and bottom  rows are
exact by construction. Since~$S'\subset T$ and $S''\subset T''$
are finite extensions of normal rings, the maps~$\alpha$ and
$\beta$ are injective.
Let us remark that the image of $S'/p^nS'$ in $\widetilde{T}_S/p^n\widetilde{T}_S=
\widetilde{T\otimes_{S'}T}/p^n\widetilde{T\otimes_{S'}T}$ is $0$, therefore
the image of $\alpha$ factors via $\Ker'_{S,T,n}=(T/p^nT)^{G_{S,T}}$.

Define $Z$ as $\Coker\bigl(S'/p^n S'\to
(T/p^nT)^{G_{S,T}}\bigr)\subset \Coker(\alpha)$ and $Y$ as
$\Coker(S/p^nS\lra \Ker_{S,T,n})$. Since $\Ker_{S,T,n}$ is
$G_{S,T}$-invariant, the image of\/ $Y$  in $\Coker(\alpha)$ is
contained in~$Z$. Since~$\alpha$ and $\beta$ are injective, the
map $Y\to Z$ is injective. Consider the exact sequence
$$ 0\lra S'/p^n S'=T^{G_{S,T}}/p^n
T^{G_{S,T}} \lra \bigl(T/p^nT\bigr)^{G_{S,T}} \lra {\rm H}^1\bigl(G_{S,T},T\bigr).$$Then $Y \subset Z\subset {\rm H}^1\bigl(G_{S,T},T\bigr)$. Since
$R_{\cU',\infty}\to T$ is almost \'etale, the group ${\rm H}^1\bigl(G_{S,T},T\bigr)$ is annihilated by any element of the maximal ideal of~$\OKbar$;
see~\cite[Thm.~I.2.4(ii)]{faltingsJAMS}. This implies the first claim of lemma \ref{lemma:localizationofbarOmodpn}.

\smallskip

{\it Study of the projection $\Ker'_{S,T,n+1}\to \Ker'_{S,T,n}$.}
It is induced by the natural projection $T/p^{n+1} T\to T/p^n T$.
Consider the  commutative diagram

$$
 \begin{array}{ccrcrcccc}
0 & \lra & T & \stackrel{p^{n+1}}{\lra} & T & \lra & T/p^{n+1} T &
\lra 0 \cr & & p \big\downarrow & & \Vert & & \big\downarrow  \cr
0 & \lra & T & \stackrel{p^n}{\lra} & T & \lra & T/p^n T & \lra 0
.\cr
\end{array}
$$Taking $G_{S,T}$--invariants we get the following commutative
diagram:

$$ \begin{array}{ccccccc} 0 & \lra & S'/p^{n+1} S' & \lra &
\bigl(T/p^{n+1}T\bigr)^{G_{S,T}}& \lra & {\rm
H}^1\bigl(G_{S,T},T\bigr) \cr & & \big\downarrow & &
\big\downarrow & & \big\downarrow \cdot p \cr 0 & \lra & S'/p^n S'
& \lra & \bigl(T/p^nT\bigr)^{G_{S,T}}& \lra & {\rm
H}^1\bigl(G_{S,T},T\bigr).\cr
 \end{array}
$$Since multiplication by $p$ annihilates ${\rm
H}^1\bigl(G_{S,T},T\bigr)$, we conclude that the projection $\Ker'_{S,T,n+1}\to \Ker'_{S,T,n}$ factors via $S/p^n S$. Hence the image of
$\bigl(\cO_{\fX_M}/p^{n+1}\cO_{\fX_M}\bigr)(\Rbar_\cU)\to \bigl(\cO_{\fX_M}/p^n\cO_{\fX_M}\bigr)(\Rbar_\cU)$ factors via $\Rbar_\cU/p^n\Rbar_\cU \subset
\bigl(\cO_{\fX_M}/p^n\cO_{\fX_M}\bigr)(\Rbar_\cU)$. This proves the second claim of the lemma.
\smallskip

Consider the map of sets $T/p T\to T/p^2 T$ sending an element $a$ to the $p$--th power $\widetilde{a}^p$ of a lift $\widetilde{a}$ of $a$ in $T/p^2 T $. It is well
defined since it does not depend on the choice of the lift~$\widetilde{a}$. It induces a map $\rho\colon \bigl(T/pT\bigr)^{G_{S,T}} \to \bigl(T/p^2T\bigr)^{G_{S,T}}
$. Frobenius on $\bigl(T/pT\bigr)^{G_{S,T}}$ factors as the composite of $\rho$ and the projection $\bigl(T/p^2T\bigr)^{G_{S,T}}\to \bigl(T/pT\bigr)^{G_{S,T}}$. It
follows from the above discussion that Frobenius $\Ker'_{S,T,1}\to \Ker'_{S,T,1}$ factors via $S/p S$. Hence the image of Frobenius on
$\bigl(\cO_{\fX_M}/p\cO_{\fX_M}\bigr)(\Rbar_\cU)$ factors via $\Rbar_\cU/p\Rbar_\cU \subset \bigl(\cO_{\fX_M}/p\cO_{\fX_M}\bigr)(\Rbar_\cU)$. This proves the last
claim of the lemma and the proposition \ref{prop:localization}.
\end{proof}

\
\bigskip
\noindent {\it The map $\theta_M$.} We define a morphism $\theta_M\colon \bA_{\rm inf,M}^+\lra \widehat{\cO}_{\fX_M}$  of objects of~$\Sh(\fX_M)^\N$ as follows. We
work in the formal setting. Fix a non-negative integer $n$. Let $\bigl(\cU,\cW\bigr)$ be an object of $\fX_M$. Let $S=\cO_{\fX_M}(\cU,\cW)$ and  consider the
diagram of sets and maps
$$
\begin{array}{ccccccccc}
(S/p^nS)^n&\stackrel{a_n}{\lra}&S/p^nS\\
\downarrow b_n\\
(S/pS)^n
\end{array}
$$
where $\ds
a_n(s_0,s_1,\ldots,s_{n-1}):=\sum_{i=0}^{n-1}p^is_i^{p^{n-1-i}}$
and $b_n$ is the natural projection. Remark that there is a unique
map of sets, $c_n\colon (S/pS)^n\lra S/p^nS$ which makes the
diagram commutative, i.e. such that $c_n\circ b_n=a_n$. Moreover,
$c_n$ induces a ring homomorphism
$c_{n,(\cU,\cW)}\colon\WW_n(S/pS)\lra S/p^nS$ functorial in
$(\cU,\cW)$ i.e. a morphism of presheaves
$\WW_{n,M}\stackrel{c_n}{\lra} \cO_{\fX_M}/p^n\cO_{\fX_M}$. Denote
by $\theta_{M,n}$ the induced morphism on the associated sheaves.

\begin{lemma}
\label{lemma:theta} a) The following diagram of sheaves and
morphisms commutes for varying~$n\in\N$:
$$
\begin{array}{ccccccccc}
\WW_{n+1,M}&\stackrel{\theta_{M,n+1}}{\lra}&
\cO_{\fX_M}/p^{n+1}\cO_{\fX_M}\\
\downarrow u_n&&\downarrow v_n\\
\WW_{n,M}&\stackrel{\theta_{M,n}}{\lra}&\cO_{\fX_M}/p^n\cO_{\fX_M}
\end{array}
$$where $u_n$ is the composition of the natural projection and
Frobenius and $v_n$ is the natural projection. \smallskip

b) For $\cU=\Spf(R_\cU)$ connected open affine in $X^{\rm
et}$, the localization of
$$\theta_M=\{\theta_{M,n}\} \colon\bA_{\rm inf,M}^+ \lra
\hatcO_{\fX_M}$$is the map $\theta_\cU\colon A_{\rm
inf,M}^+(\Rbar_\cU)\lra \hR_\cU$  of \cite[Prop.~5.1.1]{brinon}.
\end{lemma}

\begin{proof} a) It is enough to prove that the diagram commutes at the level of
pre-sheaves, so let as before $(\cU,\cW)$ be an object of $\fX_M$
and let $S=\cO_{\fX_M}(\cU,\cW)$. We are reduced to checking that
the following diagram of sets and maps commutes
$$
\begin{array}{ccccccccc}
(S/pS)^{n+1}&\stackrel{c_{n+1}}{\lra}&
S/p^{n+1}S\\
\downarrow u_n&&\downarrow v_n\\
(S/pS)^n&\stackrel{c_n}{\lra}&S/p^nS
\end{array}
$$
where $u_n(s_0,s_1,\ldots,s_n)=(s_0^p,s_1^p,\ldots,s_{n-1}^p)$.
This is a simple calculation which we leave to the reader.

b) Using the proposition \ref{prop:localization}, the map
$(\theta_{n,M})_\cU$ is induced by $\ds
\lim_{\rightarrow,\cW}c_{n,(\cU,\cW)}$ and therefore coincides
with the map defined in  \cite[Prop.~5.1.1]{brinon}.
\end{proof}

Fix an object $(\cU,\cW)$ of $\fX_M$. Write
$S=\cO_{\fX_M}(\cU,\cW)$.

\begin{lemma}\label{lemma:krcn} Assume that $p^{1/p^{n-1}}\in S$.
Then $\xi_n:=[p^{1/p^{n-1}}]-p$ (see section \ref{sec:Notation}) is a well defined element of\/~$\WW_n(S/pS)$ and it generates the kernel
of\/~$c_n\colon\WW_n(S/pS)\lra S/p^nS$.
\end{lemma}

\begin{proof}
Let us first remark that
$c_n(\xi_n)=(p^{1/p^{n-1}})^{p^{n-1}}-p=0$, therefore $\xi_n\in
\Ker(c_n)$.

\noindent We show that if $x\in \Ker(c_n)$ then $x\in
\xi_n\WW_n(S/pS)$. We'll prove this statement by induction on $n$.
For $n=1$, $c_1={\rm Id}$ and $\xi_1=0\in \WW_1(S/pS)=S/pS$. Let
now $n>1$ and suppose that our statement is true for $n-1.$ Let
$\alpha\in \Ker(c_n)$. \smallskip

{\it Claim 1} There are $\beta\in \WW_n(S/pS)$ and $\gamma\in
\WW_{n-1}(S/pS)$ such that $\alpha=\xi_n\beta+\V(\gamma)$, where
$\V\colon\WW_{n-1}(S/pS)\lra \WW_n(S/pS)$ is Vershiebung, i.e.
$\V(s_0,s_1,\ldots,s_{n-2})=(0,s_0,s_1,\ldots,s_{n-2})$, for
$(s_0,s_1,\ldots,s_{n-2})\in \WW_{n-1}(S/pS)$.

To prove this claim let us write $\alpha=(\alpha_0,\alpha_1,\ldots,\alpha_{n-1})$ and so $\ds 0=c_n(\alpha)=\sum_{i=0}^{n-1} p^i\widetilde{\alpha}_i^{p^{n-1-i}}$,
where if $x\in S/pS$, then $\widetilde{x}$ denotes any lift of $x$ to $S/p^nS$. Therefore $ \widetilde{\alpha}_0^{p^{n-1}}=p c$ for some~$c\in S$. Let~$R_\cU\subset
S' (\subset S)$ be a finite and normal extension containing both~$\widetilde{\alpha}_0$ and~$p^{1/p^{n-1}}$. For every height one prime ideal~$\fp$ of~$S'$ we
have~$\widetilde{\alpha}_0^{p^{n-1}}/p=c$ which lies in~$S'_\fp$ so that~$\widetilde{\alpha}_0/ p^{1/p^{n-1}}\in S'_\fp$ since the latter is a dvr. We conclude that
$\widetilde{\alpha}_0/ p^{1/p^{n-1}}$ lies in the intersection of the localizations of~$S'$ at every height one prime ideal so that, since~$S'$ is noetherian, we
must have $\widetilde{\alpha}_0/ p^{1/p^{n-1}}\in S'$. Denote by $\beta_0$ the image of $\widetilde{\alpha}_0/p^{1/p^{n-1}}$ in $S/pS$. Then
$\alpha_0=p^{1/p^{n-1}}\beta_0$ in $S/pS$. Let $\beta:=(\beta_0,0,\ldots,0)\in \WW_n(S/pS)$ and let us compute: $\alpha-\xi_n\cdot\beta=\alpha-\widetilde{p}_n\cdot
\beta+p\beta= (\alpha_0,\alpha_1,\ldots,\alpha_{n-1})-(\alpha_0,0,\ldots,0)+(0,\beta_0^p,0,\ldots,0) \in \V\bigl(\WW_{n-1}(S/pS)\bigr)$.

\bigskip
\noindent Let $\gamma\in \WW_{n-1}(S/pS)$ be such that
$\alpha-\xi_n\beta=\V(\gamma)$. Then
$c_n(\V(\gamma))=c_n(\alpha-\xi_n\beta)=0$ and
$c_n(\V(\gamma))=w_n(c_{n-1}(\gamma))$, where $w_n$ is the
isomorphism $w_n\colon S/p^{n-1}S\cong pS/p^nS$. Therefore
$c_{n-1}(\gamma)=0$ and by the inductive hypothesis there is
$\delta\in \WW_{n-1}(S/pS)$ such that $\gamma=\xi_{n-1}\delta$.
The lemma now follows from\smallskip

{\it Claim 2} $\V(\xi_{n-1}\delta)=\xi_n\V(\delta)$. This is a simple calculation with Witt vectors. Write $\delta=(\delta_0,\delta_1,\ldots,\delta_{n-2})$, then
$$
\xi_{n-1}\delta=(p^{1/p^{n-2}},0,\ldots,0)\cdot
(\delta_0,\delta_1,\ldots,\delta_{n-2})- p\delta=
$$
$$
=(p^{1/p^{n-2}}\delta_0,p^{1/p^{n-3}}\delta_1,\ldots,p^{1/p}\delta_{n-2})-p\delta.
$$
Therefore
$$\V(\xi_{n-1}\delta)=
(0,p^{1/p^{n-2}}\delta_0,\ldots,p^{1/p}\delta_{n-2})-\V(p\delta).
$$

On the other hand,
$$\xi_n\V(\delta)=\xi_n\cdot (0,\delta_0,\delta_1,\ldots,\delta_{n-2})
=(0,p^{1/p^{n-2}}\delta_0,\ldots,p^{1/p}\delta_{n-2})-p(\V\delta),
$$
which proves the second claim and the lemma because $\V$ is an
additive map.
\end{proof}

Note that if $(\cU,\cW)$ is an object of $\fXKbar$ then $\OKbar\subset S$ and  $\ds \widetilde{p}_n \in \WW_n(S/pS)$. In particular  $\xi$ is naturally a section of
the sheaf $\bA_{\rm inf}^+(\cO_\fX)$ over $(\cU,\cW)$. We deduce from \ref{lemma:krcn}:

\begin{corollary}\label{corollary:krthetan} We have
$\Ker\left(\theta_\Kbar\colon \bA_{\rm inf}\lra
\hatcO_\fXKbar\right)=\xi\cdot \bA_{\rm inf}$ as sheaves in
$\Sh(\fXKbar)^\N$.
\end{corollary}

Let $\cU$ be a small affine as in \S\ref{sec:formal_Groth}. Write
$q^\prime:=\frac{[\varepsilon]-1}{[\varepsilon]^{\frac{1}{p}}-1}=1+[\varepsilon]^{\frac{1}{p}}+\cdots + [\varepsilon]^{\frac{p-1}{p}}\in A_{\rm inf}^+ $. Then
\begin{lemma}\label{lemma:uglylemma} (1) For every positive integer $r$
the Frobenius morphism  $\varphi$ induces an
isomorphism $\WW_n(\Rbar_\cU/p\Rbar_\cU)/
\varphi^{-r-1}\big(q^\prime\big) \WW_n(\Rbar_\cU/p\Rbar_\cU) \cong
\WW_n(\Rbar_\cU/p\Rbar_\cU)/ \varphi^{-r}\big(q^\prime\big)
\WW_n(\Rbar_\cU/p\Rbar_\cU) $;\smallskip

(2) for every $n\in\N$ the $W_n$--module
$\WW_n(\Rbar_\cU/p\Rbar_\cU)$ is flat;\smallskip

(3) assume that we are in the formal case. The sequence $0\lra
\Z/p^n \Z \lra \WW_n(\Rbar_\cU/p\Rbar_\cU)
\stackrel{\varphi-1}{\lra} \WW_n(\Rbar_\cU/p\Rbar_\cU)\lra 0$ is
exact where $\Z/p^n \Z$ is the constant group over
$\Spec\bigl(R_{\cU}\otimes_{\cO_K} M\bigr)$.

\end{lemma}
\begin{proof} (1) We proceed by induction on $n$. The kernel of the reduction
$\WW_{n+1}\big(\Rbar_\cU/p\Rbar_\cU\big) \to
\WW_n\big(\Rbar_\cU/p\Rbar_\cU\big) $ is $\V^n
\WW_{n+1}\big(\Rbar_\cU/p\Rbar_\cU\big)$, where $\V$ is the
Vershiebung. The latter  is isomorphic to $\Rbar_\cU/p\Rbar_\cU$
as an abelian group with structure of
$\WW_{n+1}\bigl(\Rbar_\cU/p\Rbar_\cU\bigr)$--module via the map
$\WW_{n+1}\bigl(\Rbar_\cU/p\Rbar_\cU\bigr)\to\Rbar_\cU/p\Rbar_\cU$,
$(a_0,\ldots,a_n)\mapsto a_0^{p^n}$.

Let $A_{r,n}$ be the assertion that $\varphi$ induces an
isomorphism $$\WW_n(\Rbar_\cU/\Rbar_\cU)/
\varphi^{-r-1}\big(q^\prime\big) \WW_n(\Rbar_\cU/\Rbar_\cU)
\lra \WW_n(\Rbar_\cU/\Rbar_\cU)/
\varphi^{-r}\big(q^\prime\big) \WW_n(\Rbar_\cU/\Rbar_\cU).$$
Let $B_{r,n}$ be the following assertion: the sheaf
$$\V^n
\WW_{n+1}\big(\Rbar_\cU/p\Rbar_\cU\big)/
\varphi^{-r}\big(q^\prime\big) \V^n
\WW_{n+1}\big(\Rbar_\cU/p\Rbar_\cU\big),$$
which is isomorphic to
$\Rbar_\cU/\bigl((1+\zeta_2+\cdots
\zeta_2^{p-1})^{\frac{1}{p^{r-n}}},p\bigr) \Rbar_\cU$, injects in
$$\WW_{n+1}\big(\Rbar_\cU/p\Rbar_\cU\big)/
\varphi^{-r}\big(q^\prime\big)
\WW_{n+1}\big(\Rbar_\cU/p\Rbar_\cU\big).$$

We prove by induction on
$r$ and $n$ that the claims $A_{r,n}$ and $B_{r,n}$ hold.  For
$n=1$ and any $r$ the fact that the map $\varphi$ is injective
follows from the fact that $\Rbar_\cU$ is normal, cf.~proof of
\ref{lemma:barOmodpnisseparated}. It is surjective by \cite[Prop.
2.0.1]{brinon}. Thus $A_{r,1}$ holds. Since $1+\zeta_2+\cdots
\zeta_2^{p-1}$ has $p$-adic valuation $1$ assertion $B_{r,n}$
holds for every $n\in\N$ and every $r \geq n$. It then follows by
descending induction on $r< n$ and from $A_{r,1}$ that $B_{r+1,n}$
implies $B_{r,n}$. Thus $B_{r,n}$ holds for every $r$ and $n$.
Proceeding by induction on $n$ one then proves that $A_{r,n}$ and
$B_{r,n+1}$ together with $A_{r,1}$ imply $A_{r,n+1}$.\smallskip

(2) It suffices to prove that  the map $J \tensor_{W_n}
\WW_n\big(\Rbar_\cU/p\Rbar_\cU\big) \to
\WW_n\big(\Rbar_\cU/p\Rbar_\cU\big) $ is injective for every
finitely generated ideal $ J \subset W_n$. We proceed by induction
on~$n\in\N$. Let~$n=1$. Since $\OKbar$ is the union of
discrete valuation rings and $J$ is finitely generated we have
$J=p^{\delta} \OKbar/p\OKbar$ for some $0\leq \delta\leq 1$ i.~e.,
$J\cong \OKbar/p^{1-\delta} \OKbar$ and the inclusions $J\subset
\OKbar/p\OKbar$ is given by multiplication by~$\cdot
p^\delta\colon \OKbar/p^{1-\delta} \OKbar \to \OKbar/p\OKbar$.
Then $J\tensor_{\OKbar} \Rbar_\cU/p\Rbar_\cU= \Rbar_\cU/
p^{1-\delta} \Rbar_\cU$ and the map $J\tensor_{\OKbar}
\Rbar_\cU/p\Rbar_\cU\to \Rbar_\cU/p\Rbar_\cU$ is the map
$\Rbar_\cU/ p^{1-\delta} \Rbar_\cU \to \Rbar_\cU/p\Rbar_\cU$ given
by $s\mapsto p^\delta s$. Since~$\Rbar_\cU$ is normal, this map is
injective as well.

Assume that the claim holds for ideals of $W_n$ for~$n\leq N$.
Let~$J\subset W_{N+1}$ be an ideal. Let $\pi_N\colon W_{N+1}\to
W_N$ be the natural projection. Its kernel is $\V^N W_{N+1}$ which
is isomorphic to $\OKbar/p\OKbar$. Let~$J_N$ be the image of~$J$
via~$\pi_N$ and put $J':=J\cap \Ker(\pi_N)$ which we view as an
ideal of $\OKbar/p\OKbar$ via the identification above. Since
Frobenius is surjective on $\Rbar_\cU/p\Rbar_\cU$, and hence on
$\WW_{N+1}\big(\Rbar_\cU/p\Rbar_\cU\big)$, and multiplication by
$p$ is $\V\circ \varphi$, we have $\V
\WW_{N+1}\big(\Rbar_\cU/p\Rbar_\cU\big)=p
\WW_{N+1}\big(\Rbar_\cU/p\Rbar_\cU\big)$. Then
$J'\otimes_{W_{N+1}} \WW_{N+1}\big(\Rbar_\cU/p\Rbar_\cU\big) \cong
J'\otimes_{\OKbar/p\OKbar} \Rbar_\cU/p \Rbar_\cU$ which is
isomorphic to $J'\otimes_{\OKbar/p\OKbar}
\Rbar_\cU/\bigl(\varphi^{-N+1}\big(q^\prime\big),p\bigr)
\Rbar_\cU $ since multiplication by
$\varphi^{-N+1}\big(q^\prime\big)$ is multiplication by
$p\equiv 0$ on $J'$. Moreover the map from
$J'\otimes_{\OKbar/p\OKbar}
\Rbar_\cU/\bigl(\varphi^{-N+1}\big(q^\prime\big), p\bigr)
\Rbar_\cU $   to $\WW_{N+1}\big(\Rbar_\cU/p\Rbar_\cU\big)$ factors
via $\Rbar_\cU/p \Rbar_\cU \cong \V^N
\WW_{N+1}\big(\Rbar_\cU/p\Rbar_\cU\big)$ and via these
identifications it is the map $a\otimes b \mapsto a b^{p^{N-1}}$.
This is proven to be injective as in the proof of the $n=1$ case.

We have an exact sequence $0 \to J' \to J \to J_N \to 0$ of
$W_{N+1}$--modules. The induced map $J'\otimes_{W_{N+1}}
\WW_{N+1}\big(\Rbar_\cU/p\Rbar_\cU\big) \to \V^N
\WW_{N+1}\big(\Rbar_\cU/p\Rbar_\cU\big)$ has been proven to be
injective. Note that $J_N\otimes_{W_{N+1}}
\WW_{N+1}\big(\Rbar_\cU/p\Rbar_\cU\big)\cong J_N\otimes_{W_N}
\WW_N\big(\Rbar_\cU/p\Rbar_\cU\big)$ since the kernel $\V^N
\WW_{N+1}\big(\Rbar_\cU/p\Rbar_\cU\big)=\V^N \varphi^N
\WW_{N+1}\big(\Rbar_\cU/p\Rbar_\cU\big)= p^N
\WW_{N+1}\big(\Rbar_\cU/p\Rbar_\cU\big)$ and $p^N\equiv 0\in W_N$.
Furthermore, the  map $J_N\otimes_{W_{N}}
\WW_N\big(\Rbar_\cU/p\Rbar_\cU\big) \to
\WW_N\big(\Rbar_\cU/p\Rbar_\cU\big)$ is injective by inductive
hypothesis. Consider the following commutative diagram
$$\begin{array}{ccccccc}
 J'\otimes \WW_{N+1}\big(\Rbar_\cU/p\Rbar_\cU\big) &
\longrightarrow & J\otimes \WW_{N+1}\big(\Rbar_\cU/p\Rbar_\cU\big)
& \longrightarrow & J_N\otimes
\WW_{N+1}\big(\Rbar_\cU/p\Rbar_\cU\big)\cr \big\downarrow & &
\big\downarrow & & \big\downarrow\cr 0\longrightarrow  \V^N
\WW_{N+1}\big(\Rbar_\cU/p\Rbar_\cU\big) & \longrightarrow &
\WW_{N+1}\big(\Rbar_\cU/p\Rbar_\cU\big) & \longrightarrow &
\WW_N\big(\Rbar_\cU/p\Rbar_\cU\big),\cr
\end{array}$$where in the first row the tensor products are over
the ring $W_{N+1}$. The rows are exact and we have proven that the
left and right vertical maps are injective. Therefore the map
$J\otimes_{W_{N+1}} \WW_{N+1}\big(\Rbar_\cU/p\Rbar_\cU\big) \to
\WW_{N+1}\big(\Rbar_\cU/p\Rbar_\cU\big)$ is injective as
well.\smallskip

(3) Proceeding by induction on $n$ it suffices to prove the claim for $n=1$ i.~e., that the sequence $0 \lra \F_p \lra \Rbar_\cU/p\Rbar_\cU
\stackrel{\varphi-1}{\lra} \Rbar_\cU/p\Rbar_\cU \lra 0$ is exact.  Since $R_\cU$ is $p$-adically complete every finite extensions of $R_\cU$ is also $p$-adically
complete and by Hensel's lemma the connected components of the associated spectrum are in bijection with the connected components of its reduction modulo $p$. In
particular the scheme $\Spec\bigl(\Rbar_\cU/p\Rbar_\cU\bigr)$ has as many connected components as $\Spec\bigl(\Rbar_\cU\otimes_{\cO_K} {\cO_M}\bigr)$ has. By
construction these coincide with the connected components of $\Spec\bigl(R_\cU\otimes_{\cO_K} M\bigr)$. The exactness on the left and in the middle follow from
Artin--Schreier theory. The cokernel of $\varphi-1$ on~$\Rbar_\cU/p\Rbar_\cU$ is contained in ${\rm H}^1\bigl(\Rbar_\cU/p\Rbar_\cU,\F_p\bigr)$ by Artin--Schreier
theory. Let $\overline{Z}\to \Spec\bigl(\Rbar_\cU/p\Rbar_\cU\bigr)$ be an $\F_p$--torsor. It can be lifted to an $\F_p$--torsor $Z \to \Spec(S)$ over a finite and
normal extension $R_\cU\subset S$, \'etale after inverting~$p$. In particular $Z=\Spec(T)$ is affine with~$T$ normal so that $Z(\Rbar_\cU)$ admits a section by
definition of~$\Rbar_\cU$. Then $\overline{Z}$ admits a section as well and thus it is the trivial torsor. Therefore ${\rm H}^1(\Rbar_\cU/p\Rbar_\cU,\F_p)=0$ and
the  claim follows.

\end{proof}

\begin{corollary}\label{cor:WnXisflatoverWnVbar}
For every $n$ the sheaf $\WW_{n,\Kbar}$ is a sheaf of flat $W_n$-modules. Furthermore, $\varphi$ induces an isomorphism $\WW_{n,\Kbar}/
\bigl([\varepsilon]^{\frac{1}{p^{r+1}}}-1\bigr) \WW_{n,\Kbar} \lra \WW_{n,\Kbar}/ \bigl([\varepsilon]^{\frac{1}{p^r}}-1\bigr) \WW_{n,\Kbar} $ for every $r\in\N$. In
the formal case the sequence $0\lra \Z_p \lra \WW_{n,\Kbar} \stackrel{\varphi-1}{\lra} \WW_{n,\Kbar}\lra 0$ is exact.
\end{corollary}
\begin{proof} It suffices to show the claims for the pre-sheaf $\WW_n\bigl(\cO_{\fXKbar}/p\cO_{\fXKbar}\bigr)$
and further after passing to its localizations at every small
affine $\cU\subset X$. The claims follow from
lemma \ref{lemma:uglylemma}.
\end{proof}

\subsection{The sheaf $\bA_{\rm cris,M}^\nabla$.} \label{sec:sheafAcrisnabla}

We start with a formal definition. A $\WW(k)$--{\it divided power}
($\WW(k)$--DP) sheaf of algebras in $\Sh(\fX_M)$ or
$\Sh(\fX_M)^\N$ is a triple $(\cF, \cI,\gamma)$ consisting
of\enspace (1) a sheaf of $\WW(k)$--algebras $\cF\in \Sh(\fX_M)$
(resp.~an inverse system of sheaves of $\WW(k)$--algebras
$\{\cF_n\}\in \Sh(\fX_M)^\N$),\enspace (2) a sheaf of ideals
$\cI\subset \cF$ (resp.~an inverse system of sheaves of ideals
$\{\cI_n\subset \cF_n\}$), \enspace (3) maps $\gamma_i\colon \cI
\to \cI$ for~$i\in\N$ such that for every object $(\cU,\cW)$ the
triple $\bigl(\cF(\cU,\cW),\cI(\cU,\cW),\gamma_{(\cU,\cW)}\bigr)$
(resp.~for every~$n$ the triple
$\bigl(\cF_n(\cU,\cW),\cI_n(\cU,\cW),\gamma_{(\cU,\cW)}\bigr)$) is
a DP algebra compatible with the standard DP structure on the
ideal $p\WW(k)$ in the sense of \cite[Ch.~3]{berthelot_ogus}.
Given a sheaf of $\WW(k)$--algebras~$\cG$ and an ideal~$\cJ\subset
\cG$ (resp.~an inverse system of sheaves of
$\WW(k)$--algebras~$\cG$ and ideals $\cJ\subset \cG$) the
$\WW(k)$--divided power envelope of~$\cG$ with respect to~$\cJ$ is
a $\WW(k)$--DP sheaf of algebras $(\cF, \cI,\gamma)$ and a
morphism $\cG \to \cF$ of sheaves (or inverse systems of sheaves)
of $\WW(k)$--algebras, such that~$\cJ$ maps to~$\cI$, which is
universal for morphisms as sheaves (or inverse systems of sheaves)
of $\WW(k)$--algebras from~$\cG$ to $\WW(k)$--DP sheaves of
algebras $\cF'$ such that $\cJ$ maps to the sheaf of ideals of
$\cF'$ on which the divided power structure is defined. \smallskip

We'd like to consider the $\WW(k)$--DP envelope of the sheaf
$\bA_{\rm inf,M}^+\in \Sh(\fX_M)$ with respect to the sheaf of
ideals $\Ker(\theta_M)$. One could use the general machinery of
\cite[Thm. I.2.4.1]{berthelot1} to guarantee that it exists but we
prefer to provide a different more explicit description.
We start with:

\begin{lemma}\label{lemma:restrictenvelope} Let $\cG$ be a sheaf
of $\WW(k)$--algebras and  let $\cJ\subset \cG$ be an ideal. Assume that  the $\WW(k)$--divided power envelope $(\cF, \cI,\gamma)$ of $\cG$ with respect to $\cI$
exists. Then for every $\cU\in X^{\rm et}$ the restriction of $(\cF, \cI,\gamma)$ to $\fU_M$ is the $\WW(k)$--divided power (DP) envelope of $\cG\vert_{\fU_M}$ with
respect to $\cI\vert_{\fU_M}$.
\end{lemma}
\begin{proof} Let $j\colon \fX_M \to \fU_M$ be the continuous morphism of
sites sending $(\cV,\cW)\mapsto (\cV,\cW)\times_{(\fX,\fX_K)} (\cU,\cU_K)$. Let $j_!\colon \Sh(\fU_M)\to \Sh(\fX_M)$ be the functor of extension by zero. It is the
right adjoint of the functor $j^\ast\colon \Sh(\fX_M)\to \Sh(\fU_M)$ which is the restriction functor from $\fX_M$ to the subcategory $\fU_M\subset \fX_M$; see
\cite{erratum}. Let~$f\colon \cG\vert_{\fU_M}\to\cF'$ be a morphism of sheaves such that $(\cF',\cI',\gamma')$ is a $\WW(k)$--DP sheaf of algebras on $\fU_M$ and
$f(\cJ)\subset \cI'$. Then $\bigl(j_!(\cF'),j_!(\cI'),j_!(\gamma')\bigr)$ is  a $\WW(k)$--DP sheaf of algebras and by adjointness of $j_!$ we get a morphism
$j_!(f)\colon \cG \to j_!(\cF')$. The latter extends uniquely to a morphism $\cF \to j_!(\cF')$ of $\WW(k)$--DP sheaves of algebras by the universal property of
$\cF$. Restricting to $\fU_M$ we get a morphism $\cF\vert_\fUM \to \cF'$ of $\WW(k)$--DP sheaves of algebras extending $f$. Using the adjointness of $j_!$ and the
universal property of $\cF$ one proves that such a morphism is unique. The claim follows.
\end{proof}

By lemma \ref{lemma:betaastG} for every object $(\cU,\cW)\in \fX_M$ we
have a natural identification $\WW_{n,M}(\cU,\cW)\cong
\beta_{M,\Kbar,\ast}\left(\WW_n(\cU,\cW)\right)^{{\rm Gal
}(\Kbar/M)}$. We used the fact that
$\beta_{M,\Kbar}^{-1}\bigl(\WW_{n,M}\bigr)
(\cU,\cW)=\WW_n(\cU,\cW)$. Define $\bA_{\rm cris,n,M}^\nabla$ to
be the sheaf on $\fX_M$ associated to the pre-sheaf given by
$$(\cU,\cW)\mapsto \left(A_{\rm cris, n}\otimes_{W_n}
\left(\WW_n(\cU,\cW)\right)\right)^{{\rm Gal}(\Kbar/M)}.$$Let
$\theta_{M,n}\colon\bA_{\rm cris,n,M}^\nabla \lra
\cO_{\fX_M}/p^n\cO_{\fX_M}$ be the map of sheaves induced by the
map of pre-sheaves $$\left(A_{\rm cris, n}\otimes_{W_n}
\Bigl(\WW_n(\cU,\cW))\Bigr)\right)^{{\rm Gal}(\Kbar/M)}\to
\bigl(\cO_{\fX_M}/p^n\cO_{\fX_M}\bigr)(\cU,\cW),
$$given by $\theta_n\otimes
\theta_{n,M}.$ Here using again lemma \ref{lemma:betaastG}  we have
identified
$$\left(\beta_{M,\Kbar,\ast}\bigl(\cO_\fXKbar/p^n\cO_\fXKbar\bigr)
(\cU,\cW)\right)^{{\rm Gal}(\Kbar/M)}$$ with $\bigl(\cO_{\fX_M}/p^n\cO_{\fX_M}\bigr)(\cU,\cW)$. We get from lemma \ref{lemma:krcn}
that~$\Ker(\theta_{M,n})(\cU,\cW)$ coincides with the ${\rm Gal }(\Kbar/M)$--invariants of  the ideal of $A_{\rm cris, n}\otimes_{W_n}
\WW_n\bigl(\cO_{\fX_M}(\cU,\cW)/p\cO_\fXKbar(\cU,\cW)\bigr)$ generated by $\Ker(\theta_n)^{\rm DP}$. Such an ideal has $\WW(k)$--DP structure thanks to corollary
\ref{cor:WnXisflatoverWnVbar}. In particular the sheaf $\Ker(\theta_{n,M})$ is endowed with $\WW(k)$--DP structure as well. Using the identification
$\WW_{n,M}(\cU,\cW)\cong \left(\WW_n(\cU,\cW)\right)^{{\rm Gal }(\Kbar/M)}$ we also have a natural map
$$h_{M,n}\colon\WW_{n,M}\lra \bA_{\rm
cris,n,M}^\nabla.$$Since $\cO_{\fX_M}$ is a sheaf of
$\cO_M$--algebras, $\WW_{n,M}$ is a sheaf of $\OMun$--algebras.
Consider the map $r_{n+1}\colon \WW_{n+1}\to \WW_n$ defined by the
natural projection composed with Frobenius. Now we tensor with
$A_{\rm cris, n+1}$ over $W_n $. Since $\xi_n$ is the image of
$\xi_{n+1}$ via~$r_{n+1}$, taking ${\rm
Gal}(\Kbar/M)$--invariants, we get a natural map $r_{M,n+1}\colon
\bA_{\rm cris,n+1,M}^\nabla\to \bA_{\rm cris,n,M}^\nabla$. Denote
by $\bA_{\rm cris,M}^\nabla$ the sheaf in $\Sh(\fXKbar)^\N$
defined by the family $\{\bA_{\rm cris,n,M}^\nabla\}_n$ with the
transition maps $\{r_{M,n+1}\}_n$.

\begin{proposition}\label{lemma:Acrisnnabla}
1) The sheaf of rings $\bA_{\rm cris,n,M}^\nabla$, with the sheaf
of ideals~$\Ker\bigl(\theta_{n,M}\bigr)$ and the natural map
$h_{n,M}\colon\WW_{n,M}\lra \bA_{\rm cris,n,M}^\nabla$, is the
$\WW(k)$--DP envelope of $\WW_{n,M}$ with respect to
$\Ker\bigl(\theta_{n,M}\bigr)$.\smallskip

2) The system of sheaves of rings $\bA_{\rm cris,M}^\nabla$, with
the sheaf of ideals~$\{\Ker(\theta_{n,M})\}_n$ and the natural map
$h_M=\{h_{n,M}\}_n\colon \bA_{\rm inf,M}^+ \lra \bA_{\rm
cris,M}^\nabla$ is the $\WW(k)$--DP envelope of $\bA_{\rm
inf,M}^+$ with respect to $\{\Ker(\theta_{n,M})\}_n$.\smallskip

3) The Frobenius map~$\varphi\colon\WW_{n,M}\to \WW_{n,M} $
defines a map $\varphi_n\colon \bA_{\rm cris,n,M}^\nabla\to
\bA_{\rm cris,n,M}^\nabla$.\smallskip

4) For varying~$n$ the maps $\{\varphi_n\}_n$ define a morphism
$\varphi\colon  \bA_{\rm cris,M}^\nabla\to \bA_{\rm
cris,M}^\nabla$ in~$\Sh(\fX_M)^\N$;\smallskip

5) If $M_1\subset M_2$ is a Galois field extension we have a
natural isomorphism $\beta_{M_1,M_2}^\ast\left(\bA_{\rm
cris,n,M_1}^\nabla\right)\cong \bA_{\rm cris,n,M_2}^\nabla $ of
$\WW(k)$--DP sheaves of algebras, compatible with Frobenius and
the natural structures of $\WW_{n,M_2}$--sheaves of
modules;\smallskip

6) We have a natural isomorphism
$\beta_{M_1,M_2}^\ast\left(\bA_{\rm cris,M_1}^\nabla\right)\cong
\bA_{\rm cris,M_2}^\nabla $ of $\WW(k)$--DP sheaves of algebras,
compatible with Frobenius and the natural structures of $\bA_{\rm
inf,M_2}^+$--sheaves of modules.\smallskip

\end{proposition}
\begin{proof} (1) Let $\cG$ be a $\WW(k)$--DP sheaf of  algebras
and let $f\colon \WW_{n,M} \to \cG$ be a morphism sending $\Ker\bigl(\theta_{n,M}\bigr)$ to the DP ideal of $\cG$. Due to~\ref{lemma:betaastG}, for every
$(\cU,\cW)\in \fX_M$, we have $\cG(\cU,\cW)=\left(\beta_{M,\Kbar}^{-1}(\cG)(\cU,\cW)\right)^{{\rm Gal}(\Kbar/M)}$. Since $\beta_{M,\Kbar}^{-1}\bigl(\WW_{n,M}\bigr)
(\cU,\cW)$ is equal to $\WW_n(\cU,\cW)$, the map $\beta_{M,\Kbar}^{-1}(f)$ extends uniquely to a map $A_{\rm cris, n}\otimes_{W_n} \WW_n(\cU,\cW) \to
\beta_{M,\Kbar}^{-1}(\cG)(\cU,\cW)$. Taking ${\rm Gal}(\Kbar/M)$--invariants and using the definition of $\bA_{\rm cris,n,M}^\nabla$, we get a unique map $\bA_{\rm
cris,n,M}^\nabla\to \cG$ as $\WW(k)$--DP sheaves of  algebras extending $f$. This proves the universal property of $\bA_{\rm cris,n,M}^\nabla$. Claim (2) follows
from claim (1). Recall that Frobenius on $W_n$ extends to an operator $\varphi$ on~$A_{\rm cris, n}$. Claims~(3) and~(4) follow. (5) The existence of a natural map
$\beta_{M_1,M_2}^\ast\left(\bA_{\rm cris,n,M_1}^\nabla\right)\to \bA_{\rm cris,n,M_2}^\nabla $, compatible with Frobenius and the structure of
$\WW_{n,M_2}$--modules follows from the definition of $\bA_{\rm cris,n}^\nabla$. To check that it is an isomorphism it suffices to prove it for the stalks. The
stalk of $\beta_{M_1,M_2}^{-1}\bigl(\WW_{n,M_1}\bigr)$ at a point~$x\in X$ is $\WW_n\bigl(\cO_{\fX_{M_1},x}/p\cO_{\fX_{M_1},x}\bigr)$; see \cite[Prop.
4.4]{andreatta_iovita}. The latter is a $W_n$--algebra.  It then follows from~\ref{lemma:krcn} that the stalk $\bA_{\rm cris,n,M_1,x}^\nabla$ is $A_{\rm cris,
n}\otimes_{W_n} \WW_n\left(\cO_{\fX_{M_1},x}/p\cO_{\fX_{M_2},x}\right)$. Since $\cO_{\fX_{M_1},x}\cong \cO_{\fX_{M_2},x}$, the claim follows. Claim~(6) follows
from~(5)
\end{proof}

The next step is to study the localization $\bA_{\rm
cris,M}^\nabla$ over small affines. In analogy with the classical
case of~$A_{\rm cris}$ recalled in \S \ref{sec:Notation} we provide
a second essentially equivalent definition of $\bA_{\rm
cris,M}^\nabla$ via the system $\bA_{\rm cris,M}^\nabla/p^n
\bA_{\rm cris,M}^\nabla$ for varying $n\in\N$. Let $\bA_{\rm
cris,n,M}^{'\nabla}$ be the sheaf on $\fX_M$ associated to the
pre-sheaf given by

$$(\cU,\cW)\mapsto \bigl(\bigl(A_{\rm cris}/p^n A_{\rm cris}\bigr)\otimes_{W_n}
\WW_n(\cU,\cW)\bigr)^{{\rm Gal}(\Kbar/M)}.$$

Let $\theta_{M,n}'\colon\bA_{\rm cris,n,M}^{'\nabla} \lra
\cO_{\fX_M}/p^n\cO_{\fX_M}$ be the  map $\theta_{M,n}\circ
\bigl(q_n \tensor \varphi\bigr)$ (we refer to \S \ref{sec:Notation}
for the map $q_n\colon A_{\rm cris}/p^n A_{\rm cris} \to A_{\rm
cris, n}$). Denote by $\bA_{\rm cris,M}^{'\nabla}$ the sheaf in
$\Sh(\fX_M)^\N$ defined by the family $\{\bA_{\rm
cris,n,M}^{'\nabla}\}$ with the transition maps $r_{M,n+1}'\colon
\bA_{\rm cris,n+1,M}^{'\nabla}\to \bA_{\rm cris,n,M}^{'\nabla}$
induced by $r_{n+1}\colon \WW_{n+1,M}\to \WW_{n,M}$. For
every~$n\in\N$ define the map of sheaves
$$q_{M,n}\colon \bA_{\rm cris,n,M}^{'\nabla} \lra\bA_{\rm
cris,n,M}^\nabla$$associated to the map of pre-sheaves inducing
$q_n\colon A_{\rm cris}/p^n A_{\rm cris} \to A_{\rm cris, n}$ and
Frobenius  on $\WW_{n,M}(\cU,\cW)$. Consider the map of sheaves
$$u_{n,M}\colon \bA_{\rm cris,n+1,M}^\nabla\lra \bA_{\rm
cris,n,M}^{'\nabla} $$associated to the map of pre-sheaves which
induces $u_n\colon A_{\rm cris, n+1}\to A_{\rm cris}/p^n A_{\rm
cris}  $ (see \ref{sec:Notation}) and the natural projection $
\WW_{n+1,M}(\cU,\cW)\to \WW_{n,M}(\cU,\cW)$.

\begin{proposition}\label{lemma:Acrisnnabla'}
1) The sheaf of rings $\bA_{\rm cris,n,M}^{'\nabla}$ with the
sheaf of ideals~$\Ker(\theta_{n,M}')$ and the natural map
$h_{n,M}'\colon\WW_{n,M}\lra \bA_{\rm cris,n,M}^{'\nabla}$ is the
$\WW(k)$--DP envelope of $\WW_{n,M}$ with respect to
$\Ker\bigl(\theta_{n,M}\circ \varphi\bigr)$.\smallskip

2) The system of sheaves of rings $\bA_{\rm cris,M}^{'\nabla}$,
with the sheaf of ideals~$\{\Ker(\theta_{n,M}')\}_n$ and the
natural map $h_M'=\{h_{n,M}'\}_n\colon \bA_{\rm inf,M}^+ \lra
\bA_{\rm cris,M}^{'\nabla}$ is the $\WW(k)$--DP envelope of
$\bA_{\rm inf,M}^+$ with respect to $\{\Ker(\theta_{n,M}\circ
\varphi)\}_n$. \smallskip

3) Frobenius on~$\WW_{n,M}$ defines  maps $\varphi_n'\colon
\bA_{\rm cris,n,M}^{'\nabla}\to \bA_{\rm cris,n,M}^{'\nabla}$
which are compatible for varying~$n$ and give a morphism
$\varphi:=\{\varphi_n'\}\colon \bA_{\rm cris,M}^{'\nabla}\to
\bA_{\rm cris,M}^{'\nabla}$.

4)  For every~$n\in\N$ we have  $q_{M,n} \circ u_{M,n}= r_{M,n+1}$
and~$u_{M,n} \circ q_{M,n+1}= r_{M,n+1}'$. Furthermore, the
following diagrams commute
$$\begin{array}{ccccc} \bA_{\rm
cris,n+1,M}^\nabla & \stackrel{u_{n,M}}{\lra} & \bA_{\rm
cris,n,M}^{'\nabla} & \stackrel{q_{n,M}}{\lra} & \bA_{\rm
cris,n,M}^\nabla\cr \big\downarrow \varphi_{n+1} & &
\big\downarrow \varphi_n'  & & \big\downarrow \varphi_n\cr
\bA_{\rm cris,n+1,M}^\nabla & \stackrel{u_{n,M}}{\lra} & \bA_{\rm
cris,n,M}^{'\nabla} & \stackrel{q_{n,M}}{\lra} & \bA_{\rm
cris,n,M}^\nabla\cr
\end{array}
$$and

$$\begin{array}{ccccc}
\WW_{n+1,M} & \lra & \WW_{n,M} & \stackrel{\varphi}{\lra} &
\WW_{n,M} \cr \big\downarrow h_{n+1,M} & & \big\downarrow h_{n,M}'
& & \big\downarrow h_{n,M}\cr \bA_{\rm cris,n+1,M}^\nabla &
\stackrel{u_{n,M}}{\lra} & \bA_{\rm cris,n,M}^{'\nabla} &
\stackrel{q_{n,M}}{\lra} & \bA_{\rm cris,n,M}^\nabla\cr
\big\downarrow \theta_{n+1,M} & & \big\downarrow \theta_{n,M}'  &
& \big\downarrow \theta_{n,M}\cr \cO_{\fX_M}/p^{n+1}\cO_{\fX_M} &
\lra & \cO_{\fX_M}/p^n\cO_{\fX_M} & = &
\cO_{\fX_M}/p^n\cO_{\fX_M}.\cr
\end{array}$$

5) If $M_1\subset M_2$ is a Galois field extension we have a
natural isomorphism $\beta_{M_1,M_2}^\ast\left(\bA_{\rm
cris,n,M_1}^{'\nabla}\right)\cong \bA_{\rm cris,n,M_2}^{'\nabla} $
of $\WW(k)$--DP sheaves of algebras, compatible with Frobenius and
the natural structures of $\WW_{n,M_2}$--sheaves of
modules;\smallskip

6) We have a natural isomorphism
$\beta_{M_1,M_2}^\ast\left(\bA_{\rm
cris,M_1}^{'\nabla}\right)\cong \bA_{\rm cris,M_2}^{'\nabla} $ of
$\WW(k)$--DP sheaves of algebras, compatible with Frobenius and
the natural structures of $\bA_{\rm inf,M_2}^+$--sheaves of
modules.\smallskip

\end{proposition}

\noindent Write $q_M:=\{q_{n,M}\}_n\colon \bA_{\rm
cris,M}^{'\nabla} \to \bA_{\rm cris,M}^\nabla$
and~$u_M:=\{u_{n,M}\}_n\colon \bA_{\rm cris,M}^\nabla \to \bA_{\rm
cris,M}^{'\nabla}$.

\begin{lemma}\label{lemma:exactsequenceAcrisnnabla'}

a) For every $n$ we have an exact sequence $$0 \lra \bA_{\rm
cris,n,M}^{'\nabla} \stackrel{a}{\lra} \bA_{\rm
cris,n+1,M}^{'\nabla} \stackrel{b}{\lra} \bA_{\rm
cris,1,M}^{'\nabla}\lra 0,$$where $b= r_{2,M}'\circ \cdots \circ
r_{n,M}' \circ r_{n+1,M}'$ and $a$ is the map of sheaves
associated to the Vershiebung $ \V\colon \WW_{n,M}\lra
\WW_{n+1,M}$.\smallskip

b) We have $\bA_{\rm
cris,1,\Kbar}^{'\nabla}=\cO_\fXKbar/p\cO_\fXKbar
\bigl[\delta_0,\delta_1,\ldots\bigr]/(\delta_m^p\bigr)_{m\geq
0}$.\smallskip

c)  We have $\bA_{\rm cris,n,\Kbar}^{'\nabla}(\cU,\cW)=\bigl(A_{\rm cris}/p^n A_{\rm cris}\bigr)\otimes_{W_n} \bigl(\WW_{n,\Kbar}(\cU,\cW)\bigr)$ for every
$(\cU,\cW)\in \fXKbar$. In particular the sequence in~(a) is exact also as sequence of presheaves for $M=\Kbar$.
\end{lemma}
\begin{proof}

a) Certainly $b\circ a=0$. To check exactness we study the stalks. Since for any sheaf~$\cG$ on~$\fX_M$ we have $\beta_{M,\Kbar}^\ast(\cG)_x=\cG_x$ by \cite[Prop.
4.4]{andreatta_iovita} and since $\beta_{M,\Kbar}^\ast\left(\bA_{\rm cris,n,M}^{'\nabla}\right)\cong \bA_{\rm cris,n,\Kbar}^{'\nabla} $ by proposition
\ref{lemma:Acrisnnabla'} it suffices to prove the claim for~$M=\Kbar$. The kernel $H$ of the natural projection $s\colon \WW_{n+1}\to
\WW_1=\cO_\fXKbar/p\cO_\fXKbar$ is identified with $\WW_n$ via Vershiebung. It is a $W_{n+1}$--module via the projection $W_{n+1} \to W_n$ composed with Frobenius
on~$W_n$  and the  structure of $W_n$--module of $\WW_n$. Hence $$\bigl(A_{\rm cris}/p^{n+1} A_{\rm cris}\bigr)\otimes H\cong \bigl(A_{\rm cris}/p^{n+1} A_{\rm
cris}\bigr)\otimes \WW_n,$$where $\otimes$ stands for $\otimes_{W_{n+1}}$. Since~$s(\xi_{n+2}^p)\equiv p^{\frac{1}{p^n}}$ we have
$$\bigl(A_{\rm cris}/p^{n+1} A_{\rm
cris}\bigr)\otimes_{W_{n+1}} \WW_1\cong
\cO_\fXKbar/p^{\frac{1}{p^n}}\cO_\fXKbar
\bigl[\delta_0,\delta_1,\ldots\bigr]/(\delta_m^p\bigr)_{m\geq 0}
.$$Recall from \ref{cor:WnXisflatoverWnVbar} that Frobenius to the
$n$--th power~$\varphi^n$ gives an isomorphism
$\cO_\fXKbar/p^{1/p^n}\cO_\fXKbar \ra \cO_\fXKbar/p\cO_\fXKbar$.
Hence
$$\bigl(A_{\rm cris}/p^{n+1} A_{\rm cris}\bigr)\otimes_{W_{n+1}}
\bigl(\cO_\fXKbar/p\cO_\fXKbar\bigr) \cong
\cO_\fXKbar/p\cO_\fXKbar
\bigl[\delta_0,\delta_1,\ldots\bigr]/(\delta_m^p\bigr)_{m\geq
0}.$$The composite of~$s$ with~$\varphi^n$ is $r_2\circ \cdots
\circ r_n \circ r_{n+1}\colon \WW_{n+1}\lra
\cO_\fXKbar/p\cO_\fXKbar$. This proves the exactness of the
sequence displayed in Claim a) with the exception of the exactness
on the left. We prove the left exactness on stalks. Let $x$ be a
point of $X$. Note that $\xi=\bigl[\widetilde{p}\bigr]-p$. Since
the ideal generated by~$p$ admits $\WW(k)$--DP in $A_{\rm cris}$,
the $\WW(k)$--DP envelope of~$\xi$ in $A_{\rm cris}$ coincides
with the $\WW(k)$--DP envelope of $\bigl[\widetilde{p}\bigr]$ in
$A_{\rm cris}$ i.~e.,
$$A_{\rm cris}/p^n A_{\rm cris} \cong
W_n \langle \widetilde{p}_{n+1} \rangle=
W_n\bigl[\delta_0,\delta_1,\ldots\bigr]/(p\delta_0-\widetilde{p}_{n+1}^p,p
\delta_{m+1}-\delta_m^p\bigr)_{m\geq 0}.$$Define
$B:=\WW_n\bigl(\cO_{\fXKbar,x}/p\cO_{\fXKbar,x}\bigr)
\bigl[\delta_0,\delta_1,\ldots\bigr]/(p
\delta_{m+1}-\delta_m^p\bigr)_{m\geq 0}$. Similarly, denote by
$$C:=\WW_{n+1}\bigl(\cO_{\fXKbar,x}/p\cO_{\fXKbar,x}\bigr)
\bigl[\delta_0,\delta_1,\ldots\bigr]/(p
\delta_{m+1}-\delta_m^p\bigr)_{m\geq 0}$$ and write
$D:=\cO_{\fXKbar,x}/p\cO_{\fXKbar,x}
\bigl[\delta_0,\delta_1,\ldots\bigr]/(\delta_m^p\bigr)_{m\geq 0}$.
Note that $B/ (p\delta_0-\widetilde{p}_{n+1}^p) B$ is the stalk
$\bA_{\rm cris,n,\Kbar,x}^{'\nabla}$ of $\bA_{\rm
cris,n,\Kbar}^{'\nabla} $ at $x$, $C/
(p\delta_0-\widetilde{p}_{n+2}^p) C$ is  $\bA_{\rm
cris,\Kbar,n+1,x}^{'\nabla}$ and $D/\widetilde{p}_{n+2}^p
D=\bA_{\rm cris,1,\Kbar,x}^{'\nabla}$. We have the following
commutative diagram:
$$\begin{array}{cclclclcc}
0 & \lra & B & \stackrel{a_x}{\lra} & C & \stackrel{s_x}{\lra} & D
& \lra & 0\cr & & \big\downarrow(p\delta_0-\widetilde{p}_{n+1}^p)
& & \big\downarrow(p\delta_0-\widetilde{p}_{n+2}^p) & &
\big\downarrow -\widetilde{p}_{n+2}^p\cr 0 & \lra & B &
\stackrel{a_x}{\lra} & C & \stackrel{s_x}{\lra} & D & \lra & 0.\cr
\end{array}$$Here, $a_x$ sends $\delta_i\mapsto \delta_i$ and induces Vershiebung
$\WW_n\bigl(\cO_{\fXKbar,x}/p\cO_{\fXKbar,x}\bigr)\to \WW_{n+1}\bigl(\cO_{\fXKbar,x}/p\cO_{\fXKbar,x}\bigr)$. Since~$B$ (resp.~$C$) is a free
$\WW_n\bigl(\cO_{\fXKbar,x}/p\cO_{\fXKbar,x}\bigr)$--module (resp.~$\WW_{n+1}\bigl(\cO_{\fXKbar,x}/p\cO_{\fXKbar,x}\bigr)$--module) with basis  given by the
monomials in the $\delta_i$'s and Vershiebung is injective the map~$a_x$ is injective. The map~$s_x$ is the natural projection. Since also~$D$ is a free
$\cO_{\fXKbar,x}/p\cO_{\fXKbar,x}$--module with basis given by the monomials in the $\delta_i$'s the rows in the displayed diagram are exact. The sequence of
cokernels $B/ (p\delta_0-\widetilde{p}_{n+1}^p) B\to C/ (p\delta_0-\widetilde{p}_{n+2}^p) C$ is the map on stalks associated to~$a$. Note that
$\widetilde{p}_{n+1}=\widetilde{p}_{n+2}^p=p^{1/p^n}$ in~$\OKbar/p\OKbar$ and the kernel of multiplication by $p^{1/p^n}$ on~$D$ is $p p^{-1/p^n}
D=p^{\frac{p^n-1}{p^n}} D=\widetilde{p}_{n+1}^{p^n-1} D$. Choose $y \in D$ and let~$x\in C$ be the lift defined taking the Teichm\"uller lifts of the coefficients
of $x$ with respect to the $\cO_{\fXKbar,x}/p\cO_{\fXKbar,x}$--basis of $D$ given by the monomials in the $\delta_i$'s.  In particular $\widetilde{p}_{n+1}^{p^n}
y=\widetilde{p} y=0$. Put $z:=\sum_{i=0}^{p^n-1} p^i \delta_0^i \widetilde{p}_{n+1}^{p^n-i-1} y$. Then
$$(p\delta_0-\widetilde{p}_{n+1}) z= \sum_{i=0}^{p^n-1}  \delta_0^{i+1}
p^{i+1} \widetilde{p}_{n+1}^{p^n-i-1} y-\sum_{i=0}^{p^n-1}
\delta_0^i p^i \widetilde{p}_{n+1}^{p^n-i} y=\delta_0^{p^n}
p^{p^n} y -\widetilde{p}_{n+1}^{p^n} y=0$$and
$s_x(z)=p^{\frac{p^n-1}{p^n}} y$. This proves that the kernel of
multiplication by $p\delta_0-\widetilde{p}_{n+2}^p$ on $C$
surjects onto the kernel of multiplication
by~$\widetilde{p}_{n+2}^p$ on~$D$. The claimed left exactness
follows from this using the snake lemma in the displayed diagram.

b)  follows using that $A_{\rm cris}/p A_{\rm cris}\cong
\OKbar/p\OKbar
\bigl[\delta_0,\delta_1,\ldots\bigr]/(\delta_m^p\bigr)_{m\geq 0}$;
see \ref{sec:Notation}.

c) We prove the claim by induction on~$n$. For~$n=1$ it follows
from~(b) since~$\bA_{\rm cris,1,\Kbar}^{'\nabla}$ is a direct sum
of copies of~$\cO_\fXKbar/p\cO_\fXKbar$. Suppose the claim proved
for~$n$. Write $\A'_n$ for the presheaf $(\cU,\cW)\mapsto
\bigl(A_{\rm cris}/p^n A_{\rm cris}\bigr)\otimes_{W_n}
\Bigl(\WW_n(\cU,\cW)\Bigr)$. Consider the following commutative
diagram:

$$
\begin{array}{ccccccccc}  & & \A'_n &
\stackrel{1\tensor a}{\lra} & \A'_{n+1} & \stackrel{1\tensor
b}{\lra} & \A'_1 & \lra & 0 \cr & & \big\downarrow & &
\big\downarrow & & \big\downarrow \cr 0 & \lra & \bA_{\rm
cris,n,\Kbar}^{'\nabla}(\cU,\cW) & \stackrel{a}{\lra} & \bA_{\rm
cris,n+1,\Kbar}^{'\nabla}(\cU,\cW) & \stackrel{b}{\lra}& \bA_{\rm
cris,1,\Kbar}^{'\nabla}(\cU,\cW) \cr
\end{array}
$$The bottom row is exact due to~(a). The top one is exact as well (see the proof of~(a)).
This fact together with the inductive hypothesis and a diagram chase imply the
claim.
\end{proof}

To conclude the comparison between $\bA_{\rm cris,n,M}^{'\nabla}$
and $\bA_{\rm cris,n,M}^\nabla$ we prove the following:

\begin{lemma}\label{lemma:compareAcrisnablaAcrisnabla'} For
positive integers $m >  n$ the map $u_{n,M}\circ r_{n+2,M}
\circ\cdots\circ r_{m,M}\colon \bA_{\rm cris,m,M}^\nabla\lra
\bA_{\rm cris,n,M}^{'\nabla}$ induces an isomorphism $\bA_{\rm
cris,m,M}^\nabla/p^n \bA_{\rm cris,m,M}^\nabla \lra \bA_{\rm
cris,n,M}^{'\nabla}$.
\end{lemma}
\begin{proof} As in the proof of lemma \ref{lemma:exactsequenceAcrisnnabla'}(b) it
suffices to prove the lemma for $M=\Kbar$. We can write multiplication by~$p^n$ on~$\WW_m$ as the composite of $\V^n\circ \varphi^n$ with $\V=$Vershiebung and
$\varphi=$Frobenius. Since~$\varphi$ is surjective on~$\cO_\fXKbar/p\cO_\fXKbar$ by~\cite[Lemma 4.3(v)]{andreatta_iovita} we deduce that~$\WW_m/p^n \WW_m\cong
\WW_n$ where the map is the natural projection. Via this identification the map $u_n \circ r_{n+1} \circ\cdots\circ r_m\colon \WW_n \lra \WW_n$ is~$\varphi^{m-n-1}$
and, hence, sends~$\xi_m\mapsto \xi_{n+1}$. Note that~$\widetilde{p}_m=\xi_m+p$. Since~$p$ and~$\xi_m$ admit DP in~$A_{\rm cris, m}(\OKbar)$ also~$\widetilde{p}_m$
admits DP. We compute $\V^s\bigl(\widetilde{p}_m\bigr)^{p^n}=\left(\varphi^s \bigl(V^s(\widetilde{p}_m)\bigr)\right)^{p^{n-s}}=\bigl(p^s
\widetilde{p}_m\bigr)^{p^{n-s}}= p^{s p^{n-s}} \widetilde{p}_m^{p^{n-s}}=p^{s p^{n-s}} p^{n-s}! \widetilde{p}_m^{[p^{n-s}]} $. This is~$0$ in~$A_{\rm cris, m}/p^n
A_{\rm cris, m}$ since~$s p^{n-s}+n-s\geq n$. The element~$\widetilde{p}_m^{p^n}$ generates the kernel of~$\varphi^{m-n-1}$ on~$\cO_\fX/p\cO_\fX$; see the proof
of~\ref{lemma:exactsequenceAcrisnnabla'}. Hence $\bigl(\V^s\bigl(\widetilde{p}_m\bigr)^{p^n} \bigr)_{0\leq s\leq n}$ generates the kernel of~$\varphi^{m-n-1}$
on~$\WW_n$. Similarly $W_m/p^n W_m)\cong W_n$ and $\bigl(\V^s\bigl(\widetilde{p}_m\bigr)^{p^n} \bigr)_{0\leq s\leq n}$ is the kernel of~$\varphi^{m-n-1}$ on~$W_n$.
Hence $A_{\rm cris, m}/p^n A_{\rm cris, m}$ is the DP envelope of~$W_n$ with respect to~$\varphi^{m-n-1}(\xi_m)=\xi_{n+1}$ i.~e. it coincides with $A_{\rm cris}/p^n
A_{\rm cris}$. We conclude that $p^n \bA_{\rm cris,m,\Kbar}^\nabla$ contains the kernel of the map $u_{n,\Kbar}\circ r_{n+1,\Kbar} \circ\cdots\circ
r_{m,\Kbar}\colon \bA_{\rm cris,m,\Kbar}^\nabla\lra \bA_{\rm cris,n,\Kbar}^{'\nabla}$. It is also clearly contained in this kernel. The claim follows.
\end{proof}

Let $\cU=\Spf(R_\cU)$ be an object in $X^{\rm et}$. Assume it is
small in the sense of \S\ref{sec:formal_Groth}. As in
\S\ref{sec:localization} fix a geometric generic point and define
$\Rbar_\cU$ as in loc.~cit. Following \cite[\S6]{brinon} define
$A_{\rm cris}^\nabla(\Rbar_\cU)$ as the $p$--adic completion of
the $\WW(k)$--DP envelope of $\WW\bigl(\crR(\Rbar_\cU)\bigr)$ with
respect to the kernel of the map~$\vartheta$ defined as follows.
For every~$n$ let~$\vartheta_n$ be the composite of the projection
$ \WW\bigl(\crR(\Rbar_\cU)\bigr)\to
\WW_n\bigl(\crR(\Rbar_\cU)\bigr)$, of the map $
\WW_n\bigl(\crR(\Rbar_\cU)\bigr) \to
\WW_n\bigl(\Rbar_\cU/p\Rbar_\cU\bigr)$ associated to the projection
$\ds \crR(\Rbar_\cU)=\lim_\leftarrow \Rbar_\cU/p\Rbar_\cU\to
\Rbar_\cU/p\Rbar_\cU$ on the $n$--th component and
of~$\theta_n\colon \WW_n\bigl(\Rbar_\cU/p\Rbar_\cU\bigr)\to
\Rbar_\cU/p^n \Rbar_\cU$.
Let~$\vartheta\colon\WW\bigl(\crR(\Rbar_\cU)\bigr)\to \hR_\cU$ be
the map $\ds x\mapsto \lim_{\infty \leftarrow n}\vartheta_n(x)$.
It is proven in loc.~cit.~that $\Ker(\vartheta)$ is a principal
ideal generated by~$\xi$ and that Frobenius
on~$\WW\bigl(\crR(\Rbar_\cU)\bigr)$  induces Frobenius~$\varphi$
on~$A_{\rm cris}^\nabla(\Rbar_\cU)$. For every~$n$ let~$g_n$ be
the composite of the projection $
\WW\bigl(\crR(\Rbar_\cU)\bigr)\to
\WW_n\bigl(\crR(\Rbar_\cU)\bigr)$ and of the map $v_n\colon
\WW_n\bigl(\crR(\Rbar_\cU)\bigr) \to
\WW_n\bigl(\Rbar_\cU/\Rbar_\cU\bigr)$ associated to the projection
$\ds \crR(\Rbar_\cU)=\lim_\leftarrow \Rbar_\cU/p\Rbar_\cU\to
\Rbar_\cU/p\Rbar_\cU$ on the $n+1$--th component. We get that
$$A_{\rm cris}^\nabla(\Rbar_\cU)/p^n A_{\rm
cris}^\nabla(\Rbar_\cU)\cong
\WW_n\bigl(\crR(\Rbar_\cU)\bigr)[\delta_0,\delta_1,\ldots]/(p\delta_0-\xi^p,p
\delta_{i+1}-\delta_i^p)_{i\geq 0}.$$Since~$g_n(\xi)=\xi_{n+1}$,
we have a map $g_n\colon  A_{\rm cris}^\nabla(\Rbar_\cU)/p^n
A_{\rm cris}^\nabla(\Rbar_\cU) \lra \bA_{\rm cris,
n,\Kbar}^{'\nabla}(\Rbar_\cU)$. Note that $\bA_{\rm cris,
n,M}^{'\nabla}(\Rbar_\cU)=\beta_{M,\Kbar}^\ast\left(\bA_{\rm cris,
n,M}^{'\nabla}\right)(\Rbar_\cU)$  by \ref{lemma:betaastG} and the
latter coincides with $\bA_{\rm cris,
n,\Kbar}^{'\nabla}(\Rbar_\cU)$ thanks to proposition \ref{lemma:Acrisnnabla'}.
We then get a map $g_{M,n}\colon  A_{\rm
cris}^\nabla(\Rbar_\cU)/p^n A_{\rm cris}^\nabla(\Rbar_\cU) \lra
\bA_{\rm cris,n, M}^{'\nabla}(\Rbar_\cU)$.

\begin{proposition}
\label{prop:acrisnabla} 1) The map $\bA_{\rm
cris,M}^{'\nabla}(\Rbar_\cU) \to \bA_{\rm
cris,M}^\nabla(\Rbar_\cU) $ defined by~$q_M$ is an
isomorphism.\smallskip

2) For every $n\in \N$ the map $g_{n,M}\colon  A_{\rm
cris}^\nabla(\Rbar_\cU)/p^n A_{\rm cris}^\nabla(\Rbar_\cU) \lra
\bA_{\rm cris,n,M}^{'\nabla}(\Rbar_\cU)$ is injective, commutes
with the Frobenius maps and its cokernel is annihilated by any
element of~$\II$.\smallskip

3) The induced map $ A_{\rm cris}^\nabla(\Rbar_\cU) \to \bA_{\rm
cris,M}^\nabla(\Rbar_\cU)$ is an isomorphism and commutes with
Frobenius.
\end{proposition}
\begin{proof} (1) It follows
from proposition \ref{lemma:Acrisnnabla'}(3) that~$r_M$ defines an inverse.
Claims (2) and~(3) follow  from the next lemma.\end{proof}

\begin{lemma}\label{lemma:acrisnabla}
For every~$n\in\N$ the map $A_{\rm cris}^\nabla(\Rbar_\cU)/p^n
A_{\rm cris}^\nabla(\Rbar_\cU)\to \bA_{\rm
cris,\Kbar,n}^{'\nabla}(\Rbar_\cU)$ is injective, its cokernel is
annihilated by any element of~$\II$  and the transition map
$\bA_{\rm cris,\Kbar,n+1}^{'\nabla}(\Rbar_\cU)\to \bA_{\rm
cris,\Kbar,n}^{'\nabla}(\Rbar_\cU)$ factors via $A_{\rm
cris}^\nabla(\Rbar_\cU)/p^n A_{\rm cris}^\nabla(\Rbar_\cU)$.
\end{lemma}
\begin{proof} Since~$\Rbar_\cU$ is a normal ring and Frobenius is surjective
on~$\Rbar_\cU/p\Rbar_\cU$ by~\cite[Prop.~2.0.1]{brinon}, the kernel of the projection~$\ds \crR(\Rbar_\cU)=\lim_\leftarrow \Rbar_\cU/p\Rbar_\cU\to
\Rbar_\cU/p\Rbar_\cU$ on the $n+1$--factor is~$\widetilde{p}^{p^n}$. Since~$\xi=\widetilde{p}-p$ and both~$p$ and~$\xi$ have DP in~$A_{\rm cris}^\nabla(\Rbar_\cU)$,
also~$\widetilde{p}$ admits DP. As in the proof of lemma \ref{lemma:compareAcrisnablaAcrisnabla'} it follows that~$\V^s\bigl([\widetilde{p}]\bigr)^{p^n}\equiv 0$
in~$A_{\rm cris}^\nabla(\Rbar_\cU)/p^n A_{\rm cris}^\nabla(\Rbar_\cU)$. These elements generate the kernel of~$v_n$.  Hence $$A_{\rm cris}^\nabla(\Rbar_\cU)/p^n
A_{\rm cris}^\nabla(\Rbar_\cU)\cong \WW_n\bigl(\Rbar_\cU/p\Rbar_\cU\bigr)[\delta_0,\delta_1,\ldots]/(p\delta_0-\xi^p,p \delta_{i+1}-\delta_i^p)_{i\geq 0}$$where the
isomorphism induces the map $g_n\colon \WW_n\bigl(\crR(\Rbar_\cU)\bigr)\to \WW_n\bigl(\Rbar_\cU/p\Rbar_\cU\bigr)$.

We prove the lemma by induction on $n$. It follows from what we
have just seen \ref{lemma:exactsequenceAcrisnnabla'} and
\ref{lemma:localizationofbarOmodpn} that the map $A_{\rm
cris}^{\nabla}(\Rbar_\cU)/pA_{\rm cris}^\nabla(\Rbar_\cU)\to
\bA_{\rm cris,1,\Kbar}^{'\nabla}(\Rbar_\cU)$ is injective and that
its cokernel is annihilated by any element of~$\II$.

Since $A_{\rm cris}^\nabla(\Rbar_\cU)$ has no $p$--torsion
by~\cite[Prop.~6.1.4]{brinon} we have an exact sequence $$0 \lra
A_{\rm cris}^\nabla(\Rbar_\cU)/p^n A_{\rm cris}^\nabla(\Rbar_\cU)
\stackrel{p}{\lra} A_{\rm cris}^\nabla(\Rbar_\cU)/p^{n+1}A_{\rm
cris}^\nabla(\Rbar_\cU) \lra A_{\rm cris}^\nabla(\Rbar_\cU)/p
A_{\rm cris}^\nabla(\Rbar_\cU)\lra 0.$$One checks that it is
compatible with the exact sequence obtained by taking the
localizations of the sequence
in~\ref{lemma:exactsequenceAcrisnnabla'}(b).  Due
to~\ref{lemma:exactsequenceAcrisnnabla'}(b) the map $\bA_{\rm
cris,n+1,\Kbar}^{'\nabla}(\Rbar_\cU) \to \bA_{\rm
cris,1,\Kbar}^{'\nabla}(\Rbar_\cU)$ is the map
$$\frac{A_{\rm cris}}{p^n A_{\rm
cris}}\otimes_{W_n} \Bigl(\WW_n(\Rbar_\cU)\Bigr)\to \frac{A_{\rm cris}}{p A_{\rm cris}}\otimes_{\OKbar/p\OKbar} \Bigl(\cO_\fX/p\cO_\fX(\Rbar_\cU)\Bigr)$$which by
construction induces the map $\WW_n(\Rbar_\cU) \to (\cO_\fXKbar/p\cO_\fXKbar)(\Rbar_\cU)$ given by the natural projection and Frobenius to the $n-1$--th power. In
particular since~$n\geq 2$ this map factors via $\Rbar_\cU/p \Rbar_\cU$ by~\ref{lemma:localizationofbarOmodpn}. The inductive step follows from this using the
inductive hypothesis.
\end{proof}

\subsection{The sheaf $\bA_{\rm cris,M}$.} \label{sec:sheafAcris}

In this section we assume that $\cO_K=\WW(k)$ is absolutely unramified. Write ${\OMun}$ for the ring of integers of the maximal unramified extension $\Mun$ of $K$
in $M$. Recall that in lemma \ref{lemma:vastOX} we have described $v_{X,M}^\ast(\cO_X)$ as the subsheaf $\cO_{\fX_M}^{\rm un}$ of $\cO_{\fX_M}$ introduced in
\ref{def:OX}. It is a sheaf of $\cO_{\Mun}$-algebras. For every $n\ge 1$ let us define the sheaf
$\WW_{X,n,M}:=\WW_n(\cO_{\fX_M}/p\cO_{\fX_M})\otimes_{\OMun}\cO_{\fX_M}^{\rm un}$ of $\OMun$--algebras and the morphism of sheaves of $\cO_{\fX_M}^{\rm
un}$--algebras $\theta_{X,n,M}\colon\WW_{X,n,M}\lra \cO_{\fX_M}/p^n\cO_{\fX_M}$ associated to the following map of pre-sheaves. Let $(\cU,\cW)$ be an object of
$\fX_M$ such that $\cU=\Spf(R_\cU)$ is affine, let $R_\cU^{\rm un}:=\cO_{\fX_M}^{\rm un}(\cU,\cW)$) and let $S=\cO_{\fX_M}(\cU,\cW)$. Then $S$ contains $R_\cU^{\rm
un}$ and we define
$$
\theta_{n,(\cU,\cW)}\colon\WW_n(S/pS)\otimes_{\OMun}R_\cU^{\rm
un}\lra S/p^nS\mbox{ by } (x\otimes r)\rightarrow c_n(x)r.
$$
Let $\cI_{X,n,M}$ denote the sheaf of ideals
$\Ker(\theta_{X,n,M})$. Due to \cite[Thm. I.2.4.1]{berthelot1} one
knows that the $\WW(k)$--DP envelope $\bA_{\rm cris,n,M}$
of\/~$\WW_{X,n,M}$ with respect to~$\cI_{X,n,M}$ exists. The main
point of this section is an explicit description of $\bA_{\rm
cris,n,M}$ in theorem \ref{thm:gluing} which will be used in the sequel.

Let $\cU=\Spf(R_\cU)$ denote a small affine open as in
\ref{sec:Groth} with parameters $T_1,T_2,\ldots,T_d\in
R_\cU^\times$. For every $n\ge 0$ define
$R_{\cU,n}:=R_\cU[\zeta_n,T_1^{1/p^n},\ldots,T_d^{1/p^n}]$, where
$R_{\cU,0}=R_\cU$, $\zeta_n$ is a primitive $p^n$-th root of $1$
such that $\zeta_{n+1}^p=\zeta_n$ and $T_i^{1/p^n}$ is a fixed
$p^n$-th root of $T_i$ in $\Rbar_\cU$ such that
$\bigl(T_i^{1/p^{n+1}}\bigr)^p=T_i^{1/p^n}$. We consider the
category $\fU_{n,M}$ consisting of objects $(\cV,\cW)$ and a
morphism to $\bigl(\cU,\Spf\bigl(R_{\cU,n}\bigr)\tensor_{\cO_K}
M\bigr)$. The morphisms are the morphisms as objects over
$\bigl(\cU,\Spf\bigl(R_{\cU,n}\bigr)\tensor_{\cO_K} M\bigr)$. The
covering families of an object $(\cV,\cW)$ are the covering
families as an object of $\fX_M$. There is a morphism of sites
$\iota\colon \fX_M \lra \fU_{n,M} $ sending $(\cV,W)$ to
$(\cV,W)\times_{(X,X_K)}
\bigl(\cU,\Spf\bigl(R_{\cU,n}\bigr)\tensor_{\cO_K} M\bigr)$. Given
a sheaf on $\fX_M $ we write $\cF\vert_{\fU_{n,M}}$ for
$\iota^\ast(\cF)$.

Let~$(\cV,\cW)\in \fU_{n,M}$ with~$\cV=\Spf(R_\cV)$ affine and
put~$S:=\cO_{\fX_M}(\cV,\cW)$. Note that $T_i^{1/p^n}\in
R_{\cU,n}\subset S$ for all $1\le i\le d$.  Denote by $$\ds\tT_i:=
\bigl([T_i],[T_i^{1/p}],\cdots,[T_i^{1/p^n}],\cdots\bigr)\in
\lim_{\infty\leftarrow n}
\WW_n\bigl(R_{\cU,n}/pR_{\cU,n}\bigr),$$where the inverse limit is
taken with respect to
$\WW_{n+1}\bigl(R_{\cU,n+1}/pR_{\cU,n+1}\bigr)\to
\WW_n\bigl(R_{\cU,n}/pR_{\cU,n}\bigr)$ the natural projection
$\WW_{n+1}\bigl(R_{\cU,n+1}/pR_{\cU,n+1}\bigr)\to
\WW_n\bigl(R_{\cU,n+1}/pR_{\cU,n+1}\bigr)$ and the map induced by
Frobenius $R_{\cU,n+1}/pR_{\cU,n+1}\to R_{\cU,n}/pR_{\cU,n} $. The
image of~$\tT_i$ in~$ \WW_n\bigl(R_{\cU,n}/pR_{\cU,n}\bigr)$ is
the Teichm\"uller lift $(T_i^{1/p^n},0,\ldots,0)$ of
$T_i^{1/p^n}$. Write
$$X_i:=1\otimes T_i- \tT_i\otimes 1 \in
\WW_n\bigl(R_{\cU,n}/pR_{\cU,n}\bigr) \otimes_{\cO_K} R_\cU$$
for~$i=1,\ldots,d$. They are naturally elements
of~$\WW_n\bigl(S/pS\bigr) \otimes_{\cO_K} R_\cV$.

\begin{lemma}
\label{lemma:kerthetaonfUn} The kernel of the map
$\theta_{n,(\cV,\cW)}\colon \WW_n\bigl(S/pS\bigr) \otimes_{\OMun}
R_\cV^{\rm un} \to S/p^n S $ is the ideal generated by
$\bigl(\xi_n,X_1,\ldots,X_d\bigr)$.
\end{lemma}
\begin{proof}
The kernel of the ring homomorphism $R_0/p^n R_0 \otimes R_0/p^n R_0 \to R_0/p^nR_0 $ defined by $x\tensor y\to xy$ is the ideal $I=(T_1\otimes 1-1\otimes
T_1,\ldots,T_d\otimes 1-1\otimes T_d)$. Since $R_0/p^n R_0\to R_\cV/p^n R_\cV\to R_\cV^{\rm un}/p^n R_\cV^{\rm un}$ are \'etale morphisms, then $I$ generates the
kernel of $R_\cV^{\rm un}/p^n R_\cV^{\rm un} \otimes R_\cV^{\rm un}/p^n R_\cV^{\rm un} \to R_\cV^{\rm un}/p^n R_\cV^{\rm un} $. Base changing via $R_\cV^{\rm
un}/p^n R_\cV^{\rm un}\to S/p^n S$ we conclude that $I$ generates the kernel of $S/p^n S \otimes R_\cV^{\rm un}/p^n R_\cV^{\rm un} \to S/p^n S $. By lemma
\ref{lemma:krcn} the kernel of $c_n\colon \WW_n\bigl(S/pS\bigr) \to S/p^n S$ is generated by~$\xi_n$. The conclusion follows.
\end{proof}

For every $(\cV,\cW)\in \fU_{n,M}$ let $\bA_{\rm cris,n,M}^\nabla(\cV,\cW)
\langle X_1,\ldots,X_d\rangle$ be the DP envelope
of the polynomial algebra $\bA_{\rm cris,M,
n}^\nabla(\cV,\cW)[X_1,\ldots,X_d]$ with respect to the ideal
$(X_1,\ldots,X_d)$. As explained in~\cite[\S6]{brinon}, we have
$$\bA_{\rm cris, n,M}^\nabla(\cV,\cW) \langle
X_1,\ldots,X_d\rangle=\bA_{\rm cris,n,M}^\nabla(\cV,\cW)[X_{i,0},X_{i,1},\ldots]_{1\le i\le d}/(pX_{i,m+1}- X_{i,m}^p)_{1\le i\le d,m\ge 0}$$where~$ X_{i,j}=
\gamma^{j+1}(X_i)$ and $\gamma\colon z\mapsto (p-1)!z^{[p]}$. In particular it is a free $\bA_{\rm cris,n,M}^\nabla(\cV,\cW)$--module with bases given by the
monomials in the variables $X_{i,0},X_{i,1},\ldots$ for~$1\le i\le d$ in which each variable $X_{i,j}$ appears with degree $\leq p-1$. We conclude that $(\cV,\cW)
\to \bA_{\rm cris,n,M}^\nabla(\cV,\cW)\langle X_1,\ldots,X_d\rangle$ is a sheaf and that the ideal generated by $\Ker(\theta_{n,M}) (\subset \bA_{\rm
cris,n,M}^\nabla)$ and $(X_{1,j},\cdots,X_{d,j})_{j\in\N}$ has $\WW(k)$--DP structure. Write $\bA_{\rm cris,\fU_{n,M}}$ for the sheaf $\bA_{\rm
cris,n,M}^\nabla\vert_{\fU_{n,M}}\langle X_1,\ldots,X_d\rangle$. Let $\theta_{\fU_{n,M}}\colon \bA_{\rm cris,\fU_{n,M}}\lra \cO_{\fX_M}/p^n\cO_{\fX_M}
\vert_{\fU_{n,M}}$ be the map sending $X_{i,j}\mapsto 0$ and coinciding with~$\theta_{n,M}\vert_{\fU_{n,M}}$ on $\bA_{\rm cris,n,M}^\nabla\vert_{\fU_{n,M}}$. For
every~$(\cV,\cW)\in\fU_{n,M}$ with $\cV:=\Spf(R_\cV)$ affine, define the map
$$
R_0=\WW(k)\{T_1^{\pm 1},\ldots,T_d^{\pm 1}\}\lra \bA_{\rm
cris,n,M} ^\nabla(\cV,\cW) \langle X_1,\ldots,X_d\rangle
$$
of $\WW(k)$--algebras setting $T_i\rightarrow \tT_i\otimes
1+1\otimes X_i$ for every $1\le i\le d$. As $T_i$ is a unit in
$R_\cU$, then $\tT_i$ is a unit in $\WW_n(S/pS)$
where~$S:=\cO_{\fX_M}(\cV,\cW)$. Since $X_i^p=p X_i^{[p]}$ is
nilpotent in $\bA_{\rm cris, n,M}^\nabla(\cV,\cW) \langle
X_1,\ldots,X_d\rangle$, also $\tT_i\otimes 1+1\otimes X_i$ is a
unit and hence the displayed ring homomorphism is well defined.
Let us now look at the following diagram of $\WW(k)$--algebras
$$
\begin{array}{cccccccc}
\bA_{\rm cris,n,M}^\nabla(\cV,\cW) \langle X_1,\ldots,X_d\rangle
&\stackrel{\theta_{\fU_{n,M}}}{\lra} &\bigl(\cO_{\fX_M}/p^n\cO_{\fX_M}\bigr)(\cV,\cW)\\
\uparrow & &\uparrow\\
R_0& \lra  & R_\cV^{\rm un}.
\end{array}
$$The diagram is
commutative since $\theta_{\fU_{n,M}}(\tT_i\otimes 1+1\otimes
X_i)=\theta_{\fU_{n,M}}(\tT_i)=
\bigl(T_i^{1/p^n}\bigr)^{p^n}=T_i$. The kernel of
$\theta_{\fU_{n,M}}(\cV,\cW)$ is the DP ideal generated by
$\Ker(\theta_{n,M})(\cU,\cW)$ and by
$\bigl(X_{1,j},\ldots,X_{d,j}\bigr)_{j\in\N}$ so that it is
nilpotent. As $R_\cU$ is \'etale over $R_0$ and $R_\cV^{\rm un}$  \'etale
over $R_\cU$ there exists a unique homomorphism
\begin{equation}\label{display:Rcvalgebrastructure}
R_\cV^{\rm un}\lra \bA_{\rm cris, n,M}^\nabla(\cV,\cW) \langle
X_1,\ldots,X_d\rangle
\end{equation}
\noindent of $R_0$--algebras making both triangles in the above
diagram commutative. We thus get a map $
\WW_{n,M}(\cV,\cW)\tensor_{\OMun} \cO_{\fX_M}^{\rm un}(\cV,\cW)
\to \bA_{\rm cris, n,M}^\nabla(\cV,\cW) \langle
X_1,\ldots,X_d\rangle$. Passing to the associated sheaves we get a
map of sheaves
$$h_{\fU_{n,M}}\colon\WW_{X,n,M}\vert_{\fU_{n,M}}\to \bA_{\rm
cris,\fU_{n,M}}.$$It follows from~\ref{lemma:kerthetaonfUn} that
the image of~$\cI_{X,n,M}\vert_{\fU_{n,M}}$ is contained in the
ideal generated by $\Ker(\theta_{n,M})$ and
$(X_{1,j},\cdots,X_{d,j})_{j\in\N}$ which has $\WW(k)$--DP
structure.

\begin{theorem} \label{thm:gluing} (1) The $\WW(k)$--DP envelope $\bA_{\rm
cris,n,M}$ of\/~$\WW_{X,n,M}$ with respect to~$\cI_{X,n,M}$
exists.\smallskip

(2) For every small affine $\cU$ of $X^{\rm et}$  the sheaf $\bA_{\rm cris,\fU_{n,M}}$ with the ideal $\Ker\bigl( \theta_{\fU_{n,M}}\bigr) $ endowed with its
$\WW(k)$--DP structure, is the $\WW(k)$--DP envelope of $\WW_{X,n,M}\vert_{\fU_{n,M}}$ with respect to $\cI_{X,n,M}\vert_{\fU_{n,M}}$. In particular the restriction
of $\bA_{\rm cris,n,M}$ to~$\fU_{n,M}$  is $\bA_{\rm cris,\fU_{n,M}}$.

(3) Let $M_1\subset M_2$ be a Galois extension. The morphism
$\beta_{M_1,M_2}^\ast\left(\WW_{X,n,M_1} \otimes_{\cO_{M_1}^{\rm
un}}\cO_{\fX_{M_1}}^{\rm un}\right) \lra
\WW_{X,n,M_2}\otimes_{\cO_{M_2}^{\rm un}} \cO_{\fX_{M_2}}^{\rm un}
$ induces an isomorphism $\beta_{M_1,M_2}^\ast\bigl(\bA_{\rm
cris,n,M_1}\bigr) \cong \bA_{\rm cris,n,M_2}$ of $\WW(k)$--DP
sheaves of algebras.\smallskip

\end{theorem}

\begin{proof} (1) The existence is a formal consequence of
\cite[Thm. I.2.4.1]{berthelot1}.

(2)  Consider  a $\WW(k)$--DP sheaf of algebras $\cG$ and a
morphism $f\colon \WW_{X,n,M}\vert_{\fU_{n,M}} \to \cG$ sending
$\cI_{X,n,M}\vert_{\fU_{n,M}}$ to the DP ideal of $\cG$. Due to
\ref{lemma:Acrisnnabla} and \ref{lemma:restrictenvelope} the
induced map $\WW_{n,M}\vert_{\fU_{n,M}} \to \cG$ extends uniquely
to a morphism $\bA_{\rm cris,n,M}^\nabla\vert_ {\fU_{n,M}}\to \cG
$ of $\WW(k)$--DP sheaves  of algebras.  Since the
section~$X_i=1\otimes T_i-\tT_i\otimes 1$ of~$\WW_{X,n,M}$ lies in
$\Ker\bigl( \theta_{\fU_{n,M}}\bigr) $ we can uniquely extend
such map to a morphism $\bA_{\rm cris,n,M}
^\nabla\vert_{\fU_{n,M}} \langle X_1,\ldots,X_d\rangle\to \cG$.
This proves the first part of the claim. The second part follows
from the first and lemma \ref{lemma:restrictenvelope}.

(3) Via the morphism $\beta_{M_1,M_2}^\ast\left(\WW_{n,M_1}
\otimes_{\cO_{M_1}^{\rm un}}\cO_{\fX_{M_1}}^{\rm un}\right) \lra
\WW_{n,M_2}\otimes_{\cO_{M_2}^{\rm un}} \cO_{\fX_{M_2}}^{\rm un}
$ one checks that
$\beta_{M_1,M_2}^\ast\bigl(\theta_{X,n,M_1}\bigr)$ and
$\theta_{X,n,M_2}$ are compatible so that
$\beta_{M_1,M_2}^\ast(\cI_{X,n,M_1})$ is sent to $\cI_{X,n,M_2}$.
Due to the universal property of $\bA_{\rm cris,n,M_2} $ we
get a map $\beta_{M_1,M_2}^\ast\bigl(\bA_{\rm
cris,n,M_1}\bigr)\lra \bA_{\rm cris,n,M_2}$. By construction it
induces the isomorphism $\beta_{M_1,M_2}^\ast\bigl(\bA_{\rm
cris,n,M_1}^\nabla\bigr) \lra \bA_{\rm cris,n,M_2}^\nabla $ of
proposition \ref{lemma:Acrisnnabla}. It follows from (2) that the restriction
to $\fU_{n,M_1}$ sends $X_i\mapsto X_i$ for every small affine
$\cU$  and in particular it is an isomorphism. This implies the
claim.

\end{proof}

Let $\cU=\Spf(R_\cU)$ denote a small affine open in $X^{\rm et}$. Choose~$T_1,\ldots T_d\in R_\cU^\times$ parameters and let $F_\cU\colon R_\cU \to R_\cU$ be the
unique map inducing Frobenius on~$\cO_K$ and sending $T_i\mapsto T_i^p$. Denote by~$F_\cU\colon \cO_\cU \to \cO_\cU$ the induced map of sheaves on~$\cU^{\rm et}$.
Taking $v_{\cU,M}^\ast$ it provides  a morphism $F_\cU$ on $\cO_{\fX_M}^{\rm un}\vert_{\fU_M}$. Let
$$\varphi_{\fU_M,n}\colon \WW_{X,n,M}\vert_{\fU_M}\cong
\WW_{\cU,M,n}\lra  \WW_{\cU,M,n} \cong \WW_{X,n,M}\vert_{\fU_M}$$
be the map of sheaves associated to the map of pre-sheaves
$\varphi_{\cU,n} \tensor F_\cU \colon
\WW_{n,\cU,M}\otimes_{\OMun}\cO_{\fU_M}^{\rm un} \to
\WW_{n,\cU,M}\otimes_{\OMun}\cO_{\fU_M}^{\rm un}$.

\begin{corollary}\label{cor:extensionofFrobenius}
1) The morphism~$\varphi_{\fU_M,n}$ on $\WW_{X,n,M}\vert_{\fU_M}$
extends uniquely to an operator~$\varphi_{\fU_M,n}$ on~$\bA_{\rm
cris,n,M}\vert_{\fU_M}$, called {\it Frobenius}, compatible with
Frobenius on  $\bA_{\rm cris,n,M}^\nabla\vert_{\fU_M}$ defined in
proposition \ref{lemma:Acrisnnabla}.\smallskip

2) Via the identification given in theorem \ref{thm:gluing} the
restriction~$\varphi_{\fU_{n,M}}$ of~$\varphi_{\fU_M,n}$ to
$\fU_{n,M}$ is uniquely determined by requiring that it induces
Frobenius on $\bA_{\rm cris,n,M}^\nabla\vert_{\fU_{n,M}}$ and
sends $X_i\mapsto 1\otimes T_i^p- \tT_i^p\otimes 1= X_i
\bigl(\sum_{h=1}^{p-1} T_i^h \tT_i^{p-h}\bigr)$
for~$i=1,\ldots,d$.\smallskip

3) The  isomorphism $\beta_{M_1,M_2}^\ast\bigl(\bA_{\rm cris,
n,M_1}\vert_{\fU_{M_1}}\bigr)\cong \bA_{\rm cris,
n,M_2}\vert_{\fU_{M_2}}$ of $\WW(k)$--DP sheaves of algebras is
compatible with Frobenius.

\end{corollary}
\begin{proof} The fact that Frobenius on $\WW_{X,n,M}\vert_{\fU_M}$ extends to
$\bA_{\rm cris,n,M}^\nabla\vert_{\fU_M}$ follows from
proposition \ref{lemma:Acrisnnabla}.

(2)   For~$i=1,\ldots,d$ we compute that
$\varphi_{\fU_{n,M}}(X_i)= 1\otimes T_i^p- \tT_i^p\otimes 1= X_i
\bigl(\sum_{h=1}^{p-1} T_i^h \tT_i^{p-h}\bigr)$ so that
$\varphi_{\fU_{n,M}}(X_i)$ admits divided powers. This implies
that~$\varphi_{\fU_{n,M}}\bigl(\cI_{\cU,n,M}\bigr)$ admits divided
powers so that by the universal property of~$\bA_{\rm
cris,n,M}\vert_{\fU_{n,M}}$ (see \ref{thm:gluing}) the morphism
$\varphi_{\fU_{n,M}}$ extends to~$\bA_{\rm
cris,n,M}\vert_{\fU_{n,M}}$.

(1) Since~$\varphi_{\fU_{n,M}}\bigl(\cI_{\cU,n,M}\bigr)$ admits
divided powers in ~$\bA_{\rm cris,n,M}\vert_{\fU_{n,M}}$ by (2)
then also~$\varphi_{\fU_M,n}\bigl(\cI_{\cU,n,M}\bigr)$ admits
divided powers in~$\bA_{\rm cris,n,M}\vert_{\fU_M}$. By the
universal property of~$\bA_{\rm cris,n,M}\vert_{\fU_M}$, which
follows from \ref{thm:gluing} and \ref{lemma:restrictenvelope},
the morphism $\varphi_{\fU_{n,M}}$ extends to~$\bA_{\rm
cris,n,M}\vert_{\fU_M,n}$.

(3) It suffices to prove the claim after restricting to
$\fU_{n,M_2}$. In this case it follows from (1) and
theorem \ref{thm:gluing}.
\end{proof}

Let~$r_{X,n+1,M}\colon \WW_{n+1,M}\otimes_{\OMun}\cO_{\fX_M}^{\rm
un} \to \WW_{n,M}\otimes_{\OMun}\cO_{\fX_M}^{\rm un}$ be the
morphism which is the identity on~$\cO_{\fX_M}^{\rm un}$ and is
reduction composed with Frobenius on $\WW_{n+1,M}\to \WW_{n,M}$.
Then we have an inclusion
$r_{X,n+1,M}\bigl(\Ker(\theta_{X,n+1})\bigr)\subset
\Ker(\theta_{X,n}) $. Hence $r_{X,n+1,M}$ defines a map $\bA_{\rm
cris,n+1,M}\to \bA_{\rm cris,n,M}$. Let $\bA_{\rm cris,M}$ denote
the sheaf in $\Sh(\fX_M)^\N$ defined by the family $\{\bA_{\rm
cris,n,M}\}_n$ with the transition functions~$r_{X,n,M}$. It is
the $\WW(k)$--DP envelope of
$\{\WW_{n,M}\otimes_{\OMun}\cO_{\fX_M}^{\rm un}\}_n$ with respect
to the ideals $\{\Ker\bigl(\theta_{X,n}\bigr)\}_n$. For every
small affine~$\cU$ and parameters $T_1,\ldots,T_d$ denote
by~$\varphi_\cU\colon \bA_{\rm cris,M}\to \bA_{\rm cris,M}$ the
map in~$\Sh(\fUM)^\N$ defined by~$\{\varphi_{\cU,n}\}_n$.

\begin{lemma}\label{lemma:Acrismodpn} For every~$m'> m>n$ the
maps $r_{X,m',M}\circ \ldots \circ r_{X,m+1,M}\colon \bA_{\rm
cris,m',M}\to \bA_{\rm cris, m,M}$ induce an isomorphism $\bA_{\rm
cris,m',M}/p^n \bA_{\rm cris,m',M}\lra \bA_{\rm cris,m,M}/p^n
\bA_{\rm cris,m,M}$.

With the notations of\/~\ref{thm:gluing}, for every small $\cU$ of
$X^{\rm et}$ the restriction of this sheaf to~$\fU_{n,M}$ is
isomorphic to $\bA_{\rm
cris,n,M}^{'\nabla}\vert_{\fU_{M,n}}\langle X_1,\ldots,X_d\rangle
$.
\end{lemma}
\begin{proof}
It suffices to prove the two claims restricting to~$\fU_{n,M}$ for
every small object~$\cU$. They follow from theorem \ref{thm:gluing}
and lemma \ref{lemma:compareAcrisnablaAcrisnabla'}.
\end{proof}

Denote by~$\bA_{\rm cris,n,M}'$ the sheaf $\bA_{\rm cris,m,M}/p^n \bA_{\rm cris,m,M}$ for $m>n$ introduced in~\ref{lemma:Acrismodpn}. It is the $\WW(k)$--DP
envelope of\/~$\WW_{X,n,M}$ with respect to the kernel of the map~$\theta_{X,n}'\colon \WW_{n,M}\otimes_{\OMun}\cO_{\fX_M}^{\rm un}\to \cO_{\fX_M}/p^n\cO_{\fX_M}$
induced by~$\theta_{n,M}'$ on~$\WW_{n,M}$ and the natural projection $\cO_{\fX_M}^{\rm un}\to \cO_{\fX_M}/p^n\cO_{\fX_M}$.  For every small affine~$\cU$ and
parameters $T_1,\ldots,T_d$ denote by~$\varphi_{\cU,n}'\colon \bA_{\rm cris,\cU,n,M}'\to \bA_{\rm cris,\cU,n,M}$ the map defined by $\varphi_n \tensor F_\cU$ on
$\WW_{n,M}\vert_{\fU_{n,M}}\otimes_{\OMun} \cO_{\fX_M}^{\rm un} $. Let~$\bA_{\rm cris,M}':=\left\{\bA_{\rm cris, n,M}' \right\}_n$ be the associated system of
sheaves, where the transition maps $r_{X,n+1,M}'\colon \bA_{\rm cris, n+1,M}'\to \bA_{\rm cris,n,M}'$ a induced by the transition maps~$\{r_{X,m,M}\}_m$. By
construction we have a natural morphism $q_{X,n}\colon \bA_{\rm cris,n,M}'\to \bA_{\rm cris,n,M} $ for every~$n\in\N$ and hence a map $
q_X:=\left\{q_{X,n}\right\}_n\colon \bA_{\rm cris,M}'\to \bA_{\rm cris,M}$. For every small affine~$\cU$ and parameters $T_1,\ldots,T_d$ denote
by~$\varphi_{\cU}':=\left\{\varphi_{\cU,n}'\right\}_n$.

Following \cite{brinon} define~$ A_{\rm cris,M}(\Rbar_\cU)$ as the
$p$--adic completion of the $\WW(k)$--DP envelope of
$\WW\bigl(\crR(\Rbar_\cU)\bigr)\otimes_{\cO_K} R_\cU$ with respect
to the kernel of the map
$\WW\bigl(\crR(\Rbar_\cU)\bigr)\otimes_{\cO_K} R_\cU \to \hR_\cU$
given by $x \otimes y \mapsto \vartheta(x) y$. Furthermore, it is
proved that the operator $\varphi\tensor F_\cU$ on
$\WW\bigl(\crR(\Rbar_\cU)\bigr)\otimes_{\cO_K} R_\cU$ defines an
operator~$\varphi$ on~$A_{\rm cris,M}(\Rbar_\cU)$. It is shown in
\cite [Prop.~6.1.8]{brinon}  that $A_{\rm cris,M}(\Rbar_\cU)$ is
the $p$--adic completion of the algebra $A_{\rm
cris}^\nabla(\Rbar_\cU)\langle X_1,\ldots,X_d\rangle$. Hence
$$A_{\rm cris}(\Rbar_\cU)/p^n A_{\rm cris}(\Rbar_\cU)\cong
\left(A_{\rm cris}^\nabla(\Rbar_\cU)/p^nA_{\rm
cris}^\nabla(\Rbar_\cU)\right)\langle X_1,\ldots,X_d\rangle.$$This
and~\ref{lemma:Acrismodpn} provide a natural map $g_{\cU,n}\colon
A_{\rm cris}(\Rbar_\cU)/p^n A_{\rm cris}(\Rbar_\cU) \to \bA_{\rm
cris,n,M}'(\Rbar_\cU)$ and hence  a map
$$\ds g_\cU:=\lim_n g_{\cU,n}\colon A_{\rm cris}(\Rbar_\cU)\lra
\bA_{\rm cris,M}'(\Rbar_\cU).$$

\begin{proposition}
\label{prop:crislocalization} 1) The map $\bA_{\rm
cris,M}'(\Rbar_\cU)\to \bA_{\rm cris,M}(\Rbar_\cU) $ induced
by~$q_X$ is an isomorphism.\smallskip

2) The map $g_\cU$  is an isomorphism and commutes with the two
Frobenii.
\end{proposition}
\begin{proof} The first claim follows from~\ref{lemma:Acrismodpn}
and~\ref{prop:acrisnabla}. The second statement follows from the
next lemma.\end{proof}

\begin{lemma}\label{lemma:localizeAcris}
For every~$n\in\N$ the map $g_{\cU,n}$ is injective, its cokernel
is annihilated by any element of~$\II$, it commutes with Frobenii
and the transition map $\bA_{\rm cris,n+1,M}'(\Rbar_\cU)\to
\bA_{\rm cris,n,M}'(\Rbar_\cU)$ factors via $A_{\rm
cris}(\Rbar_\cU)/p^n A_{\rm cris}(\Rbar_\cU)$.
\end{lemma}
\begin{proof}
It follows from lemma \ref{lemma:Acrismodpn} and lemma \ref{lemma:acrisnabla}.
\end{proof}

\subsection{Further properties of $\bA_{\rm cris,M}^\nabla$ and $\bA_{\rm cris,M}$.}
\label{sec:propAcris}

Let us recall that we write~$\bA_{\rm cris,M}^{\nabla}$ both for
the system of sheaves~$\{\bA_{\rm cris,n,M}^\nabla\}_n$ and
for~$\{\bA_{\rm cris,M,n}^{'\nabla}\}_n$. Similarly, we
write~$\bA_{\rm cris,M}$ both for~$\{\bA_{\rm cris,n,M}\}_n$ and
for~$\{\bA_{\rm cris,n,M}'\}_n$.  We specify which system is used
when needed. Whenever~$\bA_{\rm cris,M}$ appears we implicitly
assume, as in the previous section, that $\cO_K=\WW(k)$. Consider
the filtration~$\{\Fil^r\left(\bA_{\rm
cris,M}^{\nabla}\right)\}_{r\in\N}$ defined
by~$\{\Ker(\theta_{n,M})^{[r]}\}_n$
(resp.~$\{\Ker(\theta_{n,M}')^{[r]}\}_n$). Analogously define the
filtration~$\{\Fil^r\left(\bA_{\rm cris,M}\right)\}_{r\in\N}$
given by the subsheaves $\{\Ker(\theta_{X,M,n})^{[r]}\}_n$
(resp.~$\{\Ker(\theta_{X,M,n}')^{[r]}\}_n$).

Let~$\cU$ be a small affine  of $X^{\rm et}$ and choose
parameters~$T_1,\ldots,T_d$ in~$R_\cU^\times$. Then $\bA_{\rm
cris,M}\vert_{\fU_{n,M}}=\bA_{\rm
cris,M}^\nabla\vert_{\fU_{n,M}}\left\{\langle
X_1,\ldots,X_d\rangle\right\} $ by~\ref{thm:gluing} and
$\Fil^r\left(\bA_{\rm cris,M}\right)\vert_{\fU_{n,M}}$ is $\sum
\Fil^{s_0}\left(\bA_{\rm cris,M}^{\nabla}\right) X_1^{[s_1]}
\cdots X_d^{[s_d]}$ over all~$s_0,\ldots,s_d\in\N$ such
that~$s_0+\cdots+s_d \geq r$.

We remark that the element~$t$ is an element of $\Fil^1\left(A_{\rm cris}\right)$ and, hence, of $\Fil^1\left(\bA_{\rm cris,M}\right)(\Rbar_\cU)$ as well.
Write~$B_{\rm cris}^\nabla(\Rbar_\cU)=A_{\rm cris}^\nabla(\Rbar_\cU)\left[\frac{1}{t}\right]$ and~$B_{\rm cris}(\Rbar_\cU)=A_{\rm
cris}(\Rbar_\cU)\left[\frac{1}{t}\right]$. Note that since~$t$ lies in~$\Ker(\theta)$, it admits divided powers in~$A_{\rm cris}(\Rbar_\cU)$ so that~$t^p=p
!t^{[p]}$ and~$p$ is invertible in~$B_{\rm cris}(\Rbar_\cU)$. In particular the definition given here agrees with the one given in~\cite[Def.~6.1.11]{brinon}.
In~\cite[\S6.2.1]{brinon} decreasing filtrations~$\{\Fil^r B_{\rm cris}^\nabla(\Rbar_\cU)\}_{r\in \Z}$ on~$B_{\rm cris}^\nabla(\Rbar_\cU)$ and~$\{\Fil^r B_{\rm
cris}(\Rbar_\cU)\}_{r\in \Z}$ on~$B_{\rm cris}(\Rbar_\cU)$ are defined. Then

\begin{proposition}\label{prop:filtAcrisnabla} The filtrations~$\{\Fil^r\left(\bA_{\rm
cris,M}^{\nabla}\right)\}_{r\in\N}$ and~$\{\Fil^r\left(\bA_{\rm
cris,M}\right)\}_{r\in\N}$ are decreasing, separated and
exhaustive.\smallskip

Let~$\cU$ be a small affine. Via the identifications $\bA_{\rm
cris,M}^{\nabla}(\Rbar_\cU) \cong A_{\rm
cris}^{\nabla}(\Rbar_\cU)$ and $\bA_{\rm cris,M}(\Rbar_\cU) \cong
A_{\rm cris}(\Rbar_\cU)$ given in~\ref{prop:acrisnabla}
(resp.~\ref{prop:crislocalization}), we have for every~$r\in \Z$
the identifications $\Fil^a A_{\rm cris}^\nabla(\Rbar_\cU)=\Fil^a
\bA_{\rm cris,M}^{\nabla}(\Rbar_\cU)$ and $\Fil^a A_{\rm
cris}(\Rbar_\cU)=\Fil^a \bA_{\rm cris,M}(\Rbar_\cU)$. In
particular,
$$\Fil^r B_{\rm cris}^\nabla(\Rbar_\cU)=\sum_{a+b\geq
r} t^b \Fil^a \bA_{\rm
cris,M}^{\nabla}(\Rbar_\cU)\bigl[p^{-1}\bigr] \mbox{ and } \Fil^r B_{\rm
cris}(\Rbar_\cU)=\sum_{a+b\geq r} t^b \Fil^a \bA_{\rm
cris,M}(\Rbar_\cU)\bigl[p^{-1}\bigr].$$
\end{proposition}
\begin{proof} The first claim is clear. The filtrations on~$B_{\rm
cris}^{\nabla}(\Rbar_\cU)$ and~$B_{\rm cris}(\Rbar_\cU)$ are defined in loc.~cit.~as the pull--back of the natural filtrations on~$B_{\rm dR}(\Rbar_\cU)^{\nabla +}$
and~$B_{\rm dR}(\Rbar_\cU)^+$ via the inclusions~$B_{\rm cris}^{\nabla}(\Rbar_\cU)\subset B_{\rm dR}(\Rbar_\cU)^{\nabla}$ and~$ B_{\rm cris}(\Rbar_\cU)\subset
B_{\rm dR}(\Rbar_\cU)$. In particular this induces a filtration on~$A_{\rm cris}^{\nabla}(\Rbar_\cU)$ by restriction. It is proved in~\cite[Pf.~Prop.~6.2.1]{brinon}
that it coincides with the filtration induced by $\Fil^\bullet \bA_{\rm cris,M}^{\nabla}$ via the identification $\bA_{\rm cris,M}^{\nabla}(\Rbar_\cU)=A_{\rm
cris}^{\nabla}(\Rbar_\cU)$. By loc.~cit.~$A_{\rm cris}^{\nabla}(\Rbar_\cU)$ maps to the ring $B_{\rm dR}(\Rbar_\cU)^+=B_{\rm
dR}(\Rbar_\cU)^{\nabla+}[\![X_1,\ldots,X_d]\!]$ and due to~\cite[Prop.~5.2.5]{brinon}  we have~$\Fil^r\bigl(B_{\rm dR}(\Rbar_\cU)^+\bigr)=\sum_{s_0+\ldots s_d=r}
\Fil^{s_0}\bigl(B_{\rm dR}(\Rbar_\cU)^{\nabla+}\bigr) X_1^{s_1}\cdots X_d^{s_d}$. Since~$A_{\rm cris}(\Rbar_\cU)=A_{\rm cris}^{\nabla}(\Rbar_\cU)\left\{\langle
X_1,\ldots,X_d\rangle\right\}$, also the filtration on $A_{\rm cris}(\Rbar_\cU)$ induced from~$B_{\rm cris}(\Rbar_\cU)$ coincides with the filtration associated to
$\Fil^\bullet \bA_{\rm cris,M}$ via the identification $\bA_{\rm cris,M}(\Rbar_\cU)=A_{\rm cris}(\Rbar_\cU)$. Since multiplication by~$t$ induces a shift by~$-1$
on~$\Fil^\bullet B_{\rm cris}^\nabla(\Rbar_\cU)$ and on~$\Fil^\bullet B_{\rm cris}(\Rbar_\cU)$ by~\cite[Prop.~5.2.1]{brinon} and~\cite[Prop.~5.2.5]{brinon} and
since multiplication by $p$ on $B_{\rm dR}(\Rbar_\cU)$ is an isomorphism and preserves the filtration, the claim follows.
\end{proof}

For every~$i\in\N$ let~$\Omega^i_{X/\cO_K}\in \Sh(X^{\rm et})$ be the sheaf of continuous K\"ahler differentials on the  \'etale  site of~$X$ relative to~$\cO_K$.
Then $v_{X,M}^\ast\bigl(\Omega^i_{X/\cO_K}\bigr)$ is a locally free sheaf of $v_{X,M}^\ast(\cO_X)\cong \cO_{\fX_M}^{\rm un}$-modules over $\fX_M$. The de Rham
complex on~$X$ defines a de Rham complex $v_{X,M}^\ast\bigl(\Omega^\bullet_{X/\cO_K}\bigr)$ on~$\fX_M$. For every~$n$ we get a complex $\WW_{X,M,n}
\otimes_{\cO_{\fX_M}^{\rm un}} v_{X,M}^\ast\bigl(\Omega^\bullet_{X/\cO_K}\bigr)$ with $\WW_{n,M}$--linear maps~$\nabla^{i+1}\colon \WW_{X,n,M}
\otimes_{\cO_{\fX_M}^{\rm un}} v_{X,M}^\ast\bigl(\Omega^i_{X/\cO_K}\bigr) \to \WW_{X,n,M} \otimes_{\cO_{\fX_M}^{\rm un}}
v_{X,M}^\ast\bigl(\Omega^{i+1}_{X/\cO_K}\bigr)$.\smallskip

{\it Convention:} In order to simplify the notation, for every
sheaf of $\cO_{\fX_M}^{\rm un}$-modules ${\cal E}$ and any sheaf
of $\cO_X$-modules ${\cal M}$ we write $\cE\otimes_{\cO_X} \cM$
for $\cE\otimes_{\cO_{\fX_M}^{\rm un}}
v_{X,M}^\ast\bigl(\cM\bigr)$.\smallskip

Let $d$ be the relative dimension of\/ $X$ over~$\cO_K$. Then we have

\begin{proposition}\label{prop:deRhamcomplex}
The complex $\WW_{X,n,M} \otimes_{\cO_X} \Omega^\bullet_{X/\cO_K}$
extends uniquely to a complex $$\bA_{\rm cris,M}
\stackrel{\nabla^1}{\lra} \bA_{\rm cris,M} \otimes_{\cO_X}
\Omega^1_{X/\cO_K} \stackrel{\nabla^2}{\lra} \bA_{\rm cris,M}
\otimes_{\cO_X} \Omega^2_{X/\cO_K} \lra \cdots
\stackrel{\nabla^d}{\lra} \bA_{\rm cris,M} \otimes_{\cO_X}
\Omega^d_{X/\cO_K} \lra 0 $$with the following property: for
every~$(\cU,\cW)\in \fX_M$, for~$m$, $n$ and $i\in\N$ and for
$x\in \Ker\bigl(\theta_{X,n}\bigr)(\cU,\cW)\in \bA_{\rm
cris,n,M}(\cU,\cW)$ and~$\omega\in \Omega^i_{\cU/\cO_K}$ we have
$\nabla^{i+1}(x^{[m]}\tensor \omega)=x^{[m-1]} \tensor
\nabla^{i+1}(x \omega)$.  Furthermore

\begin{enumerate}

\item[i.] the sequence above is exact;

\item[ii.] the natural inclusion $\bA_{\rm cris,M}^\nabla\subset
\bA_{\rm cris,M}$  identifies $\Ker(\nabla^1)$ with~$\bA_{\rm
cris,M}^\nabla$;

\item[iii.] {\rm (Griffith's transversality)} we have
$\nabla\left(\Fil^r\bigl(\bA_{\rm cris,M}\bigr)\right)\subset
\Fil^{r-1}\left(\bA_{\rm cris,M}\right)\otimes_{\cO_X}
\Omega^1_{X/\cO_K}$ for every~$r$;

\item[iv.] for every $r\in \N$ the sequence $0\lra \Fil^r \bA_{\rm
cris,M}^\nabla \lra \Fil^r \bA_{\rm cris,M}
\stackrel{\nabla^1}{\lra}  \Fil^{r-1} \bA_{\rm cris,M}
\otimes_{\cO_X} \Omega^1_{X/\cO_K} \stackrel{\nabla^2}{\lra}
\Fil^{r-2} \bA_{\rm cris,M} \otimes_{\cO_X} \Omega^2_{X/\cO_K}
\stackrel{\nabla^3}{\lra} \cdots$, with the convention that
$\Fil^s \bA_{\rm cris,M}=\bA_{\rm cris,M} $ for $s<0$, is exact;

\item[v.] the connection $\nabla\colon \bA_{\rm cris,M}\lra
\bA_{\rm cris,M}\otimes_{\cO_X} \Omega^1_{X/\cO_K}$ is
quasi--nilpotent;

\item[vi.] Let~$\cU$ be a  small affine, choose parameters
$T_1,\ldots,T_d\in R_\cU^\times$ and let $F_\cU$ be the induced
lift of absolute Frobenius to~$R_\cU$. Then
Frobenius~$\varphi_\cU$ on~$\bA_{\rm cris,M}\vert_\cU$ is
horizontal with respect to~$\nabla_\cU$ i.~e., $\nabla\vert_\cU
\circ \varphi_\cU=\bigl(\varphi_\cU\tensor dF_\cU\bigr)\circ
\nabla\vert_\cU$.

\end{enumerate}

\end{proposition}
\begin{proof} The uniqueness is clear. We have to prove that the
formula defining~$\nabla^i$ is well defined. By uniqueness it suffices to show it after passing to the subcategory~$\fU_{n,M}$ where~$\cU$ is a small affine  of
$X^{\rm et}$. Write $\bA_{\rm cris,M}\vert_{\fU_{n,M}}\cong \bA_{\rm cris,n,M}^\nabla\vert_{\fU_{n,M}} \langle X_1,\ldots,X_d\rangle$ as in~\ref{thm:gluing}. Then
$\cO_{\fX_M}\otimes_{\cO_X} \Omega^1_{X/\cO_K}$ restricted to $\fU_M$ is a free $\cO_{\fU_M}$--module with basis~$dT_1,\ldots,dT_d$ and the element $X_i\in
\WW_{X,n}(\fU_{n,M})$ satisfies~$\nabla(X_i)=1 \tensor d T_i$. In particular the complex above extends uniquely to a complex $\bA_{\rm cris,M}\vert_{\fU_{n,M}}
\otimes_{\cO_{\cU}} \Omega^\bullet_{\cU/\cO_K}$ characterized by the property that~$\nabla\bigl(X_i^{[m]}\bigr)=X_i^{[m-1]}\tensor d T_i$.

Claims~(i)--(vi) can be also checked after passing
to~$\fU_{n,M}$. Claims~(ii) and~(iii) follow from the formulae
given above. Claims~(i) and~(iv) follow from the formulae and
Poincar\'e's lemma for the PD polynomial algebras (\cite[Proof of
Thm.~6.12]{berthelot_ogus}). Claim~(v) follows remarking that
$\nabla(\partial/\partial T_i)^p \equiv 0$ modulo~$p \bA_{\rm
cris,M}$.

Note that $F_\cU(T_i)=T_i^p$ so that $\varphi_\cU(X_i)=1\tensor T_i^p-\tT_i^p\tensor 1$. Hence $\varphi_\cU(X_i^{[m]})=\bigl(1\tensor T_i^p-\tT_i^p\tensor
1\bigr)^{[m]}$. We compute $\nabla\bigl(\varphi_\cU(X_i^{[m]})\bigr)=\bigl(1\tensor T_i^p-\tT_i^p\tensor 1\bigr)^{[m-1]}\tensor \nabla(1\tensor T_i^p-\tT_i^p\tensor
1)=\bigl(1\tensor T_i^p-\tT_i^p\tensor 1\bigr)^{[m-1]}\tensor d T_i^p=(\varphi_\cU\tensor d F_\cU) \nabla(X_i^{[m]} )$. This proves~(vi).
\end{proof}

We conclude this section with a variant of the constructions above
considering Tate twists. For every integer $r$ define  the inverse
systems of sheaves of $\bA_{\rm cris,n,M}^{'\nabla}$-modules
(resp.~$\bA_{\rm cris,n,M}'$-modules):
$$\bA_{\rm cris, M}^\nabla(r):=\left\{u_{X,M,\ast}\bigl(\Z_p/p^n
\Z_p(r) \bigr) \tensor_{\Z_p} \bA_{\rm
cris,n,M}^{'\nabla}\right\}_n$$and
$$\bA_{\rm cris,
M}(r):=\left\{u_{X,M,\ast}\bigl(\Z_p/p^n \Z_p(r) \bigr)
\tensor_{\Z_p} \bA_{\rm cris, n, M}'\right\}_n$$For $i\geq 1$
define
$$\nabla^{i-1}(r)\colon \bA_{\rm
cris,M}(r)\tensor_{\cO_X}\Omega^i_{X/\cO_K}  \lra \bA_{\rm
cris,M}(r)\tensor_{\cO_X}\Omega^{i+1}_{X/\cO_K}$$to be induced by
the system of morphisms on $u_{X,M,\ast}\bigl(\Z_p/p^n
\Z_p(r)\bigr) \tensor_{\Z_p} \bA_{\rm cris,n,M}'
\tensor_{\cO_X}\Omega^i_{X/\cO_K} \to u_{X,M,\ast}\bigl(\Z_p/p^n
\Z_p(r)\bigr)\tensor_{\Z_p} \bA_{\rm
cris,n,M}'\tensor_{\cO_X}\Omega^{i+1}_{X/\cO_K}$ given by
$1\otimes \nabla^{i-1}$. Put $\nabla(r):=\nabla^1(r)$.

We define an exhaustive, separated decreasing filtration on $\bA_{\rm
cris, M}^\nabla(r)$ (resp.~on $\bA_{\rm cris,M}(r)$) by inverse
systems of sub-sheaves by setting
$$\Fil^i \bA_{\rm
cris,M}^\nabla(r):=\left\{u_{X,M}^\ast\bigl(\Z_p/p^n
\Z_p(r)\bigr)\tensor_{\Z_p} \Fil^{i-r}\bA_{\rm cris,n,
M}^{'\nabla}(\cU,\cW) \right\}_n$$and $$\Fil^i \bA_{\rm
cris,M}(r):=\left\{u_{X,M}^\ast\bigl(\Z_p/p^n
\Z_p(r)\bigr)\tensor_{\Z_p} \Fil^{i-r}\bA_{\rm cris,n,
M}'(\cU,\cW) \ \right\}_n$$for every~$i\geq r$ and setting it to
be $\bA_{\rm cris,M}^\nabla(r)$ (resp.~$\bA_{\rm cris,M}(r)$) for
$i\leq r$.

Recall that $p^{-1}=(p-1)! \frac{t^{[p]}}{t^p}\in A_{\rm cris}
\cdot t^{-p}$. Thus, $p^{-r}$ lies in $A_{\rm cris} t^{-pr}$ and,
since it is invariant under $G_M$, it is a well defined element of
$\bA_{\rm cris,n,M}^{'\nabla}(pr)$ for every $n\in\N$. Define
Frobenius $\varphi\colon \bA_{\rm cris,M}^{'\nabla}(r) \to
\bA_{\rm cris,M}^{'\nabla}(pr)$ to be the system of morphisms
$$u_{X,M}^\ast\bigl(\Z_p/p^n \Z_p(r)\bigr)\tensor_{\Z_p} \bA_{\rm cris,n, M}^{'\nabla}
\lra u_{X,M}^\ast\bigl(\Z_p/p^n
\Z_p(r)\bigr)\tensor_{\Z_p}\bA_{\rm cris,n, M}^{'\nabla},$$given
by $a\otimes b \mapsto a\otimes p^{-r}\varphi(b)$. Assume
that~$\cU$ is a small affine, choose parameters $T_1,\ldots,T_d\in
R_\cU^\times$ and let $F_\cU$ be the induced lift of absolute
Frobenius to~$R_\cU$. Then using the same formula we get a
Frobenius $\varphi_\cU\colon \bA_{\rm cris,M}'(r)\vert_\cU \to
\bA_{\rm cris,M}'(pr)\vert_\cU$ which is horizontal with respect
to the connection $\nabla(r)_\cU$.

\begin{lemma}\label{lemma:FilAcris(r)} The filtrations~$\{\Fil^i\left(\bA_{\rm
cris,M}^{\nabla}(r)\right)\}_{r\in\N}$ and~$\{\Fil^i\left(\bA_{\rm
cris,M}(r)\right)\}_{i\in\N}$ are decreasing, separated and
exhaustive.

Let~$\cU$ be a small affine  of $X^{\rm et}$. Then
$$\Fil^i\left(\bA_{\rm cris,M}^{\nabla}(r)\right)(\Rbar_\cU)\cong
\Fil^{i-r} A_{\rm cris}^\nabla(\Rbar_\cU)\cdot t^r \mbox{ and }
\Fil^i\left(\bA_{\rm cris,M}(r)\right)(\Rbar_\cU)\cong \Fil^{i-r}
A_{\rm cris}(\Rbar_\cU)\cdot t^r.$$

The analogue of claims (i)--(vi) of \ref{prop:deRhamcomplex} hold
for the connection $\nabla(r)$ and the system of sheaves $\bA_{\rm
cris,M}^{\nabla}(r)$ and $\bA_{\rm cris,M}(r)$.
\end{lemma}
\begin{proof} The first and second statements follow
from proposition \ref{prop:filtAcrisnabla}. The last statement follows from
the definition and proposition \ref{prop:deRhamcomplex}.
\end{proof}

\subsection{The ind-sheaves $\bB_{\rm cris}^\nabla$ and $\bB_{\rm
cris}$.}\label{def:bBcris}

{\it Generalities on inductive systems.}  Let $\mathcal{A}$ be an
abelian category. We denote by ${\rm Ind}\bigl(\mathcal{A}\bigr)$,
called the category of inductive systems of objects of
$\mathcal{A}$, the category whose objects are
$\bigl(A_i,\gamma_i\bigr)_{i\in\Z}$ with $A_i$ object of
$\mathcal{A}$ and  $\gamma_i\colon A_i \to A_{i+1}$ morphism in
$\mathcal{A}$ for every $i\in\Z$. Given an integer $N\in\Z$ a
morphism $f\colon \underline{A}:=\bigl(A_i,\gamma_i\bigr)_{i\in\Z}
\lra \underline{B}:=\bigl(B_j,\delta_j\bigr)_{j\in\Z}$ of degree
$N$ is a system of morphisms $f_i\colon A_i \to B_{i+N}$ for
$i\in\Z$ such that $\delta_{i+N} \circ f_i=f_{i+1} \circ
\gamma_i$. Since $\mathcal{A}$ is an additive category the set of
morphisms of degree $N$ form an abelian group with the zero map,
the sum of two functions and the inverse of a function defined
componentwise. Given a morphism $f=(f_i)_{i\in \Z}\colon
\underline{A} \lra \underline{B}$ of degree $N$ we get a morphism
of degree $N+1$ given by $\bigl(\delta_{i+N} \circ
f_i\bigr)_{i\in\Z}$. This defines a group homomorphism from the
morphisms $\Hom^N\bigl(\underline{A},\underline{B}\bigr)  $ of
degree $N$ to the morphisms
$\Hom^{N+1}\bigl(\underline{A},\underline{B}\bigr) $ of degree
$N+1$. We define the group of morphisms $f\colon
\bigl(A_i,\gamma_i\bigr)_{i\in\Z} \lra
\bigl(B_j,\delta_j\bigr)_{j\in\Z}$ in  ${\rm
Ind}\bigl(\mathcal{A}\bigr)$ to be the inductive limit
$\lim_{N\in\Z} \Hom^N\bigl(\underline{A},\underline{B}\bigr)$ with
respect to the transition maps just defined.

Given any such morphism $f$ we let $\Ker(f)$ be the inductive system
$\bigl(\Ker(f_i)\bigr)_{i\in\Z}$ with transition morphisms defined
by the $\gamma_i$'s. We let $\Coker(f)$ be the inductive system
$\bigl(\Coker(f_{i-N}\bigr)_{i\in\Z}$ with transition morphisms
induced by the $\delta_i$'s.  One verifies that with these
definitions  the category ${\rm Ind}\bigl(\mathcal{A}\bigr)$ is
an abelian category. Note that we have a natural functor
$$\mathcal{A}\lra {\rm Ind}\bigl(\mathcal{A}\bigr)$$sending $A$ to
the inductive system $(A,{\rm Id})_{i\in\Z}$ which is exact and
fully faithful.

Assume furthermore that $\mathcal{A}$ is a tensor category. Given
objects $\underline{A}:=\bigl(A_i,\gamma_i\bigr)_{i\in\Z}$ and $
\underline{B}:=\bigl(B_j,\delta_j\bigr)_{j\in\Z}$ in ${\rm
Ind}(\mathcal{A})$ we define $\underline{A}\tensor \underline{B}$
to be the inductive system $\bigl(A_i\tensor B_i,\gamma_i\tensor
\delta_i\bigr)_{i\in\Z}$. In this way ${\rm
Ind}\bigl(\mathcal{A}\bigr)$ is endowed with the structure of a
tensor category so that the functor $\mathcal{A}\lra {\rm
Ind}\bigl(\mathcal{A}\bigr)$  is a morphism of tensor categories.
By abuse of notation given an object $B\in \mathcal{A}$ we write
$\underline{A}\tensor B$ for the inductive system
$\underline{A}\tensor \underline{B}$ with
$\underline{B}=\bigl(B,{\rm Id}\bigr)$.

Let $ \mathcal{B}$ be an abelian category in which direct limits
of inductive systems indexed by $\Z$ exist. Consider the induced
functor
$$\displaystyle{\lim_{\to}}\colon {\rm Ind}\bigl(\mathcal{B}\bigr)
\lra \mathcal{B}.$$Suppose we are given $\delta$-functors
$T^n\colon \mathcal{B} \to \mathcal{A}$ with ${n\in\N}$. Define
$$\displaystyle{\lim_{\to}} T^n\colon{\rm
Ind}\bigl(\mathcal{A}\bigr)\lra \mathcal{B}$$as the composite of
the functor ${\rm Ind}\bigl(\mathcal{A}\bigr) \to {\rm
Ind}\bigl(\mathcal{B}\bigr) $, given by $(A_i)_{in\Z}\mapsto
\bigl(T^n(A_i)\bigr)_{i\in\Z}$ and of the functor $
\displaystyle{\lim_{\to}}$.

\begin{lemma} If $\displaystyle{\lim_{\to}}$ is left exact then
the functors $\displaystyle{\lim_{\to}} T^n$, for varying
$n\in\N$, define a $\delta$-functor. If $(T^n)_{n\in\N}$ is
universal then also $\bigl(\displaystyle{\lim_{\to}}
T^n\bigr)_{n\in\N}$ is universal.
\end{lemma}
\begin{proof}
Due to the universal property of direct limits the functor
$\displaystyle{\lim_{\to}}$ is always right exact. Thus the
assumption is equivalent to the requirement that it is exact. The
claim follows.
\end{proof}

\bigskip

\begin{remark}\label{remark:enlargemorphism} One can relax the
definition of a morphism in ${\rm Ind}\bigl(\mathcal{A}\bigr)$.
Consider a non-decreasing function $\alpha\colon \Z\to \Z$. Given
objects $\underline{A}:=\bigl(A_i,\gamma_i\bigr)_{i\in\Z}$ and $
\underline{B}:=\bigl(B_j,\delta_j\bigr)_{j\in\Z}$, we define a
morphism $f\colon \underline{A}\to \underline{B}$ of type $\alpha$
to be a collection of morphisms $f_i\colon A_i\to B_{\alpha(i)}$
such that $f_{i+1}\circ \gamma_i= \prod_{\alpha(i) \leq j <
\alpha(i+1)} \delta_j \circ f_i $. We denote by
$\Hom^\alpha\bigl(\underline{A},\underline{B}\bigr)  $ the group
of homomorphisms of type $\alpha$. We say that two morphisms $f$
and $g$ of type $\alpha$ (resp.~$\beta$) are equivalent if there
exists $N\in\N$ such that $f_i$  composed with $B_{\alpha(i)}\to
B_{\max(\alpha(i),\beta(i))+N}$ and $g_i$ composed with
$B_{\beta(i)}\to B_{\max(\alpha(i),\beta(i))+N}$ coincide. One
checks that this defines an equivalence relation. We define a
morphism $\underline{A}\to \underline{B} $ to be a class of
morphisms with respect to this equivalence relation. The morphisms
in the more restrictive sense given before inject into this new
class of morphisms. Since this complicates the notation we will
work mainly with the previous more restrictive notion.
\end{remark}

Recall that given an integer~$r$ the sheaf $\bA_{\rm
cris,n,M}^\nabla(r)$ is characterized by the property that for
every small affine open $\cU\in X^{\rm et}$ its localization
$\bA_{\rm cris,n,M}^\nabla(r)(\Rbar_\cU)$ is the group $\bA_{\rm
cris,n,M}^\nabla(\Rbar_\cU)$ with action of $\cG_{\cU,M}$ twisted
by the $r$-th power of the cyclotomic character. Let
$t:=\log[\epsilon]\in A_{\rm cris}$ (see section
\ref{sec:Notation}). Fix integers $r\geq s$. For every $\cU$ as
above we have  a map $\bA_{\rm cris,n,M}^\nabla(s)(\Rbar_\cU)
\lra \bA_{\rm cris,n,M}^\nabla(r)(\Rbar_\cU)$ of $\bA_{\rm
cris,n,M}^\nabla(\Rbar_\cU)$-module sending $1 \mapsto 1\tensor
t^{[r-s]}$. Since they are equivariant with respect to the action
of $\cG_{\cU,M}$, these maps for varying $\cU$  arise from a
unique morphism $$j_{r,s}\colon  \bA_{\rm cris,n,M}^\nabla(s)\to
\bA_{\rm cris,n,M}^\nabla(r).$$ They are compatible for varying
$n\in\N$ and define a morphism of continuous sheaves
$j_{r,s}\colon \bA_{\rm cris,M}^\nabla(s)\to \bA_{\rm
cris,M}^\nabla(r)$. Define $\iota_{r,s}:=(r-s)! j_{r,s}$. It
follows from the construction that $j_{r,s}$, and hence
$\iota_{r,s}$, sends $\Fil^n \bA_{\rm cris,M}^\nabla(s)$ to
$\Fil^n \bA_{\rm cris,M}^\nabla(r)$ for every $n\in\Z$.
\smallskip

Define $\bB_{\rm cris,M}^\nabla$ in ${\rm Ind}\left(\Sh(\fX_M)^\N
\right)$ to be the inductive system of continuous  sheaves having
$\bA_{\rm cris,M}^\nabla(-r)$ in degree $r$ with transition maps
given by $\iota_{r-1,r}$.  Analogously for every $n\in\Z$ let
$\Fil^n \bB_{\rm cris,M}^\nabla$ in ${\rm Ind}\left(\Sh(\fX_M)^\N
\right)$ be the inductive system of continuous  sheaves having
$\Fil^n \bA_{\rm cris,M}^\nabla(-r)$ in degree $r$ with transition
maps induced by $\iota_{r-1,r}$. By construction it is a sub-object
of $\bB_{\rm cris,M}^\nabla$. The Frobenius morphisms on
$\varphi\colon \bA_{\rm cris,M}^{'\nabla}(r) \to \bA_{\rm
cris,M}^{'\nabla}(pr)$ are compatible for varying $r\in\Z$ with
$\iota_{r-1,r}$. Using the more general notion of a morphism of
inductive systems given in \ref{remark:enlargemorphism} it
induces a morphism $\varphi\colon \bB_{\rm cris,M}^\nabla \to
\bB_{\rm cris,M}^\nabla $ in ${\rm Ind}\left(\Sh(\fX_M)^\N
\right)$ sending $\Fil^r\bB_{\rm cris,M}^\nabla \lra
\Fil^{r+p}\bB_{\rm cris,M}^\nabla$.

Similarly,  we define the continuous sheaves $\bA_{\rm
cris,M}(r)$. As before we get the inductive systems $\bB_{\rm
cris,M}$ and $\Fil^n \bB_{\rm cris,M}$, for $n\in\Z$, in ${\rm
ind}\left(\Sh(\fX_M)^\N \right)$ as the inductive system of
continuous  sheaves having $\bA_{\rm cris,M}(-r)$ (resp.~$\Fil^n
\bA_{\rm cris,M}(-r)$) in degree $r$. Assume that~$\cU$ is a small
affine. Then  Frobenius $\varphi_\cU\colon \bA_{\rm
cris,M}'(r)\vert_\cU \to \bA_{\rm cris,M}'(pr)\vert_\cU$ induces,
in the more general framework of \ref{remark:enlargemorphism}, a
morphism $\varphi_\cU\colon \bB_{\rm cris,M}\vert_\cU \to \bB_{\rm
cris,M}\vert_\cU $ in ${\rm Ind}\left(\Sh(\fUM)^\N \right)$.

We remark that the morphisms $j_{s+p,s}$ for varying $s\in \Z$
define a morphism $j$ on $\Fil^n \bB_{\rm cris,M}^\nabla$,
$\bB_{\rm cris,M}^\nabla$, $\Fil^n \bB_{\rm cris,M}$ and $\bB_{\rm
cris,M}$ such that $p! j$ is the identity in the category of
inductive systems. We deduce that multiplication by $p$ is an
isomorphism on all the objects above.

For notational convention put $\Fil^{-\infty} \bB_{\rm
cris,M}^\nabla=\bB_{\rm cris,M}^\nabla$ and similarly without
$\nabla$. As explained before, the localization functor on
$\Sh(\fX_M)^\N$, $\cF\mapsto \cF(\Rbar_\cU)$, extends to a functor
on ${\rm Ind}\left( \Sh(\fX_M)^\N\right)$. Put $\Fil^{-\infty}
B_{\rm cris}^\nabla(\Rbar_\cU)=B_{\rm cris}^\nabla(\Rbar_\cU)$ and
similarly without $\nabla$. Summarizing and using the results of
the previous section we get :

\begin{lemma}\label{lemma:propbBcris} (1) Multiplication by $p$ is an isomorphism
on $\Fil^n \bB_{\rm cris,M}^\nabla$, $\bB_{\rm cris,M}^\nabla$,
$\Fil^n \bB_{\rm cris,M}$ and $\bB_{\rm cris,M}$. \smallskip

(2) For every $r\in\Z\cup\{-\infty\}$ we have an exact sequence
of inductive systems

$$0\lra \Fil^r
\bB_{\rm cris,M}^\nabla \lra \Fil^r \bB_{\rm cris,M} \stackrel{\nabla^1}{\lra}  \Fil^{r-1} \bB_{\rm cris,M} \otimes_{\cO_X} \Omega^{1}_{X/\cO_K}
\stackrel{\nabla^2}{\lra} \Fil^{r-2} \bB_{\rm cris,M} \otimes_{\cO_X} \Omega^{2}_{X/\cO_K} \cdots .$$ Let~$\cU$ be a small affine, choose parameters
$T_1,\ldots,T_d\in R_\cU^\times$ and let $F_\cU$ be the induced lift of absolute Frobenius to~$R_\cU$. Then,\smallskip

(3)  Frobenius~$\varphi_\cU$ on~$\bB_{\rm cris,M}\vert_\cU$ is
horizontal with respect to~$\nabla\vert_\cU$ and induces Frobenius
on $\bB_{\rm cris,M}^\nabla\vert_\cU$.\smallskip

(4) $\bB_{\rm cris,M}^{\nabla}(\Rbar_\cU)\cong B_{\rm
cris}^\nabla(\Rbar_\cU)$ and $\bB_{\rm cris,M}(\Rbar_\cU)\cong
B_{\rm cris}(\Rbar_\cU)$. Furthermore, $\Fil^i\left(\bB_{\rm
cris,M}^{\nabla}\right)(\Rbar_\cU)\cong \Fil^i B_{\rm
cris}^\nabla(\Rbar_\cU)$ and $\Fil^i\left(\bB_{\rm
cris,M}\right)(\Rbar_\cU)\cong \Fil^i B_{\rm cris}(\Rbar_\cU)$ for
every $i\in\Z$.

\end{lemma}

\subsection{The fundamental exact sequence}

Following \cite[\S 5.3.6]{Fontaineperiodes} put  $\Fil^r_p A_{\rm
cris}=\{x\in \Fil^r A_{\rm cris}\vert \varphi(x)\in p^r A_{\rm
cris}\}$ for every $r\in\N$. Let $\frac{\varphi}{p^r}\colon
\Fil^r_p A_{\rm cris} \to A_{\rm cris}$ be the induced map. Note
that $p^r \Fil^r A_{\rm cris} \subset \Fil^r A_{\rm cris} \subset
\Fil^r A_{\rm cris}$. For every $n$ and $r\in\N$ define the sheaf
$$\Fil^r_p\bA_{\rm cris,n,\Kbar}^{'\nabla}:=\bigl(\Fil^r_p
A_{\rm cris}/ p^n \Fil^r_p A_{\rm cris}\bigr) \otimes_{W_n}
\WW_{n,\Kbar}.$$For $r=0$ it coincides with $\bA_{\rm
cris,n,\Kbar}^{'\nabla}$ thanks to
lemma \ref{lemma:exactsequenceAcrisnnabla'}(c). Since $\WW_{n,\Kbar}$ is
flat as a sheaf of $W_n$-modules by corollary \ref{cor:WnXisflatoverWnVbar},
$\Fil^r_p\bA_{\rm cris,n,\Kbar}^{'\nabla}$ defines a subsheaf
of $\Fil^0_p\bA_{\rm cris,n,\Kbar}^{'\nabla}=\bA_{\rm
cris,n,\Kbar}^{'\nabla}$. Let
$$\frac{\varphi}{p^r}\colon \Fil^r_p\bA_{\rm
cris,n,\Kbar}^{'\nabla} \lra \bA_{\rm cris,n,\Kbar}^{'\nabla}$$be
the morphism defined by $\frac{\varphi}{p^r}$ on $\Fil^r_p A_{\rm
cris}/ p^n \Fil^r_p A_{\rm cris}$ and by $\varphi$ on
$\WW_{n,\Kbar}$. Let $\Fil^r_p\bA_{\rm cris,n+1,\Kbar}^{'\nabla}
\to \Fil^r_p\bA_{\rm cris,n,\Kbar}^{'\nabla}$ be the morphism
defined by reduction modulo $p^n$ on $\Fil^r_p A_{\rm cris}/
p^{n+1} \Fil^r_p A_{\rm cris}$ and $r_{n+1,\Kbar}\colon
\WW_{n+1,\Kbar}\to \WW_{n,\Kbar}$. Put $\Fil^r_p\bA_{\rm
cris,\Kbar}^{'\nabla}$ to be the associated inverse system of
sheaves. Write $\frac{\varphi}{p^r}$ for the induced morphism
$\Fil^r_p\bA_{\rm cris,\Kbar}^{'\nabla} \to \bA_{\rm
cris,\Kbar}^{'\nabla}$.

\begin{proposition}\label{prop:phioverp^ronFulrcris}
Assume we are in the formal case. Then:

(1) The  sequence $$ 0 \lra \Z/p^n \Z t^{\{r\}} \lra
\Fil^r_p\bA_{\rm
cris,n,\Kbar}^{'\nabla}\stackrel{1-\frac{\varphi}{p^r}}{\lra}
\bA_{\rm cris,n,\Kbar}^{'\nabla}\lra 0$$is exact.

(2) The morphism of continuous sheaves $\Fil^r_p\bA_{\rm
cris,\Kbar}^{'\nabla}\to \Fil^r\bA_{\rm cris,\Kbar}^{'\nabla}$ is
an isomorphism in $\Sh(\fX_M)_{\Q_p}$.

\end{proposition}

\begin{proof} We start with the proof of (2).
It follows from \ref{cor:WnXisflatoverWnVbar} and
\ref{lemma:exactsequenceAcrisnnabla'}  that  $\Fil^r \bigl(A_{\rm
cris}/p^n A_{\rm cris}\bigr)\otimes_{W_n} \bA_{\rm
cris,n,\Kbar}^{'\nabla}\to \Fil^r\bA_{\rm cris,n,\Kbar}^{'\nabla}$
is an isomorphism. We then get natural maps $p^r \Fil^r\bA_{\rm
cris,n-r,\Kbar}^{'\nabla} \to \Fil^r_p\bA_{\rm
cris,n,\Kbar}^{'\nabla} \to \Fil^r\bA_{\rm
cris,n-r,\Kbar}^{'\nabla}$ inducing morphisms of continuous
sheaves
$$p^r \Fil^r\bA_{\rm cris,\Kbar}^{'\nabla} \to \Fil^r_p\bA_{\rm
cris,\Kbar}^{'\nabla} \to \Fil^r\bA_{\rm
cris,\Kbar}^{'\nabla}.$$This proves (2).

For the proof of (1) we proceed as in the proof of \cite[Thm. A3.26]{tsuji}. Following \cite[\S 5.3.1]{Fontaineperiodes} define $I^{[s]}A_{\rm cris} :=\{x\in A_{\rm
cris}\vert \varphi^n(x)\in \Fil^s A_{\rm cris} \,\forall n\in\N\}$. For every $m\in\N$ write $m=q(m) (p-1)+ r(m)$ with $0\leq r(m) < p-1$. Let
$t^{\{m\}}:=t^{r(m)}\cdot \left(\frac{t^{p-1}}{p} \right)^{[q(m)]}$. It is proven in loc.~cit.~that $I^{[s]}A_{\rm cris} $ is the closure for the $p$-adic topology
of the $A_{\rm inf}^+$-module generated by the elements $t^{\{s\}}$ for $s\geq r$. Furthermore $A_{\rm cris}/I^{[s]}A_{\rm cris} $ is $p$-torsion free by
\cite[Prop. 5.3.5]{Fontaineperiodes}. In particular the decreasing filtration $I^{[s]}A_{\rm cris} \cap \Fil^r_p A_{\rm cris}$ on $\Fil^r_p A_{\rm cris}$ for
$s\in\N$, has torsion free graded quotients. Its reduction modulo $p^n$ injects into $\Fil^r_p A_{\rm cris}/ p^n \Fil^r_p A_{\rm cris} $ and defines a decreasing
filtration on the latter. Since $\WW_{n,\Kbar}$ is flat as a sheaf of $W_n$-modules by \ref{cor:WnXisflatoverWnVbar} taking $ \otimes_{W_n} \WW_{n,\Kbar}$ we get a
decreasing filtration on $\Fil^r_p\bA_{\rm cris,n,\Kbar}^{'\nabla}$ which we denote by $\Fil^{r,[s]}_p\bA_{\rm cris,n,\Kbar}^{'\nabla}$. We write $I^{[s]} \bA_{\rm
cris,n,\Kbar}^{'\nabla} $ if $r=0$.

Write
$q^\prime:=\frac{[\varepsilon]-1}{[\varepsilon]^{\frac{1}{p}}-1}=1+[\varepsilon]^{\frac{1}{p}}+\cdots
+ [\varepsilon]^{\frac{p-1}{p}}$.   It follows from \cite[\S
5.3.6]{Fontaineperiodes} that $\frac{I^{[s]}A_{\rm cris} \cap
\Fil^r_p A_{\rm cris}}{I^{[s+1]}A_{\rm cris} \cap \Fil^r_p A_{\rm
cris}}$ is $p$-torsion free and it is generated  as $A_{\rm
inf}^+$-module by the element $\bigl(q^{\prime}\bigr)^{r-s}
t^{\{s\}}$ for $0\leq s <r$ and by $t^{\{s\}}$ for $s\geq r$.
Since $\varphi(q')\in p A_{\rm cris}$, by \cite[\S
5.2.9]{Fontaineperiodes}, the map $\frac{\varphi}{p^r} $ sends
$I^{[s]}A_{\rm cris} \cap \Fil^r_p A_{\rm cris}$ to $I^{[s]}A_{\rm
cris}$ so that $1-\frac{\varphi}{p^r} $ sends $I^{[s]}A_{\rm cris}
\cap \Fil^r_p A_{\rm cris}$ to $I^{[s]}A_{\rm cris}$. We deduce
that the morphism $1-\frac{\varphi}{p^r}$ sends
$\Fil^{r,[s]}_p\bA_{\rm cris,n,\Kbar}^{'\nabla}$ to $I^{[s]}
bA_{\rm cris,n,\Kbar}^{'\nabla}$. The conclusion follows from
\ref{phiprongr}.
\end{proof}

\begin{lemma}\label{phiprongr} The morphism $1-\frac{\varphi}{p^r}$ induces isomorphisms
$$\Fil^{r,[r+1]}_p\bA_{\rm cris,n,\Kbar}^{'\nabla}\lra I^{[r+1]} bA_{\rm
cris,n,\Kbar}^{'\nabla}$$and $$\frac{\Fil^{r,[s]}_p\bA_{\rm
cris,n,\Kbar}^{'\nabla}} {\Fil^{r,[s+1]}_p\bA_{\rm
cris,n,\Kbar}^{'\nabla}}\lra \frac{I^{[s]} bA_{\rm
cris,n,\Kbar}^{'\nabla}}{I^{[s+1]} bA_{\rm
cris,n,\Kbar}^{'\nabla}}$$for $0\leq s<r$ and an exact sequence
$$0 \lra \Z/p^n\Z \cdot t^{\{r\}} \lra \frac{\Fil^{r,[r]}_p\bA_{\rm cris,n,\Kbar}^{'\nabla}}
{\Fil^{r,[r+1]}_p\bA_{\rm
cris,n,\Kbar}^{'\nabla}}\stackrel{1-\frac{\varphi}{p^r}}{\lra}
\frac{I^{[r]} bA_{\rm cris,n,\Kbar}^{'\nabla}}{I^{[r+1]} bA_{\rm
cris,n,\Kbar}^{'\nabla}}\lra 0.$$
\end{lemma}
\begin{proof}
By construction we have $\Fil^{r,[r+1]}_p\bA_{\rm
cris,\Kbar}^{'\nabla}=I^{[r+1]} bA_{\rm cris,\Kbar}^{'\nabla}$ so
that the operator $1-\frac{\varphi}{p^r}$ is $1- p \cdot
\frac{\varphi}{p^{r+1}}$ which is unipotent and hence an
isomorphism. This proves the first assertion.

It follows from \cite[Prop. 5.1.3 \& Rmk. 5.3.2]{Fontaineperiodes}
that $I^{[s]}A_{\rm cris} \cap A_{\rm
inf}^+=\bigl([\varepsilon]-1\bigr)^s A_{\rm inf}^+ $ and that
$$A_{\rm inf}^+/ \bigl([\varepsilon]-1\bigr)
A_{\rm inf}^+ \lra  I^{[s]}A_{\rm cris}/ I^{[s+1]}A_{\rm cris}, \quad x \mapsto x \cdot t^{\{s\}}$$is an isomorphism. In particular this isomorphism induces the
isomorphisms

$$A_{\rm inf}^+/ \bigl([\varepsilon]^{\frac{1}{p}}-1\bigr)  A_{\rm inf}^+
\lra \frac{I^{[s]}A_{\rm cris} \cap \Fil^r_p A_{\rm cris}}{
I^{[s+1]}A_{\rm cris} \cap  \Fil^r_p A_{\rm cris}}, \quad x
\mapsto x \bigl(q^{\prime}\bigr)^{r-s} t^{\{s\}}$$for $0\leq s< r$
and the isomorphism
$$A_{\rm inf}^+/ \bigl([\varepsilon]-1\bigr)  A_{\rm inf}^+
\lra \frac{I^{[s]}A_{\rm cris} \cap \Fil^r_p A_{\rm cris}}{
I^{[s+1]}A_{\rm cris} \cap  \Fil^r_p A_{\rm cris}}= I^{[s]}A_{\rm
cris}/ I^{[s+1]}A_{\rm cris}, \quad x \mapsto x  t^{\{s\}}$$for
$s\geq r$. It follows from \cite[\S 5.2.9]{Fontaineperiodes} that
$\bigl(1-\frac{\varphi}{p^r}\bigr)\left(\bigl(q^{\prime}\bigr)^{r-s}
t^{\{s\}}\right)\equiv t^{\{s\}}$ mod $I^{[s+1]}A_{\rm cris}$ for
every $0\leq s\leq r$. We then  deduce by base changing via the
flat extension $W_n \to \WW_{n,\Kbar}$ that for $0\leq s\leq r$
the following diagram

$$\begin{array}{ccc}
\WW_{n,\Kbar}/ \bigl([\varepsilon]^{\frac{1}{p}}-1\bigr)
\WW_{n,\Kbar} & \stackrel{q^{\prime, r-s}-\varphi}{\lra} &
\WW_{n,\Kbar}/ \bigl([\varepsilon]-1\bigr) \WW_{n,\Kbar} \cr
\big\downarrow & & \big\downarrow \cr \frac{\Fil^{r,[s]}_p\bA_{\rm
cris,n,\Kbar}^{'\nabla}} {\Fil^{r,[s+1]}_p\bA_{\rm
cris,n,\Kbar}^{'\nabla}} & \stackrel{1-\frac{\varphi}{p^r}}{\lra}
& \frac{I^{[s]} bA_{\rm cris,n,\Kbar}^{'\nabla}}{I^{[s+1]} bA_{\rm
cris,n,\Kbar}^{'\nabla}} \cr
\end{array}$$is commutative and that the vertical arrows are isomorphisms.
The second claim of the lemma follows remarking that the top
horizontal morphism is the sum of a nilpotent map, given by
multiplication by $q^{\prime, r-s}$ and the map $-\varphi$ which
is an isomorphism due to \ref{cor:WnXisflatoverWnVbar}. Similarly
for $s=r$ the following diagram is commutative with  vertical
arrows isomorphisms

$$\begin{array}{ccc}
\WW_{n,\Kbar}/ \bigl([\varepsilon]-1\bigr) \WW_{n,\Kbar} &
\stackrel{1-\varphi}{\lra} & \WW_{n,\Kbar}/
\bigl([\varepsilon]-1\bigr) \WW_{n,\Kbar} \cr \big\downarrow & &
\big\downarrow \cr \frac{\Fil^{r,[r]}_p\bA_{\rm
cris,n,\Kbar}^{'\nabla}} {\Fil^{r,[r+1]}_p\bA_{\rm
cris,n,\Kbar}^{'\nabla}} & \stackrel{1-\frac{\varphi}{p^r}}{\lra}
& \frac{I^{[r]} bA_{\rm cris,n,\Kbar}^{'\nabla}}{I^{[r+1]} bA_{\rm
cris,n,\Kbar}^{'\nabla}} .\cr
\end{array}$$
The last assertion follows remarking that the top horizontal arrow
is surjective with kernel $\Z/p^n\Z$ due to
corollary \ref{cor:WnXisflatoverWnVbar}.
\end{proof}

Define $\Fil^r_p \bA_{\rm cris,\Kbar}^{\nabla}(m)$ as the system
$\left\{\bigl(\Z_p/p^n \Z_p(r)\bigr)\tensor_{\Z_p}
\Fil^{i-r}_p\bA_{\rm cris,n, M}^{'\nabla} \right\}_n$. Let
$\Fil^r_p \bB_{\rm cris,\Kbar}^\nabla$ be the inductive system of
continuous sheaves $\Fil^r_p \bA_{\rm cris,\Kbar}^{\nabla}(m)$.

Since multiplication by $p$ is an isomorphism on $\Fil^r \bB_{\rm
cris,\Kbar}^\nabla$ by \ref{lemma:propbBcris} it follows from
proposition \ref{prop:phioverp^ronFulrcris} that it coincides with $\Fil^r
\bB_{\rm cris,\Kbar}^\nabla$.  The morphism $1-\varphi\colon
\bA_{\rm cris,\Kbar}^{'\nabla}(-r) \to \bA_{\rm
cris,\Kbar}^{'\nabla}(-pr)$ induces a morphism of inductive
systems $1-\varphi\colon \bB_{\rm cris,\Kbar}^\nabla \lra \bB_{\rm
cris,\Kbar}^\nabla$. Since multplication by $t$ is an isomorphism
on $\bB_{\rm cris,\Kbar}^\nabla$ by \ref{lemma:isoinBcris} we
deduce from proposition \ref{prop:phioverp^ronFulrcris} the exact sequence
$$ 0 \lra \Q_p \lra
\Fil^0\bB_{\rm cris,\Kbar}^{\nabla}\stackrel{1-\varphi}{\lra}
\bB_{\rm cris,\Kbar}^{\nabla}\lra 0.$$We then get the following
commutative diagram with exact rows, called the {\it fundamental
diagram of sheaves}:

\begin{equation}\label{display:fundamentalexactdiagram}
\begin{array}{ccccccccccc} 0& \lra & \Q_p & \lra &\Fil^0
\bB_{\rm cris,\Kbar}^{\nabla} &\stackrel{1-\varphi}{\lra}&
\bB_{\rm
cris,\Kbar}^{\nabla}&\lra& 0\\
&&\cap&&\cap&&\big\downarrow\\
0&\lra&(\bB_{\rm cris,\Kbar}^{\nabla})^{\varphi=1}&\lra&\bB_{\rm
cris,\Kbar}^{\nabla}&\stackrel{1-\varphi}{\lra}&\bB_{\rm
cris,\Kbar}^{\nabla}& \lra & 0.\cr
\end{array}\end{equation}

\section{The crystalline comparison isomorphism.}
\label{sec:compcrys}

\subsection{Crystalline \'etale sheaves.} \label{sec:cryssheaves}

In this section we assume that $X$ is defined over $\cO_K=\WW(k)$
so that the sheaf $\bA_{\rm cris,M}$ is defined. Recall that
we have natural morphisms of sites $u_{M} \colon \fX_M \lra
X_M^{\rm et}$, given by~$(\cU,\cW)\mapsto \cW$, and $v_{M} \colon
X^{\rm et} \lra \fX_M$ given by~$\cU \mapsto
(\cU,\cU^{\rm rig})$.  If~$\cL$ is a sheaf on~$X_M^{\rm et}$, to
ease the notation we simply write~$\cL$ for~$u_{M,\ast}(\cL)$.
The aim of this section is to introduce the so called
``crystalline $\Q_p$-adic sheaves"  on~$X_M^{\rm et}$. As
explained in proposition \ref{prop:equivcris} the definition amounts to a
sheaf theoretic generalization of the usual notion of crystalline
representation, due to Fontaine, in the relative setting. We show
in~\ref{prop:equivcris} that this notion coincides with the notion
of ``locally crystalline representations" introduced by
\cite{brinon} in the relative setting. We will prove in
lemma \ref{lemma:associated} that it is also equivalent to Faltings'
notion of associated sheaves. Contrary to these alternative
definitions which are checked on small enough open affines the
present definition has the advantage of being purely sheaf
theoretic.

\bigskip

{\it $\Q_p$--adic sheaves.}\enspace By a $p$--adic sheaf~$\cL$
on~$X_M^{\rm et}$ we mean a system $\{\cL_n\}\in \Sh(X_M^{\rm
et})^\N$ such that~$\cL_n$ is a locally constant and locally free
of finite rank \'etale sheaf of $\Z/p^n\Z$--modules
and~$\cL_n=\cL_{n+1}/p^n \cL_{n+1}$ for every~$n\in\N$. Given two
$p$-adic sheaves~$\cL:=\{\cL_n\}$ and~$\cM:=\{\cM_n\}$ define
$\cL\tensor_{\Z_p}\cM:=\{\cL_n\tensor_{\Z/p^n \Z}\cM_n \}_n$ and
$\uHom(\cL,\cM):=\{\uHom(\cL_n,\cM_n)\}_n$. Put~${\bf 1}=\Z_p$ to
be the sheaf~$\{\Z/p^n\Z\}_n$ with~$\Z/p^n\Z$ the constant sheaf.
This defines a structure of abelian  tensor category on $p$-adic
sheaves on  $X_M^{\rm et}$. Define~$\Z_p(1)$ to be the
sheaf~$\{\mu_{p^n}\}_n$ of $p$--power roots of unity. For
every~$m\in\N$ define~$\Z_p(m)$ to be the $m$--fold tensor product
of~$\Z_p(1)$. For~$m\leq 0$
put~$\Z_p(m):=\uHom\bigl(\Z_p(-m),\Z_p\bigr)$. For~$m\in\Z$
and~$\cL$ a $p$--adic sheaf
denote~$\cL(m):=\cL\tensor_{\Z_p}\Z_p(m)$.

Define $\Sh(X_M^{\rm et})_{\Q_p}$ to be the full subcategory of
${\rm Ind}\left(\Sh(X_M^{\rm et})^\N\right)$ (see
\S\ref{def:bBcris}) consisting of inductive systems of the form
$(\cL)_{i\in\Z}$ where $\cL$ is a $p$-adic \'etale sheaf and the
transition maps $\cL\to \cL$ are given by multiplication by $p$.
It inherits from the category of $p$-adic sheaves on ~$X_M^{\rm
et}$ the structure of an abelian tensor category.

\bigskip

Let $\cL=\{\cL_n\}$ be a $p$-adic \'etale sheaf. By definition for every $(\cU,\cW)\in \fX_M$ we have $u_{X,M,\ast}\bigl(\cL_n\bigr)(\cU,\cW)=\cL_n(\cW)$. Since
$\cL_n$ is a locally constant sheaf of finite abelian groups there exists $\cW \in \cU_{\rm M,fet}$ such that for every morphism $\bigl(\cU',\cW'\bigr)\to
\bigl(\cU,\cW\bigr)$ in $\fX_M$ the map $\cL_n\bigl(\cW'\bigr)\to \cL_n(\cW)$ is a bijection. In particular  $u_{X,M,\ast}\bigl(\cL_n\bigr)$ is locally constant on
$\fX_M$ and $u_{X,M,\ast}$ is fully faithful. Similarly if we extend $u_{X,M,\ast}$ to inductive systems of inverse systems of sheaves we get a fully faithful
morphism $u_{X,M,\ast}\colon \Sh(X_M^{\rm et})_{\Q_p} \lra {\rm Ind}\left(\Sh(\fX_M^{\rm et})^\N\right) $. We simply write $\cL_n$ for $u_{M,\ast}\bigl(\cL_n\bigr)$
and $\cL$ for the  inverse system of sheaves $\bigl\{u_{M,\ast}\bigl(\cL_n\bigr)\bigr\}_n$

If $\cU=\Spf(R_\cU)$ is affine connected then the localization
$\cL_n(\Rbar_\cU) $ as defined in \ref{sec:formal_Groth} is
given by a free $\Z_p/p^n\Z$-module with continuous action of
$\cG_{\cU_M}$ which we denote by $V_\cU(\cL_n)$. Write
$V_\cU(\cL)=\ds \lim_{\infty \leftarrow n} V_\cU(\cL_n)$.

\bigskip

{\it The categories $\Mod(\fX_M)_{\bB_{\rm cris}^\nabla}$ and
$\Mod(\fX_M)_{\bB_{\rm cris}}$.} \enspace Denote by
$\Mod(\fX_M)_{\bA_{\rm cris}^\nabla}$
(resp.~$\Mod(\fX_M)_{\bA_{\rm cris}}$) the following category. The
objects are systems $\{\cM_n\}_n\in \Sh(\fX_M)^\N$ with~$\cM_n$ a
sheaf of $\bA_{\rm cris,n,M}^{'\nabla}$--modules (resp.~$\bA_{\rm
cris,M,n}'$--modules). Given objects~$\cM$ and~$\cM'$  the
morphisms are $\Hom_{\bA_{\rm
cris,M}^{\nabla}}\bigl(\cM,\cM'\bigr)$ (resp.~$\Hom_{\bA_{\rm
cris,M}}\bigl(\cM,\cM'\bigr)$) i.~e. the subset
of~$\Hom_{\Sh(\fX_M)^\N}(\cM,\cM')$ which are  by definition
compatible systems of homomorphisms $\{f_n\colon \cM_n\to
\cM_n'\}_{n\in\N}$ commuting with the underlying structure of
$\bA_{\rm cris,M}^{\nabla}$--modules (resp.~$\bA_{\rm
cris,M}$--modules) i.~e. such that~$f_n$ is a homomorphism of
$\bA_{\rm cris,n,M}^{'\nabla}$--modules (resp.~$\bA_{\rm
cris,n,M}'$--modules) for every~$n\in\N$.

Define the sheaf $\uHom_{\bA_{\rm
cris,M}^\nabla}\bigl(\cM,\cM'\bigr)$ (resp.~$\uHom_{\bA_{\rm
cris,M}}\bigl(\cM,\cM'\bigr)$) in~$\Mod(\fX_M)_{\bA_{\rm
cris}^\nabla}$ (respectivey~in $\Mod(\fX_M)_{A_{\rm cris}}$) associated
to the pre-sheaf whose sections at~$(\cU,\cW)\in \fX_M$ consist of
the group $$\left\{\Hom_{\bA_{\rm
cris,n,M}^{'\nabla}(\cU,\cW)}\bigl(\cM_n(\cU,\cW),\cM_n'(\cU,\cW)\bigr)\right\}_n$$
respectively $$\left\{\Hom_{\bA_{\rm
cris,n,M}'(\cU,\cW)}\bigl(\cM_n(\cU,\cW),\cM_n'(\cU,\cW)\bigr)\right\}_n.$$

Define~$\cM\tensor_{\bA_{\rm cris,M}^\nabla}\cM'$ (resp.~$\cM\tensor_{\bA_{\rm cris,M}}\cM'$) to be the sheaf in $\Mod(\fX_M)_{A_{\rm cris}^\nabla}$ (respectively
in $\Mod(\fX_M)_{\bA_{\rm cris}}$) associated to the pre-sheaf valued $\left\{\cM_n(\cU,\cW)\tensor_{\bA_{\rm cris,M,n}^{'\nabla}(\cU,\cW)} \cM_n'(\cU,\cW)
\right\}_n$ on~$(\cU,\cW)\in \fX$ (respectively $\left\{\cM_n(\cU,\cW)\tensor_{\bA_{\rm cris,n,M}'(\cU,\cW)} \cM_n'(\cU,\cW) \right\}_n$). Define~$\bf 1$ to be the
element $\bA_{\rm cris}^\nabla$ (respectively $\bA_{\rm cris}$). With these structures both categories $\Mod(\fX_M)_{\bA_{\rm cris}^\nabla}$
and~$\Mod(\fX_M)_{\bA_{\rm cris}}$ are abelian tensor categories. Given any object $\cN$ we write $\cN(r)$ to be $\cN\tensor_{\bA_{\rm cris,M}^\nabla} \bA_{\rm
cris,M}^\nabla(r)$ (respectively $\cN\tensor_{\bA_{\rm cris,M}} \bA_{\rm cris,M}(r)$.

\smallskip

Define $\Mod(\fX_M)_{\bB_{\rm cris}^\nabla}$
(resp.~$\Mod(\fX_M)_{\bB_{\rm cris}}$) to be the full subcategory
of ${\rm Ind}\bigl(\Mod(\fX_M)_{\bA_{\rm cris}^\nabla} \bigr)$
(respectively ${\rm Ind}\bigl(\Mod(\fX_M)_{\bA_{\rm cris}} \bigr)$)
consisting of objects of the form $\bigl(\cM(-r)\bigr)_{r\in\Z}$
with $\cM$ a fixed object of $\Mod(\fX_M)_{\bA_{\rm cris}^\nabla}$
(resp.~$\Mod(\fX_M)_{\bA_{\rm cris}}$) and the transition
morphisms $\iota_{\cM,r,s}\colon \cM(s) \to \cM(r)$ are induced by
the morphisms $\iota_{r,s}\colon \bA_{\rm cris,M}^\nabla(s)\to
\bA_{\rm cris,M}^\nabla(r)$ (and similarly for $\bA_{\rm cris,M}$)
defined in \S \ref{def:bBcris}. Remark that this object is simply
the tensor product $\cM\tensor_{\bA_{\rm cris}^\nabla} \bB_{\rm
cris}^\nabla $ (resp.~$\cM\tensor_{\bA_{\rm cris}} \bB_{\rm cris}
$) defined in \S\ref{def:bBcris}. Given objects $\cM$ and $\cN$ we
denote by $\Hom_{\bB_{\rm cris}^\nabla}\bigl(\cM,\cN\bigr)$
(resp.~$\Hom_{\bB_{\rm cris}^\nabla}\bigl(\cM,\cN\bigr)$) the
group of homomorphisms in this category. These categories inherit
from $\Mod(\fX_M)_{\bA_{\rm cris}^\nabla}$
and~$\Mod(\fX_M)_{\bA_{\rm cris}}$ the structures of tensor
categories.

\bigskip

Consider objects $\cM=\{\cM_a\}_{a\in\N}$ and $\cN=\{\cN_a\}_{a\in\N}$ in $\Mod(\fX_M)_{\bA_{\rm cris}^\nabla}$ or in $\Mod(\fX_M)_{\bA_{\rm cris}}$. Given integers
$m$ and $n$ and a morphism $f\colon \cM(m)\lra \cN(m+n)$ in $\Mod(\fX_M)_{\bA_{\rm cris}^\nabla}$ (respectively $\Mod(\fX_M)_{\bA_{\rm cris}}$) we define a morphism
of inductive systems $\bigl(f_i\colon  \cM(i) \lra \cN(i+n)\bigr)_{i\in\Z}$ identifying $\cM(i)\cong \cM(m)\tensor_{\bA_{\rm cris,M}^\nabla} \bA_{\rm
cris,M}^\nabla(i-m)$ and $\cN(i+n)\cong \cN(m+n)\tensor_{\bA_{\rm cris,M}^\nabla} \bA_{\rm cris,M}^\nabla(i-m)$ and setting $f_i:= f \otimes {\rm Id}$ (and
similarly if we have objects in $\Mod(\fX_M)_{\bA_{\rm cris}}$).

\begin{lemma}\label{lemma:isoinBcris} The maps above define
group isomorphisms $$\lim_{ s,r\in \Z}\Hom_{\bA_{\rm
cris,M}^{\nabla}}\bigl(\cM(s),\cN(r)\bigr) \lra \Hom_{\bB_{\rm
cris}^\nabla}\bigl(\cM,\cN\bigr).$$Here, the direct limits on the
left hand side is taken via the maps $\Hom_{\bA_{\rm
cris,M}^{\nabla}}\bigl(\cM(s),\cN(r)\bigr) \to \Hom_{\bA_{\rm
cris,M}^{\nabla}}\bigl(\cM(s'),\cN(r')\bigr)$ given by $f\mapsto
\iota_{\cN,r,r'} \circ f \circ \iota_{\cM,s',s}$ for integers
$s'\geq s$ and $r\geq r'$. Similarly we get an isomorphisms
$$\lim_{s,r\in\Z} \Hom_{\bA_{\rm
cris,M}}\bigl(\cM(s),\cN(r)\bigr)\lra \Hom_{\bB_{\rm
cris}}\bigl(\cM,\cN\bigr).$$

Let\/ $f\in \Hom_{\bB_{\rm cris}^\nabla}\bigl(\cM,\cN\bigr) $
(resp.~in $\Hom_{\bB_{\rm cris}}\bigl(\cM,\cN\bigr) $) induced by
a morphism $f_{m,n}\colon \cM(m) \to \cN(n)$ in
$\Mod(\fX_M)_{\bA_{\rm cris}^\nabla}$
(resp.~$\Mod(\fX_M)_{\bA_{\rm cris}}$) for some $m$ and $n\in\Z$.
The following are equivalent:\smallskip

1) $f$ is an isomorphism;\smallskip

2) there are~$r$ and $s\in \N$ and a map $h_{r,s}\colon \cN(n+r)
\lra \cM(m-s)$ such that $f_{m,n}(s) \circ h_{r,s}$ is
$\iota_{\cN,n+r,n-s}\colon \cN(n+r) \to \cN(n-s)$ and
$h_{r,s}\circ f_{m,n}(+r)$ is  $\iota_{\cM,m+r,m-s}\cM(m+r) \to
\cM(m-s)$;\smallskip

3) there exists $N\in\N$ such that for every small affine $\cU\in
X^{\rm et}$ and every $a\in\N$ the map
$\cM_a(m)(\Rbar_\cU) \to \cN_a(n)(\Rbar_\cU)$ induced by
$f_{m,n}$ has kernel and cokernel annihilated by $t^N$.\bigskip

\noindent Furthermore multiplication by $t$ is an isomorphism in
$\Mod(\fX_M)_{\bB_{\rm cris}^\nabla}$
(resp.~$\Mod(\fX_M)_{\bB_{\rm cris}}$).

\end{lemma}
\begin{proof} The first claim follows from the definitions and is
left to the reader. The last claim follows  from
\ref{lemma:propbBcris} remarking that any object is of the form
$\cF\tensor_{\bA_{\rm cris}^\nabla} \bB_{\rm cris}^\nabla $
(resp.~$\cF\tensor_{\bA_{\rm cris}} \bB_{\rm cris} $).

The equivalence of (1) and (2) is clear. Note that the morphism
$\cM(m)(\Rbar_\cU) \to \cN(n)(\Rbar_\cU)$ is a morphism of $A_{\rm
cris}$--modules so that (3) makes sense.

(2) $\Longrightarrow$ (3) Note that
$\cM(h)(\Rbar_\cU)=\cM(\Rbar_\cU)$ and
$\cN(h)(\Rbar_\cU)=\cN(\Rbar_\cU)$ as $A_{\rm
cris}(\Rbar_\cU)$--modules for every $h\in\Z$ (only the Galois
action of $\cG_{\cU,M}$ is different). Via these identifications
the maps $\cN(n+r)(\Rbar_\cU) \to \cN(n-s)(\Rbar_\cU)$ and
$\cM(m+r)(\Rbar_\cU) \to \cM(m-s)(\Rbar_\cU)$ are multiplication
by $t^{s+r}$.

(3) $\Longrightarrow$ (2) Write $\cM=\{\cM_a\}_{a\in\N}$ and $\cN:=\{\cN_b\}_{b\in\N}$. Consider the map $\iota_{\cN,n+2N,n+N}\colon\cN(n+2N) \to \cN(n+N)$. By
assumption for every small affine $\cU\in \fX_M$ and every $a\in\N$ the image of the induced map on localizations $\cN_a(n+2N)(\Rbar_\cU) \to \cN_a(n+N)(\Rbar_\cU)$
is zero in the cokernel of $f_{a,\cU}(N)\colon \cM_a(m+N)(\Rbar_\cU) \to \cN_a(n+N)(\Rbar_\cU)$ induced by $f_{m,n}$. Hence it factors via the image ${\rm
Im}\bigl(f_{a,\cU}(N)\bigr)$ of $f_{a,\cU}(N)$. The map ${\rm Im}\bigl(f_{a,\cU}(N)\bigr) \to {\rm Im}\bigl(f_{a,\cU}\bigr) $ is multiplication by $t^N$ and hence
factors uniquely via $\cM_a(m)(\Rbar_\cU)$ by assumption. Thus the map $\cN_a(n+2N)(\Rbar_\cU) \to \cN_a(n)(\Rbar_\cU)$ obtained from $\iota_{\cN,n+2N,n}$ by
localization factors uniquely via $f_{a,\cU} \colon \cM_a(m)(\Rbar_\cU)\to \cN_a(n)(\Rbar_\cU)$. By uniqueness this factorization is $\cG_{\cU,M}$--equivariant and
compatible for varying~$\cU$'s and $a$'s. In particular it provides a map $\{h_{a,2N,0}\colon \cN_a(n+2N) \lra \cM_a(m)\}_{a\in\N}$ with the required properties.
\end{proof}

Let  $\cU\in X^{\rm et}$ be a small affine with Galois group
$\cG_{\cU,M}$. As explained in \S\ref{def:bBcris} the localization
functor $\Sh(\fX_M)^\N \to \Rep_{\cG_{\cU,M}}$ $\cF\mapsto
\cF(\Rbar_\cU)$ extend to a localization functor ${\rm
Ind}\left(\Sh(\fX_M)^\N\right) \lra \Rep_{\cG_{\cU,M}} $ which we
denote by $\cF\mapsto \cF(\Rbar_\cU)$. Restricting it to the
categories $\Mod(\fX_M)_{\bB_{\rm cris}^\nabla}$ (respectively
$\Mod(\fX_M)_{\bB_{\rm cris}}$ and using \ref{lemma:propbBcris}
they define functors

$$\Mod(\fX_M)_{\bB_{\rm cris}^\nabla} \lra \Mod-{B_{\rm cris}^\nabla(\Rbar_\cU)
\left[\cG_{\cU,M}\right]},\qquad \Mod(\fX_M)_{\bB_{\rm cris}} \lra \Mod-{B_{\rm cris}(\Rbar_\cU)\left[\cG_{\cU,M}\right]} $$to the categories of $B_{\rm
cris}^\nabla(\Rbar_\cU)$-modules (resp.~$B_{\rm cris}(\Rbar_\cU)$-modules) endowed with continuous action of $\cG_{\cU,M}$.

For later purposes we prove the following property of
localizations of tensor products. In the next lemma we suppose
that $M$ is a finite extension of $K$. Let $\cM$ be a coherent
sheaf of $\cO_{X_{\OMun}}$--modules on~$X^{\rm et}$. Write
$\cM\tensor_{(\cO_{X\otimes_{\cO_K} \OMun})} \Fil^r\bB_{\rm cris,M}$
for $v_{X,M}^\ast(\cM)\tensor_{(\cO_{\fX_M}^{\rm un}\otimes_{\cO_K}
\OMun)} \Fil^r\bB_{\rm cris,M}$.

\begin{lemma}\label{lemma:loctensorAcris}
Let $\cL$ be a $p$-adic sheaf on $X_M^{\rm et}$. Fix $r\in\Z\cup
\{-\infty\}$. Let $\cU\in X^{\rm et}$ be a small affine. If
$\cM[p^{-1}](\cU)$ is a projective $R_\cU\otimes_{\cO_K}
\Mun$-module then
$$\left(\cM\tensor_{(\cO_{X\otimes_{\cO_K} \OMun})} \Fil^r\bB_{\rm
cris,M}\right)(\Rbar_\cU)=\cM(\cU)\tensor_{R_\cU} \Fil^rB_{\rm
cris}(\Rbar_\cU) $$in $\Mod-{B_{\rm
cris}(\Rbar_\cU)\left[\cG_{\cU,M}\right]}$.

Similarly let $V_\cU(\cL)$ be the $\cG_{\cU,M}$-representation
associated to $\cL$. Then $$\left(\cL\tensor_{\Z_p}
\Fil^r\bB_{\rm cris,M}\right)(\Rbar_\cU)=V_\cU(\cL) \tensor_{\Z_p}
\Fil^r B_{\rm cris}(\Rbar_\cU) $$in $\Mod-{B_{\rm
cris}(\Rbar_\cU)\left[\cG_{\cU,M}\right]}$.

\end{lemma}

\begin{proof} We prove the first statement. We assume that $X=\cU$. Since $\cM[p^{-1}]$ is
projective and coherent it is a direct summand in a free
$\cO_\cU\otimes_{\cO_K} \Mun$-module. We may then assume that
$\cM[p^{-1}]$ is free.  The claim follows
from~\ref{lemma:propbBcris}.  The second statement follows also
from loc.~cit.
\end{proof}

Let~$v_{M,\ast}^\cont\colon \Sh(\fX_M)^\N \lra \Sh\bigl(X^{\rm
et}\bigr)$ be the functor $\ds \{\cF_n\}_n \mapsto
\lim_{\infty\leftarrow n} v_{X,M,\ast}(\cF_n)$. As explained in
\S \ref{def:bBcris} it induces a functor on the category ${\rm
Ind}\left(\Sh(\fX_M)^\N\right)$ and hence a functor
$$v_{M,\ast}\colon \Mod(\fX_M)_{\bB_{\rm cris}} \lra \Sh\bigl(X^{\rm
et}\bigr).$$Given a $p$--adic sheaf $\cL$ on~$X_M^{\rm
et}$ define
$$\bDcrisM(\cL):=
v_{M,\ast}\Bigl(\cL\tensor_{\Z_p} \bB_{\rm cris,M} \Bigr).$$Recall
that by abuse of notation we denoted $\cL$ the continuous sheaf
on $\fX_M$ given by $w_{X,M,\ast}(\cL)$. Then $\bDcrisM(\cL)$ is
a sheaf of $\cO_{X_{\Mun}}$-modules in~$\Sh\bigl(X^{\rm
et}\bigr)$. Put
$$\bDcrisgeo(\cL):=\bD_{\rm cris,\Kbar}(\cL), \qquad
\bDcrisar(\cL):=\bD_{\rm cris,M}(\cL)$$whenever $M$ is a fixed
finite extension of~$K$.

\begin{lemma}\label{lemma:DcrisarDcrisgeoGK} The sheaf~$\bDcrisgeo(\cL)$
is endowed with an action of\/ $G_M$ and
$\bDcrisM(\cL)=\bigl(\bDcrisgeo(\cL)\bigr)^{G_M}$.
\end{lemma}
\begin{proof}
One has $v_{\Kbar,\ast}^\cont\Bigl(\cL\tensor_{\Z_p} \bA_{\rm
cris,\Kbar}(r) \Bigr)=
v_{M,\ast}^\cont\left(\beta_{M,\Kbar,\ast}^\N
\Bigl(\cL\tensor_{\Z_p} \bA_{\rm cris,\Kbar}(r) \Bigr)\right)$
since $v_{\Kbar,\ast}=v_{M,\ast}\circ \beta_{M,\Kbar,\ast}$. Due
to corollary \ref{cor:extensionofFrobenius} we have
$\beta_{M,\Kbar}^\ast(\bA_{\rm cris,M})\cong\bA_{\rm cris,\Kbar}$
so that $v_{M,\ast}^\cont\left(\beta_{M,\Kbar,\ast}^\N
\Bigl(\cL\tensor_{\Z_p} \bA_{\rm cris,\Kbar}(r) \Bigr)\right)$
coincides with $v_{M,\ast}^\cont\left(\beta_{M,\Kbar,\ast}^\N\circ
\beta_{M,\Kbar}^{\N,\ast} \Bigl(\cL\tensor_{\Z_p} \bA_{\rm
cris,M}(r) \Bigr)\right)$. It then follows from
lemma \ref{lemma:betaastG}(ii) that the
module~$v_{\Kbar,\ast}^\cont\Bigl(\cL\tensor_{\Z_p} \bA_{\rm
cris,\Kbar}(r) \Bigr)$ is endowed with an action of $G_M$ and that
$$v_{M,\ast}^\cont\Bigl(\cL\tensor_{\Z_p} \bA_{\rm cris,M}(r)
\Bigr)=v_{\Kbar,\ast}^\cont\Bigl(\cL\tensor_{\Z_p} \bA_{\rm
cris,\Kbar}(r) \Bigr)^{G_M}.$$ Passing to direct limits on $r\in\Z$
the lemma follows.
\end{proof}

Let $\cU$ be a small affine of $X^{\rm et}$. Then
$\cL_n(\Rbar_\cU)$ is a free $\Z/p^n\Z$--module by assumption and
thus  the natural map~$\cL_n(\Rbar_\cU) \tensor_{\Z_p} \bA_{\rm
cris,n,M}' (\Rbar_\cU) \lra \Bigl(\cL_n \tensor_{\Z_p} \bA_{\rm
cris,n,M}'\Bigr)(\Rbar_\cU)$ is an isomorphism. Then $
v_{M,\ast}\Bigl(\cL_n\tensor_{\Z_p} \bA_{\rm cris,n,M}'
\Bigr)(\cU)= \Bigl(\cL_n \tensor_{\Z_p} \bA_{\rm
cris,n}\Bigr)(\Rbar_\cU)^{\cG_{\cU,M}}$. Then the map $\ds
\cL(\Rbar_\cU)\tensor_{\Z_p} A_{\rm cris}(\Rbar_\cU)t^r \lra
\lim_{\infty\leftarrow n} \cL_n(\Rbar_\cU)\tensor_{\Z_p} \bA_{\rm
cris,n}'(\Rbar_\cU) t^r$ is an isomorphism for every~$r\in\Z$
since~$A_{\rm cris}(\Rbar_\cU)$ is $p$--adically complete and
separated and thanks to proposition \ref{prop:crislocalization}.
Following~\cite{brinon} define
$$ D_{\rm cris,M}
\bigl(V_\cU(\cL)\bigr):=\Bigl(V_\cU(\cL)\tensor_{\Z_p} B_{\rm
cris}(\Rbar_\cU)\Bigr)^{\cG_{\cU,M}}.$$It then follows that
$$\bDcrisM(\cL)(\cU) \stackrel{\sim}{\lra} D_{\rm cris,M}
\bigl(V_\cU(\cL)\bigr)$$as $R_\cU\otimes_{\cO_K} \Mun$-modules
and since $\cG_{\cU,\Kbar}$ acts trivially on~$t$  we get
$$\Bigl(V_\cU(\cL)\tensor_{\Z_p} B_{\rm
cris}(\Rbar_\cU)\Bigr)^{\cG_{\cU,\Kbar}}\stackrel{\sim}{\lra}
\bDcrisgeo(\cL)(\cU)$$as $R_\cU\widehat{\otimes}_{\cO_K} B_{\rm
cris}$-modules. Let $M$ be a finite extension of~$K$. It follows
that $\bDcrisM$ and $\bDcrisgeo$ define functors
$$\bDcrisar\colon \Sh(X_M^{\rm et})_{\Q_p} \lra \Mod_{\cO_X\otimes_{\cO_K}\Mun}$$and
$$\bDcrisgeo\colon  \Sh(X_M^{\rm et})_{\Q_p} \lra  \Mod(\cO_X\otimes_{\cO_K}B_{\rm cris});$$here
$\cO_{X}\tensor_{\cO_K} B_{\rm cris}$ stands for
$\cO_{X}\widehat{\tensor}_{\cO_K} A_{\rm cris}\bigl[t^{-1}\bigr]$
where $\cO_{X}\widehat{\tensor}_{\cO_K} A_{\rm cris}$ is the sheaf
on $X^{\rm et}$ defined by $\ds \lim_{\infty\leftarrow n}\left(
\bigl(\cO_{X}/p^n \cO_{X}\bigr) \tensor_{\cO_K} A_{\rm
cris}\right)$. Furthermore we have.

\begin{lemma}\label{lemma:Dcriscoherent} Let $M$ be a finite extension of $K$.
The $R_\cU\tensor_{\cO_K} \Mun$--module
$\bDcrisar\bigl(\cL\bigr)(\cU)$ is projective, of finite type and
of rank less or equal to the rank of~$\cL$.
\end{lemma}
\begin{proof} It follows from~\cite[Prop.~8.3.1]{brinon}
and the identification $D_{\rm cris} \bigl(V_\cU(\cL)\bigr)=
\bDcrisar(\cL)(\cU)$.
\end{proof}

{\it Crystalline \'etale sheaves.}\enspace Let $K\subseteq M$ ($\subset \Kbar$) be a finite extension. Following \cite[Def.~1.1]{ogus} we denote by ${\rm
Coh}\bigl(\cO_X\tensor_{\cO_K} \Mun\bigr)$ to be the full subcategory of sheaves of $\cO_X\tensor_{\cO_K} \Mun$-modules  isomorphic to $F\tensor_{\cO_K} K$ for some
coherent sheaf $F$ of $\cO_X\tensor_{\cO_K} \OMun$-modules on $X$.  A $\Q_p$-adic sheaf~$\cL=\{\cL_n\}_n$ on~$X_M^{\rm et}$ is called {\it crystalline} if

\begin{enumerate}

\item[i.] $\bDcrisar\bigl(\cL\bigr)$ is in  ${\rm Coh}\bigl(\cO_X\tensor_{\cO_K} \Mun\bigr)$;

\item[ii.] the natural map $\alpha_{\rm cris,\cL}\colon
\bDcrisar\bigl(\cL\bigr)\tensor_{(\cO_X\tensor_{\cO_K} \OMun)}
\bB_{\rm cris,M}\lra \cL\tensor_{\Z_p} \bB_{\rm cris,M}$ is an
isomorphism in~$\Mod(\fX_M)_{\bB_{\rm cris}}$.

\end{enumerate}

\noindent Denote by $\Sh(X_M^{\rm et})_{\Q_p}^{\rm cris}$ the full
subcategory of $\Sh(X_M^{\rm et})_{\Q_p}$ consisting of
crystalline sheaves.

\smallskip {\it Convention:} For any
coherent $\cO_X\otimes_{\cO_K}\OMun$-module $D$ we write
$$D\otimes_{(\cO_X\otimes_{\cO_K} \OMun)}\bB_{\rm cris,M}:=
v_{X,M}^\ast(D)\tensor_{(\cO_{\fX_M}^{\rm un}\otimes_{\cO_K}\OMun)}
\bB_{\rm cris,M}.$$

\begin{remark}\label{remark:DcrisotimesBcris} To make sense
of~(ii)
note that  by
adjunction we have a morphism
$$f_m(\cL)\colon v_{X,M}^\ast\left(v_{M,\ast}\Bigl(\cL\tensor_{\Z_p}
\bA_{\rm cris,M}(m) \Bigr)\right)\tensor_{(\cO_{\fX_M}^{\rm
un}\otimes_{\cO_K}\OMun)} \bA_{\rm cris,M}\lra \cL\tensor_{\Z_p}
\bA_{\rm cris,M}(m).$$Recall that $\cO_{\fX_M}^{\rm un}\subset
\cO_{\fX_M}$ is identified with $v_{X,M}^\ast(\cO_X)$ by
lemma \ref{lemma:vastOX}. Using proposition \ref{prop:Dacrisislocallyfree} we know
that for $m\leq N$ large the $\cO_X\otimes_{\cO_K} \OMun$-module
$D(m):= v_{M,\ast}\Bigl(\cL\tensor_{\Z_p} \bA_{\rm
cris,M}(m)\Bigr)$ is coherent and its image in ${\rm
Coh}\bigl(\cO_X\tensor_{\cO_K} \Mun\bigr)$ is $
\bDcrisar\bigl(\cL\bigr)$.  Then $\alpha_{\rm cris,\cL}$ is the
map in~$\Mod(\fX_M)_{B_{\rm cris}}$ induced by the $f_m(\cL)$ for
$m \leq N$. Due to proposition \ref{prop:Dacrisislocallyfree} and since $X$ is
noetherian $$v_{X,M}^\ast\left(v_{M,\ast}\Bigl(\cL\tensor_{\Z_p}
\bA_{\rm cris,M}(m) \Bigr)\right)\tensor_{(\cO_{\fX_M}^{\rm
un}\otimes_{\cO_K}\OMun)} \bA_{\rm cris,M}$$  is also isomorphic as
inductive system to $D\otimes_{\cO_X\otimes_{\cO_K}
\OMun}\bB_{\rm cris,M}$ for any coherent
$\cO_X\otimes_{\cO_K}\OMun$-module $D$ such that $D\tensor_{\OMun}
\Mun\cong \bDcrisar\bigl(\cL\bigr)$ as
$\cO_X\otimes_{\cO_K}\Mun$-modules.
\end{remark}

\begin{proposition}\label{prop:Dacrisislocallyfree}
Let $\cL$ be a $p$--adic sheaf. If $\bDcrisar\bigl(\cL\bigr)$ is in
${\rm Coh}\bigl(\cO_X\tensor_{\cO_K} \Mun\bigr)$ there exists a
negative integer $N$ such that for every~$m\leq N$ the natural
morphism
$$\mu_m\colon v_{M,\ast}\Bigl(\cL\tensor_{\Z_p}
\bA_{\rm cris,M}(m) \Bigr)\lra \bDcrisar\bigl(\cL\bigr)$$is
injective and is an isomorphism after inverting $p$ as
$\cO_X\tensor_{\cO_K} \Mun$--modules. For any such $m\leq N$ and every
small affine $\cU\in X^{\rm et}$ the
$R_\cU\tensor_{\cO_K}\OMun$-module
$v_{M,\ast}^\cont\Bigl(\cL\tensor_{\Z_p} \bA_{\rm cris,M}(m)
\Bigr)(\cU)$ is finitely generated and $p$-torsion free.
\end{proposition}
\begin{proof} Since~$X$ is noetherian, $K\subset M$ is a
finite extension and $\bDcrisar\bigl(\cL\bigr)$ is in
${\rm Coh}\bigl(\cO_X\tensor_{\cO_K} \Mun\bigr)$ there
exists $N\in \N$ such that $\mu_m$ is surjective after inverting
$p$ for every $m\geq N$. Since the natural maps $\cL\tensor_{\Z_p}
\bA_{\rm cris,M}(r) \to \cL\tensor_{\Z_p} \bA_{\rm cris,M}(s)$ are
injective and $v_{M,\ast}^\cont$ is left exact, $\mu_m$ is also
injective. This proves the first statement.

Since $\bDcrisar\bigl(\cL\bigr)(\cU) $ is a projective and finitely generated $R_\cU\tensor_{\cO_K} \Mun$-module it is a direct summand in a finite and free
$R_\cU\tensor_{\cO_K} \Mun$-module $T_{\Mun}$. Let $T$ be a free $R_\cU\tensor_{\cO_K} \OMun$-submodule of $T_{\Mun}$ such that $T[p^{-1}]=T_{\Mun}$. Let $n\in\N$
be large enough so that the image of $V_\cU(\cL)$ in $\bDcrisar\bigl(\cL\bigr)(\cU)\tensor_{\Mun} B_{\rm cris}(\Rbar_\cU)\subset T_{\Mun} \tensor_{\Mun} B_{\rm
cris}(\Rbar_\cU)$ is contained in  $T \cdot \frac{1}{p^n} \tensor_{R_\cU\tensor_{\cO_K}\OMun} A_{\rm cris}(\Rbar_\cU)$. Then $v_{M,\ast}\big(\cL\tensor_{\Z_p}
\bA_{\rm cris,M}(m)\big)(\cU)$ is $\left(V_\cU(\cL)\tensor_{\Z_p} A_{\rm cris}(\Rbar_\cU) t^m\right)^{\cG_{\cU,M}}$ and this is contained in the submodule $\left(T
\cdot \frac{1}{p^n} \tensor_{R_\cU\tensor_{\cO_K}} A_{\rm cris}(\Rbar_\cU) t^m\right)^{\cG_{\cU,M}}$.

Put $R':=\left(A_{\rm cris}(\Rbar_\cU) t^m\right)^{\cG_{\cU,M}}$. It is $p$-adically complete and separated, it contains $R_\cU\tensor_{\cO_K} \OMun$ and it is
contained in $ \left(B_{\rm cris}(\Rbar_\cU)t^m\right)^{\cG_{\cU,M}}=R_\cU\tensor_{\cO_K} \OMun[p^{-1}]$ by \cite[Prop.~6.2.9]{brinon}. We claim that this implies
that there exists $n\in \N$ such that $p^n R'$ is contained in $R_\cU\tensor_{\cO_K} \OMun$ . If $R_\cU$ were a complete dvr, the above conditions would imply the
claim. In the general case replacing $R$ with the localization at a prime ideal $\mathcal{P}$ over $p$ and $\Rbar_\cU$ with $\Rbar_{\cU,\mathcal{\cP}}$ we deduce
that there exists $n_{\mathcal{P}}\in\N$ such that $p^{n_{\mathcal{P}}} R'\subset \widehat{R}_{\cU,\mathcal{P}} \tensor_{\cO_K} \OMun$. Taking $n$ to be the maximum
of all  the $n_{\mathcal{P}}$'s we deduce that $R'\subset   R_\cU\tensor_{\cO_K} \OMun[p^{-1}]$ and also $p^n R'\subset \widehat{R}_{\cU,\mathcal{P}}
\tensor_{\cO_K}\OMun$. Since $R_\cU\tensor_{\cO_K} \OMun$ is normal we deduce the claim. Since $T$ is free $R_\cU\tensor_{\cO_K} \OMun$-module, we conclude that
$v_{M,\ast}\big(\cL\tensor_{\Z_p} \bA_{\rm cris,M}(m)\big)(\cU)$ is contained in $T\tensor_{\cO_K}\OMun \cdot \frac{1}{p^{n}}$. In particular it is a finitely
generated $R_\cU\tensor_{\cO_K} \OMun$-module as desired.
\end{proof}

Let~$\cL$ be a $p$--adic sheaf on $X_M^{\rm et}$.
Following~\cite{brinon} for every small affine~$\cU$ of $ X^{\rm
et}$ we say that~$V_\cU(\cL)$ is {\it crystalline} if the
map $D_{\rm cris} \bigl(V_\cU(\cL)\bigr)\tensor_{R_\cU} B_{\rm
cris}(\Rbar_\cU) \lra V_\cU(\cL)\tensor_{\Z_p} B_{\rm
cris}(\Rbar_\cU)$ is an isomorphism. Then we have.

\begin{proposition}\label{prop:equivcris}
The following are equivalent:\smallskip

1) $\cL$ is crystalline;\smallskip

2) for every small affine object~$\cU$ of~$X^{\rm
et}$ the representation~$V_{\cU}(\cL)$ is
crystalline;\smallskip

3) there is a covering $\{\cU_i\}_i$ of~$X^{\rm et}$ by small affine
objects  such that~$V_{\cU_i}(\cL)$ is crystalline for
every~$i$;\smallskip

\end{proposition}
\begin{proof} (1) $\Longrightarrow $ (2)\enspace  Due to
\ref{lemma:loctensorAcris} we have $\bigl(\cL\tensor_{\Z_p}
\bB_{\rm cris,M}\bigr)(\Rbar_\cU)=V_\cU(\cL) \tensor_{\Z_p} B_{\rm
cris}(\Rbar_\cU)$. Note that $\bDcrisar\bigl(\cL\bigr)(\cU)$ is a
projective $R_\cU\tensor_{\cO_K} \Mun$-module by
\ref{lemma:Dcriscoherent} i.e. it is a direct summand in a free
module. As a consequence of remark \ref{remark:DcrisotimesBcris} and
lemma \ref{lemma:loctensorAcris} it follows that the localization of
$\bDcrisar\bigl(\cL\bigr)\tensor_{(\cO_X\otimes_{\cO_K}\OMun)}
\bB_{\rm cris,M}$ is isomorphic to
$\bDcrisar\bigl(\cL\bigr)(\cU)\tensor_{(R_\cU\otimes_{\cO_K}\OMun)}
B_{\rm cris}(\Rbar_\cU)$. The implication follows applying the
localization functor to the isomorphism
$\bDcrisar\bigl(\cL\bigr)\tensor_{(\cO_X\tensor_{\cO_K} \OMun)}
\bB_{\rm cris,M}\lra \cL\tensor_{\Z_p} \bB_{\rm cris,M}$.
\smallskip

(2) $\Longrightarrow $ (3) is clear.\smallskip

(3) $\Longrightarrow $ (1)\enspace  For a small affine open~$\cU$
of $X^{\rm et}$ and
for a negative integer $r$ let~$g_{\cU,r}$ be the natural map
$$g_{\cU,r}\colon \Bigl(V_\cU(\cL)\tensor_{\Z_p} A_{\rm
cris}(\Rbar_\cU)t^r\Bigr)^{\cG_{\cU,M}}\tensor_{R_\cU} A_{\rm cris}(\Rbar_\cU) \lra V_\cU(\cL)\tensor_{\Z_p} A_{\rm cris}(\Rbar_\cU) t^r.$$It is injective
by~\cite[Prop.~8.2.6]{brinon}. In particular $V_\cU(\cL)$ is crystalline if and only if~$V_\cU(\cL)$ is in the image of~$g_{\cU,r}$ for some $r<0$. We deduce that
if~$V_\cU(\cL)$ is crystalline then~$V_{\cV}(\cL)$ is crystalline for every open affine~$\cV\to \cU$.

Fix a small affine $\cU$ which factors through one of the $\cU_i$'s. In particular $V_\cU(\cL)$ is crystalline. Assume that~$V_\cU(\cL)$ is in the image of
$g_{\cU,r}$. Let~$\cV\to \cU$ be an \'etale morphism with~$\cV$ affine. Let~$D'\subset D_{\rm cris} \bigl(V_\cV(\cL)\bigr)$ be the image of~$D_{\rm cris}
\bigl(V_\cU(\cL)\bigr)\to D_{\rm cris} \bigl(V_\cV(\cL)\bigr)$. Since~$V_\cU(\cL)=V_\cV(\cL)$, then~$V_\cV(\cL)$ is in the image of~$g_{\cV,r}$ and
also~$V_\cV(\cL)$ is crystalline. The extension $R_\cV\tensor_{\cO_K} \Mun \to B_{\rm cris}(\Rbar_\cV)$ is faithfully flat by~\cite[Thm.~6.3.8]{brinon} so that the
maps
$$D'\tensor_{R_\cU} R_{\cV}\tensor_{(R_\cV\tensor_{\cO_K}
\Mun)} B_{\rm cris}(\Rbar_\cV) \to D_{\rm cris}
\bigl(V_\cV(\cL)\bigr)\tensor_{R_\cV} B_{\rm cris}(\Rbar_\cV) \to
V_\cV(\cL)\tensor_{\Z_p} B_{\rm cris}(\Rbar_\cV)$$are all
injective and  the composite is surjective. Thus
$D'\tensor_{R_\cU} R_{\cV}= D_{\rm cris} \bigl(V_\cV(\cL)\bigr)$.
This proves that~$\cV \mapsto D_{\rm cris} \bigl(V_\cV(\cL)\bigr)$
is a coherent $\cO_{X_{\Mun}}$--module. Since $D_{\rm cris}
\bigl(V_\cV(\cL)\bigr) \cong \bDcrisM\bigl(\cL\bigr)(\cV)$ it follows that
$\bDcrisM\bigl(\cL\bigr)\vert_\cU$ is a coherent
$\cO_{\cU_{\Mun}}$--module as well. We deduce  from \cite[Prop.~1.2]{ogus}
that $\bDcrisM\bigl(\cL\bigr)$ lies in ${\rm Coh}\bigl(\cO_X\tensor_{\cO_K} \Mun\bigr)$ i.~e.,
condition (i) in the definition of a crystalline \'etale sheaf holds.

We can be more explicit. Take $N<0$ as in \ref{prop:Dacrisislocallyfree}. Put $D:=\Bigl(V_\cU(\cL)\tensor_{\Z_p} A_{\rm cris}(\Rbar_\cU)t^N\Bigr)^{\cG_{\cU,M}}$ and
$V:=V_\cU(\cL)$. From the proof of \ref{prop:Dacrisislocallyfree} it follows that $D$ is a finitely generated $R_\cU\otimes_{\cO_K}\OMun$-module and by construction
$D'\otimes_{\cO_K} K= \bDcrisM\bigl(\cL\bigr)(\cU)$. Since $X$ is a noetherian topological space, this implies that $\bDcrisM\bigl(\cL\bigr)$ lies in ${\rm
Coh}\bigl(\cO_X\tensor_{\cO_K} \Mun\bigr)$. Consider the commutative diagram
$$\begin{array}{ccccccccc}
 & &  D \tensor A_{\rm cris}(\Rbar_\cU) & \stackrel{p^n}{\lra}  &
D\tensor A_{\rm cris}(\Rbar_\cU) & \lra & \left(D/p^n D\right)
\tensor A_{\rm cris}(\Rbar_\cU) & \lra & 0 \cr & &
g_{\cU,N}\big\downarrow & & g_{\cU,N}\big\downarrow & &
g_{\cU,N,n}\big\downarrow \cr 0 & \lra & V\tensor_{\Z_p} A_{\rm
cris}(\Rbar_\cU)t^N & \stackrel{p^n}{\lra}  & V\tensor_{\Z_p}
A_{\rm cris}(\Rbar_\cU)t^N & \lra & \bigl(V/p^n V\bigr)
\tensor_{\Z_p} A_{\rm cris}(\Rbar_\cU)t^N & \lra & 0,\cr
\end{array}$$where in the first row $\tensor$ stands for $\tensor_{(R_\cU\tensor_{\cO_K}
\OMun)}$. Since $A_{\rm cris}(\Rbar_\cU)$ is $p$--torsion free
by~\cite[Prop.~6.1.10]{brinon} and $V$ is a free $\Z_p$--module,
the bottom row is exact. Recall that~$g_{\cU,N}$ is injective.
Since (3) holds there exists $N$ such that $V$ is in the image
of~$g_{\cU,N}$. Then the cokernel of $g_{\cU,N}$ is annihilated
by $t^{-N}$. Since for every open affine $\cV\in \fU_K$ we
have~$V_\cU(\cL)\cong V_\cV(\cL)$ we deduce that also the
cokernel of~$g_{\cV,N}$ is annihilated by~$t^{-N}$. Thus the kernel
and cokernel of~$g_{\cV,N,n}$ are annihilated by~$t^{-N}$.

Write $f_N$ for the system of morphisms $$f_{N,n}\colon
v_{M,\ast}\Bigl(\cL\tensor_{\Z_p} \bA_{\rm cris,n,M}'(m)
\Bigr)\tensor_{(\cO_X\otimes_{\cO_K}\OMun)} \bA_{\rm cris,M}\lra
\cL\tensor_{\Z_p} \bA_{\rm cris,n,M}(m)$$given by adjunction. For
every~$x\in \cU$ the stalk~$\bA_{\rm cris,n,M,x}'$ at~$x$
contains~$\ds \lim_{x\in \cV} A_{\rm cris}(\Rbar_\cV)/p^n A_{\rm
cris}(\Rbar_\cV)$ where the limit is taken over all affine
opens~$\cV$ of~$x$. The cokernel of $\ds \lim_{x\in \cV} A_{\rm
cris}(\Rbar_\cV)/p^n A_{\rm cris}(\Rbar_\cU)\subset \bA_{\rm
cris,n,M,x}'$ is annihilated by any element of~$\II$
by~\ref{lemma:localizeAcris} and hence also by~$t$.
Since~$\cL_x=V_\cU(\cL)$  and $f_{N,n,x}$ is~$\ds \lim_{x\in \cV}
g_{\cV,N,n}$ on $\ds \lim_{x\in \cV} D\otimes A_{\rm
cris}(\Rbar_\cV)/p^n A_{\rm cris}(\Rbar_\cU)$, we conclude that
kernel and cokernel of~$f_{N,n,x}$ is annihilated by~$t^{-N+1}$.
Since $X$ is a noetherian space and taking a smaller~$N$ if
necessary, we may assume that kernel and cokernel of~$f_{N,n,x}$
is annihilated by~$t^{-N+1}$ for every $x\in X$. Thus the same
applies to $f_{N,n}$ and (1) follows from
\ref{lemma:isoinBcris}(3).

\end{proof}

\subsection{The functors $\bDcrisar$ and $\bVcrisarQ$ on crystalline
sheaves.}\label{sec:propcrissheaves}

Assume as before that  $\cO_K=\WW(k)$ and let $K\subseteq M$ be a
field extension. The goal of this section is to prove in
\ref{thm:crisistannakian} that $\bDcrisar$ defines an exact, fully
faithful functor, commuting with tensor products, duals and Tate
twists, from the category of $\Q_p$--adic crystalline shaves on
$X_M^{\rm et}$ to the category of admissible filtered convergent
$F$--isocrystals on the special fiber of~$X$ relatively to~$\Mun$.
We also construct an inverse $\bVcrisarQ$ on the essential image.

\smallskip

Given a $p$--adic sheaf $\cL$  on $X_M^{\rm et}$ and $r\in\Z$ we
get a well defined subsheaf
$$\Fil^r\bDcrisM(\cL):= v_{M,\ast}\Bigl(\cL\tensor_{\Z_p} \Fil^r\bB_{\rm
cris,M} \Bigr)\subset \bDcrisM(\cL).$$Put $
\Fil^r\bDcrisgeo(\cL):=\Fil^r\bD_{\rm cris,\Kbar}(\cL) $ and if
$K\subset M$ is a fixed finite extension,  $
\Fil^r\bDcrisar(\cL):= \Fil^r\bDcrisM(\cL)$.  It follows from
\ref{lemma:DcrisarDcrisgeoGK} that
$\Fil^r\bDcrisM(\cL)=\bigl(\Fil^r\bDcrisgeo(\cL)\bigr)^{G_M}$.

Since the connections~$\nabla(r)\colon \bA_{\rm cris,M}(r) \lra
\bA_{\rm cris,M}(r) \tensor_{\cO_X} \Omega^1_{X/\cO_K}$ are
compatible for varying $r$ they induce a connection
$$\nabla_\cL\colon \bDcrisM(\cL) \lra \bDcrisM(\cL) \tensor_{\cO_X}
\Omega^1_{X/\cO_K}.$$Given a small affine $\cU$ of $X^{\rm et}$,
write~$V_\cU(\cL)$ for $p$--adic representation of~$\cG_{\cU_M}$
defined by~$\cL(\Rbar_\cU)$ and put
$$\Fil^r D_{\rm cris}
\bigl(V_\cU(\cL)\bigr):=\Bigl(V_\cU(\cL)\tensor_{\Z_p} \Fil^r
B_{\rm cris}(\Rbar_\cU)\Bigr)^{\cG_{\cU,M}}.$$Due
to~\ref{prop:filtAcrisnabla} we deduce that via the identification
$D_{\rm cris} \bigl(V_\cU(\cL)\bigr)\cong
\bDcrisM(\cL)(\Rbar_\cU)$ we have $\Fil^r D_{\rm cris}
\bigl(V_\cU(\cL)\bigr)\cong \Fil^r \bDcrisM(\cL)(\Rbar_\cU)$ for
every~$r\in \Z$. For every $r\in\Z$ define the filtrations
$$\Fil^r\left(\bDcrisar(\cL)\otimes_{(\cO_X\otimes_{\cO_K} \cO_{\Mun})}
\bB_{\rm cris,M}\right):=\sum_{a+b=r} \Fil^a \bDcrisar(\cL)\otimes_{(\cO_X\otimes_{\cO_K} \cO_{\Mun})} \Fil^b \bB_{\rm cris,M}$$and $\Fil^r\left(\cL\tensor_{\Z_p}
\bB_{\rm cris,M}\right):= \cL\tensor_{\Z_p} \Fil^r \bB_{\rm cris,M} $ by sub-objects in the category ${\rm Ind}\bigl(\Sh(\fX_M)^\N\bigr)$ of inductive systems of
continuous sheaves. We denote by $${\rm Gr}^r\left(\bDcrisar(\cL)\otimes_{(\cO_X\otimes_{\cO_K} \cO_{\Mun})} \bB_{\rm cris,M}\right)$$and ${\rm Gr
}^r\left(\cL\tensor_{\Z_p} \bB_{\rm cris,M}\right)$ the $r$-th graded quotients of the two filtrations (which are objects in ${\rm Ind}\bigl(\Sh(\fX_M)^\N\bigr)$).

We will assume now until the end of this section that $\bL$ is a
crystalline sheaf on $X^{\rm et}_M$. By construction the
isomorphism
$$\alpha_{\rm cris,\cL}\colon
\bDcrisar(\cL)\otimes_{(\cO_X\otimes_{\cO_K} \cO_{\Mun})} \bB_{\rm
cris,M}\cong \cL\tensor_{\Z_p} \bB_{\rm cris,M}$$has the property
that $\alpha_{\rm
cris,\cL}\left(\Fil^r\left(\bDcrisar(\cL)\otimes_{(\cO_X\otimes_{\cO_K}
\cO_{\Mun})} \bB_{\rm cris,M}\right)\right) \subset
\Fil^r\left(\cL\tensor_{\Z_p} \bB_{\rm cris,M}\right)$ and thus
it induces a morphism ${\rm Gr}^r \alpha_{\rm cris,\cL}$ on ${\rm
Gr}^r$.

\begin{lemma}\label{lemma:gradediso} For every $r\in\Z$ the
natural morphism $$f\colon \bigoplus_{a+b=r} {\rm Gr}^a \bDcrisar(\cL)\otimes_{(\cO_X\otimes_{\cO_K} \cO_{\Mun})} {\rm Gr}^b \bB_{\rm cris,M}  \lra {\rm Gr
}^r\left(\cL\tensor_{\Z_p} \bB_{\rm cris,M}\right) $$is an isomorphism. In particular ${\rm Gr}^r \alpha_{\rm cris,\cL}$ is an isomorphism.

\end{lemma}

\begin{proof}  The surjective morphisms $$\bigoplus_{a+b=s}
\Fil^a \bDcrisar(\cL)\otimes_{(\cO_X\otimes_{\cO_K} \cO_{\Mun})} \Fil^b \bB_{\rm cris,M}  \lra \Fil^s \left(\bDcrisar(\cL)\otimes_{(\cO_X\otimes_{\cO_K}
\cO_{\Mun})} \bB_{\rm cris,M}\right)$$ for $s=r$ and $r+1$ induce a surjective morphism $$\bigoplus_{a+b=r} {\rm Gr}^a \bDcrisar(\cL)\otimes_{(\cO_X\otimes_{\cO_K}
\cO_{\Mun})} {\rm Gr}^b \bB_{\rm cris,M}  \lra {\rm Gr}^r\left(\bDcrisar(\cL)\otimes_{(\cO_X\otimes_{\cO_K} \cO_{\Mun})} \bB_{\rm cris,M}\right).$$The map $f$ is
the composite of this surjection and ${\rm Gr}^r \alpha_{\rm cris,\cL}$. In particular to deduce that ${\rm Gr}^r \alpha_{\rm cris,\cL}$ is an isomorphism we are
left to prove that $f$ is an isomorphism.

For every integer $N$ define $D(N):=
v_{M,\ast}\Bigl(\cL\tensor_{\Z_p} \bA_{\rm cris,M} (N)\Bigr)$ with
the induced filtration. It follows from
\ref{prop:Dacrisislocallyfree} that $\bigoplus_{a+b=r} {\rm Gr}^a
\bDcrisar(\cL)\otimes_{(\cO_X\otimes_{\cO_K} \cO_{\Mun})} {\rm
Gr}^b \bB_{\rm cris,M}$ is, for $N$ sufficiently small the
inductive system of continuous sheaves $\bigoplus_{a+b=r} {\rm
Gr}^a D(N)\otimes_{(\cO_X\otimes_{\cO_K} \cO_{\Mun})} {\rm Gr}^b
\bA_{\rm cris,M}(m)$ and ${\rm Gr }^r\left(\cL\tensor_{\Z_p}
\bB_{\rm cris,M}\right)$ is the inductive system of continuous
sheaves $\cL\tensor_{\Z_p} {\rm Gr}^r\left(\bA_{\rm
cris,M}(m)\right)$. Furthermore $f$ is induced by the natural
morphisms $$f_m\colon \bigoplus_{a+b=r} {\rm Gr}^a
D(N)\otimes_{(\cO_X\otimes_{\cO_K} \cO_{\Mun})} {\rm Gr}^b
\bA_{\rm cris,M}(m)\lra \cL\tensor_{\Z_p} {\rm Gr}^r\left(\bA_{\rm
cris,M}(N+m)\right).$$To conclude it would be enough to show that
there exists a negative integer $N$ and morphisms
$$g_m\colon \cL\tensor_{\Z_p} {\rm Gr}^r\left(\bA_{\rm
cris,M}(m)\right)\lra \bigoplus_{a+b=r}  {\rm Gr}^a D(N)
\otimes_{(\cO_X\otimes_{\cO_K} \cO_{\Mun})} {\rm Gr}^b \bA_{\rm
cris,M}(m+N)$$such that $g_{N+m} \circ f_m$ and $f_{m+N}\circ g_m$
induce automorphisms on the two inductive systems.

Consider a small affine $\cU$ of $X^{\rm et}$. Let
$V_\cU(\cL)=\cL\bigl(\Rbar_\cU\bigr)$ be the associated
representation of $\cG_\cU$. It is crystalline in the sense of
\cite{brinon} thanks to \ref{prop:equivcris} with $D_{\rm
cris}\bigl(V_\cU(\cL)\bigr)=\bDcrisar(\cL)(\cU)$. It follows from
\cite[Prop. 8.2.12]{brinon} that it is de Rham with $D_{\rm
dR}\bigl(V_\cU(\cL)\bigr)=D_{\rm cris}\bigl(V_\cU(\cL)\bigr)$ and
from \cite[Prop. 8.3.2]{brinon} that it is Hodge-Tate with $D_{\rm
HT}\bigl(V_\cU(\cL)\bigr)={\rm Gr} D_{\rm
dR}\bigl(V_\cU(\cL)\bigr)={\rm Gr} D_{\rm
cris}\bigl(V_\cU(\cL)\bigr)$ i.e. we have an isomorphism $D_{\rm
HT}\bigl(V_\cU(\cL)\bigr)\tensor_{(R_\cU\otimes_{\cO_K}
\cO_{\Mun})} B_{\rm HT}(\Rbar_\cU)\cong V_\cU(\cL)\tensor_{\Z_p}
B_{\rm HT}(\Rbar_\cU)$ as graded modules. Using the
identifications $B_{\rm HT}(\Rbar_\cU)={\rm Gr}\bigl(B_{\rm
cris}(\Rbar_\cU)\bigr)$ (see \cite[Cor. 5.2.7]{brinon}) and ${\rm
Gr}\bigl(B_{\rm cris}(\Rbar_\cU)\bigr)=\displaystyle{\lim_{\to,m
}} {\rm Gr} \bA_{\rm cris,M}(m)(\Rbar_\cU)$ (see
\ref{lemma:propbBcris}) we obtain that in the following diagram
the vertical arrows are isomorphisms: $$\begin{array}{ccc}
\displaystyle{\lim_{\to,m }} \left( {\rm Gr}
D(N)(\cU)\tensor_{(R_\cU\otimes_{\cO_K} \cO_{\Mun})} {\rm Gr}
\bA_{\rm cris,M}(m)(\Rbar_\cU)\right) &
\stackrel{f_m(\Rbar_\cU)}{\lra} & V_\cU(\cL)\tensor_{\Z_p}
\displaystyle{\lim_{\to,m }} {\rm Gr} \bA_{\rm
cris,M}(m)(\Rbar_\cU) \cr \big\downarrow & & \big\downarrow\cr
D_{\rm HT}\bigl(V_\cU(\cL)\bigr)\tensor_{(R_\cU\otimes_{\cO_K}
\cO_{\Mun})} B_{\rm HT}(\Rbar_\cU) & \stackrel{\sim}{\lra} &
V_\cU(\cL)\tensor_{\Z_p} B_{\rm HT}(\Rbar_\cU). \cr
\end{array}$$In particular since $V_\cU(\cL)$ is a free $\Z_p$-module of finite rank
there exists a negative integer $Q$  such that $V_\cU(\cL)$ is
contained in the image of $f_{Q}(\Rbar_\cU)$. Since ${\rm Gr}
D(N)(\cU)[p^{-1}]={\rm Gr} \bDcrisar(\cU)=D_{\rm
HT}\bigl(V_\cU(\cL)\bigr)$ and the latter is a projective
$R_\cU\tensor_{\cO_K}\Mun$-module by \cite[Prop.~8.3.2]{brinon},
there exists $h\in\N$ and a free $R_\cU\tensor_{\cO_K}
\OMun$-module $T$ and maps $a\colon {\rm Gr} D(N)(\cU) \to T$ and
$b\colon T\to {\rm Gr} D(N)(\cU)$ such that $b\circ a$ is
multiplication by $p^h$.

\noindent
We claim that $p^h$ annihilates the
kernel of the natural map
$$f\colon {\rm Gr}
D(N)(\cU)\otimes_{(R_\cU\otimes_{\cO_K} \cO_{\Mun})} {\rm Gr}
\bA_{\rm cris,M}(m)(\Rbar_\cU)\lra D_{\rm
HT}\bigl(V_\cU(\cL)\bigr)\tensor_{(R_\cU\otimes_{\cO_K}
\cO_{\Mun})} B_{\rm HT}(\Rbar_\cU.)$$
To see this let us first remark that
as $T$ is a free
$R_\cU\otimes_{\cO_K}\cO_{\Mun}$-module,
the natural map $$T\otimes_{(R_\cU\otimes_{\cO_K}\cO_{\Mun})}
{\rm Gr}\bA_{\rm cris,M}(m)(\Rbar_\cU)
\lra T\otimes_{(R_\cU\otimes_{\cO_K}
\cO_{\Mun})} B_{\rm HT}(\Rbar_\cU)$$ is injective.
We have the following commutative diagram
in which all the tensor products are over
$R_\cU\otimes_{\cO_K}\cO_{\Mun}$ and we have denoted
by $A_{M,m,\cU}$ the ring ${\rm Gr}
\bA_{\rm cris,M}(m)(\Rbar_\cU)$.
$$
\begin{array}{ccccccccccc}
{\rm Gr}
D(N)(\cU)\otimes A_{M,m,\cU}&\stackrel{a\otimes 1}{\lra}&
T\otimes A_{M,m,\cU}&\stackrel{b\otimes 1}{\lra}&
{\rm Gr} D(N)(\cU)\otimes A_{M,m,\cU}\\
f\downarrow&&\cap&&\downarrow\\
 D_{\rm HT}(V_\cU(\bL))\otimes B_{\rm HT}(\Rbar_\cU)&
\stackrel{a\otimes 1}{\lra}& T\otimes
 B_{\rm HT}(\Rbar_\cU)&\stackrel{b\otimes 1}{\lra}&D_{\rm HT}
(V_\cU(\bL))  \otimes B_{\rm HT}(\Rbar_\cU)
\end{array}
$$
Let $x\in \Ker(f)$. Then $p^hx=(b\otimes 1)\bigl(a\otimes 1(x)\bigr)=0$
which proves the claim.

\noindent Now we choose the  pre-image of basis elements of $V_\cU(\cL)$ in $$\left( {\rm Gr} D(N)(\cU)\tensor_{(R_\cU\otimes_{\cO_K} \cO_{\Mun})} {\rm Gr} \bA_{\rm
cris,M}(Q)(\Rbar_\cU)\right).$$ This determines a map of $\Z_p$-modules $\zeta_\cU\colon V_\cU(\cL) \lra {\rm Gr} D(N)(\cU)\tensor_{(R_\cU\otimes_{\cO_K}
\cO_{\Mun})} {\rm Gr} \bA_{\rm cris,M}(Q)(\Rbar_\cU)$. The composite with the projection onto $V_\cU(\cL)\tensor_{\Z_p} B_{\rm HT}(\Rbar_\cU)$ is the natural map $a
\mapsto a\tensor 1$. For every negative integer $m$ extend $\zeta_\cU$ as a ${\rm Gr} \bA_{\rm cris,M}(m)(\Rbar_\cU) $ linear  map $$t_m(\Rbar_\cU)\colon
V_\cU(\cL)\tensor_{\Z_p}{\rm Gr} \bA_{\rm cris,M}(m)(\Rbar_\cU) \lra  {\rm Gr} D(N)(\cU)\tensor_{(R_\cU\otimes_{\cO_K} \cO_{\Mun})} {\rm Gr} \bA_{\rm
cris,M}(Q+m)(\Rbar_\cU).$$Then $f_{Q+m}(\Rbar_\cU)\circ t_m(\Rbar_\cU)$ is multiplication by $p^h$ times the shift by $N+Q$. Similarly, replacing $\zeta_\cU$ and
$t_m(\Rbar_\cU)$ with $p^h \zeta_\cU$ and $p^h t_m(\Rbar_\cU)$, the composite $t_{m+N}(\Rbar_\cU)\circ f_m(\Rbar_\cU)$ induces the identity on ${\rm Gr} D(N)(\cU)$
so that $t_{N+m}(\Rbar_\cU)\circ f_m(\Rbar_\cU)$ is also $p^h$ times the shift by $Q+N$. In particular $t_m(\Rbar_\cU)$ and $f_m(\Rbar_\cU)$ define inverses one of
the other for the two inverse systems defined by varying $m$.

Recall that  the composite of $\zeta_\cU$ with the projection onto $V_\cU(\cL)\tensor_{\Z_p} B_{\rm HT}(\Rbar_\cU)$ is the natural map $a \mapsto a\tensor 1$. In
particular it is unique with this property and it is $\cG_\cU$-equivariant. This implies that multiplying it by $p^h$ gives a $\cG_\cU$-equivariant map $\zeta_\cU$.
Note that $\zeta_U$ determines $\zeta_{U'}$ for any small affine $\cU'\to \cU$ and since $X$ can be covered by finitely many small affine opens by taking  $N$ and
$Q$ sufficiently small and $h$ sufficiently large and reducing modulo $p^n$ the morphisms $\zeta_\cU$ glue and define a morphism
$$\zeta_n\colon \cL_n \to {\rm Gr} D(N)
\otimes_{(\cO_X\otimes_{\cO_K} \cO_{\Mun})} {\rm Gr}\left(
\bA_{\rm cris,n,M}' (Q)\right).$$For every negative integer $m$
we extend it ${\rm Gr} \bA_{\rm cris,n,M}'(m)$-linearly  to get
morphisms
$$g_{m,n}\colon \cL_n\tensor_{\Z_p} {\rm Gr}\bA_{\rm
cris,n,M}' (m) \lra  {\rm Gr} D(N) \otimes_{(\cO_X \otimes_{\cO_K}
\cO_{\Mun})} {\rm Gr}\left( \bA_{\rm cris,n,M}'
(m+Q)\right),$$which are compatible for varying $m$ and $n$. Let
$\cU$ be a small affine. The composite $f_{m+Q,n}(\Rbar_\cU)\circ
g_{m,n}(\Rbar_\cU)$  is multiplication by $p^h$ on $V_\cU(\cL)/p^n
V_\cU(\cL)=\cL_n(\Rbar_\cU)$ and hence  it is multiplication by
$p^h$ times the shift by $Q+N$ by linearity. Similarly
$g_{m+N,n}(\Rbar_\cU) \circ f_{m,n}(\Rbar_\cU)$ is multiplication
by $p^h$ on ${\rm Gr} D(N)(\cU)$ and hence it is multiplication
by $p^h$ times the shift by $N+Q$ by linearity. This implies that
$f_{m+Q,n} \circ g_{m,n}$ and  $g_{m+N,n}(\Rbar_\cU) \circ c$ are
multiplication by $p^h$ times the shift by $N+Q$. Thus since
multiplication by $p$ and shifts define automorphisms of inductive
systems, we conclude that  $\{f_{m,n}\}_n$ and $\{g_{m,n}\}_n$
define automorphisms of inductive systems as claimed.

\end{proof}

\begin{proposition}\label{prop:connFilDcriar} Fix a finite extension $K\subset M$.
Assume that $\cL$ is a crystalline \'etale sheaf on $X_M^{\rm et}$
and take $N\in\Z$ as in \ref{prop:Dacrisislocallyfree}.
Then we have\smallskip

1) for varying~$r\in\N$ the $\cO_{X_{\Mun}}$--modules
$\Fil^r\bDcrisar(\cL)$ define a decreasing, exhaustive and
separated filtration of~$\Fil^r\bDcrisar(\cL)$ by locally free
$\cO_{X_{\Mun}}$--modules having the property that the
quotient~$\Fil^r\bDcrisar(\cL)/ \Fil^{r+1}\bDcrisar(\cL)$ is
locally free for every~$r\in\N$;\smallskip

2) the connection $\nabla_\cL$ on~$\bDcrisar\bigl(\cL\bigr)$ is
integrable, quasi--nilpotent and satisfies Griffith's
transversality relatively to the given filtration;\smallskip

3) the isomorphism $\alpha_{\rm cris,\cL}\colon
\bDcrisar(\cL)\otimes_{(\cO_X\otimes_{\cO_K} \cO_{\Mun})} \bB_{\rm
cris,M}\cong \cL\tensor_{\Z_p} \bB_{\rm cris,M}$ preserves the
connection where on the left we consider the composite of the
connection $\nabla_\cL$ and the connection on $\bB_{\rm cris,M}$
while on the right we consider the connection induced from the one
on $\bB_{\rm cris,M}$ which is trivial on $\cL$. Furthermore it
induces an isomorphism, in the category ${\rm
Ind}\bigl(\Sh(\fX_M)^\N\bigr)$ of inductive systems of continuous
sheaves, on filtrations;
\smallskip

4) the map $\bDcrisar(\cL)\widehat{\tensor}_{\OMun} A_{\rm
cris}\lra \bDcrisgeo(\cL)$ is injective and induces an isomorphism
after inverting~$t$. Furthermore $\sum_{a+b=r}\Fil^a
\bDcrisar(\cL)\widehat{\tensor}_{\OMun} \Fil^b A_{\rm cris} =
\left(\bDcrisar(\cL)\widehat{\tensor}_{\OMun} A_{\rm
cris}\right)\cap \Fil^r \bDcrisgeo(\cL)$ for every~$r\in\Z$.
\smallskip

\end{proposition}
\begin{proof} 1)  follows from \cite[Prop.~8.3.2]{brinon}.

2) follows from  \cite[Prop.~8.3.4]{brinon} using
\ref{prop:Dacrisislocallyfree} (which is implicitly assumed in
loc.~cit.).

3) The assertion regarding the connection is clear.  By construction the given morphism preserve the filtrations. Since $X$ is noetherian there exists $H\in\Z$ such
that $\Fil^H \bDcrisar(\cL) = \bDcrisar(\cL)$. Due to Lemma \ref{lemma:isoinBcris} the fact that $\alpha_{\rm cris,\cL}$ is an isomorphism implies that there exists
$N\in\N$ and morphisms of $\bA_{\rm cris,M}$-modules $\gamma_m\colon \cL\tensor \bA_{\rm cris,M}(m) \to  \bDcrisar(\cL) \tensor \bA_{\rm cris,M}(N+m)$, compatible
for varying $m$, defining the inverse of $\alpha_{\rm cris,\cL}$. Since $\gamma_m$ is $\bA_{\rm cris,M}$-linear and $\Fil^r \bA_{\rm cris,M}(m)=\left(\Fil^{r-m}
\bA_{\rm cris,M}\right)\cdot \bA_{\rm cris,M}(m)$ the image via $\gamma_m$ of $\Fil^r$ on the left hand side is contained in $\left(\Fil^{r-m} \bA_{\rm
cris,M}\right)\cdot \bDcrisar(\cL) \tensor \bA_{\rm cris,M}(N+m)$ which is $\bDcrisar(\cL) \tensor \Fil^{r+N}\bA_{\rm cris,M}(N+m)$ and is contained in
$\Fil^{r+N+H} \left(\bDcrisar(\cL) \tensor \bA_{\rm cris,M}(N+m)\right)$. In particular $\cL\tensor \Fil^r \bB_{\rm cris,M}$ is contained in the image of
$\Fil^{r+N+H} \left(\bDcrisar(\cL) \tensor \bB_{\rm cris,M}\right)$ via $\alpha_{\rm cris,\cL} $.

We are left to prove that the map induced by $\alpha_{\rm
cris,\cL}$ on the quotient inductive systems $\Fil^r/\Fil^s$ for
$r\leq s$ is injective. Proceeding inductively it suffices to
consider the case that $r=s+1$ and this follows from
\ref{lemma:gradediso}.
\smallskip

4) It follows from (3) that $\bDcrisgeo(\cL)$ is $v_{\Kbar,\ast}\left(\bDcrisar(\cL)\otimes_{(\cO_X\otimes_{\cO_K} \cO_{\Mun})} \bB_{\rm cris,M}\right)$ as filtered
module. Since $ \bDcrisar(\cL)$ is projective as $\cO_X\tensor_{\cO_K}\OMun$-module it is locally a direct summand in a free module. To prove the first claim it
then suffices to show that for every small affine $\cU$, the map  $R_\cU \widehat{\tensor}_{\cO_K} A_{\rm cris}(\OKbar) \lra \left(A_{\rm
cris}(\Rbar_\cU\right)^{\cG_{\cU,\Kbar}} $ is injective and it has kernel annihilated by a fixed power of~$t$. This is proven
in~\cite[Cor.~31]{andreatta_brinonacyclicity}.

To prove the second statement it suffices to show that the map
induced on graded pieces $\sum_{a+b=r}{\rm Gr}^a
\bDcrisar(\cL)\widehat{\tensor}_{\OMun} {\rm Gr}^b A_{\rm cris}
\to {\rm Gr}^r \bDcrisgeo(\cL)$ is injective for every $r$. It
follows from \ref{lemma:gradediso} and the fact ${\rm Gr}^a
\bDcrisar(\cL)$ is a projective $\cO_X\otimes_{\cO_K}
\cO_{\Mun}$-module that $v_{\Kbar,\ast}\left({\rm Gr
}^r\left(\cL\tensor_{\Z_p} \bB_{\rm cris,M}\right)\right)$ is $
\bigoplus_{a+b=r} {\rm Gr}^a
\bDcrisar(\cL)\otimes_{(\cO_X\otimes_{\cO_K} \cO_{\Mun})}
v_{\Kbar,\ast}\left({\rm Gr}^b \bB_{\rm cris,M}\right) $. The
claim follows remarking that ${\rm Gr}^b A_{\rm cris}$ injects in
$v_{\Kbar,\ast}\left({\rm Gr}^b \bB_{\rm cris,M}\right)$.

\end{proof}

Let~$\cU$  be a small affine and choose parameters
$T_1,\ldots,T_d\in R_\cU^\times$. Write~$F_\cU$
(resp.~$\varphi_\cU$) for the associated Frobenius on~$\cU$
(resp.~on~$\bA_{\rm cris,M}\vert_{\fUM}$). Define a Frobenius
$\varphi_r$ on $\bA_{\rm cris,M}^\nabla(r)$ (resp.~$\bA_{\rm
cris,M}(r)\vert_{\fUM}$) to be the map $p^r \tensor \varphi$ on
$\Z_p(r) \tensor_{\Z_p} \bA_{\rm cris,M}^\nabla$ (resp.~$\Z_p(r)
\tensor_{\Z_p} \bA_{\rm cris,M}\vert_{\fUM}$).  We then get
$F_\cU$--linear maps $$\varphi_\cU\colon \bDcrisar(\cL) \lra
\bDcrisar(\cL),\qquad \varphi_\cU\colon
v_{M,\ast}\bigl(\cL\tensor \bA_{\rm cris,M}(N)\bigr) \lra
v_{M,\ast}\bigl(\cL\tensor \bA_{\rm cris,M}(N)\bigr).$$

\begin{proposition}\label{prop:FrobDcriar}
Assume that $\cL$ is a crystalline \'etale sheaf. Then we have.\smallskip

1) $\varphi_\cU$ is horizontal with respect to the
connection~$\nabla_\cL\vert_\cU$ i.e., $\nabla_\cL\vert_\cU \circ
\varphi\vert_\cU=\bigl(\varphi\vert_\cU\tensor d F_\cU\bigr)\circ
\nabla_\cL\vert_\cU$,\smallskip

2) $\bDcrisar(\cL)\vert_{\cU} $ is an \'etale $F_\cU$--module
i.~e., $\varphi_\cU\tensor 1\colon
\bDcrisar(\cL)\vert_{\cU}\tensor_{\cO_\cU}^{F_\cU} \cO_\cU \to
\bDcrisar(\cL)\vert_{\cU}$ is an isomorphism.
\end{proposition}
\begin{proof} (1) follows since $\varphi_\cU$
on~$\bA_{\rm cris,M}\vert_\cU$ is horizontal
by~\ref{prop:deRhamcomplex}.

(2) follows since~$\bDcrisar(\cL)$ is a coherent module
$\bDcrisar(\cL)(\cU)=D_{\rm cris}\bigl(V_\cU(\cL)\bigr)$ and the
latter is \'etale thanks to~\cite[Prop.~8.3.3]{brinon}.
\end{proof}

Let $\F$ be the residue field of~$\OMun$ and write $\cU_\F$ for
$\cU\tensor_{\cO_K}\F$. Assume that $\cL$ is a crystalline \'etale
sheaf on $X_M^{\rm et}$. It follows from
\ref{prop:connFilDcriar}(2) and  from~\ref{prop:FrobDcriar} that $\bigl(
\bDcrisar\bigl(\cL\bigr)\vert_\cU,\varphi_\cU\bigr)$  is a
convergent $F$-isocrystal $\cU_\F$ relatively to $\OMun$ in the
sense of \cite[Def. 2.3.7]{berthelot}.

\begin{lemma}\label{lemma:changeofFrobenius}
Suppose we have two choices of parameters~$T_1,\ldots,T_d$
and~$T_1',\ldots,T_d'$ of~$R_\cU^\times$. Denote by~$\varphi_\cU$
and~$\varphi_\cU'$ the corresponding Frobenius morphisms on
$\bDcrisar\bigl(\cL\bigr)\vert_\cU$. Then
$\left(\bDcrisar\bigl(\cL\bigr)\vert_\cU,\varphi_\cU\right)$ and
$\left(\bDcrisar\bigl(\cL\bigr)\vert_\cU,\varphi_\cU'\right)$
define the same convergent $F$-crystals on~$\cU_\F$ relatively
to~$\OMun$.

\end{lemma}
\begin{proof} Let $F_\cU$ and~$F_\cU'$ be the Frobenii on~$R_\cU$
defined by the two choices of parameters. Let $p_1$, $p_2\colon
\cU\times_{\cO_K} \cU \to \cU$ be the two projections.   The
convergent connection~$\nabla_\cL$
on~$\bD:=\bDcrisar\bigl(\cL\bigr)\vert_\cU$ defines an isomorphism
of $\cO_{\cU_K}$--modules $\epsilon\colon p_2^\ast\bigl(\bD\bigr)
\lra p_1^\ast\bigl(\bD\bigr)$ on the tube of the
diagonal $\cU \to \cU\times \cU$ defining the structure of
isocrystal on~$\bD$. Consider $F=(F_\cU,F_\cU')\colon \cU_K \to
\cU_K\times \cU_K$. Then $F^\ast(\epsilon)$ induces an
isomorphism $F_\cU^{'\ast}\bigl(\bD\bigr) \lra
F_\cU^\ast\bigl(\bD\bigr)$. The claim amounts to prove that
$\varphi_\cU'\tensor 1 =\varphi_\cU\tensor 1 \circ
F^\ast(\varepsilon) $. Since~$\bD$ is coherent it suffices to
verify this on $\cU_K$--sections. Write $D:=\bD(\cU_K)$. The
map~$\epsilon$ on~$\cU$ is the map $D\ni m \mapsto
\sum_{\underline{n}\in\N^d} \left( \prod_{i=1}^d N_i^{n_i}
\right)(m) \tensor \bigl(1\tensor T_i-T_i\tensor 1\bigr)^{[n_i]} $
where~$N_i$ is the endomorphism of~$D$ given by~$\nabla_\cL \circ
1\tensor \frac{\partial}{\partial T_i}$ see
\cite[Pf.~Thm.~4.12]{berthelot_ogus}. Thus $F^\ast(\epsilon)$ is
the map sending $m\tensor 1 \in D\tensor_{R_\cU}^{F_\cU'} R_\cU$
to $ \sum_{\underline{n}\in\N^d} \left( \prod_{i=1}^d N_i^{n_i}
\right)(m) \tensor \bigl(F_\cU'(T_i)-F_\cU(T_i)\bigr)^{[n_i]}$.
Eventually $\varphi_\cU \tensor 1 \circ F^\ast(\epsilon)$ is the
map $m \mapsto \sum_{\underline{n}\in\N^d}\varphi_\cU\left(
\prod_{i=1}^d N_i^{n_i} \right)(m) \tensor
\bigl(F_\cU'(T_i)-F_\cU(T_i)\bigr)^{[n_i]}$. This is the
expression for~$\varphi_\cU'$ computed
in~\cite[Prop.~7.2.3]{brinon}.
\end{proof}

In particular the $F$--isocrystals defined by~$\bigl(\bDcrisar\bigl(\cL\bigr)\vert_\cU,\varphi_\cU\bigr)$ glue for different choices of~$\cU$'s and parameters  and
define a convergent $F$--isocrystal $\bigl(\bDcrisar\bigl(\cL\bigr),\nabla_\cL,\varphi_{\cL,M}\bigr)$. Denote by $\Isoc(X_\F/\Mun)$ the category of filtered
convergent $F$--isocrystals. It is a tensor category and it is abelian if we consider only convergent $F$--isocrystals (forgetting the filtrations); see
\cite[Rmk.~2.3.3(iii)\&\S2.3.7]{berthelot}. For every~$n\in\Z$ define ${\bf 1}(n)$ to be the isocrystal $\cO_{X_{\Mun}}$ with the connection defined by the usual
derivation, Frobenius given by $p^{-n}$ times the Frobenius on $\cO_{X_{\Mun}}$ and filtration which is $0$ for $r>n$ and is $\cO_{X_{\Mun}}$ for $r\leq n$. Given a
convergent filtered $F$--isocrystal $\cE$ we put $\cE(n):=\cE\tensor {\bf 1}(n)$ and we call it the $n$--th Tate twist of $\cE$. We get a functor
$$\bDcrisar\colon \Sh(X_M^{\rm et})_{\Q_p}^{\rm cris}\lra
\Isoc(X_\F/\Mun)$$ given by $$ \cL\mapsto
\bigl(\bDcrisar\bigl(\cL\bigr),\nabla_\cL,\{\Fil^r
\bDcrisar\bigl(\cL\bigr)\},\varphi_{\cL,M}\bigr).
$$Define $\Isoc(X_\F/\Mun)^{\rm adm}$, the
category of {\it admissible filtered convergent $F$--isocrystals},
to be the essential image of~$\bDcrisar$.

Let~$\underline{\cE}:=\bigl((\cE,\nabla),\{\Fil^r \cE\}_{r\in
\Z},\Phi\bigr)$ be a filtered convergent $F$--isocrystal on~$X_\F$
relative to~$\Mun$. Due to~\cite[Thm.~2.4.2]{berthelot} there
exists an $\cO_X\tensor_{\cO_K}\OMun$--module
$\bigl(\cM,\nabla,\Phi\bigr)$ with integrable and nilpotent
connection and non--degenerate Frobenius such that~$(\cM^{\rm
rig},\nabla^{\rm rig},\Phi^{\rm rig})$ is the Tate
twist~$(\cE(n),\nabla,\Phi(n)\bigr)$ for some~$n\in\N$. By
loc.~cit. such crystal is unique up to isogeny and up to Tate twist. Define
$$\bV_{\rm cris}^{\rm ar}(\underline{\cE}):=\Fil^0\left(v_{M}^\ast(\cM)(-n)
\tensor_{\cO_{\fX_M}^{\rm un}\tensor_{\cO_K}\OMun} \bA_{\rm cris,K}\right)^{\nabla=0,\Phi=1}\in {\rm Ind}\left(\Sh(\fX_M)^\N\right).$$ Recall that we have a fully
faithful functor $u_{X,M,\ast}\colon \Sh(X_M^{\rm et})_{\Q_p}\lra {\rm Ind}\left(\Sh(\fX_M)^\N\right)$ and we identify $\Sh(X_M^{\rm et})_{\Q_p}$ with its essential
image.

\begin{theorem}\label{thm:crisistannakian} The following
hold:\smallskip

1) the sub-category $\Sh(X_M^{\rm et})^{\rm cris}_{\Q_p}$ of
$\Sh(X_M^{\rm et})^\N$ is an abelian tensor sub-category, closed
under Tate twists, duals and tensor products;\smallskip

2) the sub-category ${\rm Isoc}(X_\F/\Mun)^{\rm adm}$ of admissible
filtered convergent $F$--isocrystals of  ${\rm Isoc}(X_\F/\Mun)$ is an
abelian tensor sub-category, closed under Tate twists, duals and
tensor products.\smallskip

3) the functor $\bDcrisar$ is an exact functor of abelian tensor
categories, it is fully faithful and commutes with duals and Tate
twists;\smallskip

4) the functor $\bVcrisarQ$ factors via $\Sh(X_M^{\rm et})^{\rm cris}_{\Q_p}$ and $\bVcrisarQ\circ \bDcrisar$ is equivalent to the identity. In particular
$\bDcrisar$ defines an equivalence of categories $$\Sh(X_M^{\rm et})_{\Q_p}^{\rm cris} \cong {\rm Isoc}(X_\F/\Mun)^{\rm adm}.$$
\end{theorem}
\begin{proof} Claim~(2) follows from~(1)\&(3). By construction $\bDcrisar$
is essentially surjective. To verify~(1) and the rest of~(3) one reduces
to the case that~$X=\cU$ is
small affine. To verify that $\Sh(X_M^{\rm et})^{\rm cris}_{\Q_p}$ is closed
under tensor product, internal Hom and Tate twists and
that~$\bDcrisar$ commutes with tensor products, internal~$\Hom$ and
Tate twists one reduces to the case that~$X=\cU$ is a small
affine. Using that~$\bDcrisar(\cL)(\cU)=D_{\rm cris}(V_\cU(\cL))$
and that~$\bDcrisar(\cL)$ is coherent, these claims follow from
analogous statements for~$D_{\rm cris}(V_\cU(\cL))$ proven
in~\cite[Thm.~8.4.2]{brinon}.

It follows from loc.~cit.~that if~$\cL$ is crystalline and~$\cU$
is a small affine as in~\ref{prop:equivcris} one can recover the
$\Q_p$-representation~$V_\cU(\cL)$ from~$D_{\rm cris}
\bigl(V_\cU(\cL)\bigr)$ by the formula
$$V_\cU(\cL)=\Fil^0\left(D_{\rm cris}
\bigl(V_\cU(\cL)\bigr)\tensor_{R_\cU\otimes_{\cO_K} \cO_{\Mun}} B_{\rm cris}(\Rbar_\cU) \right)^{\nabla=0,\varphi=1}.$$This is equal to
$\bVcrisarQ\bigl(\bDcrisar(\cL)\bigr)(\Rbar_\cU)[p^{-1}]$ thanks to~\ref{lemma:loctensorAcris}. This allows to recover $\cL$ up to isogeny. In particular $\bDcrisar
$  is fully faithful. Being essentially surjective it defines an equivalence of categories.
\end{proof}

It is a difficult question to characterize ${\rm Isoc}(X_\F/\Mun)^{\rm adm}$ in  ${\rm Isoc}(X_\F/\Mun)$. If $X=\Spf(\cO_K)$, a satisfactory answer is provided in
\cite{colmez_fontaine} in terms of the so called weakly admissible modules. In more generality a complete answer exists for admissible filtered convergent
$F$--isocrystals of rank~$1$ thanks to \cite{brinon}. Let $\underline{\cE}:=\bigl((\cE,\nabla),\{\Fil^r \cE\}_{r\in \Z},\Phi\bigr)$ be a filtered convergent
$F$--isocrystal on~$X_\F$ relative to~$\Mun$ of rank~$1$ i.~e., $\cE$ is a locally free $\cO_{X_K^\rig}$--module of rank~$1$. Define $t_H(\cE)$ to be the locally
constant function on $X\otimes_{\cO_K} \cO_{\Mun}$ locally defined as the largest integer such that $\Fil^r\cE=\cE$.

Let~$\cU=\Spf(R_\cU)$ be a small \'etale open  affine of $X$ such that $\cE\vert_\cU=\cO_\cU\otimes_{\cO_K} \Mun e$. Since~$\cE$ is an isocrystal we have
$\Phi(e)=a_\cU e$ with~$a_\cU\in R_\cU\otimes_{\cO_K} \OMun\bigl[p^{-1}\bigr]^\times$. For every connected component $\cU_i=\Spec(R_{\cU_i})$ of
$\Spec(R_\cU\otimes_{\cO_K} \OMun )$, the element~$p$ generates a prime ideal of~$R_{\cU_i}$ and we must have~$a_\cU=p^{n_i} \alpha$ with~$\alpha_{\cU_i}$ a unit
in~$R_{\cU_i}$. Then the integer $t_N(\cE)(\cU_i):=n_i$ does not depend on the choice of~$e$ and is locally constant on $X\otimes_{\cO_K} \cO_{\Mun}$. Then we have.

\begin{proposition}\label{prop:admissibleforrank1}
A rank~$1$, filtered convergent $F$--isocrystal $\underline{\cE}$
is admissible if and only if is locally free as
$\cO_X\otimes_{\cO_K} \Mun$ for the \'etale topology on $X$ and we
have $t_H(\cE)=t_N(\cE)$.
\end{proposition}
\begin{proof} It follows from \ref{prop:equivcris} and
\cite[Prop.~8.6.2]{brinon}.
\end{proof}

Finally we compare our notion of a crystalline \'etale sheaf with the notion of ``associated sheaves" given in \cite[p.~67]{faltingscrystalline}.
Let~$\underline{\cE}$ be a filtered convergent $F$--isocrystal on~$X_\F$ relative to~$\Mun$. As explained above we may assume that up to Tate twist it is the
generic fiber of a $\cO_X\tensor_{\cO_K}\OMun$--module $\cM$ endowed with an integrable and nilpotent connection and a non--degenerate Frobenius. We identify $\cM$
with the associated crystal on~$X_\F/ \cO_{\Mun}$. Given a small affine $\cU:=\Spf(R_\cU)$ of~$X$ we write $\cE\bigl(B_{\rm cris}^\nabla(\Rbar_\cU)\bigr)$ for
$\cM\bigl(A_{\rm cris}^\nabla(\Rbar_\cU)\bigr)\tensor_{A_{\rm cris}} B_{\rm cris}$ where $\cM\bigl(A_{\rm cris}^\nabla(\Rbar_\cU)\bigr)$ is the value of the
crystal~$\cM$ on the PD--thickening $\theta\colon A_{\rm cris}^\nabla(\Rbar_\cU) \to \widehat{\Rbar}_\cU$. Note that given a morphism $\sigma\colon R_\cU
\otimes_{\cO_K} \cO_{\Mun}\to A_{\rm cris}^\nabla(\Rbar_\cU)$ as $\cO_{\Mun}$--algebras inducing the identity on~$R_\cU\otimes_{\cO_K} \cO_{\Mun}$ via the
projection $A_{\rm cris}^\nabla(\Rbar_\cU) \to \widehat{\Rbar}_\cU$ (for example the one sending $T_i\mapsto \tT_i$) we get that $\cE\bigl(B_{\rm
cris}^\nabla(\Rbar_\cU)\bigr)\cong \cM(R_\cU)\tensor_{R_\cU\otimes_{\cO_K} \cO_{\Mun}} B_{\rm cris}^\nabla(\Rbar_\cU)\cong \cE(R_\cU)\tensor_{R_\cU\otimes_{\cO_K}
\cO_{\Mun}} B_{\rm cris}^\nabla(\Rbar_\cU)$. Since~$\cM$ is a crystal the first identification does not depend on the choice of~$\sigma$. More precisely, given two
sections $\sigma$ and $\sigma'$ there is a canonical isomorphism between $\cM(R_\cU)\tensor_{R_\cU\otimes_{\cO_K} \cO_{\Mun}}^\sigma B_{\rm cris}^\nabla(\Rbar_\cU)$
and $\cM(R_\cU)\tensor_{R_\cU\otimes_{\cO_K} \cO_{\Mun}}^{\sigma'} B_{\rm cris}^\nabla(\Rbar_\cU)$ whose Taylor expansion is defined using the connection on~$\cM$.
In particular $\cE\bigl(B_{\rm cris}^\nabla(\Rbar_\cU)\bigr)$ is endowed with an action of $\cG_{\cU,M}$ and Frobenius and since the connection on~$\cE$ satisfies
Griffith's transversality the filtration on $\cE\bigl(B_{\rm cris}^\nabla(\Rbar_\cU)\bigr)$ induced from the filtration on~$\cE(R_\cU)$ does not depend on~$\sigma$.
Let $\cL$ be a $\Q_p$--adic \'etale sheaf on~$X_M$. Following Faltings one says that~$\cL$ and~$\cE$ are {\it associated} if there is an isomorphism
$$\rho_\cU\colon \cE\bigl(B_{\rm cris}^\nabla(\Rbar_\cU)\bigr)\cong
V_\cU(\cL)\tensor_{\Z_p} B_{\rm cris}^\nabla(\Rbar_\cU)$$of
$B_{\rm cris}^\nabla(\Rbar_\cU)$--modules commuting with
filtrations action of $\cG_{\cU,M}$ and Frobenius for every small
affine $\cU$ which is functorial in~$\cU$. Recall that
$V_\cU(\cL)$ is the $\cG_{\cU,M}$--representation
$\cL(\Rbar_\cU)$. Then we have.

\begin{lemma}\label{lemma:associated} The sheaves $\underline{\cE}$ and $\cL$ are
associated in Faltings' sense if and only if\/ $\cL$ is
crystalline and $\bDcrisar(\cL)=\underline{\cE}$.
\end{lemma}
\begin{proof} First of all we remark that to be associated in
Faltings' sense it suffices  that there is a covering of $X$ by small affines $\{\cU_i\}_i$ such that we have an isomorphism $\rho_{\cU_i}$ for every~$i$ and
$\rho_{\cU_i}$ and $\rho_{\cU_j}$ are compatible on $\cU_i\cap \cU_j$. By \ref{prop:equivcris} we have that $\cL$ is crystalline if and only if there is a covering
of $X$ by small affines such that the natural map $g_{\cU_i}\colon \bDcrisar(\cL)(\cU_i)\tensor_{R_{\cU_i}\otimes_{\cO_K} \cO_{\Mun}} B_{\rm
cris}(\Rbar_{\cU_i})\cong V_{\cU_i}(\cL)\tensor_{\Z_p} B_{\rm cris}(\Rbar_{\cU_i})$ is an isomorphism. By construction $g_{\cU_i}$ is an isomorphism of $B_{\rm
cris}(\Rbar_\cU)$--modules and it commutes with filtrations, action of $\cG_{\cU,M}$, Frobenius and connections. For every~$i$ take the section $\sigma_i \colon
R_{\cU_i}\otimes_{\cO_K} \cO_{\Mun} \to A_{\rm cris}^\nabla(\Rbar_{\cU_i})$ sending $T_j \to \tT_j$. It induces the section of the inclusion $A_{\rm
cris}^\nabla(\Rbar_{\cU_i}) \to A_{\rm cris}(\Rbar_{\cU_i})\cong A_{\rm cris}^\nabla(\Rbar_{\cU_i})\langle 1\tensor T_j - \tT_j\tensor 1 \rangle_{j=1,\ldots,d}$
sending $1\tensor T_j - \tT_j\tensor 1\mapsto 0$. Therefore it defines a section $\widetilde{\sigma}_i$ of $B_{\rm cris}^\nabla(\Rbar_{\cU_i}) \to B_{\rm
cris}(\Rbar_{\cU_i})$ compatible with filtrations, Frobenius and $\cG_{\cU,M}$--action (considering on $B_{\rm cris}(\Rbar_{\cU_i})$ the Galois action twisted via
the connection).

If~$\cL$ is crystalline then we get an isomorphism $\rho_{\cU_i}:=\widetilde{\sigma}_i^\ast(g_{\cU_i})$, compatible with all the supplementary structures and $\cL$
and $\bDcrisar(\cL)$ are associated. In the other direction, assume that~$\cL$ and~$\cE$ are associated. Write $\cE\bigl(B_{\rm cris}(\Rbar_{\cU_i})\bigr)$ for
$\cM\bigl(A_{\rm cris}(\Rbar_\cU)\bigr)\tensor_{A_{\rm cris}} B_{\rm cris}$ where $\cM\bigl(A_{\rm cris}(\Rbar_\cU)\bigr)$ is the value of the crystal~$\cM$ on the
PD--thickening $A_{\rm cris}(\Rbar_\cU) \to \widehat{\Rbar}_\cU$. Since~$\cM$ is a crystal by definition \cite[Def.~6.1]{berthelot_ogus} we have canonical
isomorphisms $\cE\bigl(B_{\rm cris}(\Rbar_{\cU_i})\bigr)\cong \cE(\cU)\tensor_{R_\cU\otimes_{\cO_K} \cO_{\Mun}} B_{\rm cris}(\Rbar_{\cU_i})$ and $\cE\bigl(B_{\rm
cris}(\Rbar_{\cU_i})\bigr)\cong \cE\bigl(B_{\rm cris}^\nabla(\Rbar_{\cU_i})\bigr)\tensor_{B_{\rm cris}^\nabla(\Rbar_{\cU_i})} B_{\rm cris}(\Rbar_{\cU_i})$ of
$B_{\rm cris}(\Rbar_{\cU_i})$--modules. Being canonical they commute with the $\cG_{\cU,M}$--action and Frobenius. Since the connection on~$\cE$ satisfies
Griffith's transversality the isomorphisms preserve also the filtrations. Define $ g_{\cU_i}$ to be the extension of scalars of $\rho_{\cU_i}$ using the second
isomorphism. It follows from the first isomorphism that $\bDcrisar(\cL)=\cE$ as filtered convergent Frobenius isocrystal. In particular $\cL$ is crystalline. The
claim follows.
\end{proof}

\subsection{The cohomology of crystalline sheaves.}
\label{sec:cohocryssheaves}

As before we assume that $\cO_K=\WW(k)$ and we fix a finite
extension~$K\subset M$ contained in $\Kbar$.   Let $X$ be a
smooth $p$-adic formal scheme over~$\cO_K$.

\begin{theorem}
\label{thm:formalcomparison} We have  canonical isomorphisms of
$\delta$-functors  from $\Sh(X_M^{\rm et})^{\rm cris}_{\Q_p}$ to
the category of filtered $B_{\rm cris}$-modules endowed with the
action of $G_M$ and Frobenius
$$
{\rm H}^i\bigl(\fXKbar, \bL\otimes\bB_{\rm
cris,\Kbar}^\nabla\bigr) \cong {\rm H}^i_{\rm dR}\bigl(X^{\rm et},
\bD_{\rm cris}^{\rm geo}(\bL)\bigr),
$$for $\cL$ a crystalline $\Q_p$-adic \'etale sheaf.  In fact
for every $r\in \Z$ we have isomorphisms of
$B_{\rm cris}$-modules which are $G_M$-equivariant and compatible for varying $r$'s
and compatible with the previous isomorphism
$$
{\rm H}^i\bigl(\fXKbar, \bL\otimes\Fil^r \bB_{\rm
cris,\Kbar}^\nabla\bigr) \cong  \bH^i\bigl(X^{\rm et},\Fil^{\rm
r-\bullet} \bDcrisgeo(\bL)\otimes_{\cO_X}
\Omega^\bullet_{X/\cO_K}\bigr).
$$
\end{theorem}

The cohomology ${\rm H}^i\bigl(\fXKbar, \bL\otimes\Fil^r \bB_{\rm
cris,\Kbar}^\nabla\bigr)$  of the inductive system
$\bL\otimes\Fil^r \bB_{\rm cris,\Kbar}^\nabla$ is taken as
explained in \S\ref{def:bBcris} for every $r\in\Z\cup
\{-\infty\}$.

The cohomology group $\bH^i\bigl(X^{\rm et}, \Fil^{\rm
r-\bullet}\bDcrisgeo(\bL)\otimes_{\cO_X}\Omega^\bullet_{X/\cO_K}\bigr)$
means the following. Recall that for every $r\in \Z$ we have a
complex of sheaves on $X$ denoted $\Fil^{\rm r-\bullet}
\bDcrisgeo(\bL)\otimes_{\cO_X}\Omega_{X/\cO_K}^\bullet$ and given
by
$$
\Fil^r\bDcrisgeo(\bL)\stackrel{\nabla}{\lra}\Fil^{r-1}\bDcrisgeo(\bL)
\otimes_{\cO_X}\Omega_{X/\cO_K}^1\stackrel{\nabla}{\lra}\Fil^{r-2}\bDcrisgeo(\bL)
\otimes_{\cO_X}\Omega_{X/\cO_K}^2\stackrel{\nabla}{\lra}\ldots.
$$We denote by $\bH^i\bigl(X^{\rm et}, \Fil^{\rm r-\bullet}
\bDcrisgeo(\bL)\otimes_{\cO_X}\Omega_{X/\cO_K}^\bullet\bigr)$ the
$i$-th hypercohomology group of the respective complex.\smallskip

{\it The filtrations:}  For $r\in \Z$ the $r$-th filtration of ${\rm H}^i_{\rm
dR}\bigl(X^{\rm et}, \bDcrisgeo(\bL)\bigr)$ is by definition
 the image of
$\bH^i\bigl(X^{\rm et}, \Fil^{\rm r-\bullet}
\bDcrisgeo(\bL)\otimes_{\cO_X}\Omega^\bullet_{X/\cO_K}\bigr)$
while the $r$-the filtration on ${\rm H}^i\left(\fXKbar,
\cL\tensor_{\Z_p} \bB_{\rm cris}^\nabla\right)$ is the image of
$\left\{{\rm H}^i\left(\fXKbar, \cL\tensor_{\Z_p} \Fil^r \bB_{\rm
cris}^\nabla\right)\right\}_r$.
\smallskip

{\it Galois action:} The Galois action on ${\rm H}^i_{\rm
dR}\bigl(X^{\rm et}, \Fil^r\bDcrisgeo(\bL)\bigr)$ is induced by
the Galois action on $\bDcrisgeo(\bL)$ defined
in~\ref{lemma:DcrisarDcrisgeoGK}. The Galois action on ${\rm
H}^i(\fXKbar, \bL\otimes\Fil^r\bB_{\rm cris,\Kbar}^\nabla) $
arises as follows. Since $\beta_{M,\Kbar}^\ast$ is exact and sends
flasque objects to flasque objects by \ref{lemma:betaastG}
and~\cite[Pf. Prop.~4.4(4)]{andreatta_iovita} and since
$\beta_{M,\Kbar}^\ast\bigl(\cL\tensor \Fil^r \bA_{\rm
cris,M}^\nabla\bigr)=\cL\tensor \Fil^r \bA_{\rm
cris,\Kbar}^\nabla$, one can compute  ${\rm H}^i(\fXKbar,
\bL\otimes\Fil^r\bA_{\rm cris,\Kbar}^\nabla)$ taking global
sections of the pull--back via $\beta_{M,\Kbar}^\ast$ of an
injective resolution $\cI^\bullet$ of $\cL\tensor\Fil^r \bA_{\rm
cris,K}^\nabla$. Since
$v_{\Kbar,\ast}\bigl(\beta_{M,\Kbar}^\ast(\cI^\bullet)\bigr)
=v_{M,\ast}\bigl(\beta_{K,\Kbar,\ast}\circ
\beta_{M,\Kbar}^\ast(\cI^\bullet)\bigr)$ and
$\beta_{M,\Kbar,\ast}\circ \beta_{M,\Kbar}^\ast(\cI^\bullet)$ is
endowed with an action of $G_M$, we get the claimed action of
$G_M$.\smallskip

{\it Frobenius:} The Frobenius on  ${\rm H}^i\left(\fXKbar,
\cL\tensor_{\Z_p} \bB_{\rm cris}^\nabla\right)$ is induced by
Frobenius on $\bB_{\rm cris}^\nabla$. Frobenius on ${\rm H}^i_{\rm
dR}\bigl(X^{\rm et}, \bDcrisgeo(\bL)\bigr)$ is constructed as
follows. Fix a covering of $X$ by small affines $\cU_i$ for $i\in
I$, and for each of them choose parameters $T_{i,1},\ldots,
T_{i,d}\in R_{\cU_i}^\ast$. This choice provides a lift of
Frobenius $F_i$ on each $\cU_i$ as the unique $\cO_K$--linear map
sending $T_{i,j}\mapsto T_{i,j}^p$. Fix a total ordering on $I$.
For every non--empty subset $J\subset I$ put $\cU_J:=\prod_{i\in
J}\cU_i$ as formal schemes over $\cO_K$ and $\cU^J:=\cap_{i\in J}
\cU_i$. Note that $\cU^J\subset \cU_J$ is closed in an open
$\cU_J^o\subset \cU_J$. Let $\cU_J^{\rm DP}$ be the $p$--adic
completion of the $\WW(k)$--divided power envelope of $\cU_J^o$
with respect to the ideal defining $\cU^J_k$. Let
$F_J:=\prod_{i\in J} F_i \colon \cU_J \to \cU_J$. It induces a
morphism $F_J\colon \cU_J^{\rm DP} \to \cU_J^{\rm DP}$. Note that
$\cU_J^{\rm DP}$ and $F_J$ define a complex for varying $J$.

Put $\bD(\cL)_J:=\bD(\cL)\vert_{\cU_J^{\rm PD}}$ and let
$\Phi_J\colon \bD(\cL)_J\to \bD(\cL)_J$ be the $F_J$--linear
morphism defined by the non--degenerate Frobenius morphism $\Phi$
on $\bD(\cL)$. For varying $J$ we get  a morphism of double
complexes
$$\Phi_{\bullet,\ast}\colon \bD(\cL)_\bullet\otimes
\Omega^\ast_{\cU_\bullet,\cO_K} \widehat{\otimes} A_{\rm cris} \to
\bD(\cL)_\bullet \otimes \Omega^\ast_{\cU_\bullet,\cO_K}
\widehat{\otimes} A_{\rm cris}.
$$We have natural morphisms
$\cU^J\to \cU_J^{\rm DP}$ of PD thickenings of $\cU^J_k$ which
are compatible for varying $J$.  It
follows from the crystalline Poincar\'e lemma \cite[Th.
6.14]{berthelot_ogus} that the induced morphism
$$\bD(\cL)_J\otimes \Omega^\ast_{\cU_J,\cO_K}\lra \bD(\cL)
\vert_{\cU^J}\otimes \Omega^\ast_{\cU^J,\cO_K} $$is a quasi
isomorphism. Thus the cohomology of $\bD(\cL)_\bullet \otimes
\Omega^\ast_{\cU_\bullet,\cO_K} \widehat{\otimes} B_{\rm cris}$ is
the cohomology of the double complex
$\bD(\cL)\vert_{\cU^\bullet}\otimes \Omega^\ast_{\cU^\bullet,\cO_K}
\widehat{\otimes} B_{\rm cris}$ i.e. of the simple complex
$\bD(\cL) \otimes \Omega^\ast_{X,\cO_K} \widehat{\otimes} B_{\rm
cris}$. Taking cohomology and using this identification, we get from
$\Phi_{\bullet,\ast}$ the Frobenius map
$$\varphi\colon {\rm
H}^i\left(\Sh(X^{\rm et})^\N,\bD(\cL)\otimes
\Omega^\bullet_{X,\cO_K} \widehat{\otimes} B_{\rm cris} \right) \lra
{\rm H}^i\left(\Sh(X^{\rm et})^\N,\bD(\cL)\otimes
\Omega^\bullet_{X,\cO_K} \widehat{\otimes} B_{\rm cris}
\right).$$Since $\bDcrisgeo(\bL) \cong \bD(\cL)\widehat{\otimes}
B_{\rm cris}$  we get the claimed Frobenius on ${\rm H}^i_{\rm
dR}\bigl(X^{\rm et}, \bDcrisgeo(\bL)\bigr)$.

\smallskip

\begin{remark}\label{remark:limHiAcris(m)isHiBcris}
By construction  ${\rm H}^i\bigl(\fXKbar, \bL\otimes \bA_{\rm cris,\Kbar}^\nabla(m)\bigr)\cong {\rm H}^i\bigl(\fXKbar, \bL\otimes \bA_{\rm cris,\Kbar}^\nabla\bigr)$
as $A_{\rm cris}$--modules and for $n\leq m$ using this identification the natural morphism ${\rm H}^i\bigl(\fXKbar, \bL\otimes \bA_{\rm
cris,\Kbar}^\nabla(m)\bigr)\to {\rm H}^i\bigl(\fXKbar, \bL\otimes \bA_{\rm cris,\Kbar}^\nabla(n)\bigr)$ is simply multiplication by $t^{m-n}$ on ${\rm
H}^i\bigl(\fXKbar, \bL\otimes \bA_{\rm cris,\Kbar}^\nabla\bigr)$. Thus $${\rm H}^i\bigl(\fXKbar, \bL\otimes \bB_{\rm cris,\Kbar}^\nabla(m)\bigr)= {\rm
H}^i\bigl(\fXKbar, \bL\otimes \bA_{\rm cris,\Kbar}^\nabla\bigr)\tensor_{A_{\rm cris}}B_{\rm cris}.$$In particular we can replace the expression on the right in the
statements of theorem \ref{thm:formalcomparison}.
\end{remark}

We show first how to calculate explicitly the sheaves ${\rm R}^j
v_{\Kbar,\ast}^{\rm cont} \left(\cL\tensor_{\Z_p} \bB_{\rm
cris,\Kbar} \right)$ for a $\Q_p$--adic \'etale sheaf $\cL$
over~$X_\Kbar^{\rm et}$ according to the conventions of
\S\ref{def:bBcris}.

\smallskip {\it Computation of ${\rm R}^j v_{M,\ast}^{\rm cont}$
via continuous Galois cohomology.}  Consider an inverse system of sheaves $\cF=\{\cF_n\}_n\in \Sh\bigl(\fX_M\bigr)^\N$. Then for every $\cU$  connected and \'etale
affine over $X$ consider the inverse system $\{\cF_n(\Rbar_\cU)\}_n\in \Rep\bigl(\cG_{\cU_M}\bigr)^\N$. Recall that we have defined the localization $\ds
\cF\bigl(\Rbar_\cU\bigr):=\lim_{\infty \leftarrow n}\cF_n\bigl(\Rbar_\cU\bigr)$. Given an abelian category $\cA$ in \cite[\S5.1]{andreatta_iovita} general results
for the category of inverse systems $\cA^\N$ are recalled. In particular if $\cA$ admits enough injectives then also $\cA^\N$ has enough injectives. For example one
can derive the functor associating to an inverse systems of discrete $\cG_{\cU_M}$-modules $G:=\{G_n\}_n$ the abelian group $\bigl(\lim_{\infty \leftarrow n}
G_n\bigr)^{\cG_{\cU_M}}$ and we get a $\delta$-functor ${\rm H}^\ast\left(\cG_{\cU_M},G\right)$. If the system $G$ is Mittag-Leffler this is shown to coincide with
the usual continuous group cohomology. In particular given $\cF=\{\cF_n\}_n\in \Sh\bigl(\fX_M\bigr)^\N$ we can define ${\rm H}_{\rm Gal}^i\bigl(\cF\bigr)$ to be the
sheaf associated to the contravariant functor sending $\cU$, a connected and \'etale affine over $X$ to ${\rm H}^i\left(\cG_{\cU_M},\{\cF_n(\Rbar_\cU)\}_n\right)$.
We also have the sheaf ${\rm R}^n v_{M,\ast}(\cF)$ on $X^{\rm et}$ obtained by deriving the functor $\cF \mapsto \ds \lim_{\infty \leftarrow n}v_{X,M,\ast}\cF_n$.
Then we have.

\begin{lemma}\label{lemma:howtocomputeRivast}
For every $n\in\N$ there is a functorial homomorphism of sheaves
$$f_n(\cF)\colon {\rm H}_{\rm Gal}^n\bigl(\cF\bigr) \lra {\rm R}^n
v_{M,\ast}(\cF).$$
\end{lemma}
\begin{proof}
This is a variant of \cite[Lemma 3.5]{erratum} which is stated and proven for sheaves on $\fX_M$. We provide the main ingredients. Let $\fU_M$ be Faltings's site
associated to $\cU$. We then have a morphism of sites $j_\cU\colon \fX_M\to \fU_M$ sending $(\cV,W)\mapsto (\cV\times_\X \cU, W\times_X \cU)$. It induces a morphism
$j_\cU^\ast\colon \Sh\bigl(\fX_M\bigr)^\N \lra \Sh\bigl(\fU_M\bigr)^\N$ which admits an exact left adjoint $j_{\cU,!}$ given componentwise  by extension by zero. In
particular one deduces from this as in \cite[Lemma 3.5]{erratum} that the sheaf $T^n(\cF)$ associated to the pre-sheaf $\cU \mapsto {\rm
H}^n\bigl(\fU_M,j_\cU^\ast(\cF)\bigr)$ is a universal $\delta$-functor. Since ${\rm H}^0\bigl(\fU_M,j_\cU^\ast(\cF)\bigr)={\rm H}^0\bigl(\cU,
v_{M,\ast}(\cF)\bigr)$, we conclude that $T^n(\cF)\cong {\rm R}^n v_{M,\ast}(\cF) $.

Note that ${\rm H}^n\bigl(\fU_M,\_\bigr)$ is the composite of
the localization functor $\cF \mapsto  \{\cF_n(\Rbar_\cU)\}_n$
with ${\rm H}^0\left(\cG_{\cU_M},\_\right)$. The induced spectral
sequence provides  ${\rm
H}^0\left(\cG_{\cU_M},\{\cF_n(\Rbar_\cU)\}_n\right) \lra {\rm
H}^n\bigl(\fU_M,j_\cU^\ast(\cF)\bigr) $. Composing this with the
morphism to $T^n(\cF)(\cU)\cong {\rm R}^n v_{M,\ast}(\cF)(\cU)$ we
get the claimed map.
\end{proof}

In \cite[Lemma 3.5]{erratum} it is shown that if we work with
sheaves on $\fX_M$, not with continuous sheaves, the above map is
an isomorphism. This is not true in the general context of
continuous sheaves. Assume that $M=\Kbar$ then we
have the following:

\begin{proposition}\label{prop:whenRvconisgaloiscoho} For every $n\in\N$
the morphism $f_n(\cF)$ has kernel and cokernel annihilated by
any element of $\II^{2i}$ where $\II$ is the ideal  introduced in
\S \ref{sec:Notation}  in the following cases:\enspace (a)
$\cF:=\{\WW_{n,\Kbar}\}_n$;\enspace (b) $\cF:=\{\cO_{\fX_M}/p^n
\cO_{\fX_M}\}_n$;\enspace (c) $\cF=\bA_{\rm cris,\Kbar}(m)$ for
every $m\in\Z$; \enspace (d) $\cF={\rm Gr}^r\bA_{\rm
cris,\Kbar}(m)$ for every $m$ and $r\in \Z$.
\end{proposition}
\begin{proof} The proposition follows from
\cite[Thm.~6.12]{andreatta_iovita} after minor changes if the assumptions of loc.~cit.~are satisfied. The statement in loc.~cit.~provides an isomorphism after
inverting $[\varepsilon]-1$ and working with the pointed site $\fX_\Kbar^{\bullet}$. An inspection of the proof gives our claim. First of all the proof works for
$\fX_\Kbar$ and not only for $\fX_\Kbar^{\bullet}$ using \cite[Lemma 3.5]{erratum} which is the analogue of \cite[Prop~4.4]{andreatta_iovita} for $\fX_\Kbar$
instead of $\fX_\Kbar^\bullet$. Secondly the two cohomology groups are related by two spectral sequences, one in \cite[Formula (19)]{andreatta_iovita} and the other
in \cite[Prop.~6.15]{andreatta_iovita} and the proofs of \cite[lemma 6.13 \& Prop.~6.15]{andreatta_iovita} show that each degenerates if we multiply by any element
of $\II^i$ and not only after inverting $[\varepsilon]-1$.

We now verify that the assumptions hold. Assumptions (i)
and~(ii)  state the existence of enough small affines of~$X$. The
fact that the other Assumptions hold for $\{\WW_{n,\Kbar}\}_n$ and
for  $\cF:=\{\cO_{\fX_M}/p^n \cO_{\fX_M}\}_n$ is precisely the
content of \cite[Thm.~6.16(A)\&(B)]{andreatta_iovita}. We pass to
$\bA_{\rm cris,\Kbar}(m)$.  Assumptions (iii)--(vi) concern the
behavior of the sheaves $\cF_n$ restricted to the subsite
$\fU_{\Kbar,n}$ of $\fX_\Kbar$ for $n\gg 0$ for every small $\cU$
of $X^{\rm et}$; see \ref{sec:sheafAcris} for the definition of
$\fU_{\Kbar,n}$. Since  $\bA_{\rm
cris,n,\Kbar}'(m)\vert_{\fU_{\Kbar,n}} \cong \bA_{\rm
cris,n,\Kbar}^{'\nabla}(m)\vert_{\fU_{\Kbar,n}}\langle
X_1,\ldots,X_d\rangle$ by lemma \ref{lemma:Acrismodpn} this is a sheaf
of free $\bA_{\rm
cris,n,\Kbar}^{'\nabla}(m)\vert_{\fU_{\Kbar,n}}$--modules. On the
other hand  $\bA_{\rm cris,n,\Kbar}^{'\nabla}(m)\cong A_{\rm
cris}/p^n A_{\rm cris}\tensor \WW_{n,\Kbar}$ by
lemma \ref{lemma:exactsequenceAcrisnnabla'}. We conclude since
assumptions (iii)--(vi) hold for the continuous sheaf
$\{\WW_{n,\Kbar}\}_n$.

The statement concerning ${\rm Gr}^r\bA_{\rm cris,\Kbar}(m)$ is
proven similarly. We have
$$\Fil^{r}\bA_{\rm cris,n,\Kbar}'(m)\vert_{\fU_{\Kbar,n}} \cong
\sum_{s_0+\ldots+s_d\geq r-m} \Fil^{s_0} \bigl(A_{\rm cris}/p^n A_{\rm cris}\bigr)\tensor \WW_{n,\Kbar}\vert_{\fU_{\Kbar,n}} X_1^{[s_1]}\cdots X_d^{[s_d]}$$with
$X_j:=1\tensor T_j - \tT_j\tensor 1$. In particular $${\rm Gr}^{r} \bA_{\rm cris,n,\Kbar}'\vert_{\fU_{\Kbar,n}} \cong \oplus_{s_0+\ldots+s_d=r-m}
\bigl(\cO_{\fX_M}/p^n \cO_{\fX_M}\bigr)\vert_{\fU_{\Kbar,n}} \xi^{[s_0]} \cdot X_1^{[s_1]}\cdots X_d^{[s_d]}$$is a free  $\cO_{\fX_M}/p^n
\cO_{\fX_M}\vert_{\fU_{\Kbar,n}}$-module. Since the Assumptions of loc~cit.~for the inverse system   $\{\cO_{\fX_M}/p^n \cO_{\fX_M}\}_n$ are satisfied we are done.
\end{proof}

Denote by $\cH^i(m)$ (resp.~$\cH^i(\Fil^r,m)$, resp.~$\cH^i({\rm Gr}^r,m)$) the sheaf associated to the contravariant functor which associates to a small affine
open $\cU={\rm Spf}(R_\cU)$ of $X^{\rm et}$ the value
$${\rm H}^i\bigl(\cG_{\cU,\Kbar}, A_{\rm
cris,\Kbar}(m)(\Rbar_\cU)\bigr),\mbox{ resp. } {\rm
H}^i\bigl(\cG_{\cU,\Kbar}, {\rm Fil}^r A_{\rm
cris,\Kbar}(m)(\Rbar_\cU)\bigr), \mbox{ resp. }{\rm
H}^i\bigl(\cG_{\cU,\Kbar}, {\rm Gr}^r A_{\rm
cris,\Kbar}(m)(\Rbar_\cU)\bigr).
$$
The cohomology
considered is the continuous Galois cohomology. Then we have.

\begin{lemma}\label{lemma:howtocompute}
For every $r\in\N$ and $m\in \Z$ we have morphisms $\cH^i(m)  \lra
{\rm H}_{\rm Gal}^i\left(\bA_{\rm cris,\Kbar}(m)\right)$,
$\cH^i({\rm Gr}^r,m) \lra {\rm H}_{\rm Gal}^i\left({\rm
Gr}^r\bA_{\rm cris,\Kbar}(m)\right)$  and $\cH^i({\rm Fil}^r,m)
\lra {\rm H}_{\rm Gal}^i\left(\Fil^r\bA_{\rm
cris,\Kbar}(m)\right)$ which are compatible with the maps induced
by the projection $\Fil^r \to {\rm Gr}^r$  and have  kernel and
cokernel annihilated by  $\II$.
\end{lemma}
\begin{proof} It follows from lemma \ref{lemma:acrisnabla} and
lemma \ref{lemma:localizeAcris}
that the inverse system of  $\cG_{\cU,M}$-modules  $ A_{\rm
cris,n,\Kbar}'(m)(\Rbar_\cU)$ is contained in  the localization
$\bA_{\rm cris,n,\Kbar}'(m)\bigr)(\Rbar_\cU)$  for every small
affine~$\cU$ of~$X$ with cokernel annihilated by any element of
$\II$. This provides the first morphism and the claim regarding
its kernel and cokernel. Similarly using the description of ${\rm
Gr}^r\bA_{\rm cris,\Kbar}(m)$ given in the proof of
\ref{prop:whenRvconisgaloiscoho}, we deduce that ${\rm Gr}^{r}
A_{\rm cris,n,\Kbar}(\Rbar_\cU)$ is contained in ${\rm Gr}^{r}
\bA_{\rm cris,n,\Kbar}(\Rbar_\cU)$ with cokernel annihilated by
$\II$. This provides the second map and the subsequent statement.
Using induction on $r$ and the fact that ${\rm Gr}^{r}={\rm
Fil}^{r}/{\rm Fil}^{r+1}$ we deduce that also ${\rm Fil}^{r}
A_{\rm cris,n,\Kbar}(\Rbar_\cU)$ is contained in $\Fil^{r}
\bA_{\rm cris,n,\Kbar}(\Rbar_\cU)$ with cokernel annihilated by
$\II$.  This gives the last morphism and proves the assertion
concerning its kernel and cokernel.
\end{proof}

\begin{corollary}\label{cor:crysisacyclic}
Let $\cM$ be a coherent  $\cO_{X_K}$--module such that for every
small affine $\cU$ the $R_\cU\tensor_{\cO_K} K$-module
$\cM(\cU)\tensor_{\cO_K} K$ is projective. We view $\cM$ as a
sheaf on $\fXKbar$ via $v_{\Kbar}^{\ast}$. For every $r\in
\Z\cup\{-\infty\}$ we have:

$$
{\rm R}^j v_{\Kbar,\ast}^{\rm cont} \left(\Fil^r \bB_{\rm
cris,\Kbar}\tensor_{\cO_X} \cM \right)=\begin{cases} 0 &
\hbox{{\rm if }} j\geq 1 \cr  \Fil^r B_{\rm cris}
\widehat{\tensor} \cM & \hbox{{\rm if }} j=0 \cr
\end{cases}$$
\end{corollary}
\begin{proof} Note the we have a natural map $\Fil^r B_{\rm cris}
\widehat{\tensor} \cM \lra v_{\Kbar,\ast}^{\rm cont} \left(\Fil^r
\bB_{\rm cris,\Kbar}\tensor_{\cO_{\fX_M}^{\rm un}} \cM \right)$. Thus both
statements are local on $X$. We may then assume that $X=\cU$ is a
small affine so that $\cM\tensor_{\cO_K} K$ is a direct summand in
a free $\cO_{X_K}$-module. Since ${\rm R}^j v_{\Kbar,\ast}^{\rm
cont}$ commutes with direct sums  we may assume that
$\cM\tensor_{\cO_K} K$ is a free module and we are reduced to
prove the corollary in the case that $\cM=\cO_X$. Fix integers $m$
and $r\in\Z$ with $m\leq N$. We start by considering the following
statements:

\begin{enumerate}
\item[$\alpha$)] there exists $a\in\N$ depending on~$r-m$ such that ${\rm
R}^j v_{\Kbar,\ast}^{\rm cont} \left(\Fil^r \bA_{\rm
cris,\Kbar}(m)\right)$ is annihilated by $t^a$  if~$j\geq 1$.

\item[$\beta$)] there exists $c\in\N$, depending on $a$,   such that the
kernel of ${\rm R}^j v_{\Kbar,\ast}^{\rm cont} \left(\Fil^r \bA_{\rm
cris,\Kbar}(m-a)\right) \lra {\rm R}^j v_{\Kbar,\ast}^{\rm cont}
\left(\Fil^{r-a} \bA_{\rm cris,\Kbar}(m-a) \right)$ is annihilated
by $p^c$ for every $j\geq 1$.

\item[$\gamma$)] the map  $\cM \widehat{\tensor} \Fil^r A_{\rm
cris}\to v_{\Kbar,\ast}^{\rm cont} \left(\Fil^r \bA_{\rm
cris,\Kbar} \tensor_{\cO_{\fX_M}^{\rm un}} \cM\right)$ is an isomorphism for all $r$.

\end{enumerate}

First of all we remark that these statements imply the corollary in
the case that $\cM=\cO_X$. Indeed together with
\ref{lemma:isoinBcris} these claims imply all the statements of the
corollary except for the vanishing of ${\rm R}^j
v_{\Kbar,\ast}^{\rm cont} \left(\cL\tensor_{\Z_p} \Fil^r \bB_{\rm
cris,\Kbar}\tensor_{\cO_X} \cM \right)$ for $r\in \Z$ and $j\geq 1$.
Note that the image of the continuous sheaf $\Fil^r \bA_{\rm
cris,\Kbar}(m)$ in $\Fil^{r-a} \bA_{\rm cris,\Kbar}(m-a)$ is $t^a
\cdot \Fil^{r-a} \bA_{\rm cris,\Kbar}(m-a)$. It follows from
($\alpha$) that the map
$${\rm R}^j v_{\Kbar,\ast} \left(
\Fil^r \bA_{\rm cris,\Kbar}(m)\right)\to {\rm R}^j
v_{\Kbar,\ast} \left(\Fil^{r} \bA_{\rm
cris,\Kbar}(m-a)\right)$$factors via the kernel of the map $${\rm
R}^j v_{\Kbar,\ast} \left(\Fil^{r} \bA_{\rm
cris,\Kbar}(m)\right) \lra {\rm R}^j v_{\Kbar,\ast}
\left(\Fil^{r-a} \bA_{\rm cris,\Kbar}(m-a)\right)$$which is
annihilated by $p^c$ by ($\beta$). Since multiplication by $p$ is an
isomorphism on $\Fil^r \bB_{\rm cris,M}$ by \ref{lemma:propbBcris},
also the vanishing of ${\rm R}^j v_{\Kbar,\ast}^{\rm cont}
\left(\Fil^r \bB_{\rm cris,\Kbar}\cM \right)$ for $r\in \Z$ and
$j\geq 1$ follows.\smallskip

Now we start proving the statements $\alpha), \beta), \gamma).$ In view of \ref{prop:filtAcrisnabla} to prove statement~($\gamma$) we need to prove that for every
small affine the map $R_\cU \widehat{\tensor} \Fil^r A_{\rm cris}\to \Fil^r A_{\rm cris,\Kbar}(\Rbar_\cU)^{\cG_{\cU,\Kbar}}$ is an isomorphism. This follows
from~\cite[Prop.~41]{andreatta_brinonacyclicity}.
\smallskip

Recall that $ \Fil^r \bA_{\rm cris,\Kbar}(m)\cong \Fil^{r-m} \bA_{\rm cris,\Kbar}$ as continuous sheaves on~$\fXKbar$. Thus, we may also assume that~$m=0$.
Given~$r\in\N$, since $t^r \in \Fil^r A_{\rm cris}$, the cokernel of the inclusion $\Fil^r \bA_{\rm cris,\Kbar} \subset \bA_{\rm cris,\Kbar}$ is annihilated by
$t^r$. Hence it suffices to prove~($\alpha$) for~$r=0$. Recall from \ref{sec:Notation} that $t\in\II$ since  $t=(1-[\varepsilon])u$ with $u$ a unit in $A_{\rm
cris}$ by \cite[\S5.2.4\&\S5.2.8(ii)]{Fontaineperiodes}. Then claim ($\alpha$) follows from proposition \ref{prop:whenRvconisgaloiscoho}, lemma
\ref{lemma:howtocompute} and the fact that ${\rm H}^i\bigl(\cG_{\cU,\Kbar}, A_{\rm cris,\Kbar}(\Rbar_\cU)(m)\bigr)$ is annihilated
by~$(1-[\varepsilon])^{2(d+1)}\II^2$ proven in corollary \cite[Cor.~24]{andreatta_brinonacyclicity}. Here $d$ is the relative dimension of $X$ over $\cO_K$.

We are left to show ($\beta$). Proceeding inductively on $a$ it
suffices to show that for every $r\in\Z$ and every $n\in\N$ the
cokernel of the map ${\rm R}^{j-1} v_{\Kbar,\ast}^{\rm cont}
\left(\Fil^{r} \bA_{\rm cris,\Kbar}'\right) \to {\rm R}^{j-1}
v_{\Kbar,\ast}^{\rm cont} \left({\rm Gr}^{r} \bA_{\rm
cris,\Kbar}'\right)$ is annihilated by a power of $p$
(independent of $n$).  Consider the commutative diagram

$$\begin{array}{ccc} \cH^{j-1}({\rm Fil}^r,m) & \lra & \cH^{j-1}({\rm Gr}^r,m)
\cr\big\downarrow & & \big\downarrow\\
{\rm R}^{j-1} v_{\Kbar,\ast}^{\rm cont}
\left(\Fil^{r} \bA_{\rm cris,\Kbar}'\right) & \lra &{\rm R}^{j-1}
v_{\Kbar,\ast}^{\rm cont} \left({\rm Gr}^{r} \bA_{\rm
cris,\Kbar}'\right);
\end{array}
$$obtained from \ref{prop:whenRvconisgaloiscoho} and
lemma \ref{lemma:howtocompute}. For every small affine $\cU$, the map
$${\rm H}^{j-1}\bigl(\cG_{\cU,\Kbar}, \Fil^{r}A_{\rm cris,\Kbar}(m)(\Rbar_\cU)\bigr)\lra {\rm H}^{j-1}\bigl(\cG_{\cU,\Kbar},
{\rm Gr}^{r}A_{\rm cris,\Kbar}(m)(\Rbar_\cU)\bigr)$$ has cokernel annihilated  by a power of $p$ by \cite[Pf. Lemme 36]{andreatta_brinonacyclicity}. This also
applies to the associated sheaves i.e., to the map $\cH^{j-1}({\rm Fil}^r,m) \lra \cH^{j-1}({\rm Gr}^r,m) $. The right vertical morphism in the diagram  has kernel
and cokernel annihilated by $\II^{2(j-1)+1}$ by proposition \ref{prop:whenRvconisgaloiscoho} and lemma \ref{lemma:howtocompute}. We conclude that the same applies
to the cokernel of the lower horizontal arrow. The conclusion follows.
\end{proof}

\begin{proof} ({\it of theorem \ref{thm:formalcomparison}})
Thanks to lemma \ref{lemma:propbBcris}, if $\cL$ is a $p$-adic sheaf
sheaf, the sequence
\begin{align*}{(\ast)}\ 0 \lra \cL\tensor_{\Z_p} \Fil^r\bB_{\rm
cris,\Kbar}^\nabla  \lra \cL\tensor_{\Z_p} \Fil^r  \bB_{\rm
cris,\Kbar} &  \stackrel{\nabla}{\lra} \cL\tensor_{\Z_p}
\Fil^{r-1}\bB_{\rm cris,\Kbar}\otimes_{\cO_X} \Omega^1_{X/\cO_K}
\stackrel{\nabla}{\lra}  \cdots \cr & \cdots
\stackrel{\nabla}{\lra} \cL \tensor_{\Z_p} \Fil^{r-d}\bB_{\rm
cris,\Kbar} \otimes_{\cO_X} \Omega^d_{X/\cO_K}\lra 0\cr
\end{align*} is exact for every $r\in\Z$. Due to \ref{prop:connFilDcriar}
the complex $\cL\tensor_{\Z_p} \Fil^{r-\bullet}\bB_{\rm cris,\Kbar}\otimes_{\cO_X} \Omega^\bullet_{X/\cO_K}$ is isomorphic to the complex
$\Fil^{r-\bullet}\bigl(\bDcrisgeo(\cL)\otimes_{\cO_X\otimes_{\cO_K} A_{\rm cris}} \bB_{\rm cris,\Kbar} \otimes_{\cO_X} \Omega^\bullet_{X/\cO_K}\bigr)$. This
provides an isomorphism

$$
{\rm H}^i\bigl(\fXKbar, \bL\otimes{\rm Fil}^r \bB_{\rm cris,\Kbar}^\nabla\bigr)  \stackrel{\sim}{\lra}  {\mathbb H}^i\left(\fXKbar,
\Fil^{r-\bullet}\bigl(\bDcrisgeo(\cL)\otimes_{\cO_X\otimes_{\cO_K} A_{\rm cris}} \bB_{\rm cris,\Kbar}\otimes_{\cO_X} \Omega^\bullet_{X/\cO_K}\bigr)\right).$$Here we
view $\Fil^{r-\bullet}\bigl(\bDcrisgeo(\cL)\otimes_{\cO_X\otimes_{\cO_K} A_{\rm cris}} \bB_{\rm cris,\Kbar}\otimes_{\cO_X} \Omega^\bullet_{X/\cO_K}\bigr)$ as the
inductive system of complexes of continuous sheaves $\Fil^{r-\bullet}\bigl(\bDcrisgeo(\cL)\otimes_{\cO_X\otimes_{\cO_K} A_{\rm cris}} \bA_{\rm cris,\Kbar}(m)
\otimes_{\cO_X} \Omega^\bullet_{X/\cO_K}\bigr)$ and we apply the construction of definition \ref{def:bBcris} to the continuous hyper-cohomology ${\mathbb
H}^\ast\bigl(\fXKbar, \_\bigr)$ of these complexes. It follows from corollary \ref{cor:crysisacyclic} that
$\Fil^{r-\bullet}\bigl(\bDcrisgeo(\cL)\otimes_{\cO_X\otimes_{\cO_K} A_{\rm cris}} \bB_{\rm cris,\Kbar}\otimes_{\cO_X} \Omega^\bullet_{X/\cO_K}\bigr)$ is acyclic for
$v_{\Kbar,\ast}$ and that its image via $v_{\Kbar,\ast}$  is $\Fil^{r-\bullet}\bigl(\bDcrisgeo(\cL)\otimes_{\cO_X} \Omega^\bullet_{X/\cO_K}\bigr)$.  Thus, we have
an isomorphism $$ \bH^i\left(X^{\rm et}, \Fil^{\rm r-\bullet}\bDcrisgeo(\cL)\otimes_{\cO_X} \Omega^\bullet_{X/\cO_K}\right)\stackrel{\sim}{\lra}{\mathbb
H}^i\left(\fXKbar, \Fil^{r-\bullet}\bigl(\bDcrisgeo(\cL)\otimes_{\cO_X\otimes_{\cO_K} A_{\rm cris}} \bB_{\rm cris,\Kbar}\otimes_{\cO_X}
\Omega^\bullet_{X/\cO_K}\bigr)\right).$$

\smallskip

{\it Compatibility with $G_M$--action:} These are isomorphisms of $G_M$--modules with $G_M$--structure given as explained after the theorem resolving
$\cL\tensor\Fil^r \bA_{\rm cris,\Kbar}^\nabla(m)$ (resp.~$\cL\tensor_{\Z_p} \Fil^{r-\bullet}\bA_{\rm cris,\Kbar}(m) \otimes_{\cO_X} \Omega^\bullet_{X/\cO_K}$) with
the pull--back via $\beta_{M,\Kbar}^\ast$ of an injective resolution of $\cL\tensor\Fil^r \bA_{\rm cris,M}^\nabla(m)$ (resp.~of the complex $\cL\tensor_{\Z_p}
\Fil^{r-\bullet}\bA_{\rm cris,M}(m) \otimes_{\cO_X} \Omega^\bullet_{X/\cO_K}$).\smallskip

{\it Compatibility with Frobenius:}  Fix a covering of $X$ by small
affines $\cU_i$, for $i\in I$ and for each of them choose parameters
$T_{i,1},\ldots, T_{i,d}\in R_{\cU_i}^\ast$. For every subset
$J\subset I$ let $\cU^J\subset \cU_J^o \subset \cU_J$ be as in the
notation introduced after theorem \ref{thm:formalcomparison}. Let
$\fU_{\Kbar}^J$ be Faltings' site associated to $\cU^J$ and consider
the continuous morphism $j_J\colon \fX_{\Kbar}\to \fU_{\Kbar}^J$
sending $(\cV,W)$ to $(\cV,W)\times_{(X,X_K)} (\cU^J,\cU^J_K)$. The
inverse image  $j_J^\ast$ of a sheaf on $\fX_\Kbar$ is the restriction of
$\cF$ to $\fU_\Kbar^J$ viewed as a subcategory of $\fX_\Kbar$.

Define~$\bA_{{\rm cris},\cU_J,n}^\nabla$ to be the $\WW(k)$--DP
sheaf of algebras in $\Sh(\fU^J_{\Kbar})$ of $\bA_{{\rm
inf},\cU_J,n}$ with respect to the kernel of
$$\WW_n\left(\cO_{\fU^J_{\Kbar}}/p\cO_{\fU^J_{\Kbar}}\right) \lra
\cO_{\fU^J_{\Kbar}}/p^n\cO_{\fU^J_{\Kbar}} $$defined by
$\theta_M$. Its existence  is proven as  in
\S \ref{sec:sheafAcrisnabla}. Since
$\cO_{\fU^J_{\Kbar}}=\cO_{\fX_{\Kbar}}\vert_{\fU^J_{\Kbar}} $
we have a natural morphism  $$\bA_{{\rm cris},\cU_J,n}^\nabla\lra
j_J^\ast\bigl(\bA_{{\rm cris},X,n}^\nabla\bigr)$$and it follows
from loc.~cit.~that such a morphism is an isomorphism.
Define~$\bA_{{\rm cris},\cU_J,n}$ as the $\WW(k)$--DP sheaf of
$\cO_{\cU_J^o}$--algebras of
$\WW_n\left(\cO_{\fU^J_{\Kbar}}/p\cO_{\fU^J_{\Kbar}}\right)
\otimes_{\WW_n(\kbar)}
v_{\cU^J_{\Kbar}}^\ast\bigl(\cO_{\cU_J^o}\bigr)$ with respect to
the kernel of the morphism of $ \cO_{\cU_J^o}$--algebras
$$\bA_{{\rm inf},J,n}^+\otimes_{\WW_n(\kbar)} v_{\cU^J_{\Kbar}}^\ast\bigl(\cO_{\cU_J^o}\bigr)\lra
\cO_{\fU^J_{\Kbar}}/p^n\cO_{\fU^J_{\Kbar}}$$defined by the
$\cO_{\cU_J^o}$-linear extension of $\theta_M$.

\begin{lemma}\label{lemma:AcrisUJ} The sheaf $\bA_{{\rm cris},\cU_J,n}$ exists and
$$\bA_{{\rm cris},\cU_J,n}\vert_{\fU^J_n} \cong  \bA_{{\rm cris},\cU^J,n}^\nabla\vert_{{\fU^J_n}} \left\langle
X_{i,1},\ldots, X_{i,d}\right\rangle_{i\in J},$$where $X_{i,j}:=1\otimes T_{i,j}- \widetilde{T}_{i,j}\otimes 1$ for every $i\in I$ and every $1\leq j\leq d$. For
every $h\in J$ one also has an isomorphism
$$ \bA_{{\rm cris},\cU_J,n}\cong j_J^\ast\bigl(\bA_{{\rm
cris},X,n}\bigr) \left\langle Y_{i,1}^{(h)},\ldots, Y_{i,d}^{(h)}\right\rangle_{i\in J,i\neq h} $$where $Y_{i,j}^{(h)}$, for $i\in J$ with $i\neq h$ and for $1\leq
j\leq d$ are regular elements generating the ideal defining the closed immersion $\cU^J\subset \bigl(\cU_h \times_{\cO_K} \cU_i\bigr)^o$. In particular $\bA_{{\rm
cris},\cU_J,n}$ is a free $j_J^\ast\bigl(\bA_{{\rm cris},X,n}\bigr)$-module.
\end{lemma}
\begin{proof} The existence of the sheaf and the formula for its restriction to $\fU^J_n$ is
proven as in \S \ref{sec:sheafAcris}. A similar argument implies the last formula over $\fU^J_n$. A descent argument allows to conclude that the formula holds also
over $\fU^J$. The last statement follows from the properties of divided powers: a basis is given by monomials in the elements $\gamma_j\bigl(Y_{i,\ell}^{(h)}\bigr)$
for $j\in \N$, $1\leq \ell\leq d$ and $i\in J$ but $i\neq h$ taking $\gamma_j$ to be as in \S\ref{sec:Notation}; see \S \ref{sec:sheafAcris} for details.
\end{proof}

Let $\bA_{{\rm cris},J,n}^\nabla$ (resp.~$\bA_{{\rm cris},J,n}$) be $j_{J,\ast} \left(\bA_{{\rm cris},\cU^J,n}^\nabla\right)$ (resp.~$j_{J,\ast} \left(\bA_{{\rm
cris},\cU^J,n}\right)$). Define $\bA_{{\rm cris},J}^\nabla $ (resp.~$\bA_{{\rm cris},J}$) as the system $\left\{\bA_{{\rm cris,J},n}^\nabla\right\}_n$
(resp.~$\left\{\bA_{{\rm cris,J},n}\right\}_n$) and $\bB_{{\rm cris},J}^\nabla $ (resp.~$\bB_{{\rm cris},J}$) for the inductive systems given by multiplication by
$t$. We also write $\bB_{{\rm cris},J}\otimes_{\cO_{\cU_J}} \Omega^\ast_{\cU_J/\cO_K}$ for the inductive system, with respect to multiplication by $t$, associated
to the push-forward via $j_{J,\ast}$ of $\bA_{{\rm cris},\cU_J}\otimes_{\cO_{\fU^J}^{\rm un}} v_{\cU^J,\Kbar}^\ast\bigl(\Omega^\ast_{\cU_J/\cO_K}\bigr)$. We then
get  a long exact sequence of continuous sheaves $$0 \lra \bB_{{\rm cris},J}^\nabla \lra  \bB_{{\rm cris},J}\otimes_{\cO_{\cU_J}} \Omega^\ast_{\cU_J/\cO_K} \lra
0.$$The exactness is proven as in proposition \ref{prop:deRhamcomplex} using the first description given in lemma  \ref{lemma:AcrisUJ}. These complexes are
compatible if we vary $J$. In particular  we get a double complex $ \bB_{{\rm cris},\bullet}\otimes_{\cO_{\cU_\bullet}} \Omega^\ast_{\cU_\bullet/\cO_K}$ which is
equivalent to the simple complex $\bB_{{\rm cris},\bullet}^\nabla$. Since the $\cU_i$'s cover $X$ the sequence
$$0 \lra \cO_{\fXKbar}/p\cO_{\fXKbar} \lra
j_{\bullet,\ast}\left(\cO_{\fU_\Kbar^\bullet}/p\cO_{\fU_\Kbar^\bullet}\right) \lra 0$$is exact. Using \ref{lemma:exactsequenceAcrisnnabla'} we deduce that the
sequence
$$0\lra  \bB_{{\rm cris,\Kbar}}^\nabla \lra \bB_{{\rm
cris},\bullet}^\nabla\lra 0$$is also exact. Consider the commutative diagram
$$\begin{array}{ccccc}\label{diagram:chechAcris}
0 & \lra & \cL\tensor_{\Z_p}\bB_{{\rm cris,\Kbar}}^\nabla & \lra & \cL\tensor_{\Z_p}\bB_{{\rm cris,\Kbar}}\otimes_{\cO_{X}} \Omega^\ast_{X/\cO_K}\cr & &
\big\downarrow  & & \big\downarrow \cr 0 & \lra &  \cL\tensor_{\Z_p}\bB_{{\rm cris},\bullet}^\nabla & \lra & \cL\tensor_{\Z_p}\bB_{{\rm
cris},\bullet}\otimes_{\cO_{\cU_\bullet}} \Omega^\ast_{\cU_\bullet/\cO_K} \cr
\end{array}.$$Since the rows are exact and the first column is exact
it follows that the complex $\cL\tensor_{\Z_p}\bB_{{\rm cris,\Kbar}}\otimes_{\cO_{X}} \Omega^\ast_{X/\cO_K}$ is quasi--isomorphic to the double complex $
\cL\tensor_{\Z_p}\bB_{{\rm cris},\bullet}\otimes_{\cO_{\cU_\bullet}} \Omega^\ast_{\cU_\bullet/\cO_K}$. On the latter the Frobenius maps $F_J$ on $\cU_J$ and
Frobenius on $\bA_{{\rm inf},J}^+$ define a morphism of complexes $$\Phi_{\bullet,\ast}\colon \cL\tensor_{\Z_p}\bB_{{\rm cris},\bullet}\otimes_{\cO_{\cU_\bullet}}
\Omega^\ast_{\cU_\bullet/\cO_K}\lra \cL\tensor_{\Z_p}\bB_{{\rm cris},\bullet}\otimes_{\cO_{\cU_\bullet}} \Omega^\ast_{\cU_\bullet/\cO_K}$$compatible with Frobenius
on $\cL\tensor_{\Z_p}\bB_{{\rm cris,\Kbar}}^\nabla$. Let $\bD(\cL)$ be the Frobenius crystal associated to $\bDcrisar(\cL)$ (up to isogeny); see the notation
following theorem \ref{thm:formalcomparison}. By definition of crystal the $\cO_{\cU_J^{\rm PD}}$-module $\bD(\cL)_J:=\bD(\cL)\vert_{\cU_J^{\rm PD}}$ together with
Frobenius and connection, coincides with the pull--back of $\bDcrisar(\cL) $ via the $h$-th projection $\pi_h\colon \cU^J\to \cU_h$. We have

\begin{lemma}  The $\cO_{\cU_J^{\rm PD}}\widehat{\otimes} B_{\rm cris}$-module
$\bD(\cL)_J\widehat{\otimes} B_{\rm cris}$ together with Frobenius and connection, coincides with $v_{X,\Kbar,\ast}^{\rm cont}\left(\cL\tensor_{\Z_p}\bB_{{\rm
cris},J}\right)$ with the Frobenius and connection induced by those on $\bB_{{\rm cris},J}$.
\end{lemma}
\begin{proof} By definition of crystal $\bD(\cL)_J$ coincides with the
pull--back of $\bDcrisar(\cL) $ via the $h$-th projection $\pi_h\colon \cU^J\to \cU_h$. In particular $\cO_{\cU_J^{\rm PD}}\cong
\pi_h^\ast\bigl(\cO_{\cU_h}\bigr)\left\langle Y_{i,1}^{(h)},\ldots, Y_{i,d}^{(h)}\right\rangle_{i\in J,i\neq h}$ compatibly with Frobenius and connection; see lemma
\ref{lemma:AcrisUJ} for the notation. Due to the second description in \ref{lemma:AcrisUJ} we get that $v_{X,\Kbar,\ast}^{\rm cont}\left(\bB_{{\rm
cris},J}\right)\cong \cO_{\cU_J^{\rm PD}}\widehat{\otimes} B_{\rm cris}$ compatibly with Frobenius and connection. Since $\cL\tensor_{\Z_p}\bB_{{\rm cris},J}\cong
\cL\tensor_{\Z_p} j_J^\ast\bigl(\bB_{{\rm cris},X}\bigr) \tensor_{j_J^\ast\bigl(\bB_{{\rm cris},X}\bigr)}\otimes \bB_{{\rm cris},J}$ by lemma \ref{lemma:AcrisUJ}
and $\cL$ is crystalline, we conclude that $v_{X,\Kbar,\ast}^{\rm cont}$ of the former coincides with
$\bDcrisgeo\bigl(\cL\bigr)\widehat{\otimes}_{\pi_h^\ast\bigl(\cO_{\cU_h}\bigr)} \cO_{\cU_J^{\rm PD}} $ compatibly with Frobenius and connection. This coincides with
$\bD(\cL)_J\widehat{\otimes} B_{\rm cris}$ and the claim follows.
\end{proof}

By adjunction we get a morphism of complexes
$$v_{X,\Kbar}^\ast\left(\bD(\cL)_\bullet\otimes
\Omega^\ast_{\cU_\bullet/\cO_K} \widehat{\otimes} B_{\rm cris}\right) \lra \cL\tensor_{\Z_p}\bB_{{\rm cris},\bullet}\otimes_{\cO_{\cU_\bullet}}
\Omega^\ast_{\cU_\bullet/\cO_K}$$which is compatible with the Frobenius morphisms defined on the two complexes. We deduce that the morphisms

$$\begin{array}{ccc}{\rm
H}^i\left(\fXKbar,\cL\tensor_{\Z_p}\bB_{{\rm cris,\Kbar}}^\nabla\right) & \stackrel{\sim}{\lra} & {\mathbb H}\left(\fXKbar, \cL\tensor_{\Z_p}\bB_{{\rm
cris},\bullet}\otimes_{\cO_{\cU_\bullet}} \Omega^{\ast}_{\cU_\bullet/\cO_K}\right) \cr & & \big\uparrow\wr \cr {\rm H}^i_{\rm dR}\bigl(X^{\rm et},
\bDcrisar(\cL)\otimes_{\cO_X} \Omega^{\ast}_{X/\cO_K} \widehat{\otimes} B_{\rm cris}\bigr)& \stackrel{\sim}{\lra} & {\rm H}^i_{\rm dR}\bigl(X^{\rm et},
\bD(\cL)_\bullet\otimes_{\cO_{\cU_\bullet}} \Omega^{\ast}_{\cU_\bullet/\cO_K} \widehat{\otimes} B_{\rm cris}\bigr) \cr
\end{array}$$
are compatible with the Frobenius morphisms defined on each
cohomology group. By construction the group ${\rm H}^i_{\rm
dR}\bigl(X^{\rm et}, \bDcrisgeo(\bL)\bigr)$ is ${\rm H}^i_{\rm
dR}\bigl(X^{\rm et}, \bDcrisar(\cL)\otimes \Omega^\ast_{X,\cO_K}
\widehat{\otimes} B_{\rm cris}\bigr)$. The compatibility with
Frobenius follows.
\end{proof}

\subsection{The comparison isomorphism in the proper case}\label{sec:propercase}
Let us now assume that our formal scheme $X\lra \Spf(\cO_K)$ is the formal completion along the special fiber of a proper and smooth scheme $X^{\rm alg}\lra
\Spec(\cO_K)$. We have a morphism of sites $\mu_M\colon X_M^{\rm alg,et} \to X_M^{\rm et}$ associating~$U\mapsto \widehat{U}_K$ where $\widehat{U}$ is the $p$-adic
completion of $U$. Given a sheaf $\cL\in \Sh(X_M^{\rm alg,et})$ we write $\cL$ for $\mu_M^\ast(\cL)$. Define $\Sh(X_M^{\rm alg, et})^{\rm cris}_{\Q_p}$ to be the
category of $\Q_p$--adic sheaves on $X_M^{\rm alg}$ whose images via~$\mu_M^\ast$ lie in $\Sh(X_M^{\rm et})^{\rm cris}_{\Q_p}$. Given an object~$\cL$ in
$\Sh(X_M^{\rm alg, et})^{\rm cris}_{\Q_p}$, we abuse notation and  write $\bDcrisar(\cL)$ for the filtered locally free $\cO_{X_{\Mun}}$--module with integrable
connection on~$X_{\Mun}^{\rm alg}$ associated by rigid analytic GAGA to the isocrystal $\bDcrisar\bigl(\mu_M^\ast(\cL)\bigr)$. We also identify, for $i\ge 0$ the de
Rham cohomology groups ${\rm H}^i_{\rm dR}\left(X_{\Mun}^{\rm alg, Zar},\bDcrisar(\cL) \right)$ with the rigid cohomology of the $F$--isocrystal
$\bDcrisar\bigl(\mu_M^\ast(\cL)\bigr)$ to get a Frobenius structure. Then we have.

\begin{theorem}\label{thm:compiso}
There is an isomorphism of $\delta$--functors from~$\Sh(X_M^{\rm
alg, et})^{\rm cris}_{\Q_p}$ to the category of filtered $B_{\rm
cris}$--modules endowed with $G_M$--action and Frobenius:
$${\rm
H}^i\bigl(X_\Kbar^{\rm alg, et},\cL\bigr)\tensor_{\Z_p} B_{\rm
cris} \cong {\rm H}^i_{\rm cris}\left(X/\Mun,\bDcrisar(\cL)
\right)\tensor_{\Mun} B_{\rm cris}.$$
\end{theorem}

Here ${\rm H}^i_{\rm cris}\left(X/\Mun,\bDcrisar(\cL) \right)$ is the crystalline cohomology of the filtered $F$--isocrystal $\bDcrisar(\cL)$ in the sense of
\cite{berthelot}. It is endowed with  Frobenius. As it coincides with ${\rm H}^i_{\rm dR}\left(X_{\Mun}^{\rm rig},\bDcrisar(\cL) \right)\cong {\rm H}^i_{\rm
dR}\left(X_{\Mun}^{\rm alg},\bDcrisar(\cL) \right)$, it is endowed also with the Hodge filtration. The theorem implies the following corollary:

\begin{corollary}\label{cor:phigammamod}
Let $\cL$ be a crystalline \'etale sheaf on~$X_M^{\rm alg}$. Then
${\rm H}^i\bigl(X_\Kbar^{\rm alg, et},\cL\bigr)$ is a crystalline
representation of\/~$G_M$  and ${\rm H}^i_{\rm
dR}\left(X_{\Mun}^{\rm alg},\bDcrisar(\cL) \right)$ is, as
filtered $\Mun$--vector space endowed with Frobenius, the
classical $D_{\rm cris}\left({\rm H}^i\bigl(X_\Kbar^{\rm alg,
et},\cL\bigr)\right)$ associated to the $G_M$--representation
${\rm H}^i\bigl(X_\Kbar^{\rm alg, et},\cL\bigr)$.
\end{corollary}

We start with the following:

\begin{proposition}\label{prop:filterediso} The natural map $${\rm
H}^i_{\rm cris}\bigl(X/\Mun,
\bDcrisar(\bL)\bigr)\otimes_{\Mun}B_{\rm cris}\lra {\rm
H}^i\bigl(\fXKbar, \bL\otimes\bB_{\rm cris,\Kbar}^\nabla\bigr),
$$deduced from the isomorphism $\bDcrisar(\cL)\widehat{\tensor}_{\Mun} B_{\rm cris}\cong
\bDcrisgeo(\cL)$ of~\ref{prop:connFilDcriar} and from the
isomorphism in~\ref{thm:formalcomparison}, is an isomorphism of
$\delta$-functors from $\Sh(X_M^{\rm et})^{\rm cris}_{\Q_p}$ to
the category  of $B_{\rm cris}$--modules endowed with filtrations
and Galois action of~$G_M$.
\end{proposition}
\begin{proof} Recall that we have an isomorphism
$\bDcrisar(\cL)\widehat{\tensor}_{\Mun} B_{\rm cris}\cong
\bDcrisgeo(\cL)$  as filtered as $\cO_X\widehat{\tensor}_{\OMun}
B_{\rm cris}$--modules endowed with filtrations and Galois action
of~$G_M$. Due to~\ref{thm:formalcomparison} it suffices to show
that the natural map
$$\gamma^i_\bL\colon {\rm H}^i_{\rm cris}\bigl(X/\Mun, \bDcrisar(\bL)\bigr)\otimes_{\Mun}B_{\rm
cris}\lra {\rm H}^i_{\rm dR}\bigl(X^{\rm et}, \bDcrisgeo(\bL)\bigr)$$is an isomorphism of $B_{\rm cris}$--modules endowed with filtrations and Galois action
of~$G_M$. It is clear that $\gamma^i_\bL$ is $B_{\rm cris}$--linear and that it  is compatible with $G_M$--action. Let $\bD(\cL)$ be a Frobenius crystal on $X$
whose generic fiber is $\bDcrisar(\bL)$. Then ${\rm H}^i_{\rm cris}\bigl(X/\Mun, \bDcrisar(\bL)\bigr)\cong {\rm H}^i_{\rm
cris}\bigl(X_k/\OMun,\bD(\cL)\bigr)[p^{-1}]$ as $\Mun$--vector spaces with filtration and Frobenius. In particular using the definition of the filtration and of
Frobenius on  ${\rm H}^i_{\rm dR}\bigl(X^{\rm et}, \bDcrisgeo(\bL)\bigr)$ we deduce that $\gamma^i_\cL$ is compatible also with the filtrations and with Frobenius.

Since ${\rm H}^i_{\rm dR}\bigl(X, \bD(\cL)\otimes
\Omega^\bullet_{X/\cO_K}\bigr)\cong
\bH^i\left(X^{\rm et}, \bD(\cL)\otimes
\Omega^\bullet_{X/\cO_K}\right)$ is a finite $\cO_K$--module and
$A_{\rm cris}$ is $p$--torsion free the natural map
$$\rho_\bL^i\colon {\rm H}^i_{\rm dR}\bigl(X, \bD(\cL)\otimes
\Omega^\bullet_{X/\cO_K}\bigr) \otimes A_{\rm cris} \lra
\bH^i\left(X^{\rm et}, \bD(\cL)\otimes
\Omega^\bullet_{X/\cO_K} \widehat{\otimes} A_{\rm cris}\right)
$$is an isomorphism.  Inverting $t$ and using the fact that
$\bH^i\left(X^{\rm et},\_\right)$ commutes with direct limits since $X$ is noetherian gives the morphism $\gamma^i_\cL$ which is an isomorphism. We are left to
prove that $\rho_\bL^i$ is strict after inverting $p$ i.e., that it induces an isomorphism on the various steps of the filtrations.

We first treat the case that $i=2d$ where $d$ is  be the relative dimension of $X$ over $\cO_K$ and $\bL=\Z_p$ is the constant sheaf.  Then the natural map  $${\rm
H}^{d}\bigl(X_K,\Omega^d_{X_K/K}\bigr)\lra {\rm H}^{2d}_{\rm dR}\bigl(X, \Omega^\bullet_{X_K/K}\bigr)$$is an isomorphism with filtration $\Fil^n = $ everything for
$n\leq d$ and $\Fil^n=0$ for $n> d$. Via the trace map we have an identification ${\rm H}^{2d}\bigl(X,\Omega^d_{X_K/K}\bigr) \cong K(-d)$, where $K(-d)$ is $K$ as a
$K$-vector space with $\Fil^n K(-d)=K$ for $n\leq d$ and $0$ for $n>d$.

Since $ \rho_{\Z_p}^{2d}$ is an isomorphism it follows that the map
$${\rm H}^{d}\bigl(X,\Omega^d_{X/\cO_K}\bigr)
\otimes A_{\rm cris}[p^{-1}]\lra \bH^{2d}\left(X^{\rm et}, \Omega^\bullet_{X/\cO_K} \widehat{\otimes} A_{\rm cris}\right)[p^{-1}]$$is an isomorphism. The quotient
of $\Omega^d_{X/\cO_K} \widehat{\otimes} \Fil^{n-d} A_{\rm cris}[-d] \to \Omega^\bullet_{X/\cO_K} \widehat{\otimes} \Fil^{n-\bullet }A_{\rm cris}$ is a complex of
quasi-coherent sheaves with $\leq d-1$ terms so that it has trivial cohomology groups ${\rm H}^{i} $ for $i\geq 2d$. In particular $\Fil^n$ on $\bH^{d}\left(X^{\rm
et}, \Omega^\bullet_{X/\cO_K} \widehat{\otimes} A_{\rm cris}\right)$  is the image of $\bH^{d}\left(X^{\rm et}, \Omega^d_{X/\cO_K} \widehat{\otimes} \Fil^{n-d}
A_{\rm cris}\right)$. This coincides with $ {\rm H}^{d}\bigl(X,\Omega^d_{X,\cO_K}\bigr) \otimes \Fil^{n-d}A_{\rm cris}$ showing that $\rho^{2d}_{\Z_p}$ is an
isomorphism of  filtered $A_{\rm cris}$-modules.

In the general case recall that $\bD(\cL^\vee)\cong
\bD(\cL)^\vee$ as filtered isocrystals by
\ref{thm:crisistannakian} i.e., the pairing provides on
filtrations morphisms $\Fil^n \bD(\cL) \times \Fil^h \bD(\cL^\vee)
\to \Fil^{n+h} \cO_X$. This and the fact that $ \rho_{\Z_p}^{2d}$
is an isomorphism of filtered modules provides with pairings of
filtered modules

$$\begin{array}{ccc}
 {\rm H}^i_{\rm dR}\bigl(X, \bD(\cL)\otimes
\Omega^\bullet_{X/\cO_K}\bigr)  \otimes {\rm H}^{2d-i}_{\rm
dR}\bigl(X, \bD(\cL^\vee)\otimes \Omega^\bullet_{X/\cO_K}\bigr)
\otimes A_{\rm cris} & \lra & A_{\rm cris}(-d) \cr
\rho_{\bL}^i\otimes\rho_{\bL^\vee}^{2d-i}\Big\downarrow & &
\Vert\cr \bH^i\left(X^{\rm et}, \bD(\cL)\otimes
\Omega^\bullet_{X/\cO_K} \widehat{\otimes} A_{\rm cris}\right)
\otimes \bH^{2d-i}\left(X^{\rm et},
\bD(\cL^\vee)\otimes \Omega^\bullet_{X/\cO_K} \widehat{\otimes}
A_{\rm cris}\right) & \lra & A_{\rm cris}(-d) .\cr
\end{array}$$
By \cite[Prop. 2.5.3]{MSaito} Poincar\'e duality induces  an
isomorphism of filtered $K$-vector spaces
$$
\mu_\cL^i\colon {\rm H}^i_{\rm dR}\bigl(X, \bD(\cL)\otimes
\Omega^\bullet_{X/\cO_K}\bigr)\bigl[p^{-1}\bigr] \lra \Hom\left(
{\rm H}^{2d-i}_{\rm dR}\bigl(X, \bD(\cL)\otimes
\Omega^\bullet_{X/\cO_K}\bigr) , K(-d)\right).
$$Recall that $\Fil^n \Hom\left( {\rm H}^{2d-i}_{\rm dR}\bigl(X,
\bD(\cL)\otimes \Omega^\bullet_{X/\cO_K}\bigr) , K(-d)\right)$ consists of the  $f\colon {\rm H}^{2d-i}_{\rm dR}\bigl(X, \bD(\cL)\otimes
\Omega^\bullet_{X/\cO_K}\bigr)\to K(-d)$ such that $f(\Fil^h)\subset \Fil^{h+n}$ for every $h$. In particular the pairings displayed above give then morphisms of
filtered $A_{\rm cris}$-modules:

$${\rm H}^i_{\rm dR}\bigl(X, \bD(\cL)\otimes
\Omega^\bullet_{X/\cO_K}\bigr) \otimes A_{\rm cris}[p^{-1}] \lra
\bH^i\left(X^{\rm et}, \bD(\cL)\otimes
\Omega^\bullet_{X/\cO_K} \widehat{\otimes} A_{\rm
cris}\right)[p^{-1}]\lra$$
$$\lra \Hom_{A_{\rm cris}}\left(\bH^{2d-i}\left((X^{\rm et}, \bD(\cL^\vee)\otimes
\Omega^\bullet_{X/\cO_K} \widehat{\otimes} A_{\rm cris}\right),
A_{\rm cris}(-d)\right)[p^{-1}] \lra$$ $$\lra \Hom_{A_{\rm
cris}}\left({\rm H}^{2d-i}_{\rm dR}\bigl(X, \bD(\cL^\vee)\otimes
\Omega^\bullet_{X/\cO_K}\bigr) \otimes A_{\rm cris}, A_{\rm
cris}(-d)\right)[p^{-1}]$$such that the composite is
$\mu_\cL^i\otimes A_{\rm cris}$. In  particular it is an
isomorphism of filtered $A_{\rm cris}$-modules. This implies that
the first map, which is $\rho_\cL^i[p^{-1}]$, is an isomorphism as
filtered $A_{\rm cris}[p^{-1}]$-modules as claimed.

\end{proof}

Let us now assume that $\bL$ is a crystalline \'etale sheaf on
$X_M$. For $i\in \Z$ write $V_i:={\rm H}^i(X_\Kbar^{\rm et},
\bL)\otimes_{\Z_p} \Q_p$ and $D_i:={\rm H}^i_{\rm dR}(X_K,
\bDcrisar(\bL))$. Then $V_i$ is a finite dimensional $p$-adic
representation of $G_M$ for every $i\in \Z$ and $V_i=0$ unless
$0\le i\le 2d$, where $d$ is the dimension of $X_K$. Similarly,
$D_i$ is a finite dimensional filtered $\varphi$-module over
$\Mun$ for all $i\in \Z$ and $D_i=0$ unless $0\le i\le 2d$.

\begin{corollary}\label{cor:diagram} We have a canonical
commutative diagram with exact rows, referred to as the diagram
$(\ast_\bL)$, of topological $\Q_p$-vector spaces with continuous
$G_M$-action:
$$
\begin{array}{cccccccccccc}
\cdots\lra&V_i&\stackrel{\alpha_i}{\lra}&{\rm
Fil}^0\bigl(D_i\otimes_{\Mun} B_{\rm cris}\bigr)&
\stackrel{1-\varphi}{\lra}&D_i\otimes_{\Mun} B_{\rm cris}&\stackrel{\epsilon_i}{\lra}&V_{i+1}\\
&\downarrow\beta_i&&\downarrow\gamma_i&& \Vert &&\downarrow\beta_{i+1}\\
\cdots\lra& T_i &\stackrel{\omega_i}{\lra}& D_i\otimes_{\Mun}
B_{\rm cris}&\stackrel{1-\varphi}{\lra}&D_i\otimes_{\Mun} B_{\rm
cris}&\lra& \bigl(D_{i+1}\otimes_{\Mun} B_{\rm
cris}\bigr)^{\varphi=1}
\end{array}
$$
\end{corollary}
\begin{proof}
Consider the diagram attached to $\bL$ tensoring
(\ref{display:fundamentalexactdiagram}) with $\cL$ and taking the
long exact sequence in cohomology. We get a commutative diagram of
$\Q_p$-modules endowed with continuous action of $G_M$ whose rows
are exact:
$$
\begin{array}{cccccccccc}
\cdots\lra  {\rm H}^i(\fX_\Kbar^\bullet, \bL)\otimes_{\Z_p} \Q_p
\lra&{\rm H^i}(\fX_\Kbar^\bullet, \bL\otimes {\rm Fil}^0 \bB_{\rm
cris,\Kbar}^\nabla)\bigr)&\stackrel{1-\varphi}{\lra} & {\rm
H}^i\bigl(\fX_\Kbar^\bullet,
\bL\otimes\bB_{\rm cris}^\nabla\bigr) \lra  \cdots\\
\big\downarrow &\big\downarrow &&\big\downarrow \\
\cdots\lra {\rm H}^i\left(\fX_\Kbar^\bullet,
\bigl(\bL\otimes\bB_{\rm cris}^\nabla\bigr)^{\varphi=1}\right)
\lra& {\rm H}^i\bigl(\fX_\Kbar^\bullet, \bL\otimes\bB_{\rm
cris}^\nabla\bigr)& \stackrel{1-\varphi}{\lra} &  {\rm
H}^i\bigl(\fX_\Kbar^\bullet, \bL\otimes\bB_{\rm cris}^\nabla\bigr)
\lra \cdots
\end{array}
$$For every $i\in \Z$ we have canonical isomorphisms ${\rm H}^i(\fX_\Kbar,
\bL)[1/p]\cong V_i$ as $G_M$-modules by \cite[Prop. 4.9]{andreatta_iovita}. By theorem \ref{thm:formalcomparison} and proposition \ref{prop:filterediso} we have
canonical isomorphisms as filtered $\varphi$-modules, compatible with the $G_M$-action, $ {\rm H}^i(\fX_\Kbar^\bullet, \bL\otimes \bB_{\rm cris}^\nabla)\cong
D_i\otimes_{\Mun} B_{\rm cris}$. Put $T_i:={\rm H}^i\left(\fX_\Kbar^\bullet, \bigl(\bL\otimes\bB_{\rm cris}^\nabla\bigr)^{\varphi=1}\right)$. Furthermore the image
of the map $g_i$ from ${\rm H^i}(\fX_\Kbar^\bullet, \bL\otimes {\rm Fil}^0 \bB_{\rm cris,\Kbar}^\nabla)\bigr)$ to ${\rm H^i}(\fX_\Kbar^\bullet, \bL\otimes \bB_{\rm
cris,\Kbar}^\nabla)\bigr)$ is $\Fil^0\left({\rm H^i}(\fX_\Kbar^\bullet, \bL\otimes {\rm Fil}^0 \bB_{\rm cris,\Kbar}^\nabla)\bigr)\right)$. To prove the claim it
suffices to show that $g_i$ is injective. It follows from the above diagram of long exact sequences that the kernel of $g_i$ is in the image of $V_i$. In particular
it is  a finite dimensional $\Q_p$-vector space. Since $\Ker(g_i)$ is a $B_{\rm cris}^+$-module and $B_{\rm cris}^+$ is an algebra over the maximal unramified
extension $K^{\rm un}$ of $K$, then $\Ker(g_i)$ is a $K^{\rm un}$-vector space. Since $\Ker(g_i)$ is a finite dimensional $\Q_p$-vector space we conclude that it
must be $0$ which proves the claim.
\end{proof}

As in corollary \ref{cor:diagram} take $\bL$ to be a crystalline sheaf. Then its $\Z_p$-dual $\bL^\vee$ is also a crystalline sheaf on $X_K$ and $\bD_{\rm
cris}^{\rm ar}(\bL^\vee)\cong \bD_{\rm cris}^{\rm ar}(\bL)^\vee:=\Hom_{\rm Isoc(X)}(\bD_{\rm cris}^{\rm ar}(\bL), \cO_{X_K})$ where the isomorphism is as filtered
$F$--isocrystals on $X$ (or $X_\F$). Let us denote for every $i\in \Z$, $V_i^\ast:={\rm H}^i(X_K^{\rm et}, \bL^\vee)[1/p]$ and by $D_i^\ast:={\rm H}^i_{\rm dR}(X_K,
\bD_{\rm cris}^{\rm ar}(\bL^\vee))$ the $G_M$-representations, respectively the filtered $\varphi$-modules attached to $\bL^\vee$. By Poincar\'e duality for \'etale
cohomology and de Rham cohomology respectively we have canonical isomorphisms as $G_M$-representations (respectively as filtered $\varphi$-modules) $V_i^\ast\cong
V_{2d-i}^\vee$ (respectively $D_i^\ast\cong D_{2d-i}^\vee$.) Let us remark that we have a canonical diagram with exact rows attached to $\bL^\vee$, denoted
$(\ast_{\bL^\vee})$, which involves $V_i^\ast,D_i^\ast$ and in which the maps are denoted $\alpha_i^\ast,\beta_i^\ast,\gamma_i^\ast,$....

\bigskip
\noindent Suppose $\bL$ is a crystalline sheaf on $X_M^{\rm et}$
and assume all the notations above. The idea of the proof of
Theorem \ref{thm:compiso} is very simple. We prove by induction on
$j\ge 0$ that $D_{2d-j}$ is an admissible filtered
$\varphi$-module and that $V_{\rm cris}(D_{2d-j})=V_{2d-j}$.
Granting this it follows that $V_{2d-j}$ is a crystalline
representation of $G_M$ and that $D_{\rm cris}(V_{2d-j})\cong
D_{2d-j}$ and so we are done.

Let us first recall a criterion of admissibility from \cite{colmez_fontaine}. In our setting $\Mun$ is a finite unramified extension of $\Q_p$ and  $D$ is a finite
dimensional filtered $\varphi$-module over $\Mun$. Let
$$
\delta(D)\colon (D\otimes_{\Mun} B_{\rm cris})^{\varphi=1}\lra
\frac{D\otimes_{\Mun} B_{\rm cris}}{{\rm
Fil}^0\bigl(D\otimes_{\Mun} B_{\rm cris}\bigr)}
$$
be the natural map. Put $V_{\rm cris}(D):=\Ker(\delta_D)$

\begin{proposition}[\cite{colmez_fontaine}]
\label{prop:admis} The filtered $\varphi$-module $D$ over $\Mun$
is admissible  if and only if \enspace {\rm (a)} $V_{\rm cris}(D)$
is a finite dimensional $\Q_p$-vector space and \enspace {\rm (b)}
$\delta(D)$ is surjective.

\smallskip Moreover, if $V=V_{\rm cris}(D)$ is finite dimensional
then it is a crystalline representation of $G_M$ and $D_{\rm cris}(V)\subseteq D$. This inclusion is an equality if and only if $D$ is admissible.
\end{proposition}

\begin{proof} Let us first remark that we are in
the situation of \cite{colmez_fontaine}, i.e. the natural map
$$
\frac{D\otimes_{\Mun} B_{\rm cris}}{{\rm
Fil}^0\bigl(D\otimes_{\Mun} B_{\rm cris}\bigr)}\lra
\frac{D\otimes_{\Mun} B_{\rm dR}}{{\rm Fil}^0\bigl(D\otimes_{\Mun}
B_{\rm dR}\bigr)}
$$
is an isomorphism for every finite dimensional filtered module $D$
over $\Mun$ due to the fact that the natural inclusion $B_{\rm
cris} \lra B_{\rm dR}$ of filtered rings induces isomorphisms on
the graded quotients.

It is proven in \cite[Prop. 4.5]{colmez_fontaine} that the $\Q_p$-vector space $V_{\rm cris}(D)$ is finite dimensional if and only if for every sub-object
$D'\subset D$ we have $t_H(D')\leq t_N(D')$. Moreover, it also shown in loc.~cit.~that in this case $V_{\rm cris}(D)$ is a crystalline representation of $G_M$ whose
associated filtered $\varphi$-module is contained in $D$. It coincides with $D$ if and only if $\dim_{\Q_p} V_{\rm cris}(D)=\dim_{\Mun} D$.

It follows from the proof of of \cite[Prop. 5.7]{colmez_fontaine}
that, if $V_{\rm cris}(D)$ is finite dimensional, then
$\dim_{\Q_p} V_{\rm cris}(D)=\dim_{\Mun} D$ if and only if
$\delta(D)$ is surjective. The claim follows.

\end{proof}

\bigskip
\noindent {\it The proof of Theorem \ref{thm:compiso}}. Let $\bL$
be a crystalline sheaf on $X_K$ and let us consider the part of
the diagram $(\ast_\bL)$ relevant for $j=0$:
$$
\begin{array}{ccccccccccc}
\cdots\stackrel{\epsilon_{2d-1}}{\lra}&V_{2d}&\stackrel{\alpha_{2d}}{\lra}&{\rm
Fil}^0(D_{2d}\otimes_{\Mun} B_{\rm
cris})&\stackrel{1-\varphi}{\lra}&
D_{2d}\otimes_{\Mun} B_{\rm cris}&\lra&0\\
&\downarrow\beta_{2d}&&\downarrow\gamma_{2d}&&||\\
0 \lra&\bigl(D_{2d}\otimes_{\Mun} B_{\rm
cris}\bigr)^{\varphi=1}&\stackrel{\omega_{2d}}{\lra}&
D_{2d}\otimes_{\Mun} B_{\rm
cris}&\stackrel{1-\varphi}{\lra}&D_{2d}\otimes_{\Mun} B_{\rm
cris}&\lra&0
\end{array}
$$
Let us remark that $\delta(D_{2d})$ can be seen as the composition 
$$
\bigl(D_{2d}\otimes_{\Mun} B_{\rm
cris}\bigr)^{\varphi=1}\stackrel{\omega_{2d}}{\lra}D_{2d}\otimes_{\Mun}
B_{\rm cris} \lra \Coker(\gamma_{2d})
$$
and also that $\Ker(\delta(D_{2d}))=\Ker\bigl((1-\varphi)\colon
{\rm Fil}^0(D_{2d}\otimes_{\Mun} B_{\rm cris})\lra
D_{2d}\otimes_{\Mun} B_{\rm cris}\bigr)$. It follows that
$\alpha_{2d}$ induces a surjective $\Q_p$-linear map $V_{2d}\lra
\Ker(\delta(D_{2d}))$ and that $\delta(D_{2d})$ is surjective. We
deduce from proposition \ref{prop:admis} that $D_{2d}$ is
admissible and that we have a $\Q_p$-linear, surjective
homomorphism $V_{2d}\lra V_{\rm cris}(D_{2d})$ which is
$G_M$-equivariant.

Now we look at the part near $i=0$ of the diagram
$(\ast_{\bL^\vee})$, remarking that $D_0^\ast\cong D_{2d}^\vee$ by
\cite[Prop. 2.5.3]{MSaito} and therefore it is admissible.
$$
\begin{array}{ccccccccccc}
0&\lra&V_{0}^\ast&\stackrel{\alpha_{0}^\ast}{\lra}&{\rm
Fil}^0(D_0^\ast\otimes_{\Mun} B_{\rm
cris})&\stackrel{1-\varphi}{\lra}&
D_0^\ast\otimes_{\Mun} B_{\rm cris}&\stackrel{\epsilon_0^\ast}{\lra}\cdots\\
&&\downarrow\beta_0^\ast&&\downarrow\gamma_0^\ast&&||\\
0&\lra&\bigl(D_0^\ast\otimes_{\Mun} B_{\rm
cris}\bigr)^{\varphi=1}&\stackrel{\omega_0^\ast}{\lra}&
D_0^\ast\otimes_{\Mun} B_{\rm
cris}&\stackrel{1-\varphi}{\lra}&D_0^\ast\otimes_{\Mun} B_{\rm
cris}&\lra\cdots
\end{array}
$$
It follows that $V_0^\ast\cong \Ker(\delta(D_0^\ast))=V_{\rm cris}(D_0^\ast)$. Therefore we deduce that ${\rm dim}_{\Q_p}(V_{2d})={\rm dim}_{\Q_p}(V_0^\ast)={\rm
dim}_{K}(D_0^\ast)={\rm dim}_{K}(D_{2d}) ={\rm dim}_{\Q_p}(V_{\rm cris}(D_{2d}))$ and hence $V_{2d}\cong V_{\rm cris}(D_{2d})$ and $V_0^\ast \cong V_{\rm
cris}(D_{0}^\ast)$. This proves our statement for $j=0$ for $\bL$ and $j=2d$ for $\bL^\vee$.

Let us remark at the same time that as $\alpha_{2d}$ is injective,
$\epsilon_{2d-1}=0$. Since $V_0^\ast \cong V_{\rm
cris}(D_{0}^\ast)  $ an easy diagram chase shows that
$\epsilon^\ast_0=0$ and therefore the map $\alpha_1^\ast$ is
injective. Since $\gamma_1^\ast$ is injective, we can continue
with $j=1$ along exactly the same lines as for $j=0$. By induction
Theorem \ref{thm:compiso} follows.

\leftline{Fabrizio Andreatta} \leftline{DIPARTIMENTO DI MATEMATICA ``FEDERIGO ENRIQUES",}  \leftline{UNIVERSIT\`A STATALE DI MILANO, VIA C. SALDINI 50, 20133,
ITALIA}

\bigskip

\leftline{Adrian Iovita}
\leftline{DIPARTIMENTO DI MATEMATICA PURA ED APPLICATA, UNIVERSIT\`A DEGLI}
\leftline{STUDI DI PADOVA, VIA TRIESTE 63, PADOVA 35121, ITALIA}

and

\leftline{DEPARTMENT OF MATHEMATICS AND STATISTICS, CONCORDIA UNIVERSITY}
\leftline{1455 DE MAISONNEUVE BLVD., MONTREAL, H3G 1MB CANADA}

\end{document}